\documentclass{amsart}%
\usepackage{amssymb}
\usepackage{amsfonts}
\usepackage{geometry}%
\usepackage{amsmath}%
\usepackage{graphicx}
\providecommand{\U}[1]{\protect\rule{.1in}{.1in}}
\newtheorem{theorem}{Theorem}
\theoremstyle{plain}
\newtheorem{acknowledgement}[theorem]{Acknowledgement}

\newtheorem{conclusion}[theorem]{Conclusion}

\newtheorem{conjecture}[theorem]{Conjecture}
\newtheorem{corollary}[theorem]{Corollary}

\newtheorem{definition}[theorem]{Definition}

\newtheorem{lemma}[theorem]{Lemma}
\newtheorem{notation}[theorem]{Notation}

\newtheorem{proposition}[theorem]{Proposition}
\newtheorem{remark}[theorem]{Remark}

\numberwithin{equation}{section}
\begin{document}
\title[Probabilistic Fourier extension]{A probabilistic analogue of the Fourier extension conjecture}
\author[E. T. Sawyer]{Eric T. Sawyer$^{\dagger}$}
\address{Eric T. Sawyer, Department of Mathematics and Statistics\\
McMaster University\\
1280 Main Street West\\
Hamilton, Ontario L8S 4K1 Canada}
\email{sawyer@mcmaster.ca}
\thanks{$\dagger$ Research supported in part by a grant from the National Science and
Engineering Research Council of Canada.}
\date{\today }

\begin{abstract}
The Fourier extension conjecture in $n$ dimensions is equivalent to%
\[
\left\Vert T\mathbf{1}_{U_{0}}f\right\Vert _{L^{p}\left(  \lambda_{n}\right)
}\leq C\left\Vert f\right\Vert _{L^{p}\left(  B_{n-1}\left(  0,\frac{1}%
{2}\right)  \right)  }\ ,\ \ \ \ \ p>\frac{2n}{n-1},
\]
where $Tf\left(  \xi\right)  \equiv\int_{B_{n-1}\left(  0,\frac{1}{2}\right)
}e^{-i\Phi\left(  x\right)  \cdot\xi}f\left(  x\right)  dx$, $U_{0}\subset
B_{n-1}\left(  0,\frac{1}{2}\right)  \subset\mathbb{R}^{n-1}$, $\Phi\left(
x\right)  =\left(  x,\sqrt{1-\left\vert x\right\vert ^{2}}\right)  $ and
$\lambda_{n}$ is Lebesgue measure on $\mathbb{R}^{n}$. Noting that
$f=\sum_{I\in\mathcal{G}}\bigtriangleup_{I;\kappa}^{\eta}f$, we prove that the
following probabilistic analogue of the Fourier extension conjecture for
$p=q$,%
\[
\mathbb{E}_{2^{\mathcal{G}}}\left\Vert T\mathbf{1}_{U_{0}}\sum_{I\in
\mathcal{G}}\pm\bigtriangleup_{I;\kappa}^{\eta}f\right\Vert _{L^{p}\left(
\lambda_{n}\right)  }\leq C\left\Vert f\right\Vert _{L^{p}\left(
B_{n-1}\left(  0,\frac{1}{2}\right)  \right)  }\ ,
\]
holds for all $f\in L^{p}\left(  B_{n-1}\left(  0,\frac{1}{2}\right)  \right)
$ if and only if $p>\frac{2n}{n-1}$. The operator $\mathbb{E}_{2^{\mathcal{G}%
}} $ averages over all sequences of $\pm1$, where $\mathcal{G}$ is a grid of
dyadic subcubes containing $U_{0}$, and where $\bigtriangleup_{I;\kappa}%
^{\eta}$ is a smooth Alpert pseudoprojection, resulting in a `martingale
tranform' analogue.

By Khintchine's inequalities, the probabilistic analogue of the Fourier
extension conjecture is equivalent to the Fourier square function estimate,%
\[
\left\Vert \mathcal{S}_{T\mathbf{1}_{U_{0}}}f\right\Vert _{L^{p}\left(
\lambda_{n}\right)  }\lesssim\left\Vert f\right\Vert _{L^{p}\left(
B_{n-1}\left(  0,\frac{1}{2}\right)  \right)  },\ \ \ \ \ \text{if and only if
}\frac{2n}{n-1}<p\leq\infty,
\]
where%
\[
\mathcal{S}_{T\mathbf{1}_{U_{0}}}f\equiv\left(  \sum_{I\in\mathcal{G}%
}\left\vert T\mathbf{1}_{U_{0}}\bigtriangleup_{I;\kappa}^{n-1,\eta
}f\right\vert ^{2}\right)  ^{\frac{1}{2}}.
\]

To prove this probabilistic analogue of the extension conjecture, we use
frames for $L^{p}$ consisting of smooth compactly supported Alpert wavelets
having a large number $\kappa>\frac{n}{2}$ of vanishing moments, along with
stationary phase and interpolation of $L^{2}$ and probabilistic $L^{4}$
estimates, thus circumventing the most challenging issues arising in the
Fourier extension conjecture.

\end{abstract}
\maketitle
\tableofcontents

\section{Introduction}

In this paper we consider a probabilistic analogue of the Fourier extension
conjecture (Theorem \ref{FEC}),

\begin{conjecture}
\label{Fourier 1}Let $1<p,q<\infty$, $\sigma_{n-1}$ be surface measure on the
sphere $\mathbb{S}^{n-1}$, and $\mathcal{F}\left(  \mu\right)  \equiv
\int_{\mathbb{R}^{n}}e^{-ix\cdot\xi}d\mu\left(  x\right)  $ denote the Fourier
transform of the measure $\mu$. Then%
\begin{equation}
\left(  \int_{\mathbb{R}^{n}}\left\vert \mathcal{F}\left(  f\sigma
_{n-1}\right)  \left(  \xi\right)  \right\vert ^{q}d\xi\right)  ^{\frac{1}{q}%
}\leq C\left(  \int_{\mathbb{S}^{n-1}}\left\vert f\left(  x\right)
\right\vert ^{p}d\sigma_{n-1}\left(  x\right)  \right)  ^{\frac{1}{p}%
},\ \ \ \ \ f\in L^{p}\left(  \sigma_{n-1}\right)  ,\label{extension}%
\end{equation}
if and only if $q>\frac{2n}{n-1}$ and $\frac{1}{p}+\frac{n+1}{n-1}\frac{1}%
{q}=1$.
\end{conjecture}

\subsection{The probabilistic extension problem}

Let $\Phi\left(  x\right)  \equiv\left(  x,\sqrt{1-\left\vert x\right\vert
^{2}}\right)  \in\mathbb{S}^{n-1}$ be the standard parametization of the
northern hemisphere of $\mathbb{S}^{n-1}$. Let $B_{n-1}\left(  0,\frac{1}%
{2}\right)  $ be the ball of radius $\frac{1}{2}$ centered at the origin in
$\mathbb{R}^{n-1}$, and define%
\begin{equation}
Tf\left(  \xi\right)  \equiv\int_{B_{n-1}\left(  0,\frac{1}{2}\right)
}e^{-i\Phi\left(  x\right)  \cdot\xi}f\left(  x\right)  \frac{dx}{\left\vert
\det\nabla\Phi\left(  x\right)  \right\vert },\ \ \ \ \ \xi\in\mathbb{R}%
^{n},\label{def T}%
\end{equation}
for $f\in L^{p}\left(  B_{n-1}\left(  0,\frac{1}{2}\right)  \right)  $. Thus
$Tf=\mathcal{F}\Phi_{\ast}\left(  f\lambda_{n-1}\right)  =\widehat{\Phi_{\ast
}\left(  f\lambda_{n-1}\right)  }$, where $\Phi_{\ast}\nu$ denotes the
pushforward of a measure $\nu$ under the map $\Phi$. Then the Fourier
extension inequality (\ref{extension}) is equivalent to boundedness of the
operator $T\mathbf{1}_{U_{0}}$, i.e.
\begin{equation}
\left\Vert T\mathbf{1}_{U_{0}}f\right\Vert _{L^{q}\left(  \lambda_{n}\right)
}\leq C\left\Vert f\right\Vert _{L^{p}\left(  B\left(  0,\frac{1}{2}\right)
\right)  },\label{T_S equiv}%
\end{equation}
for a fixed sufficiently small subcube $U_{0}$ of $B_{n-1}\left(  0,\frac
{1}{2}\right)  $ (after considering finitely many rotations). The Jacobian
$\frac{1}{\left\vert \det\nabla\Phi\left(  x\right)  \right\vert }$ is roughly
$1$ on $B\left(  0,\frac{1}{2}\right)  $ and can be absorbed into the function
$f\left(  x\right)  $ - we will often abuse notation by simply ignoring it.

Now let $\left\{  \bigtriangleup_{I;\kappa}^{n-1,\eta}\right\}  _{I\in
\mathcal{G}}$ be the family of smooth Alpert pseudoprojections
\[
\bigtriangleup_{I;\kappa}^{n-1,\eta}=\sum_{a\in\Gamma_{n-1}}\left\langle
\left(  S_{\kappa,\eta}\right)  ^{-1}f,h_{I;\kappa}^{a}\right\rangle
h_{I;\kappa}^{a,\eta}%
\]
on $L^{2}\left(  \mathbb{R}^{n-1}\right)  $ as given in Theorem \ref{frame}
below, where $\mathcal{G}$ is a dyadic grid containing $U_{0}$. Then we can
rewrite (\ref{T_S equiv}) as,%
\begin{equation}
\left\Vert T\mathbf{1}_{U_{0}}\sum_{I\in\mathcal{G}}\bigtriangleup_{I;\kappa
}^{n-1,\eta}f\right\Vert _{L^{q}\left(  \lambda_{n}\right)  }\leq C\left\Vert
f\right\Vert _{L^{p}\left(  B\left(  0,\frac{1}{2}\right)  \right)
}\ .\label{T_S expan}%
\end{equation}

The \emph{probabilistic} Fourier extension problem is then to decide when the
following `martingale transform'\ analogue of (\ref{T_S expan}) holds,%
\begin{equation}
\mathbb{E}_{2^{\mathcal{G}}}^{\mu}\left\Vert T\mathbf{1}_{U_{0}}\sum
_{I\in\mathcal{G}}\pm\bigtriangleup_{I;\kappa}^{n-1,\eta}f\right\Vert
_{L^{q}\left(  \lambda_{n}\right)  }\leq C\left\Vert f\right\Vert
_{L^{p}\left(  B\left(  0,\frac{1}{2}\right)  \right)  },\label{prob ext}%
\end{equation}
where the expectation $\mathbb{E}_{2^{\mathcal{G}}}^{\mu}$ is taken over all
choices of $\pm$ for each $I\in\mathcal{G}$. We point out that it is not hard
to see that the probabilistic analogue (\ref{prob ext}) fails for the same
pairs $\left(  p,q\right)  $ that (\ref{extension}) is currently known to fail
for - see the discussion below.

By Khinchine's inequalities, (\ref{prob ext}) is equivalent to the Fourier
square function estimate%
\begin{equation}
\left\Vert \mathcal{S}_{T\mathbf{1}_{U_{0}}}f\right\Vert _{L^{q}\left(
\lambda_{n}\right)  }\lesssim\left\Vert f\right\Vert _{L^{p}\left(  B\left(
0,\frac{1}{2}\right)  \right)  }\ ,\label{prob ext square}%
\end{equation}
where $\mathcal{S}_{T\mathbf{1}_{U_{0}}}$ is the Fourier square function
defined by%
\begin{equation}
\mathcal{S}_{T\mathbf{1}_{U_{0}}}f\equiv\left(  \sum_{I\in\mathcal{G}\left[
U\right]  }\left\vert T\mathbf{1}_{U_{0}}\bigtriangleup_{I;\kappa}^{n-1,\eta
}f\right\vert ^{2}\right)  ^{\frac{1}{2}}.\label{def T square}%
\end{equation}

\subsubsection{A precise description of the martingale transform}

We begin with a more precise description of the `martingale
transform'\ inequality (\ref{prob ext}), and then establish a reduction to
certain Alpert projections. Let $\mathcal{G}$ be a grid in $\mathbb{R}^{n-1}$,
and let $\left\{  \bigtriangleup_{I;\kappa}^{n-1}\right\}  _{I\in\mathcal{G}}$
be the orthogonal family of Alpert projections $\bigtriangleup_{I;\kappa
}^{n-1}=\sum_{a\in\Gamma_{n-1}}\left\langle f,h_{I;\kappa}^{n-1,a}%
\right\rangle h_{I;\kappa}^{n-1,a}$ on $L^{2}\left(  \mathbb{R}^{n-1}\right)
$ as in Theorem \ref{frame}, and let $\left\{  \bigtriangleup_{I;\kappa
}^{n-1,\eta}\right\}  _{I\in\mathcal{G}}$ be the frame of \emph{smooth} Alpert
pseudoprojections on $L^{p}\left(  \mathbb{R}^{n-1}\right)  $. For
$\mathbf{a}=\left\{  a_{I}\right\}  _{I\in\mathcal{G}}\in\left\{
1,-1\right\}  ^{\mathcal{G}}$ and $f\in L^{p}\left(  \mathbb{R}^{n-1}\right)
$, define the \emph{Alpert martingale transform} $\mathcal{A}_{\mathbf{a}}$
by
\[
\mathcal{A}_{\mathbf{a}}f\equiv\sum_{I\in\mathcal{G}}a_{I}\bigtriangleup
_{I;\kappa}^{n-1}f,
\]
which is $\sum_{I\in\mathcal{G}}\pm\bigtriangleup_{I;\kappa}f$ for a choice of
$\pm$ determined by $\mathbf{a}$.

Given linear operators $L$ and $S$ with $S$ invertible, define the conjugation
of $L$ by $S$ as%
\[
L^{S}\equiv SLS^{-1}.
\]
Let $S_{\kappa,\eta}$ be the bounded invertible linear map on $L^{p}$ given in
Theorem \ref{frame}, that takes Alpert wavelets $h_{I;\kappa}^{n-1,a}$ to
their smooth counterparts $h_{I;\kappa}^{n-1,a,\eta}=h_{I;\kappa}^{n-1,a}%
\ast\phi_{\eta\ell\left(  I\right)  }$. For $\mathbf{a}=\left\{
a_{I}\right\}  _{I\in\mathcal{G}}\in\left\{  1,-1\right\}  ^{\mathcal{G}}$ and
$f\in L^{p}\left(  \mathbb{R}^{n-1}\right)  $, define the
\emph{\textbf{smooth} Alpert martingale transform}
\[
\mathcal{A}_{\mathbf{a}}^{S_{\kappa,\eta}}f\equiv\sum_{I\in\mathcal{G}}%
a_{I}\bigtriangleup_{I;\kappa}^{n-1,\eta}f=\sum_{I\in\mathcal{G}}%
\pm\bigtriangleup_{I;\kappa}^{n-1,\eta}f
\]
by conjugating $\mathcal{A}_{\mathbf{a}}$ with the bounded invertible map
$S_{\kappa,\eta}$, i.e.
\[
\mathcal{A}_{\mathbf{a}}^{S_{\kappa,\eta}}f\equiv S_{\kappa,\eta}%
\mathcal{A}_{\mathbf{a}}S_{\kappa,\eta}^{-1}f=S_{\kappa,\eta}\sum
_{I\in\mathcal{G}}a_{I}\left\langle S_{\kappa,\eta}^{-1}f,h_{I;\kappa}%
^{n-1}\right\rangle h_{I;\kappa}^{n-1}=\sum_{I\in\mathcal{G}}a_{I}\left\langle
S_{\kappa,\eta}^{-1}f,h_{I;\kappa}^{n-1}\right\rangle h_{I;\kappa}^{n-1,\eta
}=\sum_{I\in\mathcal{G}}a_{I}\bigtriangleup_{I;\kappa}^{n-1,\eta}f.
\]
Note that both $\mathcal{A}_{\mathbf{a}}$ and $\mathcal{A}_{\mathbf{a}%
}^{S_{\kappa,\eta}}$ are involutions, $\mathcal{A}_{\mathbf{a}}^{2}=\left(
\mathcal{A}_{\mathbf{a}}^{S_{\kappa,\eta}}\right)  ^{2}=\operatorname*{Id}$.

Since we will be using the notation $L^{S_{\kappa,\eta}}$ for various
operators $L=\mathcal{A}_{\mathbf{a}},\mathcal{A}_{\mathbf{a}}\mathsf{P}%
_{S},\mathcal{A}_{\mathbf{a}}\mathsf{Q}_{K}^{s}$ etc., we declutter the
exponent by writing%
\[
L^{\spadesuit}\equiv L^{S_{\kappa,\eta}},
\]
when the bounded invertible linear operator is $S_{\kappa,\eta}$.

Then we identify $2^{\mathcal{G}}$ and $\left\{  1,-1\right\}  ^{\mathcal{G}}$
and equip $2^{\mathcal{G}}$ with the probability measure $\mu$ that satisfies,%
\[
\mu_{\Lambda}\left(  E\right)  \equiv\mu\left(  \left\{  E\mid E\subset
2^{\Lambda}\right\}  \right)  =\frac{\left\vert E\right\vert }{\left\vert
2^{\Lambda}\right\vert },\ \ \ \ \ E\subset2^{\Lambda}\text{ with }%
\Lambda\subset\mathcal{G}\text{ finite},
\]
where $\left\vert F\right\vert $ denotes cardinality of a finite subset of
$\mathcal{G}$, and $\mu\left(  \left\{  E\mid E\subset2^{\Lambda}\right\}
\right)  $ is the conditional probability of $E$ given that $E\subset
2^{\Lambda}$ (here $2^{\Lambda}$ is a set of $\mu$-measure zero, and see e.g.
\cite{Hyt} for a construction of such a measure $\mu$). We define the
expectation operator $\mathbb{E}_{2^{\mathcal{G}}}^{\mu}$ by
\[
\mathbb{E}_{2^{\mathcal{G}}}^{\mu}F\equiv\int_{2^{\mathcal{G}}}F\left(
\mathbf{a}\right)  d\mu\left(  \mathbf{a}\right)
\]
for $F$ a nonnegative function on $2^{\mathcal{G}}=\left\{  1,-1\right\}
^{\mathcal{G}}$, so that (\ref{prob ext}) becomes,%
\begin{equation}
\mathbb{E}_{2^{\mathcal{G}}}^{\mu}\left\Vert T\mathbf{1}_{U_{0}}\left(
\mathcal{A}_{\mathbf{a}}\right)  ^{\spadesuit}f\right\Vert _{L^{q}\left(
\lambda_{n}\right)  }=\mathbb{E}_{2^{\mathcal{G}}}^{\mu}\int_{2^{\mathcal{G}}%
}\left\Vert T\mathbf{1}_{U_{0}}\left(  \mathcal{A}_{\mathbf{a}}\right)
^{\spadesuit}f\right\Vert _{L^{q}\left(  \lambda_{n}\right)  }d\mu\left(
\mathbf{a}\right)  \leq C\left\Vert f\right\Vert _{L^{p}\left(  B\left(
0,\frac{1}{2}\right)  \right)  }.\label{prob ext'}%
\end{equation}

\subsubsection{A reduction of the martingale transform inequality}

We now replace $\mathbf{1}_{U_{0}}\left(  \mathcal{A}_{\mathbf{a}}\right)
^{\spadesuit}f=\mathbf{1}_{U_{0}}S_{\kappa,\eta}\mathcal{A}_{\mathbf{a}%
}S_{\kappa,\eta}^{-1}$ in (\ref{prob ext'}) with
\[
\left(  \mathcal{A}_{\mathbf{a}}\mathsf{P}_{U}\right)  ^{\spadesuit}f\equiv
S_{\kappa,\eta}\mathcal{A}_{\mathbf{a}}\mathsf{P}_{U}S_{\kappa,\eta}%
^{-1}=S_{\kappa,\eta}\mathcal{A}_{\mathbf{a}}\sum_{I\in\mathcal{G}\left[
U\right]  }\bigtriangleup_{I;\kappa}S_{\kappa,\eta}^{-1}f=\sum_{I\in
\mathcal{G}\left[  U\right]  }a_{I}\bigtriangleup_{I;\kappa}^{\eta}f,
\]
where $\mathsf{P}_{U}g\equiv\sum_{I\in\mathcal{G}\left[  U\right]
}\bigtriangleup_{I;\kappa}g$ is the Alpert projection of a function $g$ in
which the sum over cubes $I$ is restricted to those contained in $U$, and
where $U\subset B_{n-1}\left(  0,\frac{1}{8}\right)  $. We claim that this new
inequality is sufficient for (\ref{prob ext'}) in the case
\[
U=\pi_{\mathcal{G}}^{\left(  2\right)  }U_{0}%
\]
is the $\mathcal{G}$-grandparent of $U_{0}$, where we assume $3U_{0}\subset U
$, i.e. $U_{0}$ is an \emph{interior} grandchild of $U$.

More precisely, we will show in a moment that (\ref{prob ext'}) is implied by
the following \emph{truncated} inequality,
\begin{equation}
\mathbb{E}_{2^{\mathcal{G}}}^{\mu}\left\Vert T\left(  \mathcal{A}_{\mathbf{a}%
}\mathsf{P}_{U}\right)  ^{\spadesuit}f\right\Vert _{L^{q}\left(  \lambda
_{n}\right)  }\leq C\left\Vert f\right\Vert _{L^{p}\left(  B\left(  0,\frac
{1}{2}\right)  \right)  },\label{prob ext''}%
\end{equation}
in which we have replaced $\mathbf{1}_{U_{0}}\left(  \mathcal{A}_{\mathbf{a}%
}\right)  ^{\spadesuit}f$ by the truncation $\left(  \mathcal{A}_{\mathbf{a}%
}\mathsf{P}_{U}\right)  ^{\spadesuit}f=\sum_{I\in\mathcal{G}\left[  U\right]
}a_{I}\bigtriangleup_{I;\kappa}^{\eta}f$. This latter inequality is what we
will prove in the remainder of this paper.

\begin{lemma}
\label{prob trunc}The probabilistic Fourier extension inequality
(\ref{prob ext'}) is equivalent to the truncated probabilistic extension
inequality (\ref{prob ext''}) taken over finitely many rotations.
\end{lemma}

The proof of Lemma \ref{prob trunc}, given at the end of the introduction,
also gives the following lemma upon removing the expectations $\mathbb{E}%
_{2^{\mathcal{G}}}^{\mu}$ and the random coefficients $a_{I}$ from the proof.

\begin{lemma}
\label{det trunc}The \emph{deterministic} Fourier extension inequality
(\ref{T_S equiv}) is equivalent to the truncated deterministic inequality
taken over finitely many rotations,%
\begin{equation}
\left\Vert T\sum_{I\in\mathcal{G}\left[  U\right]  }\bigtriangleup_{I;\kappa
}^{\eta}f\right\Vert _{L^{q}\left(  \lambda_{n}\right)  }\leq C\left\Vert
f\right\Vert _{L^{p}\left(  B\left(  0,\frac{1}{2}\right)  \right)
}.\label{T_S trunc}%
\end{equation}

\end{lemma}

\subsection{The main results and a brief history}

The following Fourier extension conjecture arose from unpublished work of E.
Stein in 1967, see e.g. \cite[see the Notes at the end of Chapter IX, p. 432,
where Stein proved the restriction conjecture for $1\leq p<\frac{4n}{3n+1}$%
]{Ste2} and \cite{Ste},
\begin{equation}
\left(  \int_{\mathbb{R}^{n}}\left\vert \mathcal{F}\left(  f\sigma
_{n-1}\right)  \right\vert ^{p}d\xi\right)  ^{\frac{1}{p}}\leq C\left(
\int_{\mathbb{S}^{n-1}}\left\vert f\left(  x\right)  \right\vert ^{p}%
d\sigma_{n-1}\left(  x\right)  \right)  ^{\frac{1}{p}},\ \ \ \ \ \text{for
}\frac{2n}{n-1}<p\leq\infty.\label{Stein conjecture}%
\end{equation}
Our probabilistic analogue\ of (\ref{Stein conjecture}) is the following
conjecture for the case $p=q$, where $\left(  \mathcal{A}_{\mathbf{a}}\right)
^{\spadesuit}=S_{\kappa,\eta}\mathcal{A}_{\mathbf{a}}\left(  S_{\kappa,\eta
}\right)  ^{-1}$ is the conjugation of the martingale transform $\mathcal{A}%
_{\mathbf{a}}$ with the bounded invertible linear map $S_{\kappa,\eta}$ used
in constructing the smooth Alpert wavelets in Theorem \ref{frame} below.

\begin{conjecture}
For $\kappa>\frac{n}{2}$\footnote{It seems likely this conjecture holds for
the classical Haar expansion (it is of course implied by the Fourier extension
conjecture), but we need $\kappa>\frac{n}{2}\geq1$ in our proof of the smooth
wavelet decomposition in Theorem \ref{frame}.} and notation as above,%
\begin{equation}
\mathbb{E}_{2^{\mathcal{G}}}^{\mu}\left\Vert T\mathbf{1}_{U_{0}}\left(
\mathcal{A}_{\mathbf{a}}\right)  ^{\spadesuit}f\right\Vert _{L^{p}\left(
\lambda_{n}\right)  }\lesssim\left\Vert f\right\Vert _{L^{p}\left(  B\left(
0,\frac{1}{2}\right)  \right)  }\ ,\ \ \ \ \ \text{if and only if }\frac
{2n}{n-1}<p\leq\infty,\label{Stein conjecture prob}%
\end{equation}
equivalently, the Fourier square function estimate,%
\begin{equation}
\left\Vert \mathcal{S}_{T\mathbf{1}_{U_{0}}}f\right\Vert _{L^{q}\left(
\lambda_{n}\right)  }\lesssim\left\Vert f\right\Vert _{L^{p}\left(  B\left(
0,\frac{1}{2}\right)  \right)  },\ \ \ \ \ \text{if and only if }\frac
{2n}{n-1}<p\leq\infty,\label{square equiv}%
\end{equation}
where%
\[
\mathcal{S}_{T\mathbf{1}_{U_{0}}}f\equiv\left(  \sum_{I\in\mathcal{G}\left[
U\right]  }\left\vert T\mathbf{1}_{U_{0}}\bigtriangleup_{I;\kappa}^{n-1,\eta
}f\right\vert ^{2}\right)  ^{\frac{1}{2}}.
\]

\end{conjecture}

\begin{theorem}
[Probabilistic extension conjecture]\label{FEC}The probabilistic Fourier
extension inequalities (\ref{Stein conjecture prob}) and (\ref{square equiv})
hold in all dimensions $n\geq2$.
\end{theorem}

Here the implied constant in $\lesssim$ depends only on harmless quantities
determined by context, which in the display (\ref{Stein conjecture prob}) are
$n$, $p$ and $U_{0}$.

Sections 2 through 10 are devoted to proving Theorem \ref{FEC}. Some
concluding remarks are made in Section 11.

\begin{acknowledgement}
I am indebted to Hong Wang and Ruixiang Zhang for pointing out serious gaps in
earlier versions of this paper, which claimed stronger results.
\end{acknowledgement}

There is a long history of progress on the Fourier extension conjecture in the
past half century, and we refer the reader to the excellent survey articles by
Thomas Wolff \cite{Wol}, Terence Tao \cite{Tao} and Betsy Stovall \cite{Sto}
for this history up to 2019, as well as for connections with related
conjectures and topics. Recently, a proof of the Kakeya set conjecture in
$\mathbb{R}^{3}$ has been posted to the arXiv by Hong Wang and Joshua Zahl
\cite{WaZa}. See further references below.

The following $\left(  \frac{1}{p},\frac{1}{q}\right)  $-rectangle for
boundedness of the extension operator illustrates this progression of positive
results:%
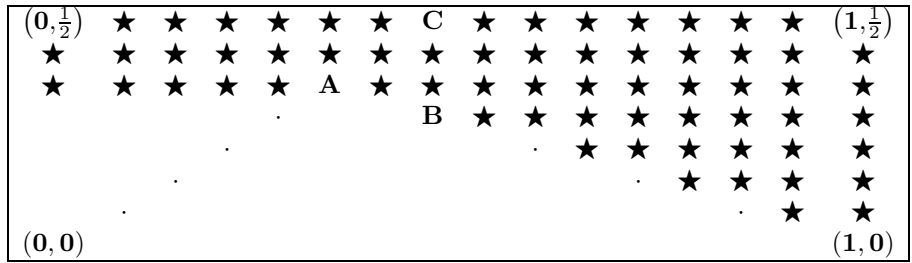
\begin{align*}
& \frame{$%
\begin{array}
[c]{cccccccccccccccc}%
\left(  \mathbf{0,}\frac{1}{2}\right)  & \bigstar & \bigstar & \bigstar &
\bigstar & \bigstar & \bigstar & \mathbf{C} & \bigstar & \bigstar & \bigstar &
\bigstar & \bigstar & \bigstar & \bigstar & \left(  \mathbf{1,}\frac{1}%
{2}\right) \\
\bigstar & \bigstar & \bigstar & \bigstar & \bigstar & \bigstar & \bigstar &
\bigstar & \bigstar & \bigstar & \bigstar & \bigstar & \bigstar & \bigstar &
\bigstar & \bigstar\\
\bigstar & \bigstar & \bigstar & \bigstar & \bigstar & \mathbf{A} & \bigstar &
\bigstar & \bigstar & \bigstar & \bigstar & \bigstar & \bigstar & \bigstar &
\bigstar & \bigstar\\
&  &  &  & \cdot &  &  & \mathbf{B} & \bigstar & \bigstar & \bigstar &
\bigstar & \bigstar & \bigstar & \bigstar & \bigstar\\
&  &  & \cdot &  &  &  &  &  & \cdot & \bigstar & \bigstar & \bigstar &
\bigstar & \bigstar & \bigstar\\
&  & \cdot &  &  &  &  &  &  &  &  & \cdot & \bigstar & \bigstar & \bigstar &
\bigstar\\
& \cdot &  &  &  &  &  &  &  &  &  &  &  & \cdot & \bigstar & \bigstar\\
\left(  \mathbf{0,0}\right)  &  &  &  &  &  &  &  &  &  &  &  &  &  &  &
\left(  \mathbf{1,0}\right)
\end{array}
$}\\
& \ \ \ \ \ \ \ \mathbf{A}=\left(  \frac{n-1}{2n},\frac{n-1}{2n}\right)
\text{ and }\mathbf{B}=\left(  \frac{1}{2},\frac{n-1}{2n+2}\right)  \text{ and
}\mathbf{C}=\left(  \frac{1}{2},\frac{1}{2}\right)
\end{align*}
The region marked with $\bigstar$ is where boundedness of the extension
operator (\ref{extension}) is known to fail, i.e. on and above the line
$\frac{1}{q}=\frac{n-1}{2n}$, and strictly above the Knapp line joining
$\mathbf{A}$ to $\left(  1,0\right)  $. The probabilistic analogue
(\ref{prob ext}) also fails for these pairs $\left(  \frac{1}{p},\frac{1}%
{q}\right)  $, as is shown below. The point $\mathbf{B}$ on the Knapp line is
the Stein-Tomas point, where boundedness is known from their 1975 result.
Since the set of points $\left(  \frac{1}{p},\frac{1}{q}\right)  $ for which
boundedness holds is both left-filled by embedding of $L^{p}$ spaces on the
sphere, and convex by interpolation, we see that as of 1975, the region
consisting of the line joining $\mathbf{B}$ to $\left(  1,0\right)  $, and
everything to the left of it, was known to be bounded for the extension
operator. The point $\left(  \frac{1}{2+\frac{4}{n}},\frac{1}{2+\frac{4}{n}%
}\right)  $ was added by Tao \cite{Tao4} in 2003, and points slightly better
than $\left(  \frac{1}{2+\frac{3}{n}},\frac{1}{2+\frac{3}{n}}\right)  $ were
added by Bourgain and Guth \cite[BoGu]{BoGu} in 2018.

Note also that any progress along the open diagonal line joining $\left(
0,0\right)  $ and $\mathbf{A}$, such as showing that $\left(  \frac{1}%
{p},\frac{1}{p}\right)  $ is bounded, yields boundedness for the convex hull
of $\left(  \frac{1}{p},\frac{1}{p}\right)  $ and the line $\frac{1}{q}=0$, as
well as all points to the left. Of course, even if the open diagonal segment
joining $\left(  0,0\right)  $ and $\mathbf{A}$ was known to be bounded, this
would still leave the open segment of the Knapp line joining $\mathbf{A}$ to
$\mathbf{B}$.

Our probabilistic theorem shows that the boundedness region for the
probabilistic extension conjecture includes all points not already eliminated
for the extension conjecture, except possibly for the open segment of the
Knapp line joining $\mathbf{A}$ to $\mathbf{B}$. Indeed, the conditions $q\geq
p^{\prime}\frac{n+1}{n-1}$ and $\frac{2n}{n-1}<q$ are necessary for the
extension inequality (\ref{extension}) to hold, see e.g. \cite{Tao}. The same
arguments show that these conditions on $p$ and $q$ are necessary for the
probabilistic analogue (\ref{prob ext}) to hold, upon considering individual
smooth Alpert wavelets $h_{I;\kappa}^{\eta}$ (see below for definitions).
Since $\sigma_{n-1}$ is a finite measure, embedding and interpolation with the
trivial $L^{1}\rightarrow L^{\infty}$ bound, together with Theorem \ref{FEC},
prove the probabilistic extension inequality for this range of exponents,
except for the range $q=p^{\prime}\frac{n+1}{n-1}$ and $1<p<\frac{2n}{n-1}$.
Since the Stein Tomas result \cite{Tom} captures the subcase of
(\ref{extension}) when $1\leq p\leq2$, this leaves only $q=p^{\prime}%
\frac{n+1}{n-1}$ and $2<p<\frac{2n}{n-1}$ open in the probabilistic extension conjecture.

\subsection{Quick overview of the proof using smooth Alpert wavelets}

We begin with a short and informal narrative.

\begin{description}
\item[Narrative] In the theory of nonhomogeneous harmonic analysis, and
especially that of two weight norm inequalities for the \emph{Hilbert}
transform, Nazarov, Treil and Volberg initiated the systematic use of weighted
Haar wavelets to analyze boundedness. The Hilbert transform has kernel
$\frac{1}{x-\xi}$ , and thus the action of a Haar wavelet against such a
kernel typically has geometric decay away from the origin, which permits
`error' off diagonal terms to be controlled. This two weight theory has
concentrated mainly on the Hilbert space case $p=2$ in the past couple of
decades, but more recently $L^{p}$ estimates and square functions have
attracted attention, especially with the recent work of Hyt\"{o}nen and
Vuorinen.\medskip\newline At this point it becomes conceivable that square
function and two weight techniques might be applicable to two weight $L^{p}$
norm inequalities for the \emph{Fourier} transform, such as the Fourier
restriction conjecture, equivalent to the norm inequality with measures
$d\sigma_{n-1}$ and $d\lambda_{n}$ in $\mathbb{R}^{n}$,%
\[
\left\Vert \mathcal{F}\left(  f\sigma_{n-1}\right)  \right\Vert _{L^{p}\left(
\lambda_{n}\right)  }\lesssim\left\Vert f\right\Vert _{L^{p}\left(
\sigma_{n-1}\right)  }.
\]
However, the kernel $K\left(  x,\xi\right)  =e^{-ix\cdot\xi}$ of the Fourier
transform $\mathcal{F}$ is purely oscillatory with no decay at all, but this
is\ partially offset by the curvature of the support of $\sigma_{n-1}$, that
produces decay from the principle of stationary phase. Moreover, the action of
a Haar wavelet against this kernel will be small if there is little variation
of the kernel over the support of the wavelet (i.e. long wavelength), since
the wavelet has vanishing mean, but this gain is limited by the absence of
higher order vanishing moments in a Haar wavelet.\medskip\newline Addressing
this defect, Alpert constructed wavelets with similar properties to those of
Haar, but with additional vanishing moments that confer extra geometric gain.
But even with Alpert wavelets in place of Haar wavelets, there is no geometric
gain when the wavelength of the kernel is small compared to the size of the
wavelet, due to the abrupt cutoffs in the dyadic construction of these
wavelets.\medskip\newline In this paper we construct \emph{smooth} Alpert
wavelets that permit geometric decay when the wavelengths are small, i.e. when
there is sufficient oscillation of the kernel over the support of the wavelet
to permit gain from repeated integration by parts. Thus we will have gain
except in the case of \emph{resonance}, when there is neither sufficient
smoothness nor oscillation in the restriction of the kernel to the support of
either the $n-1$ or $n$ dimensional wavelet. In these resonant situations,
which form the core of difficulty in the deterministic Fourier extension
conjecture, we must appeal to probability in order to obtain the desired
$L^{4}$ bound needed for interpolation. The remainder of the paper holds
without the intervention of probability.\medskip\newline
\end{description}

Our proof of the probabilistic Fourier extension conjecture uses some
techniques arising in the\ two weight testing theory of operator norms,
\cite{NTV4}, \cite{Vol}, \cite{LaSaShUr3}, \cite{SaShUr7}, \cite{AlSaUr} and
\cite{SaWi}, that were in turn based on older work with roots in \cite{FeSt},
\cite{DaJo}, \cite{Saw} and \cite{Saw3}, and followed by many other papers as
well, such as \cite{Hyt}, \cite{LaWi}, \cite{SaShUr12} and \cite{HyVu} to
mention just a few\footnote{Some of the deepest results in testing theory,
namely the good/bad machinery of Nazarov, Treil and Volberg in e.g.
\cite{NTV4}, the functional energy from \cite{LaSaShUr3}, the two weight
inequalities for Poisson integrals from \cite{Saw3}, and the upside down
corona and recursion from Lacey \cite{Lac}, are not used here. Some reasons
for this are the lack of `edge effects' in smooth Alpert wavelets, the lack of
a paraproduct/stopping form decomposition, the `niceness' of surface measure
on the sphere and Lebesgue measure, and of course that the probabilistic
conjecture is significantly weaker than the deterministic one. Indeed, the
higher frequencies are damped to a greater extent by expectation, and this is
why Kakeya phenomena do not enter into probabilistic arguments. On the other
hand we make extensive use of pigeonholing into bilinear subforms according to
the uncertainty principle, and then applying square function techniques for
Alpert frames.}. One of the main new ingredients used here is the construction
of compactly supported smooth frames in $L^{p}$ with derivative estimates
adapted to the support, and as many vanishing moments as we wish. In fact, we
will show that the wavelets $h_{I;\kappa}^{a,\eta}$ in the following theorem,
can be constructed in the spirit of symbol smoothing, as appropriate
convolutions of a certain approximate identity with the Alpert wavelets in
\cite{Alp}, see also their weighted versions in \cite{RaSaWi}.

As already noted, for the proof of the probabilistic extension conjecture, it
is enough to prove (\ref{prob ext''}),%
\[
\mathbb{E}_{2^{\mathcal{G}}}^{\mu}\left\Vert T\left(  \sum_{I\in
\mathcal{G}\left[  U\right]  }a_{I}\bigtriangleup_{I;\kappa}^{n-1,\eta
}f\right)  \right\Vert _{L^{p}}\lesssim\left\Vert f\right\Vert _{L^{p}}\ .
\]
However, we begin by writing the Fourier bilinear form $\left\langle T\left(
\sum_{I\in\mathcal{G}\left[  U\right]  }a_{I}\bigtriangleup_{I;\kappa
}^{n-1,\eta}f\right)  ,g\right\rangle _{\mathbb{R}^{n}}$ as a finite sum of
subforms%
\[
\mathsf{B}_{\mathcal{P}}\left(  f,g\right)  \equiv\sum_{\left(  I,J\right)
\in\mathcal{P}}\left\langle T\left(  a_{I}\bigtriangleup_{I;\kappa}^{n-1,\eta
}f\right)  ,\bigtriangleup_{J;\kappa}^{n,\eta}g\right\rangle _{\mathbb{R}^{n}}%
\]
where $\mathcal{P}$ is a collection of pairs of dyadic cubes $I\in
\mathcal{G}\left[  U\right]  $ and $J\in\mathcal{D}$, and where
$\bigtriangleup_{I;\kappa}^{n-1,\eta}$ and $\bigtriangleup_{J;\kappa}^{n,\eta
}$ are smooth Alpert pseudoprojections in $\mathbb{R}^{n-1}$ and
$\mathbb{R}^{n}$ respectively. This decomposition into subforms follows that
used by Nazarov, Treil and Volberg in the setting of singular integrals with
weighted Haar wavelets, but using the uncertainty principle to compare sizes
of cubes here. There are six main subforms, the below $\mathsf{B}%
_{\operatorname*{below}}\left(  f,g\right)  $, above $\mathsf{B}%
_{\operatorname*{above}}\left(  f,g\right)  $, upper disjoint and distal
$\mathsf{B}_{\operatorname*{disjoint}}^{\operatorname*{upper}}\left(
f,g\right)  ,\mathsf{B}_{\operatorname*{distal}}^{\operatorname*{upper}%
}\left(  f,g\right)  $, and lower disjoint and distal $\mathsf{B}%
_{\operatorname*{disjoint}}^{\operatorname*{lower}}\left(  f,g\right)
,\mathsf{B}_{\operatorname*{distal}}^{\operatorname*{lower}}\left(
f,g\right)  $ subforms.

The first two subforms are handled by the classical methods of integration by
parts and stationary phase, but also use the smoothness and moment vanishing
properties of the Alpert wavelets constructed in the next theorem, while the
next two upper forms also use tangential integration by parts.

Finally, the last two most challenging forms, namely the lower disjoint and
distal forms \footnote{challenging because of the resonance that arises when
the cubes $I$ and $J$ are appropriately positioned and sized, with the
consequence that neither integration by parts nor moment vanishing can be put
to use. In fact, it was precisely this difficulty that led to the serious gap
in an earlier version \texttt{v4} of this paper, and which was pointed out to
the author by Hong Wang and Ruixiang Zhang.}, are handled using properties of
smooth Alpert wavelets with expectation taken over involutive smooth Alpert
multipliers. While the deterministic \emph{form} estimates for the previous
four forms imply corresponding deterministic \emph{norm} estimates by duality,
this is no longer true for the probabilistic estimates we obtain, and it is
important that we obtain the stronger probabilistic \emph{norm} estimates in
these cases. In fact, we will obtain $L^{2}$ and average $L^{4}$ norm
estimates for smooth Alpert pseudoprojections (essentially because these
spaces have the upper majorant property), which can then be interpolated to
obtain the required norm bounds. However, this argument fails without
expectation, and so fails to obtain the Fourier extension conjecture, whose
attack requires far more sophisticated techniques. See Proposition
\ref{prop interp}, and Lemmas \ref{fattened red} and \ref{interp red} below.

Here is the smooth compactly supported frame of wavelets for $L^{p}$ that we
will use\footnote{This particular theorem does not appear to be in the
literature on frames.}.

\begin{theorem}
\label{frame}Let $n,\kappa\in\mathbb{N}$ with $\kappa>\frac{n}{2}$, and
$\eta>0$ be sufficiently small depending on $n$ and $\kappa$. Then there are a
bounded invertible linear map $S_{\kappa,\eta}:L^{p}\rightarrow L^{p} $
($1<p<\infty$) satisfying%
\begin{equation}
\left\Vert \operatorname{Id}-S_{\kappa,\eta}\right\Vert _{L^{p}\rightarrow
L^{p}}\leq C_{n,p}\eta\ ,\label{little oh}%
\end{equation}
and `wavelets' $\left\{  h_{I;\kappa}^{a}\right\}  _{I\in\mathcal{D}%
,\ a\in\Gamma_{n}}$ and $\left\{  h_{I;\kappa}^{a,\eta}\right\}
_{I\in\mathcal{D},\ a\in\Gamma_{n}}$ (with $\Gamma_{n}$ a finite index set
depending only on $\kappa$ and $n$), and corresponding projections and
pseudoprojections $\left\{  \bigtriangleup_{I;\kappa}\right\}  _{I\in
\mathcal{D}}$ and $\left\{  \bigtriangleup_{I;\kappa}^{\eta}\right\}
_{I\in\mathcal{D}}$ defined by%
\[
\bigtriangleup_{I;\kappa}f\equiv\sum_{a\in\Gamma_{n}}\left\langle
f,h_{I;\kappa}^{a}\right\rangle h_{I;\kappa}^{a}\text{ and }\bigtriangleup
_{I;\kappa}^{\eta}f\equiv\sum_{a\in\Gamma_{n}}\left\langle \left(
S_{\kappa,\eta}\right)  ^{-1}f,h_{I;\kappa}^{a}\right\rangle h_{I;\kappa
}^{a,\eta}\ ,
\]
satisfying

\begin{enumerate}
\item the standard properties,%
\begin{align}
\left\Vert h_{I;\kappa}^{a,\eta}\right\Vert _{L^{2}}  & \approx\left\Vert
h_{I;\kappa}^{a}\right\Vert _{L^{2}}=1,\label{wavelet prop}\\
\operatorname*{Supp}h_{I;\kappa}^{a}  & \subset I\text{ and }%
\operatorname*{Supp}h_{I;\kappa}^{a,\eta}\subset\left(  1+\eta\right)
I,\nonumber\\
\left\Vert \nabla^{m}h_{I;\kappa}^{a,\eta}\right\Vert _{\infty}  & \leq
C_{m}\left(  \frac{1}{\eta\ell\left(  I\right)  }\right)  ^{m}\frac{1}%
{\sqrt{\left\vert I\right\vert }},\ \ \ \ \ \text{for all }m\geq0,\nonumber\\
\int h_{I;\kappa}^{a}\left(  x\right)  x^{\alpha}dx  & =\int h_{I;\kappa
}^{a,\eta}\left(  x\right)  x^{\alpha}dx=0,\ \ \ \ \ \text{for all }%
0\leq\left\vert \alpha\right\vert <\kappa.\nonumber
\end{align}

\item and for each $a\in\Gamma_{n}$ the wavelets $h_{I;\kappa}^{a}$ and
$h_{I;\kappa}^{a,\eta}$ are translations and $L^{2}$-dilations of the unit
wavelets $h_{Q_{0};\kappa}^{a}$ and $h_{Q_{0};\kappa}^{a,\eta}$ respectively,
where $Q_{0}=\left[  0,1\right)  ^{n}$ is the unit cube in $\mathbb{R}^{n}$,%
\begin{equation}
h_{I;\kappa}^{a}=\sqrt{\frac{\left\vert Q_{0}\right\vert }{\left\vert
I\right\vert }}h_{Q_{0};\kappa}^{a}\circ\varphi_{I}\text{ and }h_{I;\kappa
}^{a,\eta}=\sqrt{\frac{\left\vert Q_{0}\right\vert }{\left\vert I\right\vert
}}h_{Q_{0};\kappa}^{a,\eta}\circ\varphi_{I}\ ,\label{mother}%
\end{equation}
where $\varphi_{I}:I\rightarrow Q_{0}$ is the affine map taking $I$ one-to-one
and onto $Q_{0}$,

\item and for all $1<p<\infty$,%
\begin{align}
& \ \ \ \ \ \ \ \ \ \ \ \ \ \ \ f=\sum_{I\in\mathcal{D},\ a\in\Gamma_{n}%
}\bigtriangleup_{I;\kappa}^{a}f=\sum_{I\in\mathcal{D},\ a\in\Gamma_{n}%
}\bigtriangleup_{I;\kappa}^{a,\eta}f,\ \ \ \ \ \text{with convergence in norm
for }f\in L^{p}\cap L^{2},\label{squ est}\\
& \left\Vert \left(  \sum_{I\in\mathcal{D},\ a\in\Gamma_{n}}\left\vert
\bigtriangleup_{I;\kappa}^{a}f\right\vert ^{2}\right)  ^{\frac{1}{2}%
}\right\Vert _{L^{p}\left(  \mathbb{R}^{n}\right)  }\approx\left\Vert \left(
\sum_{I\in\mathcal{D},\ a\in\Gamma_{n}}\left\vert \bigtriangleup_{I;\kappa
}^{a,\eta}f\right\vert ^{2}\right)  ^{\frac{1}{2}}\right\Vert _{L^{p}\left(
\mathbb{R}^{n}\right)  }\approx\left\Vert f\right\Vert _{L^{p}\left(
\mathbb{R}^{n}\right)  },\ \ \ \ \ \text{for }f\in L^{p}\cap L^{2},\nonumber
\end{align}

\item and for all $I\in\mathcal{D}$,%
\[
h_{Q;\kappa}^{a}\left(  x\right)  =h_{Q;\kappa}^{a,\eta}\left(  x\right)
,\ \ \ \ \ \text{for }x\in\mathbb{R}^{n}\setminus\mathcal{H}_{\eta}\left(
Q\right)  ,
\]
where $\mathcal{H}_{\eta}\left(  Q\right)  $ is the $\eta$-halo of the
skeleton of $Q$ defined in (\ref{def halo}) below.

\item and finally, the unsmoothed operators $\bigtriangleup_{I;\kappa}$ are
self-adjoint orthogonal projections\footnote{The operators $\bigtriangleup
_{I;\kappa}^{\eta}$ are neither self-adjoint, projections nor orthogonal, but
come close as we will see.},%
\begin{equation}
\bigtriangleup_{I;\kappa}\bigtriangleup_{J;\kappa}=\left\{
\begin{array}
[c]{ccc}%
\bigtriangleup_{I;\kappa} & \text{ if } & I=J\\
0 & \text{ if } & I\not =J
\end{array}
\right.  .\label{ortho proj}%
\end{equation}

\end{enumerate}
\end{theorem}

\begin{remark}
This theorem shows that the collection of `almost' $L^{2}$ projections
$\left\{  \bigtriangleup_{I;\kappa}^{\eta,a}\right\}  _{I\in\mathcal{D}%
,\ a\in\Gamma_{n}}$ is a `frame' for the Banach space $L^{p}$, $1<p<\infty$.
The case $\eta=0$ of (\ref{squ est}) was obtained in the generality of
doubling measures $\mu$ in \cite{SaWi}.
\end{remark}

\begin{acknowledgement}
I thank Brett Wick for instigating our work on two weight $L^{p}$ norm
inequalities in \cite{LaWi}, Michel Alexis and Ignacio Uriarte-Tuero for
completing in our joint paper \cite{AlSaUr} the work begun in \cite{Saw6} on
doubling measures, and Michel and Jose Luis Luna-Garcia for our work
\cite{AlLuSa} on $L^{p}$ frames. Ideas from these papers have played a key
role in the development of the arguments used here, as well as ideas from past
collaborations and other works. I also thank Cristian Rios for valuable
discussions, suggestions and critical reading of portions of the manuscript,
including a fruitful week long visit to Hamilton. Finally, I thank Ruixiang
Zhang for many enlightening comments, and for pointing to several problems in
the proof.
\end{acknowledgement}

\subsubsection{Organization of the paper}

In the next section we will construct and prove the required properties of
smooth Alpert wavelets, and in Section 3 we introduce the extension operator
and recall what we need regarding stationary phase. This material is
well-known but we repeat it here due to the explicit error estimates we use.
In Section 4 we discuss the initial wavelet decompositions into various
subforms and describe the classical and well-known decay principles we use.
Then in Section 5 we turn to the interpolation of $L^{2}$ and $L^{4}$
estimates using probability. Then in Sections 6, 7 and 8 we will control the
below, above and upper disjoint/distal forms respectively in the deterministic
sense. Then in Section 9 we will use probability to control the lower
disjoint/distal form by averaging over smooth Alpert martingale transforms.
Then we collect these results to finish the proof of the probabilistic Fourier
extension theorem in Section 10, and in Section 11 we make some concluding comments.

\subsection{The initial setup}

Fix a small cube $U_{0}$ in $\mathbb{R}^{n-1}$ with side length a negative
power of $2$, and such that there is a translation $\mathcal{G}$ of the
standard grid on $\mathbb{R}^{n-1}$ with the property that $U_{0}%
\in\mathcal{G}$, the grandparent $U\equiv\pi_{\mathcal{G}}^{\left(  2\right)
}U_{0}$ of $U_{0}$ has the origin as a vertex, and $U_{0}$ is an interior
grandchild of $U_{0}$, so that%
\begin{equation}
U_{0},U\in\mathcal{G}\text{ with }U_{0}\subset\frac{1}{2}U\text{ and }U\subset
B\left(  0,\frac{1}{8}\right)  \text{.}\label{support}%
\end{equation}
The radius $\frac{1}{8}$ is chosen small enough that the various definitions
of forms below are well-defined.

Now parameterize a patch of the sphere $\mathbb{S}^{n-1}$ in the usual way,
i.e. $\Phi:U\rightarrow\mathbb{S}^{n-1}$ by
\[
z=\Phi\left(  x\right)  \equiv\left(  x,\sqrt{1-\left\vert x\right\vert ^{2}%
}\right)  =\left(  x_{1},x_{2},...,x_{n-1},\sqrt{1-\left\vert x\right\vert
^{2}}\right)  .
\]
For $f\in L^{p}\left(  B_{n-1}\left(  0,\frac{1}{2}\right)  \right)  $ define
\begin{equation}
Tf\left(  \xi\right)  \equiv\mathcal{F}\left(  \Phi_{\ast}\left[  f\left(
x\right)  dx\right]  \right)  =\int_{B_{n-1}\left(  0,\frac{1}{2}\right)
}e^{-i\Phi\left(  x\right)  \cdot\xi}f\left(  x\right)  \frac{dx}{\left\vert
\det\Phi\left(  x\right)  \right\vert },\label{def T_S}%
\end{equation}
where $\Phi_{\ast}\left[  f\left(  x\right)  dx\right]  $ is the pushforward
of the measure $f\left(  x\right)  dx$ in $B_{n-1}\left(  0,\frac{1}%
{2}\right)  $ to the patch of sphere $\Phi\left(  B_{n-1}\left(  0,\frac{1}%
{2}\right)  \right)  $ lying above $B_{n-1}\left(  0,\frac{1}{2}\right)  $,
and that we typically abuse notation by ignoring the harmless factor $\frac
{1}{\left\vert \det\Phi\left(  x\right)  \right\vert }$. Recall that the
Fourier extension inequality is equivalent to (\ref{T_S equiv}). The bilinear
form associated to $T\mathbf{1}_{U_{0}}$ in (\ref{T_S equiv}) can be
decomposed by,%
\[
\left\langle T\mathbf{1}_{U_{0}}f,g\right\rangle =\left\langle T\mathbf{1}%
_{U_{0}}\left(  \sum_{I\in\mathcal{G}}\bigtriangleup_{I;\kappa}^{n-1}f\right)
,\sum_{J\in\mathcal{D}}\bigtriangleup_{J;\kappa}^{n}g\right\rangle
=\sum_{\left(  I,J\right)  \in\mathcal{G}\times\mathcal{D}}\left\langle
T\mathbf{1}_{U_{0}}\bigtriangleup_{I;\kappa}^{n-1}f,\bigtriangleup_{J;\kappa
}^{n}g\right\rangle \ ,
\]
where $\left\{  \bigtriangleup_{J;\kappa}^{n}\right\}  _{J\in\mathcal{D}}$ is
an Alpert basis of projections for $L^{2}\left(  \mathbb{R}^{n}\right)  $, and
$\left\{  \bigtriangleup_{I;\kappa}^{n-1}\right\}  _{I\in\mathcal{G}}$ is an
Alpert basis of projections for $L^{2}\left(  \mathbb{R}^{n-1}\right)  $.
Using rotation invariance, the Fourier extension conjecture is shown at the
beginning of Section 3 below, to be equivalent to boundedness of
$T\mathbf{1}_{U_{0}}$, taken over a finite collection of patches $\Phi\left(
U_{0}\right)  $.

\begin{notation}
\label{Notation Alpert'}We are using the index $n-1$ or $n$ in the superscript
of the notation $\bigtriangleup_{I;\kappa}^{n-1,\eta}f$ for an Alpert
projection, to denote whether the wavelet lives in $\mathbb{R}^{n-1}$ or in
$\mathbb{R}^{n}$. The index $\eta$ in the superscript denotes the smoothness
injected by convolution in the construction of the smooth Alpert wavelets
below. Moreover, we usually suppress the index $a\in\Gamma$ that runs over the
set of all Alpert wavelets associated with a given cube.
\end{notation}

However, in order to carry out the standard two weight approach to bounding
$T$, it will be necessary to fix $\kappa\in\mathbb{N}$, $\kappa>\frac{n}{2} $,
and instead expand the bilinear form $\left\langle T\left(  \mathsf{P}%
_{U}\right)  ^{\spadesuit}f,g\right\rangle =\left\langle T\sum_{I\in
\mathcal{G}\left[  U\right]  }\bigtriangleup_{I;\kappa}^{n-1,\eta
}f,g\right\rangle $, corresponding to the equivalent inequality
(\ref{T_S trunc}), in terms of the \emph{smooth} $\kappa$-Alpert
decompositions of $f$ and $g$,%
\[
\left\langle T\left(  \mathsf{P}_{U}\right)  ^{\spadesuit}f,g\right\rangle
=\sum_{\left(  I,J\right)  \in\mathcal{G}\left[  U\right]  \times\mathcal{D}%
}\left\langle T\bigtriangleup_{I;\kappa}^{n-1,\eta}f,\bigtriangleup_{J;\kappa
}^{n,\eta}g\right\rangle ,
\]
so as to exploit the cancellation inherent in the oscillatory kernel
$e^{-i\Phi\left(  x\right)  \cdot\xi}$ of the operator $T_{S}$.

\begin{definition}
A subset $E$ of the unit sphere $\mathbb{S}^{n-1}$ in $\mathbb{R}^{n}$ is said
to be a \emph{ball} if it is the intersection of the sphere with a halfspace,
and is said to be a \emph{pseudoball with constant }$C_{\operatorname{pseudo}%
}$, if there are concentric balls $B_{1}$ and $B_{2}$ such that
\begin{equation}
B_{1}\subset E\subset B_{2}\text{ and }\left\vert B_{2}\right\vert \leq
C_{\operatorname{pseudo}}\left\vert B_{1}\right\vert ,\label{def pseudoball}%
\end{equation}
where $\left\vert E\right\vert $ denotes surface measure on the sphere. We
simply say that $E$ is a \emph{pseudoball} when $C_{\operatorname{pseudo}}$ is
understood from context, and we will sometimes define a `center' of $E$ to be
the center (not uniquely determined) of the balls $B_{1}$ and $B_{2}$ in
(\ref{def pseudoball}).
\end{definition}

\begin{definition}
Given a subset $F$ of Euclidean space $\mathbb{R}^{n}$, we define the
tangential and radial `projections' of $F$, onto $\mathbb{S}^{n-1}$ and
$\left[  0,\infty\right)  $ respectively, by
\[
\pi_{\tan}\left(  F\right)  \equiv\left\{  \frac{\xi}{\left\vert
\xi\right\vert }:\xi\in F\right\}  \text{ and }\pi_{\operatorname{rad}}\left(
F\right)  \equiv\left\{  \left\vert \xi\right\vert :\xi\in F\right\}  .
\]

\end{definition}

Then for $C_{\operatorname{pseudo}}$ chosen large enough in
(\ref{def pseudoball}), the subsets $\Phi\left(  I\right)  $ and $\pi_{\tan
}\left(  J\right)  $ of the sphere $\mathbb{S}^{n-1}$ are pseudoballs with
constant $C_{\operatorname{pseudo}}$, for all $I\in\mathcal{G}\left[
U\right]  $ and $J\in\mathcal{D}$. For $E\subset\mathbb{S}^{n-1}$, we denote
by $-E$ the set antipodal to $E$, i.e. $-E=\left\{  \zeta\in\mathbb{S}%
^{n-1}:-\zeta\in E\right\}  $.

We now divide the collection of pairs $\left(  I,J\right)  \in\mathcal{G}%
\left[  U\right]  \times\mathcal{D}$ according to the relative size and
location of their associated pseudoballs $\Phi\left(  I\right)  $ and
$\pi_{\tan}\left(  J\right)  $, as dictated by the uncertainty principle:
\begin{align}
& \mathcal{G}\left[  U\right]  \times\mathcal{D}\subset\mathcal{P}%
\ \cup\ \mathcal{P}^{-}\ ,\label{decomp}\\
\text{where }\mathcal{P}  & =\mathcal{P}_{0}\ \cup\ \bigcup_{m=1}^{\infty
}\mathcal{P}_{m}\ \cup\ \mathcal{R}\cup\ \mathcal{X\ },\nonumber\\
\text{and }\mathcal{P}^{-}  & =\left\{  \left(  I,-J\right)  :\left(
I,J\right)  \in\mathcal{P}\right\}  \ ,\nonumber
\end{align}
and where
\begin{align*}
\mathcal{P}_{0}  & \equiv\left\{  \left(  I,J\right)  \in\mathcal{G}\left[
U\right]  \times\mathcal{D}:\pi_{\tan}\left(  J\right)  \subset\Phi\left(
C_{\operatorname{pseudo}}I\right)  \right\}  \ ,\\
\mathcal{P}_{m}  & \equiv\left\{  \left(  I,J\right)  \in\mathcal{G}\left[
U\right]  \times\mathcal{D}:2^{m+1}I\subset2U\text{, }\pi_{\tan}\left(
J\right)  \subset\Phi\left(  4U\cap2^{m+1}C_{\operatorname{pseudo}}I\right)
\setminus\Phi\left(  2^{m}\frac{1}{C_{\operatorname{pseudo}}}I\right)
\right\}  ,\ \ \ \ \ m\in\mathbb{N}\ ,\\
\mathcal{R}  & \equiv\left\{  \left(  I,J\right)  \in\mathcal{G}\left[
U\right]  \times\mathcal{D}:\Phi\left(  I\right)  \subset\pi_{\tan}\left(
C_{\operatorname{pseudo}}J\right)  \right\}  \ ,\\
\mathcal{X}  & \equiv\left\{  \left(  I,J\right)  \in\mathcal{G}\left[
U\right]  \times\mathcal{D}:J\subset\mathbb{R}_{+}^{n}\text{ and }\pi_{\tan
}\left(  C_{\operatorname{pseudo}}J\right)  \cap\Phi\left(  2U\right)
=\emptyset\right\}  \ .
\end{align*}
Note that $\left(  I,J\right)  \in\mathcal{P}_{m}$ implies that $m\leq cs$ for
some fixed constant $c$, and that there is some bounded overlap among the
pairs in this decomposition, but this overcounting turns out to be
inconsequential. Finally we point out that it suffices to show that%
\[
\left\vert \sum_{\left(  I,J\right)  \in\mathcal{P}}\left\langle
T\bigtriangleup_{I;\kappa}^{n-1,\eta}f,\bigtriangleup_{J;\kappa}^{n,\eta
}g\right\rangle \right\vert \lesssim\left\Vert f\right\Vert _{L^{p}}\left\Vert
g\right\Vert _{L^{p^{\prime}}}\ ,
\]
since $\left(  I,J\right)  \in\mathcal{P}^{-}$ if and only if $\left(
I,-J\right)  \in\mathcal{P}$, and this amounts to replacing the kernel
$e^{-i\Phi\left(  x\right)  \cdot\xi}$ with the kernel $e^{i\Phi\left(
x\right)  \cdot\xi}$, for which the estimates obtained below are identical.

\subsubsection{Proof of reduction to the truncated inequality}

Here we prove Lemma \ref{prob trunc}.

\begin{proof}
[Proof of Lemma \ref{prob trunc}]Using $f=\sum_{I\in\mathcal{G}}%
\bigtriangleup_{I;\kappa}^{n-1,\eta}f$ from the first line in (\ref{squ est})
of Theorem \ref{frame} below, we write\footnote{I thank Cristian Rios for
pointing out this simplification to an earlier proof.}%
\[
\mathbf{1}_{U_{0}}\left(  \mathcal{A}_{\mathbf{a}}\right)  ^{\spadesuit
}f=\mathbf{1}_{U_{0}}\sum_{I\in\mathcal{G}}a_{I}\bigtriangleup_{I;\kappa
}^{n-1,\eta}f=\mathbf{1}_{U_{0}}\sum_{I\in\mathcal{G}\left[  U\right]  }%
a_{I}\bigtriangleup_{I;\kappa}^{n-1,\eta}f+\mathbf{1}_{U_{0}}\sum
_{k=1}^{\infty}\sum_{I\in\mathcal{N}\left(  \pi^{\left(  k\right)  }%
U_{0}\right)  }a_{I}\bigtriangleup_{I;\kappa}^{n-1,\eta}f\equiv L_{1}%
^{\mathbf{a}}f+L_{2}^{\mathbf{a}}f.
\]
since $\mathbf{1}_{U_{0}}\bigtriangleup_{I;\kappa}^{n-1,\eta}f$ vanishes if
$I\notin\mathcal{G}\left[  U\right]  \cup\left\{  \mathcal{N}\left(
\pi^{\left(  k\right)  }U_{0}\right)  \right\}  _{k=1}^{\infty}$. Indeed,
$\operatorname*{Supp}\bigtriangleup_{I;\kappa}^{n-1,\eta}\subset\left(
1+\eta\right)  U$ which is disjoint from $U_{0}$ if $I\notin\mathcal{G}\left[
U\right]  \cup\left\{  \mathcal{N}\left(  \pi^{\left(  k\right)  }%
U_{0}\right)  \right\}  _{k=1}^{\infty}$. We will now show that%
\begin{align}
\mathbb{E}_{2^{\mathcal{G}}}^{\mu}\left\Vert TL_{1}^{\mathbf{a}}f\right\Vert
_{L^{q}}  & =\mathbb{E}_{2^{\mathcal{G}}}^{\mu}\left\Vert T\mathbf{1}_{U_{0}%
}\sum_{I\in\mathcal{G}\left[  U\right]  }a_{I}\bigtriangleup_{I;\kappa
}^{n-1,\eta}f\right\Vert _{L^{q}}\lesssim\mathbb{E}_{2^{\mathcal{G}}}^{\mu
}\left\Vert T\sum_{I\in\mathcal{G}\left[  U\right]  }a_{I}\bigtriangleup
_{I;\kappa}^{n-1,\eta}f\right\Vert _{L^{q}}\ ,\label{two ineq}\\
\sup_{\mathbf{a}}\left\Vert TL_{2}^{\mathbf{a}}f\right\Vert _{L^{q}}  &
=\sup_{\mathbf{a}}\left\Vert T\mathbf{1}_{U_{0}}\sum_{k=1}^{\infty}\sum
_{I\in\mathcal{N}\left(  \pi^{\left(  k\right)  }U_{0}\right)  }%
a_{I}\bigtriangleup_{I;\kappa}^{n-1,\eta}f\right\Vert _{L^{q}}\lesssim
\left\Vert f\right\Vert _{L^{p}\left(  B\left(  0,\frac{1}{2}\right)  \right)
}\ ,\nonumber
\end{align}
which is easily seen to complete the proof that (\ref{prob ext''}) implies
(\ref{prob ext'}).

To see the first line in (\ref{two ineq}), choose a rectangle $R_{0}$ in
$\mathbb{R}^{n}$ with base $U_{0}$ and height $1$ so that $R_{0}\cap
\mathbb{S}^{n-1}=\Phi\left(  U_{0}\right)  $. Then $\Phi_{\ast}\mathbf{1}%
_{U_{0}}=\mathbf{1}_{R_{0}}\Phi_{\ast}$, and since $\mathcal{F}\mathbf{1}%
_{R_{0}}\mathcal{F}^{-1}$ is a bounded Fourier multiplier on $L^{q}\left(
\mathbb{R}^{n}\right)  $ for all $1<q<\infty$, we obtain
\begin{align*}
& \mathbb{E}_{2^{\mathcal{G}}}^{\mu}\left\Vert TL_{1}f\right\Vert _{L^{q}%
}=\mathbb{E}_{2^{\mathcal{G}}}^{\mu}\left\Vert \mathcal{F}\Phi_{\ast
}\mathbf{1}_{U_{0}}\sum_{I\in\mathcal{G}\left[  U\right]  }a_{I}%
\bigtriangleup_{I;\kappa}^{n-1,\eta}f\right\Vert _{L^{q}}\\
& =\mathbb{E}_{2^{\mathcal{G}}}^{\mu}\left\Vert \mathcal{F}\mathbf{1}_{R_{0}%
}\Phi_{\ast}\sum_{I\in\mathcal{G}\left[  U\right]  }a_{I}\bigtriangleup
_{I;\kappa}^{n-1,\eta}f\right\Vert _{L^{q}}=\mathbb{E}_{2^{\mathcal{G}}}^{\mu
}\left\Vert \left(  \mathcal{F}\mathbf{1}_{R_{0}}\mathcal{F}^{-1}\right)
\mathcal{F}\Phi_{\ast}\sum_{I\in\mathcal{G}\left[  U\right]  }a_{I}%
\bigtriangleup_{I;\kappa}^{n-1,\eta}f\right\Vert _{L^{q}}\\
& \lesssim\mathbb{E}_{2^{\mathcal{G}}}^{\mu}\left\Vert \mathcal{F}\Phi_{\ast
}\sum_{I\in\mathcal{G}\left[  U\right]  }a_{I}\bigtriangleup_{I;\kappa
}^{n-1,\eta}f\right\Vert _{L^{q}}=\mathbb{E}_{2^{\mathcal{G}}}^{\mu}\left\Vert
T\sum_{I\in\mathcal{G}\left[  U\right]  }a_{I}\bigtriangleup_{I;\kappa
}^{n-1,\eta}f\right\Vert _{L^{q}}.
\end{align*}

Now we turn to proving the second line in (\ref{two ineq}). Let $\psi$ be a
smooth bump function that is $1$ on $U_{0}$ and supported in $U$. Then arguing
once more as above,
\begin{align*}
& \left\Vert T\mathbf{1}_{U_{0}}L_{2}^{\mathbf{a}}f\right\Vert _{L^{q}%
}=\left\Vert \mathcal{F}\Phi_{\ast}\mathbf{1}_{U_{0}}\psi L_{2}^{\mathbf{a}%
}f\right\Vert _{L^{q}}=\left\Vert \mathcal{F}\mathbf{1}_{R_{0}}\Phi_{\ast}\psi
L_{2}^{\mathbf{a}}f\right\Vert _{L^{q}}\\
& =\left\Vert \mathcal{F}\mathbf{1}_{R_{0}}\mathcal{F}^{-1}\mathcal{F}%
\Phi_{\ast}\psi L_{2}^{\mathbf{a}}f\right\Vert _{L^{q}}\lesssim\left\Vert
\mathcal{F}\Phi_{\ast}\psi L_{2}^{\mathbf{a}}f\right\Vert _{L^{q}}=\left\Vert
T\psi L_{2}^{\mathbf{a}}f\right\Vert _{L^{q}}\ ,
\end{align*}
where%
\[
\psi L_{2}^{\mathbf{a}}f=\sum_{k=1}^{\infty}\sum_{I\in\mathcal{N}\left(
\pi^{\left(  k\right)  }U_{0}\right)  }a_{I}\left\langle \mathbf{1}_{U_{0}%
}\left(  S_{\kappa,\eta}\right)  ^{-1}f,h_{I;\kappa}\right\rangle \psi
h_{I;\kappa}^{\eta}\ .
\]
Thus we see that $\psi L_{2}^{\mathbf{a}}f$ is smooth and compactly supported
upon using that (\textbf{i}) the functions $\psi h_{\pi^{\left(  k\right)
}U_{0};\kappa}^{\eta}$ are smooth and compactly supported uniformly in $k$ for
$I\in\mathcal{N}\left(  \pi^{\left(  k\right)  }U_{0}\right)  $, and that
(\textbf{ii}) we have the pointwise inequality,
\begin{align*}
& \left\vert \sum_{k=1}^{\infty}\sum_{I\in\mathcal{N}\left(  \pi^{\left(
k\right)  }U_{0}\right)  }a_{I}\left\langle \mathbf{1}_{U_{0}}\left(
S_{\kappa,\eta}\right)  ^{-1}f,h_{I;\kappa}\right\rangle \psi h_{I;\kappa
}^{\eta}\right\vert \lesssim\left\Vert \psi\right\Vert _{L^{\infty}}\sum
_{k=1}^{\infty}\sum_{I\in\mathcal{N}\left(  \pi^{\left(  k\right)  }%
U_{0}\right)  }\left\Vert \mathbf{1}_{U_{0}}\left(  S_{\kappa,\eta}\right)
^{-1}f\right\Vert _{L^{1}}\left\Vert h_{I;\kappa}\right\Vert _{L^{\infty}}%
^{2}\\
& \lesssim\sum_{k=1}^{\infty}\sum_{I\in\mathcal{N}\left(  \pi^{\left(
k\right)  }U_{0}\right)  }\left\Vert \mathbf{1}_{U_{0}}\left(  S_{\kappa,\eta
}\right)  ^{-1}f\right\Vert _{L^{p}}\left\Vert h_{I;\kappa}\right\Vert
_{L^{\infty}}^{2}\lesssim\sum_{k=1}^{\infty}\left\Vert f\right\Vert _{L^{p}%
}\frac{1}{\left\vert \pi^{\left(  k\right)  }U_{0}\right\vert }\lesssim
\left\Vert f\right\Vert _{L^{p}}.
\end{align*}

Consequently, the Fourier transform $\widehat{\Phi_{\ast}\left(  \psi
L_{2}^{\mathbf{a}}f\right)  }$ of the smooth surface measure $\Phi_{\ast
}\left(  \psi L_{2}^{\mathbf{a}}f\right)  $ has decay
\[
\left\vert \widehat{\Phi_{\ast}\left(  \psi L_{2}^{\mathbf{a}}f\right)
}\left(  \xi\right)  \right\vert \lesssim\left\Vert \psi\right\Vert
_{C^{\frac{n}{2}+2}}\left\Vert f\right\Vert _{L^{p}}\left(  1+\left\vert
\xi\right\vert \right)  ^{-\frac{n-1}{2}},
\]
by e.g. \cite[Theorem 1 page 348]{Ste2} or Theorem \ref{osc int} below. Since
this function is in $L^{q}\left(  \mathbb{R}^{n}\right)  $ for all
$q>\frac{2n}{n-1}$, it follows that%
\[
\left\Vert TL_{2}^{\mathbf{a}}f\right\Vert _{L^{q}}\lesssim\left\Vert
f\right\Vert _{L^{p}\left(  U\right)  },
\]
which proves the second line in (\ref{two ineq}), and completes the proof that
(\ref{prob ext''}) implies (\ref{prob ext'}).
\end{proof}

\section{Smooth Alpert frames in $L^{p}$ spaces}

Recall the Alpert projections $\left\{  \bigtriangleup_{Q;\kappa}\right\}
_{Q\in\mathcal{D}}$ and corresponding wavelets $\left\{  h_{Q;\kappa}%
^{a}\right\}  _{Q\in\mathcal{D},\ a\in\Gamma_{n}}$ of order $\kappa$ in
$\mathbb{R}^{n}$ that were constructed in B. Alpert \cite{Alp} - see also
\cite{RaSaWi} for an extension to doubling measures, and for the terminology
we use here. In fact, $\left\{  h_{Q;\kappa}^{a}\right\}  _{a\in\Gamma}$ is an
orthonormal basis for the finite dimensional vector subspace of $L^{2}$ that
consists of linear combinations of the indicators of\ the children
$\mathfrak{C}\left(  Q\right)  $ of $Q$ multiplied by polynomials of degree at
most $\kappa-1$, and such that the linear combinations have vanishing moments
on the cube $Q$ up to order $\kappa-1$:%
\[
L_{Q;k}^{2}\left(  \mu\right)  \equiv\left\{  f=%
%TCIMACRO{\dsum \limits_{Q^{\prime}\in\mathfrak{C}\left(  Q\right)  }}%
%BeginExpansion
{\displaystyle\sum\limits_{Q^{\prime}\in\mathfrak{C}\left(  Q\right)  }}
%EndExpansion
\mathbf{1}_{Q^{\prime}}p_{Q^{\prime};k}\left(  x\right)  :\int_{Q}f\left(
x\right)  x_{i}^{\ell}d\mu\left(  x\right)  =0,\ \ \ \text{for }0\leq\ell\leq
k-1\text{ and }1\leq i\leq n\right\}  ,
\]
where $p_{Q^{\prime};k}\left(  x\right)  =\sum_{\alpha\in\mathbb{Z}_{+}%
^{n}:\left\vert \alpha\right\vert \leq k-1\ }a_{Q^{\prime};\alpha}x^{\alpha}$
is a polynomial in $\mathbb{R}^{n}$ of degree $\left\vert \alpha\right\vert
=\alpha_{1}+...+\alpha_{n}$ at most $\kappa-1$, and $x^{\alpha}=x_{1}%
^{\alpha_{1}}x_{2}^{\alpha_{2}}...x_{n-1}^{\alpha_{n-1}}$. Let $d_{Q;\kappa
}\equiv\dim L_{Q;\kappa}^{2}\left(  \mu\right)  $ be the dimension of the
finite dimensional linear space $L_{Q;\kappa}^{2}\left(  \mu\right)  $.
Moreover, for each $a\in\Gamma_{n}$, we may assume the wavelet $h_{Q;\kappa
}^{a}$ is a translation and dilation of the unit wavelet $h_{Q_{0};\kappa}%
^{a}$, where $Q_{0}=\left[  0,1\right)  ^{n}$ is the unit cube in
$\mathbb{R}^{n}$.

\subsection{Alpert square functions}

It is shown in \cite[Corollary 14]{SaWi} (even for doubling measures in place
of Lebesgue measure) that despite the failure of the $\kappa$-Alpert expansion
to be a martingale when $\kappa\geq2$, Burkholder's proof of the martingale
transform theorem nevertheless carries over to prove, along with Khintchine's
inequalities, that the $L^{p}$ norm of the Alpert square function
$\mathcal{S}f$ of $f$ is comparable to the $L^{p}$ norm of $f$, where%
\[
\mathcal{S}f\left(  x\right)  \equiv\left(  \sum_{Q\in\mathcal{D},\ a\in
\Gamma_{n}}\left\vert \bigtriangleup_{Q;\kappa}^{a}f\left(  x\right)
\right\vert ^{2}\right)  ^{\frac{1}{2}},\ \ \ \ \ x\in\mathbb{R}^{n}.
\]
Of course $\mathcal{S}f$ also depends on the grid $\mathcal{D}$ and $\kappa$,
but we suppress this in the notation.

\begin{theorem}
[Sawyer and Wick \cite{SaWi}]\label{Alpert square thm}For $\kappa\in
\mathbb{N}$ and $1<p<\infty$, we have%
\begin{equation}
\left\Vert \mathcal{S}f\right\Vert _{L^{p}\left(  \mathbb{R}^{n}\right)  }\leq
C_{p,n,\kappa}\left\Vert f\right\Vert _{L^{p}\left(  \mathbb{R}^{n}\right)
}.\label{square}%
\end{equation}

\end{theorem}

Thus the Alpert square function enjoys $L^{p}$ inequalities, whereas
boundedness of the Fourier square function $\mathcal{S}_{T\mathbf{1}_{U_{0}}}
$ for $p>\frac{2n}{n-1}$ is the subject of this paper.

\subsection{Smoothing the Alpert wavelets}

Given a small positive constant $\eta>0$, define a smooth approximate identity
by $\phi_{\eta}\left(  x\right)  \equiv\eta^{-n}\phi\left(  \frac{x}{\eta
}\right)  $ where $\phi\in C_{c}^{\infty}\left(  B_{\mathbb{R}^{n}}\left(
0,1\right)  \right)  $ has unit integral, $\int_{\mathbb{R}^{n}}\phi\left(
x\right)  dx=1$, and vanishing moments of \emph{positive} order less than
$\kappa$, i.e.
\begin{equation}
\int_{\mathbb{R}^{n}}\phi\left(  x\right)  x^{\gamma}dx=\delta_{\left\vert
\gamma\right\vert }^{0}=\left\{
\begin{array}
[c]{ccc}%
1 & \text{ if } & \left\vert \gamma\right\vert =0\\
0 & \text{ if } & 0<\left\vert \gamma\right\vert <\kappa
\end{array}
\right.  .\label{van pos}%
\end{equation}
In fact we may take for $\phi\left(  x\right)  $ a product function
$\phi\left(  x\right)  =\prod_{i=1}^{n}\varphi\left(  x_{i}\right)  $ where
$\varphi\in C_{c}^{\infty}\left(  \left(  -1,1\right)  \right)  $ satisfies%
\begin{equation}
\int_{\mathbb{R}}\varphi\left(  x\right)  x^{\gamma}dx=\left\{
\begin{array}
[c]{ccc}%
1 & \text{ if } & \gamma=0\\
0 & \text{ if } & 0<\gamma<\kappa
\end{array}
\right.  ,\ \ \ \ \ \text{for }1\leq i\leq n.\label{van pos 1}%
\end{equation}
One way to construct a function $\varphi$ satisfying (\ref{van pos 1}) is to
pick $\chi\in C_{c}^{\infty}\left(  \left(  \frac{3}{4},1\right)  \right)  $
with $\int\chi\left(  y\right)  dy=1$, a large $N\in\mathbb{N}$, and then for
$\lambda\equiv\left(  \lambda_{1},...,\lambda_{N}\right)  $ to define,%
\[
\varphi_{\lambda}\left(  x\right)  =\sum_{m=1}^{N}\lambda_{m}\chi\left(
2^{m}x\right)  .
\]

Then with the change of variable $y=2^{m}x$ we have,%
\[
\int\varphi_{\lambda}\left(  x\right)  x^{\gamma}dx=\sum_{m=1}^{N}\lambda
_{m}\int\chi\left(  2^{m}x\right)  x^{\gamma}dx=\sum_{m=1}^{N}\lambda
_{m}2^{-m\left(  \gamma+1\right)  }\int\chi\left(  y\right)  y^{\gamma
}dy=C_{\gamma}\sum_{m=1}^{N}\lambda_{m}2^{-m\left(  \gamma+1\right)  }.
\]
In order to achieve $\int\varphi_{\lambda}\left(  x\right)  x^{\gamma
}dx=\left\{
\begin{array}
[c]{ccc}%
1 & \text{ if } & \gamma=0\\
0 & \text{ if } & 0<\gamma<\kappa
\end{array}
\right.  $ we need to solve the linear system,%
\[
1=\sum_{m=1}^{N}\lambda_{m}2^{-m}\text{ and }0=\sum_{m=1}^{N}\lambda
_{m}2^{-m\left(  \gamma+1\right)  },\ \ \ \ \ \text{for }0<\gamma<\kappa,
\]
which in matrix form is%
\[
\mathbf{e}_{1}=M_{\kappa}\mathbf{\lambda\ .}\ \ \ \ \ \text{where }M_{\kappa
}\equiv\left[  2^{-m\ell}\right]  _{\substack{1\leq m\leq N \\1\leq\ell
\leq\kappa}}\ .
\]
We take $N\geq\kappa$ and observe that the square matrix $M_{\kappa}%
\equiv\left[  2^{-m\ell}\right]  _{\substack{1\leq m\leq\kappa\\1\leq\ell
\leq\kappa}}$ has nonzero determinant, in fact $\left\vert \det M_{\kappa
}\right\vert $ is bounded below by $2^{-\frac{\kappa^{2}\left(  \kappa
-1\right)  }{2}}$. Indeed, the square Vandermonde matrix%
\[
V\left(  x\right)  =V\left(  x_{1},x_{2},...,x_{n}\right)  \equiv\left[
\begin{array}
[c]{cccc}%
x_{1} & x_{1}^{2} & \cdots & x_{1}^{n}\\
x_{2} & x_{2}^{2} & \cdots & x_{2}^{n}\\
\vdots & \vdots & \ddots & \vdots\\
x_{n} & x_{n}^{2} & \cdots & x_{n}^{n}%
\end{array}
\right]
\]
has determinant $\det V\left(  x\right)  =%
%TCIMACRO{\dprod \limits_{1\leq i<j\leq n}}%
%BeginExpansion
{\displaystyle\prod\limits_{1\leq i<j\leq n}}
%EndExpansion
\left(  x_{j}-x_{i}\right)  $. Thus with $x\left(  \kappa\right)  =\left(
2^{-1},2^{-2},...,2^{-\kappa}\right)  \in\mathbb{R}^{\kappa}$, we have
$V\left(  x\left(  \kappa\right)  \right)  =\left[  2^{-m\ell}\right]
_{\substack{1\leq m\leq\kappa\\1\leq\ell\leq\kappa}}=M_{\kappa}$ and so%
\[
\left\vert \det M_{\kappa}\right\vert =%
%TCIMACRO{\dprod \limits_{1\leq i<j\leq\kappa}}%
%BeginExpansion
{\displaystyle\prod\limits_{1\leq i<j\leq\kappa}}
%EndExpansion
\left\vert 2^{-j}-2^{-i}\right\vert \geq%
%TCIMACRO{\dprod \limits_{1\leq i<j\leq\kappa}}%
%BeginExpansion
{\displaystyle\prod\limits_{1\leq i<j\leq\kappa}}
%EndExpansion
2^{-\kappa}=2^{-\kappa\frac{\kappa\left(  \kappa-1\right)  }{2}}.
\]
Thus we can find coefficients $\lambda\equiv\left(  \lambda_{1},...,\lambda
_{N}\right)  $ such that $\varphi=\varphi_{\lambda}$ satisfies
(\ref{van pos 1}).

In the spirit of symbol smoothing for pseudodifferential operators, we define
\emph{smooth} Alpert `wavelets' by%
\[
h_{Q;\kappa}^{a,\eta}\equiv h_{Q;\kappa}^{a}\ast\phi_{\eta\ell\left(
Q\right)  },
\]
and we claim that $h_{Q;\kappa}^{a}$ and $h_{Q;\kappa}^{a,\eta}$ coincide away
from the $\eta$-neighbourhood (often referred to as a `halo')%
\begin{equation}
\mathcal{H}_{\eta}\left(  Q\right)  \equiv\left\{  x\in\mathbb{R}%
^{n}:\operatorname*{dist}\left(  x,S_{Q}\right)  <\eta\right\}
,\label{def halo}%
\end{equation}
of the skeleton $S_{Q}\equiv\bigcup_{Q^{\prime}\in\mathfrak{C}_{\mathcal{D}%
}\left(  Q\right)  }\partial Q^{\prime}$. Note that away from the skeleton,
the Alpert wavelet $h_{Q;\kappa}^{a}$ restricts to a polynomial of degree less
than $\kappa$ on each dyadic child of $Q$. We now show the same for smooth
Alpert wavelets away from the halo of the skeleton.

\begin{lemma}
With notation as above and $\phi$ satisfying (\ref{van pos}), we have%
\begin{equation}
h_{Q;\kappa}^{a}\left(  x\right)  =h_{Q;\kappa}^{a,\eta}\left(  x\right)
,\ \ \ \ \ x\in\mathbb{R}^{n}\setminus\mathcal{H}_{\eta}\left(  Q\right)
.\label{coin}%
\end{equation}

\end{lemma}

\begin{proof}
If $m_{\alpha}\left(  x\right)  \equiv x^{\alpha}=x_{1}^{\alpha_{1}}%
x_{2}^{\alpha_{2}}...x_{n}^{\alpha_{n}}$ is a multinomial, then%
\[
\left(  m_{\alpha}\ast\phi\right)  \left(  x\right)  =\sum_{0\leq\beta
\leq\alpha}\left(  c_{\alpha,\beta}\int y^{\alpha-\beta}\phi\left(  y\right)
dy\right)  x^{\beta}=x^{\alpha}=m_{\alpha}\left(  x\right)  ,
\]

which shows that (\ref{coin}) holds.
\end{proof}

We also observe that for $0\leq\left\vert \beta\right\vert <\kappa$,%
\begin{align*}
& \int h_{Q;\kappa}^{a,\eta}\left(  x\right)  x^{\beta}dx=\int\phi_{\eta
\ell\left(  I\right)  }\ast h_{Q;\kappa}^{a}\left(  x\right)  x^{\beta}%
dx=\int\int\phi_{\eta\ell\left(  I\right)  }\left(  y\right)  h_{Q;\kappa}%
^{a}\left(  x-y\right)  x^{\beta}dx\\
& =\int\phi_{\eta\ell\left(  I\right)  }\left(  y\right)  \left\{  \int
h_{Q;\kappa}^{a}\left(  x-y\right)  x^{\beta}dx\right\}  dy=\int\phi_{\eta
\ell\left(  I\right)  }\left(  y\right)  \left\{  \int h_{Q;\kappa}^{a}\left(
x\right)  \left(  x+y\right)  ^{\beta}dx\right\}  dy\\
& =\int\phi_{\eta\ell\left(  I\right)  }\left(  y\right)  \left\{  0\right\}
dy=0,
\end{align*}
by translation invariance of Lebesgue measure.

\subsection{The reproducing formula}

For the purposes of this subsection we will change notation from that in
Theorem \ref{frame} in the introduction by defining%
\[
\bigtriangleup_{I;\kappa}^{\eta}f\equiv\sum_{a\in\Gamma_{n}}\left\langle
f,h_{I;\kappa}^{a}\right\rangle h_{I;\kappa}^{a,\eta}=\left(  \bigtriangleup
_{I;\kappa}f\right)  \ast\phi_{\eta\ell\left(  I\right)  }\ .
\]
Next, for any grid $\mathcal{D}$, we wish to show that for $\eta>0$
sufficiently small, the linear map $S_{\kappa,\eta}^{\mathcal{D}}$ defined by%
\begin{equation}
S_{\kappa,\eta}^{\mathcal{D}}f\equiv\sum_{I\in\mathcal{D},\ a\in\Gamma_{n}%
}\left\langle f,h_{I;\kappa}^{a}\right\rangle h_{I;\kappa}^{a,\eta}=\sum
_{I\in\mathcal{D}}\bigtriangleup_{I;\kappa}^{\eta}f\ ,\ \ \ \ \ f\in
L^{p},\label{def S_eta}%
\end{equation}
is bounded and invertible on $L^{p}$, and that we have the reproducing
formula,
\[
f\left(  x\right)  =\sum_{I\in\mathcal{D},\ a\in\Gamma_{n}}\left\langle
\left(  S_{\kappa,\eta}^{\mathcal{D}}\right)  ^{-1}f,h_{I;\kappa}%
^{a}\right\rangle \ h_{I;\kappa}^{a,\eta}\left(  x\right)
,\ \ \ \ \ \text{for all }f\in L^{p}\cap L^{2},
\]
with convergence in the $L^{p}$ norm. Since $\kappa>\frac{n}{2}$ is fixed
throughout our arguments we will often write $S_{\eta}^{\mathcal{D}}$ instead
of $S_{\kappa,\eta}^{\mathcal{D}}$ in the sequel.

\begin{proof}
[Proof of Theorem \ref{frame}]Theorem \ref{frame} follows easily, together
with what was proved just above, from Theorem \ref{reproducing} below if we
define the pseudoprojection $\bigtriangleup_{I;\kappa}^{\eta}$ in Theorem
\ref{frame} as the pseudoprojection $\widetilde{\bigtriangleup}_{I;\kappa
}^{\eta}$ in Theorem \ref{reproducing}.
\end{proof}

We include arbitrary grids $\mathcal{D}$ in Theorem \ref{reproducing} since
this may be useful in other contexts where probability of grids plays a role,
originating with the work of Nazarov, Treil and Volberg, see e.g. \cite{NTV4}
and \cite{Vol}, and references given there.

\begin{theorem}
\label{reproducing}Let $n\geq2$ and $\kappa\in\mathbb{N}$ with $\kappa
>\frac{n}{2}$. Then there is $\eta_{0}>0$ depending on $n$ and $\kappa$\ such
that for all $0<\eta<\eta_{0}$, and for all grids $\mathcal{D}$ in
$\mathbb{R}^{n}$, and all $1<p<\infty$, there is a bounded invertible operator
$S_{\eta}^{\mathcal{D}}=S_{\kappa,\eta}^{\mathcal{D}}$ on $L^{p}$, and a
positive constant $C_{p,n,\eta}$ such that the collection of functions
$\left\{  h_{I;\kappa}^{a,\eta}\right\}  _{I\in\mathcal{D},\ a\in\Gamma_{n}}$
is a $C_{p,n,\eta}$-frame for $L^{p}$, by which we mean\footnote{See
\cite{AlLuSa} and \cite{CaHaLa} for more detail on frames in $L^{p}$ spaces.},%
\begin{align}
f\left(  x\right)   & =\sum_{I\in\mathcal{D},\ a\in\Gamma_{n}}\widetilde
{\bigtriangleup}_{I;\kappa}^{\eta}f\left(  x\right)  ,\ \ \ \ \ \text{for a.e.
}x\in\mathbb{R}^{n}\text{, and for all }f\in L^{p},\label{bounded below}\\
\text{where }\widetilde{\bigtriangleup}_{I;\kappa}^{\eta}f  & \equiv\sum
_{a\in\Gamma_{n}}\left\langle \left(  S_{\eta}^{\mathcal{D}}\right)
^{-1}f,h_{I;\kappa}^{a}\right\rangle \ h_{I;\kappa}^{a,\eta}\ ,\nonumber
\end{align}
and with convergence of the sum in the $L^{p}$ norm, and%
\begin{align*}
& \frac{1}{C_{p,n,\eta}}\left\Vert f\right\Vert _{L^{p}}\leq\left\Vert \left(
\sum_{I\in\mathcal{D}}\left\vert \widetilde{\bigtriangleup}_{I;\kappa}^{\eta
}f\right\vert ^{2}\right)  ^{\frac{1}{2}}\right\Vert _{L^{p}},\left\Vert
\left(  \sum_{I\in\mathcal{D}}\left\vert \bigtriangleup_{I;\kappa}^{\eta
}f\right\vert ^{2}\right)  ^{\frac{1}{2}}\right\Vert _{L^{p}}\leq C_{p,n,\eta
}\left\Vert f\right\Vert _{L^{p}},\\
&
\ \ \ \ \ \ \ \ \ \ \ \ \ \ \ \ \ \ \ \ \ \ \ \ \ \ \ \ \ \ \ \ \ \ \ \ \ \ \ \text{for
all }f\in L^{p}.
\end{align*}

\end{theorem}

\begin{notation}
\label{Notation Alpert} We will often drop the index $a$ parameterized by the
finite set $\Gamma_{n}$ as it plays no essential role in most of what follows,
and it will be understood that when we write
\[
\bigtriangleup_{Q;\kappa}^{\eta}f=\left\langle f,h_{Q;\kappa}\right\rangle
h_{Q;\kappa}^{\eta},
\]
we \emph{actually} mean the Alpert \emph{pseudoprojection},%
\[
\bigtriangleup_{Q;\kappa}^{\eta}f=\sum_{a\in\Gamma_{n}}\left\langle
f,h_{Q;\kappa}^{a}\right\rangle h_{Q;\kappa}^{a,\eta}\ .
\]

\end{notation}

Now we turn to two propositions that we will use in the proof of Theorem
\ref{reproducing}.

\begin{proposition}
\label{invert}For $\kappa>\frac{n}{2}$ and $\eta>0$ sufficiently small, we
have%
\[
\left\Vert S_{\eta}^{\mathcal{D}}f\right\Vert _{L^{p}}\approx\left\Vert
f\right\Vert _{L^{p}}\ ,\ \ \ \ \ \text{for }f\in L^{p}\cap L^{2}\text{ and
}1<p<\infty.
\]

\end{proposition}

\begin{proposition}
\label{invert dual}For $\kappa>\frac{n}{2}$ and $\eta>0$ sufficiently small,
we have%
\[
\left\Vert \left(  S_{\eta}^{\mathcal{D}}\right)  ^{\ast}f\right\Vert _{L^{p}%
}\approx\left\Vert f\right\Vert _{L^{p}}\ ,\ \ \ \ \ \text{for }f\in L^{p}\cap
L^{2}\text{ and }1<p<\infty.
\]

\end{proposition}

To prove these propositions, we will need some estimates on the inner products
$\left\langle h_{I;\kappa}^{\eta},h_{Q;\kappa}\right\rangle $ where one
wavelet is smooth and the other is not. Fix a dyadic grid $\mathcal{D}$. We
say that dyadic cubes $Q_{1}$ and $Q_{2}$ are \emph{siblings} if $\ell\left(
Q_{1}\right)  =\ell\left(  Q_{2}\right)  $, $Q_{1}\cap Q_{2}=\emptyset$ and
$\overline{Q_{1}}\cap\overline{Q_{2}}\neq\emptyset$, and we say they are
\emph{dyadic} siblings if in addition they have a common dyadic parent, i.e.
$\pi_{\mathcal{D}}Q_{1}=\pi_{\mathcal{D}}Q_{2}$. Finally, we define
$\operatorname*{Car}\left(  Q\right)  $ to be the set of $I\in\mathcal{D}$
with $\ell\left(  I\right)  <\ell\left(  Q\right)  $ such that $I$ and $Q$
share a face. We refer to these cubes $I$ as Carleson cubes of $Q$, and note
they can be either outside $Q$ or inside $Q$. Finally, we may assume without
loss of generality that $\eta$ is a negative integer power of $2$.

\begin{lemma}
\label{inner est}Suppose $\kappa\in\mathbb{N}$ with $\kappa>\frac{n}{2}$,
$0<\eta=2^{-k}<1$, and $I,Q\in\mathcal{D}$, where $\mathcal{D}$ is a grid in
$\mathbb{R}^{n}$. Then we have%
\begin{align*}
\left\vert \left\langle h_{Q;\kappa}^{\eta},h_{Q;\kappa}\right\rangle
\right\vert  & \approx1\text{ and }\left\vert \left\langle h_{Q;\kappa}^{\eta
},h_{Q^{\prime};\kappa}\right\rangle \right\vert \lesssim\eta
,\ \ \ \ \ \text{for }Q\text{ and }Q^{\prime}\text{ siblings},\\
\left\vert \left\langle h_{I;\kappa}^{\eta},h_{Q;\kappa}\right\rangle
\right\vert  & \lesssim\eta\left(  \frac{\ell\left(  I\right)  }{\ell\left(
Q\right)  }\right)  ^{\frac{n}{2}},\ \ \ \ \ \text{for }I\in
\operatorname*{Car}\left(  Q\right)  ,\\
\left\vert \left\langle h_{I;\kappa}^{\eta},h_{Q;\kappa}\right\rangle
\right\vert  & \lesssim\eta\left(  \frac{\ell\left(  Q\right)  }{\ell\left(
I\right)  }\right)  ^{\frac{n}{2}-1},\ \ \ \ \ \text{for }Q\in
\operatorname*{Car}\left(  I\right)  \text{ and }\ell\left(  Q\right)
\geq\eta\ell\left(  I\right)  ,\\
\left\vert \left\langle h_{I;\kappa}^{\eta},h_{Q;\kappa}\right\rangle
\right\vert  & \lesssim\frac{1}{\eta^{\kappa}}\left(  \frac{\ell\left(
Q\right)  }{\ell\left(  I\right)  }\right)  ^{\kappa+\frac{n}{2}%
},\ \ \ \ \ \text{for }\ell\left(  Q\right)  \leq\eta\ell\left(  I\right)
\text{ and }Q\cap\mathcal{H}_{\frac{\eta}{2}}\left(  I\right)  \neq
\emptyset,\\
\left\langle h_{I;\kappa}^{\eta},h_{Q;\kappa}\right\rangle  &
=0,\ \ \ \ \ \text{in all other cases}.
\end{align*}

\end{lemma}

\begin{proof}
Fix a grid $\mathcal{D}$, and take $0<\eta<1$. We have
\[
\left\langle h_{Q;\kappa}^{\eta},h_{Q;\kappa}\right\rangle =\left\langle
h_{Q;\kappa},h_{Q;\kappa}\right\rangle +\left\langle h_{Q;\kappa}^{\eta
}-h_{Q;\kappa},h_{Q;\kappa}\right\rangle =1+\int_{\mathcal{H}_{\eta}\left(
Q\right)  }\left(  h_{Q;\kappa}^{\eta}-h_{Q;\kappa}\right)  \left(  x\right)
h_{Q;\kappa}\left(  x\right)  dx,
\]
where
\[
\left\vert \int_{\mathcal{H}_{\eta}\left(  Q\right)  }\left(  h_{Q;\kappa
}^{\eta}-h_{Q;\kappa}\right)  \left(  x\right)  h_{Q;\kappa}\left(  x\right)
dx\right\vert \lesssim\left\Vert h_{Q;\kappa}^{\eta}-h_{Q;\kappa}\right\Vert
_{\infty}\left\Vert h_{Q;\kappa}\right\Vert _{\infty}\left\vert \mathcal{H}%
_{\eta}\left(  Q\right)  \right\vert \lesssim\frac{1}{\sqrt{\left\vert
Q\right\vert }}\frac{1}{\sqrt{\left\vert Q\right\vert }}\eta\left\vert
Q\right\vert =\eta.
\]

Next we note that if $I$ is a dyadic cube and $Q\in\operatorname*{Car}\left(
I\right)  $, then $Q\cap\mathcal{H}_{\eta}\left(  I\right)  \neq\emptyset$ and
$\left\langle h_{I;\kappa}^{\eta},h_{Q;\kappa}\right\rangle \neq0$ where
$\eta=2^{-k}$ imply that $\operatorname*{Supp}h_{Q;\kappa}=Q\subset
\mathcal{H}_{\eta}\left(  I\right)  $. If $Q\subset\mathcal{H}_{\eta}\left(
I\right)  $, then we have
\begin{align*}
& \left\langle h_{I;\kappa}^{\eta},h_{Q;\kappa}\right\rangle =\int
_{\mathcal{H}_{\eta}\left(  I\right)  }\mathbf{1}_{Q}h_{I;\kappa}^{\eta
}\left(  x\right)  h_{Q;\kappa}\left(  x\right)  dx=\int_{Q\cap\mathcal{H}%
_{\eta}\left(  I\right)  }\left(  h_{I;\kappa}\ast\phi_{\eta\ell\left(
I\right)  }\right)  \left(  x\right)  h_{Q;\kappa}\left(  x\right)  dx\\
& =\int_{Q\cap\mathcal{H}_{\eta}\left(  I\right)  }\left\{  \int
_{I}h_{I;\kappa}\left(  y\right)  \phi_{\eta\ell\left(  I\right)  }\left(
x-y\right)  dy\right\}  h_{Q;\kappa}\left(  x\right)  dx=\int_{I}h_{I;\kappa
}\left(  y\right)  \left\{  \int_{Q\cap\mathcal{H}_{\eta}\left(  I\right)
}\phi_{\eta\ell\left(  I\right)  }\left(  x-y\right)  h_{Q;\kappa}\left(
x\right)  dx\right\}  dy\\
& =\int_{I\cap2\eta\ell\left(  I\right)  Q}h_{I;\kappa}\left(  y\right)
\left\{  \int_{Q\cap\mathcal{H}_{\eta}\left(  I\right)  }\left[  \phi
_{\eta\ell\left(  I\right)  }\left(  x-y\right)  -\sum_{j=0}^{\kappa-1}\left(
\left(  x-c_{Q}\right)  \cdot\nabla\right)  ^{j}\phi_{\eta\ell\left(
I\right)  }\left(  c_{Q}-y\right)  \right]  h_{Q;\kappa}\left(  x\right)
dx\right\}  dy\\
& \leq\left\Vert h_{I;\kappa}\right\Vert _{\infty}\left\Vert \left(
\nabla^{\kappa}\phi_{\eta\ell\left(  I\right)  }\right)  \right\Vert _{\infty
}\ell\left(  Q\right)  ^{\kappa}\left\Vert h_{Q;\kappa}\right\Vert _{\infty
}\int_{B\left(  c_{Q},\eta\ell\left(  I\right)  \right)  }\int_{Q\cap
\mathcal{H}_{\eta}\left(  I\right)  }dxdy\\
& \lesssim\sqrt{\frac{1}{\left\vert I\right\vert }}\left\Vert \nabla^{\kappa
}\phi\right\Vert _{\infty}\left(  \frac{1}{\eta\ell\left(  I\right)  }\right)
^{n+\kappa}\ell\left(  Q\right)  ^{\kappa}\sqrt{\frac{1}{\left\vert
Q\right\vert }}\left\vert B\left(  c_{Q},\eta\ell\left(  I\right)  \right)
\right\vert \left\vert Q\cap\mathcal{H}_{\eta}\left(  I\right)  \right\vert
\lesssim\frac{1}{\eta^{\kappa}}\left(  \frac{\ell\left(  Q\right)  }%
{\ell\left(  I\right)  }\right)  ^{\kappa+\frac{n}{2}},
\end{align*}
since $\left\Vert h_{I;\kappa}\right\Vert _{\infty}\lesssim\sqrt{\frac
{1}{\left\vert I\right\vert }}$, $\left\Vert h_{Q;\kappa}\right\Vert _{\infty
}\lesssim\sqrt{\frac{1}{\left\vert Q\right\vert }}$ and $\left\Vert
\nabla^{\kappa}\phi_{\eta\ell\left(  I\right)  }\right\Vert _{\infty}%
\leq\left\Vert \nabla^{\kappa}\phi\right\Vert _{\infty}\left(  \frac{1}%
{\eta\ell\left(  I\right)  }\right)  ^{\kappa}$.

If $Q\in\operatorname*{Car}\left(  I\right)  $ and $\ell\left(  Q\right)
\geq\eta\ell\left(  I\right)  $, then we have the trivial estimate%
\[
\left\vert \left\langle h_{I;\kappa}^{\eta},h_{Q;\kappa}\right\rangle
\right\vert \lesssim\eta\ell\left(  I\right)  \ell\left(  Q\right)
^{n-1}\sqrt{\frac{1}{\left\vert I\right\vert \left\vert Q\right\vert }}%
=\eta\left(  \frac{\ell\left(  Q\right)  }{\ell\left(  I\right)  }\right)
^{\frac{n}{2}-1}.
\]

On the other hand, if $I\in\operatorname*{Car}\left(  Q\right)  $, we claim
that
\[
\left\vert \left\langle h_{I;\kappa}^{\eta},h_{Q;\kappa}\right\rangle
\right\vert \lesssim\eta\left(  \frac{\ell\left(  I\right)  }{\ell\left(
Q\right)  }\right)  ^{\frac{n}{2}}.
\]
Indeed, this is clear if $Q\cap I=\emptyset$ since then $\left\vert
\left\langle h_{I;\kappa}^{\eta},h_{Q;\kappa}\right\rangle \right\vert
\leq\eta\left\vert I\right\vert \sqrt{\frac{1}{\left\vert I\right\vert }}%
\sqrt{\frac{1}{\left\vert Q\right\vert }}$, while if $Q^{\prime}%
\in\mathfrak{C}_{D}\left(  I\right)  $ is the child containing $I$, and if
$\varphi\left(  x-c_{Q^{\prime}}\right)  $ is the polynomial whose restriction
to $Q^{\prime}$ is $\left(  \mathbf{1}_{Q^{\prime}}h_{Q;\kappa}\right)
\left(  x\right)  $, then $\left\langle h_{I;\kappa}^{\eta},\varphi
\right\rangle =0$ and so
\[
\left\vert \left\langle h_{I;\kappa}^{\eta},h_{Q;\kappa}\right\rangle
\right\vert =\left\vert \left\langle h_{I;\kappa}^{\eta},h_{Q;\kappa}%
-\varphi\right\rangle \right\vert \lesssim\eta\sqrt{\frac{\left\vert
I\right\vert }{\left\vert Q\right\vert }}=\eta\left(  \frac{\ell\left(
I\right)  }{\ell\left(  Q\right)  }\right)  ^{\frac{n}{2}}.
\]

\end{proof}

We will also need the following consequence of the Marcinkiewicz interpolation theorem.

\begin{lemma}
\label{Marcin}For $1<p<\infty$ and $\kappa\in\mathbb{N}$, we have%
\begin{align*}
& \left\Vert \left(  \sum_{I\in\mathcal{D}}\left(  \frac{\left\vert
\left\langle f,h_{I;\kappa}\right\rangle \right\vert }{\left\vert I\right\vert
^{\frac{1}{2}}}\mathbf{1}_{\mathcal{H}_{\eta}\left(  I\right)  }\left(
x\right)  \right)  ^{2}\right)  ^{\frac{1}{2}}\right\Vert _{L^{p}}\leq
C_{p,n}\eta^{\gamma_{p}}\left\Vert f\right\Vert _{L^{p}},\\
\text{where }\gamma_{p}  & \equiv\left\{
\begin{array}
[c]{ccc}%
\frac{1}{2\left(  p-1\right)  } & \text{ if } & p>2\\
\frac{1}{2} & \text{ if } & p=2\\
\frac{p-1}{p\left(  3-p\right)  } & \text{ if } & 1<p<2
\end{array}
\right.  .
\end{align*}

\end{lemma}

\begin{proof}
Define the square function $\mathcal{R}_{\eta}$ by%
\[
\mathcal{R}_{\eta}f\left(  x\right)  \equiv\left(  \sum_{I\in\mathcal{D}%
}\left(  \frac{\left\vert \left\langle f,h_{I;\kappa}\right\rangle \right\vert
}{\left\vert I\right\vert ^{\frac{1}{2}}}\mathbf{1}_{I\cap\mathcal{H}_{\eta
}\left(  I\right)  }\left(  x\right)  \right)  ^{2}\right)  ^{\frac{1}{2}}.
\]
Using $\mathbf{1}_{\mathcal{H}_{\eta}\left(  I\right)  }\left(  x\right)
\lesssim M\mathbf{1}_{I\cap\mathcal{H}_{\eta}\left(  I\right)  }\left(
x\right)  $, the Fefferman-Stein vector valued maximal inequality \cite{FeSt}
yields,%
\begin{align*}
& \left\Vert \left(  \sum_{I\in\mathcal{D}}\left(  \frac{\left\vert
\left\langle f,h_{I;\kappa}\right\rangle \right\vert }{\left\vert I\right\vert
^{\frac{1}{2}}}\mathbf{1}_{\mathcal{H}_{\eta}\left(  I\right)  }\left(
x\right)  \right)  ^{2}\right)  ^{\frac{1}{2}}\right\Vert _{L^{p}}%
\lesssim\left\Vert \left(  \sum_{I\in\mathcal{D}}\left(  \frac{\left\vert
\left\langle f,h_{I;\kappa}\right\rangle \right\vert }{\left\vert I\right\vert
^{\frac{1}{2}}}M\mathbf{1}_{I\cap\mathcal{H}_{\eta}\left(  I\right)  }\left(
x\right)  \right)  ^{2}\right)  ^{\frac{1}{2}}\right\Vert _{L^{p}}\\
& \ \ \ \ \ \ \ \ \ \ \ \ \ \ \ \ \ \ \ \ \ \ \ \ \ \ \ \ \ \ \lesssim
\left\Vert \left(  \sum_{I\in\mathcal{D}}\left(  \frac{\left\vert \left\langle
f,h_{I;\kappa}\right\rangle \right\vert }{\left\vert I\right\vert ^{\frac
{1}{2}}}\mathbf{1}_{I\cap\mathcal{H}_{\eta}\left(  I\right)  }\left(
x\right)  \right)  ^{2}\right)  ^{\frac{1}{2}}\right\Vert _{L^{p}}=\left\Vert
\mathcal{R}_{\eta}f\left(  x\right)  \right\Vert _{L^{p}}\ .
\end{align*}
Now we note that%
\[
\left\Vert \mathcal{R}_{\eta}f\right\Vert _{L^{p}}\lesssim\left\Vert \left(
\sum_{I\in\mathcal{D}}\left(  \frac{\left\vert \left\langle f,h_{I;\kappa
}\right\rangle \right\vert }{\left\vert I\right\vert ^{\frac{1}{2}}}%
\mathbf{1}_{I}\right)  ^{2}\right)  ^{\frac{1}{2}}\right\Vert _{L^{p}%
}=\left\Vert \left(  \sum_{I\in\mathcal{D}}\left(  \bigtriangleup_{I;\kappa
}f\right)  ^{2}\right)  ^{\frac{1}{2}}\right\Vert _{L^{p}}=\left\Vert
\mathcal{R}f\right\Vert _{L^{p}}\approx\left\Vert f\right\Vert _{L^{p}}%
\]
and
\begin{align*}
\left\Vert \mathcal{R}_{\eta}f\right\Vert _{L^{2}}^{2}  & =\int\sum
_{I\in\mathcal{D}}\left(  \frac{\left\vert \left\langle f,h_{I;\kappa
}\right\rangle \right\vert }{\left\vert I\right\vert ^{\frac{1}{2}}}%
\mathbf{1}_{I\cap\mathcal{H}_{\eta}\left(  I\right)  }\left(  x\right)
\right)  ^{2}dx=\int\sum_{I,I^{\prime}\in\mathcal{D}}\frac{\left\vert
\left\langle f,h_{I;\kappa}\right\rangle \right\vert }{\left\vert I\right\vert
^{\frac{1}{2}}}\frac{\left\vert \left\langle f,h_{I^{\prime};\kappa
}\right\rangle \right\vert }{\left\vert I^{\prime}\right\vert ^{\frac{1}{2}}%
}\mathbf{1}_{I\cap\mathcal{H}_{\eta}\left(  I\right)  }\left(  x\right)
\mathbf{1}_{I^{\prime}\cap\mathcal{H}_{\eta}\left(  I^{\prime}\right)
}\left(  x\right)  dx\\
& =\sum_{I,I^{\prime}\in\mathcal{D}}\frac{\left\vert \left\langle
f,h_{I;\kappa}\right\rangle \right\vert }{\left\vert I\right\vert ^{\frac
{1}{2}}}\frac{\left\vert \left\langle f,h_{I^{\prime};\kappa}\right\rangle
\right\vert }{\left\vert I^{\prime}\right\vert ^{\frac{1}{2}}}\left\vert
I\cap\mathcal{H}_{\eta}\left(  I\right)  \cap I^{\prime}\cap\mathcal{H}_{\eta
}\left(  I^{\prime}\right)  \right\vert \leq\sum_{I,I^{\prime}\in\mathcal{D}%
}\frac{\left\vert \left\langle f,h_{I;\kappa}\right\rangle \right\vert
}{\left\vert I\right\vert ^{\frac{1}{2}}}\frac{\left\vert \left\langle
f,h_{I^{\prime};\kappa}\right\rangle \right\vert }{\left\vert I^{\prime
}\right\vert ^{\frac{1}{2}}}\eta\left\vert I\cap I^{\prime}\right\vert \\
& =\eta\int\sum_{I\in\mathcal{D}}\left(  \frac{\left\vert \left\langle
f,h_{I;\kappa}\right\rangle \right\vert }{\left\vert I\right\vert ^{\frac
{1}{2}}}\mathbf{1}_{I}\left(  x\right)  \right)  ^{2}dx=\eta\int\sum
_{I\in\mathcal{D}}\frac{\left\vert \left\langle f,h_{I;\kappa}\right\rangle
\right\vert ^{2}}{\left\vert I\right\vert }\mathbf{1}_{I}\left(  x\right)
dx=\eta\sum_{I\in\mathcal{D}}\left\vert \left\langle f,h_{I;\kappa
}\right\rangle \right\vert ^{2}=\eta\left\Vert f\right\Vert _{L^{2}}^{2}\ .
\end{align*}
Thus the (linearizable) sublinear operator $\mathcal{R}_{\eta}$ maps
$L^{2}\rightarrow L^{2}$ with bound $B_{2}\equiv\eta^{\frac{1}{2}}$, and
maps$\mathcal{\ }L^{q}\rightarrow L^{q}$ with bound $B_{q}\equiv
C_{n,q}^{\prime}$ for $1<q<\infty$ and $q\neq2$.

In the case $p>2$, let $q=2p$. Then by the scaled Marcinkiewicz theorem
applied to $\mathcal{R}_{\eta}$ with exponents $2$ and $q=2p$, see e.g.
\cite[Remark 29]{Tao2}, we have
\[
\left\Vert \mathcal{R}_{\eta}f\right\Vert _{L^{p}}\leq C_{n,p}^{\prime\prime
}B_{2}^{1-\theta}B_{2p}^{\theta}=C_{n,p}^{\prime\prime}\eta^{\frac{1}%
{2}\left(  1-\theta\right)  }\left(  C_{n,2p}^{\prime}\right)  ^{\theta
}=C_{n,p}\eta^{\frac{1}{2\left(  p-1\right)  }},
\]
with $C_{n,p}=C_{n,p}^{\prime\prime}\left(  C_{n,2p}^{\prime}\right)
^{\frac{p-2}{p-1}}$, since $\frac{1}{p}=\frac{1-\theta}{2}+\frac{\theta}{2p}$
implies $1-\theta=\frac{1}{p-1}$.

In the case $1<p<2$, take $q=\frac{1+p}{2}$ and apply the scaled Marcinkiewicz
theorem to $\mathcal{R}_{\eta}$ with exponents $2$ and $q=\frac{1+p}{2}$ to
obtain%
\[
\left\Vert \mathcal{R}_{\eta}f\right\Vert _{L^{p}}\leq C_{n,p}^{\prime\prime
}B_{2}^{1-\theta}B_{\frac{1+p}{2}}^{\theta}=C_{n,p}^{\prime\prime}\eta
^{\frac{1}{2}\left(  1-\theta\right)  }\left(  C_{n,\frac{1+p}{2}}^{\prime
}\right)  ^{\theta}=C_{n,p}\eta^{\frac{p-1}{p\left(  3-p\right)  }},
\]
with $C_{n,p}=C_{n,p}^{\prime\prime}\left(  C_{n,\frac{1+p}{2}}^{\prime
}\right)  ^{\theta}$, since $\frac{1}{p}=\frac{1-\theta}{2}+\frac{\theta
}{\frac{1+p}{2}}$ implies $1-\theta=\frac{2p-2}{p\left(  3-p\right)  }$.
\end{proof}

\subsubsection{Injectivity}

We can now prove Proposition \ref{invert}.

\begin{proof}
[Proof of Proposition \ref{invert}]We have%
\[
S_{\eta}^{\mathcal{D}}f=\sum_{Q\in\mathcal{D}}\bigtriangleup_{Q;\kappa}%
S_{\eta}f=\sum_{Q\in\mathcal{D}}\left\langle S_{\eta}f,h_{Q;\kappa
}\right\rangle h_{Q;\kappa}=\sum_{Q\in\mathcal{D}}\left\langle \sum
_{I\in\mathcal{D}}\left\langle f,h_{I;\kappa}\right\rangle h_{I;\kappa}^{\eta
},h_{Q;\kappa}\right\rangle h_{Q;\kappa}=\sum_{Q,I\in\mathcal{D}}\left\langle
f,h_{I;\kappa}\right\rangle \left\langle h_{I;\kappa}^{\eta},h_{Q;\kappa
}\right\rangle h_{Q;\kappa}\ ,
\]
and%
\begin{align*}
& \left\Vert S_{\eta}^{\mathcal{D}}f\right\Vert _{L^{p}}\approx\left\Vert
\left(  \sum_{Q\in\mathcal{D}}\left\vert \left\langle S_{\eta}f,h_{Q;\kappa
}\right\rangle h_{Q;\kappa}\right\vert ^{2}\right)  ^{\frac{1}{2}}\right\Vert
_{L^{p}}=\left\Vert \left(  \sum_{Q\in\mathcal{D}}\left\vert \sum
_{I\in\mathcal{D}}\left\langle f,h_{I;\kappa}\right\rangle \left\langle
h_{I;\kappa}^{\eta},h_{Q;\kappa}\right\rangle h_{Q;\kappa}\right\vert
^{2}\right)  ^{\frac{1}{2}}\right\Vert _{L^{p}}\\
& \approx\left\Vert \left(  \sum_{Q\in\mathcal{D}}\left\vert \left\langle
f,h_{Q;\kappa}\right\rangle \left\langle h_{Q;\kappa}^{\eta},h_{Q;\kappa
}\right\rangle \right\vert ^{2}\left\vert h_{Q;\kappa}\right\vert ^{2}\right)
^{\frac{1}{2}}\right\Vert _{L^{p}}+O\left(  \left\Vert \left(  \sum
_{Q\in\mathcal{D}}\left\vert \sum_{I\in\mathcal{D}:\ I\neq Q}\left\langle
f,h_{I;\kappa}\right\rangle \left\langle h_{I;\kappa}^{\eta},h_{Q;\kappa
}\right\rangle \right\vert ^{2}\left\vert h_{Q;\kappa}\right\vert ^{2}\right)
^{\frac{1}{2}}\right\Vert _{L^{p}}\right) \\
& \approx\left\Vert \left(  \sum_{Q\in\mathcal{D}}\left\vert \left\langle
f,h_{Q;\kappa}\right\rangle \right\vert ^{2}\frac{1}{\left\vert Q\right\vert
}\mathbf{1}_{Q}\right)  ^{\frac{1}{2}}\right\Vert _{L^{p}}^{p}+O\left(
\left\Vert \left(  \sum_{Q\in\mathcal{D}}\frac{1}{\left\vert Q\right\vert
}\left\vert \sum_{I\in\mathcal{D}:\ I\neq Q}\left\langle f,h_{I;\kappa
}\right\rangle \left\langle h_{I;\kappa}^{\eta},h_{Q;\kappa}\right\rangle
\right\vert ^{2}\mathbf{1}_{Q}\right)  ^{\frac{1}{2}}\right\Vert _{L^{p}%
}\right)  ,
\end{align*}
where by the Alpert square function estimate (\ref{square}),
\[
C_{p}\left\Vert f\right\Vert _{L^{p}}^{p}\geq\left\Vert \left(  \sum
_{Q\in\mathcal{D}}\left\vert \left\langle f,h_{Q;\kappa}\right\rangle
\right\vert ^{2}\frac{1}{\left\vert Q\right\vert }\mathbf{1}_{Q}\right)
^{\frac{1}{2}}\right\Vert _{L^{p}}^{p}=\left\Vert \left(  \sum_{Q\in
\mathcal{D}}\left\vert \bigtriangleup_{Q;\kappa}f\right\vert ^{2}\right)
^{\frac{1}{2}}\right\Vert _{L^{p}}^{p}\geq c_{p}\left\Vert f\right\Vert
_{L^{p}}^{p}\ ,
\]
for some $C_{p},c_{p}>0$.

Thus we have for each $Q\in\mathcal{D}$,
\begin{align*}
& \sum_{I\in\mathcal{D}:\ I\neq Q}\left\langle f,h_{I}\right\rangle
\left\langle h_{I}^{\eta},h_{Q}\right\rangle =\sum_{\substack{I\in
\mathcal{D}:\ \ell\left(  I\right)  <\ell\left(  Q\right)  \\I\in
\operatorname*{Car}\left(  Q\right)  }}\left\langle f,h_{I}\right\rangle
\left\langle h_{I}^{\eta},h_{Q}\right\rangle +\sum_{\substack{I\in
\mathcal{D}:\ \ell\left(  I\right)  >\ell\left(  Q\right)  \\Q\cap
\mathcal{H}_{\frac{\eta}{2}}\left(  I\right)  \neq\emptyset}}\left\langle
f,h_{I}\right\rangle \left\langle h_{I}^{\eta},h_{Q}\right\rangle \\
&
\ \ \ \ \ \ \ \ \ \ \ \ \ \ \ \ \ \ \ \ \ \ \ \ \ \ \ \ \ \ \ \ \ \ \ \ \ \ \ \ \ \ \ \ \ \ \ +\sum
_{\substack{I\in\mathcal{D}:\ \ell\left(  I\right)  \geq\ell\left(  Q\right)
\geq\eta\ell\left(  I\right)  \\Q\in\operatorname*{Car}\left(  I\right)
}}\left\langle f,h_{I}\right\rangle \left\langle h_{I}^{\eta},h_{Q}%
\right\rangle .
\end{align*}
As a consequence of the estimates in Lemma \ref{inner est}, we have for each
$Q\in\mathcal{D}$,%
\begin{align*}
\left\vert \sum_{I\in\mathcal{D}:\ I\neq Q}\left\langle f,h_{I;\kappa
}\right\rangle \left\langle h_{I;\kappa}^{\eta},h_{Q}\right\rangle
\right\vert  & \lesssim\eta\sum_{\substack{I\in\mathcal{D}:\ \ell\left(
I\right)  <\ell\left(  Q\right)  \\I\in\operatorname*{Car}\left(  Q\right)
}}\left\vert \left\langle f,h_{I;\kappa}\right\rangle \right\vert \left(
\frac{\ell\left(  I\right)  }{\ell\left(  Q\right)  }\right)  ^{\frac{n}{2}%
}+\sum_{\substack{I\in\mathcal{D}:\ \ell\left(  Q\right)  \leq\eta\ell\left(
I\right)  \\Q\cap\mathcal{H}_{\frac{\eta}{2}}\left(  I\right)  \neq\emptyset
}}\left\vert \left\langle f,h_{I;\kappa}\right\rangle \right\vert \frac
{1}{\eta^{\kappa}}\left(  \frac{\ell\left(  Q\right)  }{\ell\left(  I\right)
}\right)  ^{\kappa+\frac{n}{2}}\\
& \ \ \ \ \ \ \ \ \ \ \ \ \ \ \ \ \ \ \ \ +\left\vert \sum_{\substack{I\in
\mathcal{D}:\ \ell\left(  I\right)  \geq\ell\left(  Q\right)  \geq\eta
\ell\left(  I\right)  \\Q\in\operatorname*{Car}\left(  I\right)
}}\left\langle f,h_{I;\kappa}\right\rangle \left\langle h_{I;\kappa}^{\eta
},h_{Q}\right\rangle \right\vert \\
& \equiv A\left(  Q\right)  +B\left(  Q\right)  +C\left(  Q\right)  .
\end{align*}

Altogether we have%
\begin{align}
& \left\Vert \left(  \sum_{Q\in\mathcal{D}}\frac{1}{\left\vert Q\right\vert
}\left\vert \sum_{I\in\mathcal{D}:\ I\neq Q}\left\langle f,h_{I;\kappa
}\right\rangle \left\langle h_{I;\kappa}^{\eta},h_{Q;\kappa}\right\rangle
\right\vert ^{2}\mathbf{1}_{Q}\right)  ^{\frac{1}{2}}\right\Vert _{L^{p}%
}\lesssim\left\Vert \left(  \sum_{Q\in\mathcal{D}}\frac{1}{\left\vert
Q\right\vert }A\left(  Q\right)  ^{2}\mathbf{1}_{Q}\right)  ^{\frac{1}{2}%
}\right\Vert _{L^{p}}\label{two norms}\\
& \ \ \ \ \ \ \ \ \ \ \ \ \ \ \ +\left\Vert \left(  \sum_{Q\in\mathcal{D}%
}\frac{1}{\left\vert Q\right\vert }B\left(  Q\right)  ^{2}\mathbf{1}%
_{Q}\right)  ^{\frac{1}{2}}\right\Vert _{L^{p}}+\left\Vert \left(  \sum
_{Q\in\mathcal{D}}\frac{1}{\left\vert Q\right\vert }C\left(  Q\right)
^{2}\mathbf{1}_{Q}\right)  ^{\frac{1}{2}}\right\Vert _{L^{p}}.\nonumber
\end{align}

We now claim that
\begin{equation}
\left\Vert \left(  \sum_{Q\in\mathcal{D}}\frac{1}{\left\vert Q\right\vert
}\left\vert \sum_{I\in\mathcal{D}:\ I\neq Q}\left\langle f,h_{I;\kappa
}\right\rangle \left\langle h_{I;\kappa}^{\eta},h_{Q;\kappa}\right\rangle
\right\vert ^{2}\mathbf{1}_{Q}\right)  ^{\frac{1}{2}}\right\Vert _{L^{p}%
}\lesssim\eta^{\frac{1}{2}\gamma_{p}}\left(  \log_{2}\frac{1}{\eta}\right)
\left\Vert f\right\Vert _{L^{p}}.\label{claim that}%
\end{equation}
With this established, and since $\kappa>\frac{n}{2}$, we obtain$\ $%
\[
\left\Vert \left(  \sum_{Q\in\mathcal{D}}\frac{1}{\left\vert Q\right\vert
}\left\vert \sum_{I\in\mathcal{D}:\ I\neq Q}\left\langle f,h_{I;\kappa
}\right\rangle \left\langle h_{I;\kappa}^{\eta},h_{Q;\kappa}\right\rangle
\right\vert ^{2}\right)  ^{\frac{1}{2}}\right\Vert _{L^{p}}\leq C\eta
^{\frac{1}{2}\gamma_{p}}\left(  \log_{2}\frac{1}{\eta}\right)  \left\Vert
f\right\Vert _{L^{p}}<\frac{c_{p}}{2}\left\Vert f\right\Vert _{L^{p}}\ ,
\]
with$\ \eta>0$ sufficiently small. This then gives%
\[
C_{p}\left\Vert f\right\Vert _{L^{p}}\geq\left\Vert S_{\eta}^{\mathcal{D}%
}f\right\Vert _{L^{p}}\geq c_{p}\left\Vert f\right\Vert _{L^{p}}-\frac{c_{p}%
}{2}\left\Vert f\right\Vert _{L^{p}}=\frac{c_{p}}{2}\left\Vert f\right\Vert
_{L^{p}}\ ,
\]
which completes the proof of Proposition \ref{invert} modulo (\ref{claim that}).

We prove (\ref{claim that}) by estimating each of the three terms on the right
hand side of (\ref{two norms}) separately, beginning with the term involving
$A\left(  Q\right)  $.

\textbf{Case} $A\left(  Q\right)  $: For each $Q\in\mathcal{D}$, we have for
$0<\varepsilon<1\,$ and $0<\gamma<n-\varepsilon$,%
\begin{align*}
& A\left(  Q\right)  =\eta\sum_{\substack{I\in\mathcal{D}:\ \ell\left(
I\right)  <\ell\left(  Q\right)  \\I\in\operatorname*{Car}\left(  Q\right)
}}\left\vert \left\langle f,h_{I;\kappa}\right\rangle \right\vert \left(
\frac{\ell\left(  I\right)  }{\ell\left(  Q\right)  }\right)  ^{\frac{n}{2}%
}=\eta\sum_{t=1}^{\infty}\sum_{\substack{I\in\mathcal{D}:\ \ell\left(
I\right)  =2^{-t}\ell\left(  Q\right)  \\I\in\operatorname*{Car}\left(
Q\right)  }}\left\vert \left\langle f,h_{I;\kappa}\right\rangle \right\vert
2^{-t\frac{n}{2}}\\
& \lesssim\eta\sum_{t=1}^{\infty}\sqrt{\sum_{\substack{I\in\mathcal{D}%
:\ \ell\left(  I\right)  =2^{-t}\ell\left(  Q\right)  \\I\in
\operatorname*{Car}\left(  Q\right)  }}\left\vert \left\langle f,h_{I;\kappa
}\right\rangle \right\vert ^{2}2^{-t\left(  n-\varepsilon\right)  }}=\eta
\sum_{t=1}^{\infty}2^{-t\frac{n-\varepsilon-\gamma}{2}}\sqrt{\sum
_{\substack{I\in\mathcal{D}:\ \ell\left(  I\right)  =2^{-t}\ell\left(
Q\right)  \\I\in\operatorname*{Car}\left(  Q\right)  }}2^{-t\gamma}\left\vert
\left\langle f,h_{I;\kappa}\right\rangle \right\vert ^{2}}\\
& \leq\eta\sqrt{\sum_{t=1}^{\infty}2^{-t\left(  n-\varepsilon-\gamma\right)
}}\sqrt{\sum_{t=1}^{\infty}\sum_{\substack{I\in\mathcal{D}:\ \ell\left(
I\right)  =2^{-t}\ell\left(  Q\right)  \\I\in\operatorname*{Car}\left(
Q\right)  }}2^{-t\gamma}\left\vert \left\langle f,h_{I;\kappa}\right\rangle
\right\vert ^{2}}=\eta\sqrt{\frac{2^{-\left(  n-\varepsilon-\gamma\right)  }%
}{1-2^{-\left(  n-\varepsilon-\gamma\right)  }}}\sqrt{\sum_{t=1}^{\infty}%
\sum_{\substack{I\in\mathcal{D}:\ \ell\left(  I\right)  =2^{-t}\ell\left(
Q\right)  \\I\in\operatorname*{Car}\left(  Q\right)  }}2^{-t\gamma}\left\vert
\left\langle f,h_{I;\kappa}\right\rangle \right\vert ^{2}}.
\end{align*}
and so%
\[
A\left(  Q\right)  =\eta\sum_{\substack{I\in\mathcal{D}:\ \ell\left(
I\right)  <\ell\left(  Q\right)  \\I\in\operatorname*{Car}\left(  Q\right)
}}\left\vert \left\langle f,h_{I;\kappa}\right\rangle \right\vert \left(
\frac{\ell\left(  I\right)  }{\ell\left(  Q\right)  }\right)  ^{\frac{n}{2}%
}\leq\eta\sqrt{\sum_{t=1}^{\infty}\sum_{\substack{I\in\mathcal{D}%
:\ \ell\left(  I\right)  =2^{-t}\ell\left(  Q\right)  \\I\in
\operatorname*{Car}\left(  Q\right)  }}2^{-t\left(  n-2\varepsilon\right)
}\left\vert \left\langle f,h_{I;\kappa}\right\rangle \right\vert ^{2}}%
\]
if we take $\gamma=n-2\varepsilon$. It follows that%
\begin{align*}
& \left\Vert \left(  \sum_{Q\in\mathcal{D}}\frac{1}{\left\vert Q\right\vert
}A\left(  Q\right)  ^{2}\mathbf{1}_{Q}\right)  ^{\frac{1}{2}}\right\Vert
_{L^{p}}\lesssim\eta\left\Vert \left(  \sum_{Q\in\mathcal{D}}\frac
{1}{\left\vert Q\right\vert }\sum_{t=1}^{\infty}\sum_{\substack{I\in
\mathcal{D}:\ \ell\left(  I\right)  =2^{-t}\ell\left(  Q\right)
\\I\in\operatorname*{Car}\left(  Q\right)  }}2^{-t\left(  n-2\varepsilon
\right)  }\left\vert \left\langle f,h_{I;\kappa}\right\rangle \right\vert
^{2}\mathbf{1}_{Q}\right)  ^{\frac{1}{2}}\right\Vert _{L^{p}}\\
& =\eta\left\Vert \left(  \sum_{I\in\mathcal{D}}\left\vert \left\langle
f,h_{I;\kappa}\right\rangle \right\vert ^{2}\sum_{t=1}^{\infty}\frac
{1}{\left\vert Q\right\vert }\sum_{\substack{Q\in\mathcal{D}:\ \ell\left(
I\right)  =2^{-t}\ell\left(  Q\right)  \\I\in\operatorname*{Car}\left(
Q\right)  }}2^{-t\left(  n-2\varepsilon\right)  }\mathbf{1}_{Q}\right)
^{\frac{1}{2}}\right\Vert _{L^{p}}\\
& \leq\eta\left\Vert \left(  \sum_{I\in\mathcal{D}}\left\vert \left\langle
f,h_{I;\kappa}\right\rangle \right\vert ^{2}\sum_{t=1}^{\infty}\frac
{1}{\left\vert 2^{t}I\right\vert }2^{-t\left(  n-2\varepsilon\right)
}\mathbf{1}_{2^{t}I}\right)  ^{\frac{1}{2}}\right\Vert _{L^{p}}\leq
\eta\left\Vert \left(  \sum_{I\in\mathcal{D}}\frac{\left\vert \left\langle
f,h_{I;\kappa}\right\rangle \right\vert ^{2}}{\left\vert I\right\vert }%
\sum_{t=1}^{\infty}2^{-2tn+2\varepsilon t}\mathbf{1}_{2^{t}I}\right)
^{\frac{1}{2}}\right\Vert _{L^{p}}\\
& \lesssim\eta\left\Vert \left(  \sum_{I\in\mathcal{D}}\frac{\left\vert
\left\langle f,h_{I;\kappa}\right\rangle \right\vert ^{2}}{\left\vert
I\right\vert }\left(  M\mathbf{1}_{I}\right)  ^{2\frac{2-2\varepsilon}{2}%
}\right)  ^{\frac{1}{2}}\right\Vert _{L^{p}}\lesssim\eta\left\Vert \left(
\sum_{I\in\mathcal{D}}\frac{\left\vert \left\langle f,h_{I;\kappa
}\right\rangle \right\vert ^{2}}{\left\vert I\right\vert }\left(
M_{r}\mathbf{1}_{I}\right)  ^{2}\right)  ^{\frac{1}{2}}\right\Vert _{L^{p}}\\
& \lesssim\eta\left\Vert \left(  \sum_{I\in\mathcal{D}}\frac{\left\vert
\left\langle f,h_{I;\kappa}\right\rangle \right\vert ^{2}}{\left\vert
I\right\vert }\mathbf{1}_{I}\right)  ^{\frac{1}{2}}\right\Vert _{L^{p}}%
\approx\eta\left\Vert f\right\Vert _{L^{p}},
\end{align*}
provided $1<r=\frac{2}{2-2\varepsilon}=\frac{1}{1-\varepsilon}<p$. Indeed,
\[
\sum_{t=1}^{\infty}2^{-2tn+2\varepsilon t}\mathbf{1}_{2^{t}I}\lesssim\left(
M\mathbf{1}_{I}\right)  ^{2\frac{2-2\varepsilon}{2}}=\left(  M_{r}%
\mathbf{1}_{I}\right)  ^{2},
\]
where the inequality follows from%
\begin{align*}
& \sum_{t=1}^{\infty}2^{-2tn+2\varepsilon t}\mathbf{1}_{2^{t}I}\left(
x\right)  \approx\sum_{t=1}^{\infty}2^{-2tn+2\varepsilon t}\mathbf{1}%
_{2^{t}I-2^{t-1}I}\left(  x\right) \\
& =\sum_{t=1}^{\infty}2^{-2tn\left(  1-\frac{\varepsilon}{n}\right)
}\mathbf{1}_{2^{t}I-2^{t-1}I}\left(  x\right)  \lesssim\sum_{t=1}^{\infty
}M\mathbf{1}_{I}\left(  x\right)  ^{2\left(  1-\frac{\varepsilon}{n}\right)
}\mathbf{1}_{2^{t}I-2^{t-1}I}\left(  x\right)  =M\mathbf{1}_{I}\left(
x\right)  ^{2\left(  1-\frac{\varepsilon}{n}\right)  },
\end{align*}
and the equality follows by definition of $M_{r}$ and since $\mathbf{1}%
_{I}=\left(  \mathbf{1}_{I}\right)  ^{r}$, namely%
\[
\left(  M\mathbf{1}_{I}\right)  ^{2\frac{2-2\varepsilon}{2}}=\left(  \left(
M\left(  \mathbf{1}_{I}\right)  ^{r}\right)  ^{\frac{1}{r}}\right)
^{2}=\left(  M_{r}\mathbf{1}_{I}\right)  ^{2}.
\]

\textbf{Case} $B\left(  Q\right)  $: Set $\eta=2^{-\beta}$. Note that the
function squared in the second norm in (\ref{two norms}) then satisfies%
\begin{align*}
& \sum_{Q\in\mathcal{D}}\frac{1}{\left\vert Q\right\vert }B\left(  Q\right)
^{2}\mathbf{1}_{Q}\left(  x\right)  =\sum_{Q\in\mathcal{D}}\frac{1}{\left\vert
Q\right\vert }\left(  \sum_{\substack{I\in\mathcal{D}:\ \ell\left(  Q\right)
\leq\eta\ell\left(  I\right)  \\Q\cap\mathcal{H}_{\frac{\eta}{2}}\left(
I\right)  \neq\emptyset}}\left\vert \left\langle f,h_{I;\kappa}\right\rangle
\right\vert \frac{1}{\eta^{\kappa}}\left(  \frac{\ell\left(  Q\right)  }%
{\ell\left(  I\right)  }\right)  ^{\kappa+\frac{n}{2}}\right)  ^{2}%
\mathbf{1}_{Q}\left(  x\right) \\
& =\frac{1}{\eta^{2\kappa}}\sum_{Q\in\mathcal{D}}\frac{1}{\left\vert
Q\right\vert }\sum_{\substack{I\in\mathcal{D}:\ \ell\left(  Q\right)  \leq
\eta\ell\left(  I\right)  \\Q\cap\mathcal{H}_{\frac{\eta}{2}}\left(  I\right)
\neq\emptyset}}\sum_{\substack{I^{\prime}\in\mathcal{D}:\ \ell\left(
Q\right)  \leq\eta\ell\left(  I^{\prime}\right)  \\Q\cap\mathcal{H}%
_{\frac{\eta}{2}}\left(  I^{\prime}\right)  \neq\emptyset}}\left\vert
\left\langle f,h_{I;\kappa}\right\rangle \right\vert \left\vert \left\langle
f,h_{I^{\prime};\kappa}\right\rangle \right\vert \left(  \frac{\ell\left(
Q\right)  }{\ell\left(  I\right)  }\right)  ^{\kappa+\frac{n}{2}}\left(
\frac{\ell\left(  Q\right)  }{\ell\left(  I^{\prime}\right)  }\right)
^{\kappa+\frac{n}{2}}\mathbf{1}_{Q}\left(  x\right) \\
& =\frac{1}{\eta^{2\kappa}}2\sum_{I,I^{\prime}\in\mathcal{D}\text{ and
}I\subset I^{\prime}}\left\vert \left\langle f,h_{I;\kappa}\right\rangle
\right\vert \left\vert \left\langle f,h_{I^{\prime};\kappa}\right\rangle
\right\vert \left(  \frac{1}{\ell\left(  I\right)  \ell\left(  I^{\prime
}\right)  }\right)  ^{\kappa+\frac{n}{2}}\sum_{\substack{Q\in\mathcal{D}%
:\ \ell\left(  Q\right)  \leq\eta\ell\left(  I\right)  \\_{Q\cap
\mathcal{H}_{\frac{\eta}{2}}\left(  I\right)  \neq\emptyset}}}\ell\left(
Q\right)  ^{2\kappa}\mathbf{1}_{Q}\left(  x\right) \\
& \approx\frac{1}{\eta^{2\kappa}}\sum_{I,I^{\prime}\in\mathcal{D}\text{ and
}I\subset I^{\prime}}\left\vert \left\langle f,h_{I;\kappa}\right\rangle
\right\vert \left\vert \left\langle f,h_{I^{\prime};\kappa}\right\rangle
\right\vert \left(  \frac{1}{\ell\left(  I\right)  \ell\left(  I^{\prime
}\right)  }\right)  ^{\kappa+\frac{n}{2}}\ell\left(  I\right)  ^{2\kappa}%
\sum_{t=\beta}^{\infty}\sum_{\substack{Q\in\mathcal{D}:\ \ell\left(  Q\right)
=2^{-t}\ell\left(  I\right)  \\_{Q\cap\mathcal{H}_{\frac{\eta}{2}}\left(
I\right)  \neq\emptyset}}}\mathbf{1}_{Q}\left(  x\right)  2^{-t2\kappa},
\end{align*}
where for $t\geq\beta$ and $x\in\mathcal{H}_{\frac{\eta}{2}}\left(  I\right)
$, we have%
\[
\sum_{\substack{Q\in\mathcal{D}:\ \ell\left(  Q\right)  =2^{-t}\ell\left(
I\right)  \\_{Q\cap\mathcal{H}_{\frac{\eta}{2}}\left(  I\right)  \neq
\emptyset}}}\mathbf{1}_{Q}\left(  x\right)  \leq1,
\]
so that%
\[
\sum_{Q\in\mathcal{D}}\frac{1}{\left\vert Q\right\vert }B\left(  Q\right)
^{2}\mathbf{1}_{Q}\left(  x\right)  \lesssim\frac{1}{\eta^{2\kappa}}%
\sum_{I,I^{\prime}\in\mathcal{D}\text{ and }I\subset I^{\prime}}\left\vert
\left\langle f,h_{I;\kappa}\right\rangle \right\vert \left\vert \left\langle
f,h_{I^{\prime};\kappa}\right\rangle \right\vert \left(  \frac{1}{\ell\left(
I\right)  \ell\left(  I^{\prime}\right)  }\right)  ^{\kappa+\frac{n}{2}}%
\ell\left(  I\right)  ^{2\kappa}\sum_{t=\beta}^{\infty}2^{-t2\kappa}%
\mathbf{1}_{\mathcal{H}_{\frac{\eta}{2}}\left(  I\right)  }\left(  x\right)  .
\]

Now recalling $2^{-t}=\frac{\ell\left(  Q\right)  }{\ell\left(  I\right)  }$,
we have for $t\geq\beta$,
\[
\#\left\{  Q\in\mathcal{D}:\ \operatorname*{dist}\left(  Q,\partial I\right)
\geq\ell\left(  Q\right)  =2^{-t}\ell\left(  I\right)  \text{ and }%
Q\cap\mathcal{H}_{\frac{\eta}{2}}\left(  I\right)  \neq\emptyset\right\}
\text{ is }\left\{
\begin{array}
[c]{ccc}%
\approx\eta2^{tn} & \text{ if } & t\geq\beta\\
0 & \text{ if } & 1\leq t<\beta
\end{array}
\right.  .
\]

Our blanket assumption that $\kappa>\frac{n}{2}$ shows that all of the
geometric series appearing below are convergent. Then we have%
\begin{align*}
\sum_{Q\in\mathcal{D}_{\operatorname{good}}}\frac{1}{\left\vert Q\right\vert
}B\left(  Q\right)  ^{2}\mathbf{1}_{Q}\left(  x\right)   & \lesssim\frac
{1}{\eta^{2\kappa}}\sum_{I,I^{\prime}\in\mathcal{D}\text{ and }I\subset
I^{\prime}}\frac{\left\vert \left\langle f,h_{I;\kappa}\right\rangle
\right\vert \left\vert \left\langle f,h_{I^{\prime};\kappa}\right\rangle
\right\vert }{\ell\left(  I\right)  ^{\frac{n}{2}}\ell\left(  I^{\prime
}\right)  ^{\frac{n}{2}}}\left(  \frac{\ell\left(  I\right)  }{\ell\left(
I^{\prime}\right)  }\right)  ^{\kappa}\sum_{t=\beta}^{\infty}2^{-t2\kappa
}\mathbf{1}_{\mathcal{H}_{\frac{\eta}{2}}\left(  I\right)  }\left(  x\right)
\\
& \lesssim\frac{1}{\eta^{2\kappa}}\sum_{I,I^{\prime}\in\mathcal{D}\text{ and
}I\subset I^{\prime}}\frac{\left\vert \left\langle f,h_{I;\kappa}\right\rangle
\right\vert \left\vert \left\langle f,h_{I^{\prime};\kappa}\right\rangle
\right\vert }{\ell\left(  I\right)  ^{\frac{n}{2}}\ell\left(  I^{\prime
}\right)  ^{\frac{n}{2}}}\left(  \frac{\ell\left(  I\right)  }{\ell\left(
I^{\prime}\right)  }\right)  ^{\kappa}\frac{2^{-\beta2\kappa}}{1-2^{-2\kappa}%
}\mathbf{1}_{\mathcal{H}_{\frac{\eta}{2}}\left(  I\right)  }\left(  x\right)
\\
& \lesssim\sum_{I,I^{\prime}\in\mathcal{D}\text{ and }I\subset I^{\prime}%
}\frac{\left\vert \left\langle f,h_{I;\kappa}\right\rangle \right\vert
\left\vert \left\langle f,h_{I^{\prime};\kappa}\right\rangle \right\vert
}{\ell\left(  I\right)  ^{\frac{n}{2}}\ell\left(  I^{\prime}\right)
^{\frac{n}{2}}}\left(  \frac{\ell\left(  I\right)  }{\ell\left(  I^{\prime
}\right)  }\right)  ^{\kappa}\mathbf{1}_{\mathcal{H}_{\frac{\eta}{2}}\left(
I\right)  }\left(  x\right)  ,
\end{align*}
which in turn equals,%
\begin{align*}
& \sum_{I\in\mathcal{D}}\sum_{s=1}^{\infty}\frac{\left\vert \left\langle
f,h_{I;\kappa}\right\rangle \right\vert }{\sqrt{\left\vert I\right\vert }\ell
}\frac{\left\vert \left\langle f,h_{\left(  \pi^{\left(  s\right)  }I\right)
;\kappa}\right\rangle \right\vert }{\sqrt{\left\vert \pi^{\left(  s\right)
}I\right\vert }}\left(  \frac{\ell\left(  I\right)  }{\ell\left(  \pi^{\left(
s\right)  }I\right)  }\right)  ^{\kappa}\mathbf{1}_{\mathcal{H}_{\eta}\left(
I\right)  }\left(  x\right) \\
& =\sum_{I\in\mathcal{D}}\sum_{s=1}^{\infty}\frac{\left\vert \left\langle
f,h_{I;\kappa}\right\rangle \right\vert }{\left\vert I\right\vert ^{\frac
{1}{2}}}\frac{\left\vert \left\langle f,h_{\left(  \pi^{\left(  s\right)
}I\right)  ;\kappa}\right\rangle \right\vert }{\left\vert \pi^{\left(
s\right)  }I\right\vert ^{\frac{1}{2}}}2^{-s\kappa}\mathbf{1}_{\mathcal{H}%
_{\eta}\left(  I\right)  }\left(  x\right) \\
& =\left(  \sum_{s=1}^{\infty}2^{-s\kappa}\right)  \sum_{I\in\mathcal{D}}%
\frac{\left\vert \left\langle f,h_{I;\kappa}\right\rangle \right\vert
}{\left\vert I\right\vert ^{\frac{1}{2}}}\frac{\left\vert \left\langle
f,h_{\left(  \pi^{\left(  s\right)  }I\right)  ;\kappa}\right\rangle
\right\vert }{\left\vert \pi^{\left(  s\right)  }I\right\vert ^{\frac{1}{2}}%
}\mathbf{1}_{\mathcal{H}_{\eta}\left(  I\right)  }\left(  x\right)  ,
\end{align*}
which is at most%
\[
\left(  \sum_{s=1}^{\infty}2^{-s\kappa}\right)  \sqrt{\sum_{I\in\mathcal{D}%
}\left(  \frac{\left\vert \left\langle f,h_{I;\kappa}\right\rangle \right\vert
}{\left\vert I\right\vert ^{\frac{1}{2}}}\right)  ^{2}\mathbf{1}%
_{\mathcal{H}_{\eta}\left(  I\right)  }\left(  x\right)  }\sqrt{\sum
_{I\in\mathcal{D}}\left(  \frac{\left\vert \left\langle f,h_{\left(
\pi^{\left(  s\right)  }I\right)  ;\kappa}\right\rangle \right\vert
}{\left\vert \pi^{\left(  s\right)  }I\right\vert ^{\frac{1}{2}}}\right)
^{2}\mathbf{1}_{\mathcal{H}_{\eta}\left(  \pi^{\left(  t\right)  }I\right)
}\left(  x\right)  }\approx\sum_{I\in\mathcal{D}}\left(  \frac{\left\vert
\left\langle f,h_{I;\kappa}\right\rangle \right\vert }{\left\vert I\right\vert
^{\frac{1}{2}}}\right)  ^{2}\mathbf{1}_{\mathcal{H}_{\eta}\left(  I\right)
}\left(  x\right)  .
\]
By Lemma \ref{Marcin} we thus have%
\begin{equation}
\left\Vert \left(  \sum_{Q\in\mathcal{D}}\frac{1}{\left\vert Q\right\vert
}B\left(  Q\right)  ^{2}\mathbf{1}_{Q}\right)  ^{\frac{1}{2}}\right\Vert
_{L^{p}}\lesssim\left\Vert \left(  \sum_{I\in\mathcal{D}}\left(
\frac{\left\vert \left\langle f,h_{I;\kappa}\right\rangle \right\vert
}{\left\vert I\right\vert ^{\frac{1}{2}}}\mathbf{1}_{\mathcal{H}_{\eta}\left(
I\right)  }\left(  x\right)  \right)  ^{2}\right)  ^{\frac{1}{2}}\right\Vert
_{L^{p}}\leq C_{p,n}\eta^{\frac{1}{2\left(  p-1\right)  }}\left\Vert
f\right\Vert _{L^{p}}\ .\label{est term B}%
\end{equation}

\textbf{Case} $C\left(  Q\right)  $: We have,
\begin{align*}
& \sum_{Q\in\mathcal{D}}\frac{1}{\left\vert Q\right\vert }C\left(  Q\right)
^{2}\mathbf{1}_{Q}=\sum_{Q\in\mathcal{D}}\frac{1}{\left\vert Q\right\vert
}\left\vert \sum_{\substack{I\in\mathcal{D}:\ \ell\left(  I\right)  \geq
\ell\left(  Q\right)  \geq\eta\ell\left(  I\right)  \\Q\in\operatorname*{Car}%
\left(  I\right)  }}\left\langle f,h_{I;\kappa}\right\rangle \left\langle
h_{I;\kappa}^{\eta},h_{Q;\kappa}\right\rangle \right\vert ^{2}\mathbf{1}%
_{Q}\left(  x\right) \\
& =\sum_{Q\in\mathcal{D}}\frac{1}{\left\vert Q\right\vert }\left(
\sum_{\substack{I\in\mathcal{D}:\ \ell\left(  I\right)  \geq\ell\left(
Q\right)  \geq\eta\ell\left(  I\right)  \\Q\in\operatorname*{Car}\left(
I\right)  }}\sum_{\substack{I\in\mathcal{D}:\ \ell\left(  I^{\prime}\right)
\geq\ell\left(  Q\right)  \geq\eta\ell\left(  I^{\prime}\right)
\\Q\in\operatorname*{Car}\left(  I^{\prime}\right)  }}\left\langle
f,h_{I;\kappa}\right\rangle \left\langle h_{I;\kappa}^{\eta},h_{Q;\kappa
}\right\rangle \left\langle f,h_{I^{\prime};\kappa}\right\rangle \left\langle
h_{I^{\prime};\kappa}^{\eta},h_{Q}\right\rangle \right)  \mathbf{1}_{Q}\left(
x\right) \\
& \approx\sum_{Q\in\mathcal{D}}\frac{1}{\left\vert Q\right\vert }\left(
\sum_{\substack{I,I^{\prime}\in\mathcal{D}:\ I\subset I^{\prime}\text{ and
}\ell\left(  I\right)  \geq\ell\left(  Q\right)  \geq\eta\ell\left(
I^{\prime}\right)  \\Q\in\operatorname*{Car}\left(  I\right)  \cap
\operatorname*{Car}\left(  I^{\prime}\right)  }}\left\langle f,h_{I;\kappa
}\right\rangle \left\langle h_{I;\kappa}^{\eta},h_{Q;\kappa}\right\rangle
\left\langle f,h_{I^{\prime};\kappa}\right\rangle \left\langle h_{I^{\prime
};\kappa}^{\eta},h_{Q;\kappa}\right\rangle \right)  \mathbf{1}_{Q}\left(
x\right) \\
& =\sum_{Q\in\mathcal{D}}\frac{1}{\left\vert Q\right\vert }\left(
\sum_{\substack{I,I^{\prime}\in\mathcal{D}:\ I\subset I^{\prime}
\\Q\in\operatorname*{Car}\left(  I\right)  \cap\operatorname*{Car}\left(
I^{\prime}\right)  \\\ell\left(  I\right)  \geq\ell\left(  Q\right)  \geq
\eta\ell\left(  I^{\prime}\right)  }}\left\langle f,h_{I}\right\rangle
\left\langle f,h_{I^{\prime}}\right\rangle \left\langle h_{I}^{\eta}%
,h_{Q}\right\rangle \left\langle h_{I^{\prime}}^{\eta},h_{Q}\right\rangle
\right)  \mathbf{1}_{Q}\left(  x\right)  .
\end{align*}
We first compute the diagonal sum restricted to $I=I^{\prime}$. Set
\[
\Gamma_{\eta,t}\left(  I\right)  \equiv\left\{  x\in I:\operatorname*{dist}%
\left(  x,\mathcal{H}_{\eta}\left(  I\right)  \right)  \approx2^{t}\eta
\ell\left(  I\right)  \right\}  ,\ \ \ \ \ \text{for }0\leq t\leq\beta,
\]
where we recall that $\eta=2^{-\beta}$, and note that the diagonal portion of
the sum above equals%
\begin{align*}
& \sum_{Q\in\mathcal{D}}\frac{1}{\left\vert Q\right\vert }\left(
\sum_{\substack{I\in\mathcal{D}:\ Q\in\operatorname*{Car}\left(  I\right)
\\\ell\left(  I\right)  \geq\ell\left(  Q\right)  \geq\eta\ell\left(
I\right)  }}\left\vert \left\langle f,h_{I}\right\rangle \right\vert
^{2}\left\vert \left\langle h_{I;\kappa}^{\eta},h_{Q;\kappa}\right\rangle
\right\vert ^{2}\right)  \mathbf{1}_{Q}\left(  x\right)  =\sum_{I\in
\mathcal{D}}\left\vert \left\langle f,h_{I;\kappa}\right\rangle \right\vert
^{2}\sum_{\substack{Q\in\mathcal{D}:\ Q\in\operatorname*{Car}\left(  I\right)
\\\ell\left(  I\right)  \geq\ell\left(  Q\right)  \geq\eta\ell\left(
I\right)  }}\frac{\left\vert \left\langle h_{I;\kappa}^{\eta},h_{Q;\kappa
}\right\rangle \right\vert ^{2}}{\left\vert Q\right\vert }\mathbf{1}%
_{Q}\left(  x\right) \\
& \lesssim\sum_{I\in\mathcal{D}}\left\vert \left\langle f,h_{I;\kappa
}\right\rangle \right\vert ^{2}\sum_{\substack{Q\in\mathcal{D}:\ Q\in
\operatorname*{Car}\left(  I\right)  \\\ell\left(  I\right)  \geq\ell\left(
Q\right)  \geq\eta\ell\left(  I\right)  }}\frac{\eta^{2}\left(  \frac
{\ell\left(  Q\right)  }{\ell\left(  I\right)  }\right)  ^{n-2}}{\ell\left(
Q\right)  ^{n}}\mathbf{1}_{Q}\left(  x\right)  =\eta^{2}\sum_{I\in\mathcal{D}%
}\left\vert \left\langle f,h_{I;\kappa}\right\rangle \right\vert ^{2}%
\sum_{\substack{Q\in\mathcal{D}:\ Q\in\operatorname*{Car}\left(  I\right)
\\\ell\left(  I\right)  \geq\ell\left(  Q\right)  \geq\eta\ell\left(
I\right)  }}\frac{1}{\ell\left(  I\right)  ^{n-2}\ell\left(  Q\right)  ^{2}%
}\mathbf{1}_{Q}\left(  x\right) \\
& \approx\eta^{2}\sum_{I\in\mathcal{D}}\left\vert \left\langle f,h_{I;\kappa
}\right\rangle \right\vert ^{2}\frac{1}{\ell\left(  I\right)  ^{n-2}\left[
\eta\ell\left(  I\right)  +\operatorname*{dist}\left(  x,\mathcal{H}_{\eta
}\left(  I\right)  \right)  \right]  ^{2}}\mathbf{1}_{I}\left(  x\right)
=\sum_{I\in\mathcal{D}}\frac{\left\vert \left\langle f,h_{I;\kappa
}\right\rangle \right\vert ^{2}}{\left\vert I\right\vert }\left(  \frac
{1}{1+\frac{\operatorname*{dist}\left(  x,\mathcal{H}_{\eta}\left(  I\right)
\right)  }{\eta\ell\left(  I\right)  }}\right)  ^{2}\mathbf{1}_{I}\left(
x\right)  ,
\end{align*}
which can be written as%
\[
\sum_{I\in\mathcal{D}}\frac{\left\vert \left\langle f,h_{I;\kappa
}\right\rangle \right\vert ^{2}}{\left\vert I\right\vert }\sum_{t=0}^{\beta
}\mathbf{1}_{\Gamma_{\eta,t}}\left(  x\right)  \left(  \frac{1}{1+\frac
{\operatorname*{dist}\left(  x,\mathcal{H}_{\eta}\left(  I\right)  \right)
}{\eta\ell\left(  I\right)  }}\right)  ^{2}\approx\sum_{I\in\mathcal{D}}%
\frac{\left\vert \left\langle f,h_{I;\kappa}\right\rangle \right\vert ^{2}%
}{\left\vert I\right\vert }\sum_{t=0}^{\beta}2^{-2t}\mathbf{1}_{\Gamma
_{\eta,t}\left(  I\right)  }\left(  x\right)  .
\]

Thus%
\[
\left\Vert \sum_{Q\in\mathcal{D}}\frac{1}{\left\vert Q\right\vert }\left(
\sum_{\substack{I\in\mathcal{D}:\ Q\in\operatorname*{Car}\left(  I\right)
\\\ell\left(  I\right)  \geq\ell\left(  Q\right)  \geq\eta\ell\left(
I\right)  }}\left\vert \left\langle f,h_{I;\kappa}\right\rangle \right\vert
^{2}\left\vert \left\langle h_{I;\kappa}^{\eta},h_{Q;\kappa}\right\rangle
\right\vert ^{2}\right)  \mathbf{1}_{Q}\left(  x\right)  \right\Vert _{L^{p}%
}\lesssim\sum_{t=0}^{\beta}2^{-2t}\left\Vert \sum_{I\in\mathcal{D}}%
\frac{\left\vert \left\langle f,h_{I;\kappa}\right\rangle \right\vert ^{2}%
}{\left\vert I\right\vert }\mathbf{1}_{\Gamma_{\eta,t}\left(  I\right)
}\left(  x\right)  \right\Vert _{L^{p}}.
\]
From the estimate for term $B$ in (\ref{est term B}), with $\eta$ replaced by
$2^{t}\eta$, we obtain%
\[
\left\Vert \sum_{I\in\mathcal{D}}\frac{\left\vert \left\langle f,h_{I;\kappa
}\right\rangle \right\vert ^{2}}{\left\vert I\right\vert }\mathbf{1}%
_{\Gamma_{\eta,t}\left(  I\right)  }\left(  x\right)  \right\Vert _{L^{p}%
}\lesssim C_{p,n}\left(  2^{t}\eta\right)  ^{\frac{1}{2\left(  p-1\right)  }%
}\left\Vert f\right\Vert _{L^{p}}\ ,
\]
and so altogether, the diagonal portion of $\left\Vert \sum_{Q\in\mathcal{D}%
}\frac{1}{\left\vert Q\right\vert }C\left(  Q\right)  ^{2}\mathbf{1}%
_{Q}\left(  x\right)  \right\Vert _{L^{p}}$ is at most
\begin{align*}
& \sum_{t=0}^{\beta}2^{-2t}\left\Vert \sum_{I\in\mathcal{D}}\frac{\left\vert
\left\langle f,h_{I;\kappa}\right\rangle \right\vert ^{2}}{\left\vert
I\right\vert }\mathbf{1}_{\Gamma_{\eta,t}\left(  I\right)  }\left(  x\right)
\right\Vert _{L^{p}}\lesssim\sum_{t=0}^{\beta}C_{p,n}2^{-2t}\left(  2^{t}%
\eta\right)  ^{\frac{1}{2\left(  p-1\right)  }}\left\Vert f\right\Vert
_{L^{p}}\\
& =\eta^{\frac{1}{2\left(  p-1\right)  }}\sum_{t=0}^{\beta}C_{p,n}2^{-t\left(
2-\frac{1}{2\left(  p-1\right)  }\right)  }\left\Vert f\right\Vert _{L^{p}%
}=\eta^{\frac{1}{2\left(  p-1\right)  }}\sum_{t=0}^{\beta}C_{p,n}%
2^{-t\frac{4\left(  p-1\right)  -1}{2\left(  p-1\right)  }}\left\Vert
f\right\Vert _{L^{p}}\\
& =\eta^{\frac{1}{2\left(  p-1\right)  }}\sum_{t=0}^{\beta}C_{p,n}%
2^{-t\frac{4p-5}{2p-2}}\left\Vert f\right\Vert _{L^{p}}\approx C_{p,n}\left\{
\begin{array}
[c]{ccc}%
\eta^{\frac{1}{2\left(  p-1\right)  }}\left\Vert f\right\Vert _{L^{p}} &
\text{ if } & p>\frac{5}{4}\\
\eta^{2}\left(  \log_{2}\frac{1}{\eta}\right)  \left\Vert f\right\Vert
_{L^{p}} & \text{ if } & p=\frac{5}{4}\\
\eta^{2}\left\Vert f\right\Vert _{L^{p}} & \text{ if } & 1<p<\frac{5}{4}%
\end{array}
\right.  \ .
\end{align*}

Now we use the estimate $\left\vert \left\langle h_{I;\kappa}^{\eta
},h_{Q;\kappa}\right\rangle \right\vert \lesssim\eta\left(  \frac{\ell\left(
Q\right)  }{\ell\left(  I\right)  }\right)  ^{\frac{n}{2}-1}\ $for
$Q\in\operatorname*{Car}\left(  I\right)  $ and $\ell\left(  Q\right)
\geq\eta\ell\left(  I\right)  $, see the third line of Lemma \ref{inner est},
to obtain%
\begin{align*}
& \sum_{Q\in\mathcal{D}}\frac{1}{\left\vert Q\right\vert }\left\vert
\sum_{\substack{I\in\mathcal{D}:\ \ell\left(  I\right)  \geq\ell\left(
Q\right)  \geq\eta\ell\left(  I\right)  \\Q\in\operatorname*{Car}\left(
I\right)  }}\left\langle f,h_{I;\kappa}\right\rangle \left\langle h_{I}^{\eta
},h_{Q}\right\rangle \right\vert ^{2}\mathbf{1}_{Q}\left(  x\right) \\
& \lesssim\sum_{I,I^{\prime}\in\mathcal{D}:\ I\subset I^{\prime}}\left\vert
\left\langle f,h_{I;\kappa}\right\rangle \right\vert \left\vert \left\langle
f,h_{I^{\prime};\kappa}\right\rangle \right\vert \sum_{\substack{Q\in
\mathcal{D}:\ Q\in\operatorname*{Car}\left(  I\right)  \cap\operatorname*{Car}%
\left(  I^{\prime}\right)  \\\ell\left(  I\right)  \geq\ell\left(  Q\right)
\geq\eta\ell\left(  I^{\prime}\right)  }}\left\vert \left\langle h_{I;\kappa
}^{\eta},h_{Q;\kappa}\right\rangle \right\vert \left\vert \left\langle
h_{I^{\prime};\kappa}^{\eta},h_{Q;\kappa}\right\rangle \right\vert \frac
{1}{\left\vert Q\right\vert }\mathbf{1}_{Q}\left(  x\right) \\
& \lesssim\eta^{2}\sum_{I,I^{\prime}\in\mathcal{D}:\ I\subset I^{\prime}%
}\left\vert \left\langle f,h_{I;\kappa}\right\rangle \right\vert \left\vert
\left\langle f,h_{I^{\prime};\kappa}\right\rangle \right\vert \sum
_{\substack{Q\in\mathcal{D}:\ Q\in\operatorname*{Car}\left(  I\right)
\cap\operatorname*{Car}\left(  I^{\prime}\right)  \\\ell\left(  I\right)
\geq\ell\left(  Q\right)  \geq\eta\ell\left(  I^{\prime}\right)  }}\left(
\frac{\ell\left(  Q\right)  }{\ell\left(  I\right)  }\right)  ^{\frac{n}{2}%
-1}\left(  \frac{\ell\left(  Q\right)  }{\ell\left(  I^{\prime}\right)
}\right)  ^{\frac{n}{2}-1}\frac{1}{\left\vert Q\right\vert }\mathbf{1}%
_{Q}\left(  x\right) \\
& =\eta^{2}\sum_{I,I^{\prime}\in\mathcal{D}:\ I\subset I^{\prime}}%
\frac{\left\vert \left\langle f,h_{I;\kappa}\right\rangle \right\vert
\left\vert \left\langle f,h_{I^{\prime};\kappa}\right\rangle \right\vert
}{\sqrt{\left\vert I\right\vert }\sqrt{\left\vert I^{\prime}\right\vert }}%
\sum_{\substack{Q\in\operatorname*{Car}\left(  I\right)  \cap
\operatorname*{Car}\left(  I^{\prime}\right)  \\\ell\left(  I\right)  \geq
\ell\left(  Q\right)  \geq\eta\ell\left(  I^{\prime}\right)  }}\frac
{\ell\left(  I\right)  }{\ell\left(  Q\right)  }\frac{\ell\left(  I^{\prime
}\right)  }{\ell\left(  Q\right)  }\mathbf{1}_{Q}\left(  x\right)  .
\end{align*}

At this point we observe that the conditions imposed on the cubes $I$ and
$I^{\prime}$ in the sum above are that there exists a cube $Q$ such that
$Q\subset I\subset I^{\prime}$, $Q\in\operatorname*{Car}\left(  I\right)
\cap\operatorname*{Car}\left(  I^{\prime}\right)  $, and $\ell\left(
I\right)  \geq\ell\left(  Q\right)  \geq\eta\ell\left(  I^{\prime}\right)  $.
It follows from these conditions that
\[
I\in\operatorname*{Car}\left(  I^{\prime}\right)  \text{ and }\ell\left(
I\right)  \leq\ell\left(  I^{\prime}\right)  \leq\frac{1}{\eta}\ell\left(
I\right)  =2^{\beta}\ell\left(  I\right)  .
\]
Thus we can now pigeonhole the ratio of the lengths of $I$ and $I^{\prime}$
by
\[
\frac{\ell\left(  I^{\prime}\right)  }{\ell\left(  I\right)  }=2^{s}%
,\ \ \ \ \text{for }0\leq s\leq\beta.
\]
With $s$ fixed we have $I^{\prime}=\pi^{\left(  s\right)  }I$ and%
\begin{align*}
& \eta^{2}\sum_{I\in\mathcal{D}}\frac{\left\vert \left\langle f,h_{I;\kappa
}\right\rangle \right\vert \left\vert \left\langle f,h_{\left(  \pi^{\left(
s\right)  }I\right)  ;\kappa}\right\rangle \right\vert }{\sqrt{\left\vert
I\right\vert }\sqrt{\left\vert \pi^{\left(  s\right)  }I\right\vert }}%
\sum_{\substack{Q\in\operatorname*{Car}\left(  I\right)  \cap
\operatorname*{Car}\left(  \pi^{\left(  s\right)  }I\right)  \\\ell\left(
I\right)  \geq\ell\left(  Q\right)  \geq\eta\ell\left(  \pi^{\left(  s\right)
}I\right)  }}\frac{\ell\left(  I\right)  }{\ell\left(  Q\right)  }\frac
{\ell\left(  \pi^{\left(  s\right)  }I\right)  }{\ell\left(  Q\right)
}\mathbf{1}_{Q}\left(  x\right) \\
& =\eta^{2}\sum_{I\in\mathcal{D}}\frac{\left\vert \left\langle f,h_{I;\kappa
}\right\rangle \right\vert \left\vert \left\langle f,h_{\left(  \pi^{\left(
s\right)  }I\right)  ;\kappa}\right\rangle \right\vert }{\sqrt{\left\vert
I\right\vert }\sqrt{\left\vert \pi^{\left(  s\right)  }I\right\vert }}%
\sum_{\substack{Q\in\operatorname*{Car}\left(  I\right)  \cap
\operatorname*{Car}\left(  \pi^{\left(  s\right)  }I\right)  \\\ell\left(
I\right)  \geq\ell\left(  Q\right)  \geq2^{s}\eta\ell\left(  I\right)  }%
}2^{s}\left(  \frac{\ell\left(  I\right)  }{\ell\left(  Q\right)  }\right)
^{2}\mathbf{1}_{Q}\left(  x\right) \\
& \approx2^{s}\eta^{2}\sum_{I\in\mathcal{D}}\frac{\left\vert \left\langle
f,h_{I;\kappa}\right\rangle \right\vert \left\vert \left\langle f,h_{\left(
\pi^{\left(  s\right)  }I\right)  ;\kappa}\right\rangle \right\vert }%
{\sqrt{\left\vert I\right\vert }\sqrt{\left\vert \pi^{\left(  s\right)
}I\right\vert }}2^{s}\left(  \frac{\ell\left(  I\right)  }{2^{s}\eta
\ell\left(  I\right)  +\operatorname*{dist}\left(  x,\mathcal{H}_{2^{s}\eta
}\left(  I\right)  \right)  }\right)  ^{2}\mathbf{1}_{I}\left(  x\right) \\
& =2^{s}\eta^{2}\sum_{I\in\mathcal{D}}\frac{\left\vert \left\langle
f,h_{I;\kappa}\right\rangle \right\vert \left\vert \left\langle f,h_{\left(
\pi^{\left(  s\right)  }I\right)  ;\kappa}\right\rangle \right\vert }%
{\sqrt{\left\vert I\right\vert }\sqrt{\left\vert \pi^{\left(  s\right)
}I\right\vert }}2^{s}\left(  \frac{1}{2^{s}\eta+\frac{\operatorname*{dist}%
\left(  x,\mathcal{H}_{2^{s}\eta}\left(  I\right)  \right)  }{\ell\left(
I\right)  }}\right)  ^{2}\mathbf{1}_{I}\left(  x\right) \\
& =\sum_{I\in\mathcal{D}}\frac{\left\vert \left\langle f,h_{I;\kappa
}\right\rangle \right\vert }{\sqrt{\left\vert I\right\vert }}\frac{\left\vert
\left\langle f,h_{\left(  \pi^{\left(  s\right)  }I\right)  ;\kappa
}\right\rangle \right\vert }{\sqrt{\left\vert \pi^{\left(  s\right)
}I\right\vert }}\left(  \frac{1}{1+\frac{\operatorname*{dist}\left(
x,\mathcal{H}_{2^{s}\eta}\left(  I\right)  \right)  }{2^{s}\eta\ell\left(
I\right)  }}\right)  ^{2}\mathbf{1}_{I}\left(  x\right)  ,
\end{align*}
where our sum is exactly like the diagonal portion with two exceptions, namely
that $I$\ has been replaced by $\pi^{\left(  s\right)  }I$ in the second
factor, and $\eta$ has been replaced by $2^{s}\eta$ in the third factor. Thus
we continue with,%
\begin{align*}
& \sum_{I\in\mathcal{D}}\frac{\left\vert \left\langle f,h_{I;\kappa
}\right\rangle \right\vert }{\sqrt{\left\vert I\right\vert }}\frac{\left\vert
\left\langle f,h_{\left(  \pi^{\left(  s\right)  }I\right)  ;\kappa
}\right\rangle \right\vert }{\sqrt{\left\vert \pi^{\left(  s\right)
}I\right\vert }}\left(  \frac{1}{1+\frac{\operatorname*{dist}\left(
x,\mathcal{H}_{2^{s}\eta}\left(  I\right)  \right)  }{2^{s}\eta\ell\left(
I\right)  }}\right)  ^{2}\mathbf{1}_{I}\left(  x\right) \\
& =\sum_{I\in\mathcal{D}}\frac{\left\vert \left\langle f,h_{I;\kappa
}\right\rangle \right\vert }{\sqrt{\left\vert I\right\vert }}\frac{\left\vert
\left\langle f,h_{\left(  \pi^{\left(  s\right)  }I\right)  ;\kappa
}\right\rangle \right\vert }{\sqrt{\left\vert \pi^{\left(  s\right)
}I\right\vert }}\sum_{t=0}^{\beta-s}\mathbf{1}_{\Gamma_{2^{s}\eta,t}\left(
I\right)  }\left(  x\right)  \left(  \frac{1}{1+\frac{\operatorname*{dist}%
\left(  x,\mathcal{H}_{2^{s}\eta}\left(  I\right)  \right)  }{2^{s}\eta
\ell\left(  I\right)  }}\right)  ^{2}\\
& \approx\sum_{I\in\mathcal{D}}\frac{\left\vert \left\langle f,h_{I;\kappa
}\right\rangle \right\vert }{\sqrt{\left\vert I\right\vert }}\frac{\left\vert
\left\langle f,h_{\left(  \pi^{\left(  s\right)  }I\right)  ;\kappa
}\right\rangle \right\vert }{\sqrt{\left\vert \pi^{\left(  s\right)
}I\right\vert }}\sum_{t=0}^{\beta-s}2^{-2t}\mathbf{1}_{\Gamma_{2^{s}\eta
,t}\left(  I\right)  }\left(  x\right)  ,
\end{align*}
since $\Gamma_{2^{s}\eta,t}\left(  I\right)  =\left\{  x\in
I:\operatorname*{dist}\left(  x,\mathcal{H}_{2^{s}\eta}\left(  I\right)
\right)  \approx2^{t}2^{s}\eta\ell\left(  I\right)  \right\}  $ and
$\operatorname*{dist}\left(  x,\mathcal{H}_{2^{s}\eta}\left(  I\right)
\right)  \leq\ell\left(  I\right)  $.

Now we continue to proceed as in the diagonal case to obtain,%
\begin{align*}
& \left\Vert \sum_{I\in\mathcal{D}}\frac{\left\vert \left\langle
f,h_{I;\kappa}\right\rangle \right\vert }{\sqrt{\left\vert I\right\vert }%
}\frac{\left\vert \left\langle f,h_{\left(  \pi^{\left(  s\right)  }I\right)
;\kappa}\right\rangle \right\vert }{\sqrt{\left\vert \pi^{\left(  s\right)
}I\right\vert }}\left(  \frac{1}{1+\frac{\operatorname*{dist}\left(
x,\mathcal{H}_{2^{s}\eta}\left(  I\right)  \right)  }{2^{s}\eta\ell\left(
I\right)  }}\right)  ^{2}\mathbf{1}_{I}\right\Vert _{L^{p}}\\
& \lesssim\sum_{t=0}^{\beta-s}2^{-2t}\left\Vert \sum_{I\in\mathcal{D}}%
\frac{\left\vert \left\langle f,h_{I;\kappa}\right\rangle \right\vert }%
{\sqrt{\left\vert I\right\vert }}\frac{\left\vert \left\langle f,h_{\left(
\pi^{\left(  s\right)  }I\right)  ;\kappa}\right\rangle \right\vert }%
{\sqrt{\left\vert \pi^{\left(  s\right)  }I\right\vert }}\mathbf{1}%
_{\Gamma_{2^{s}\eta,t}\left(  I\right)  }\right\Vert _{L^{p}}\\
& \lesssim\sum_{t=0}^{\beta-s}2^{-2t}\left\Vert \sqrt{\sum_{I\in\mathcal{D}%
}\frac{\left\vert \left\langle f,h_{I;\kappa}\right\rangle \right\vert ^{2}%
}{\left\vert I\right\vert }\mathbf{1}_{\Gamma_{2^{s}\eta,t}\left(  I\right)
}}\sqrt{\sum_{I\in\mathcal{D}}\frac{\left\vert \left\langle f,h_{\left(
\pi^{\left(  s\right)  }I\right)  ;\kappa}\right\rangle \right\vert ^{2}%
}{\left\vert \pi^{\left(  s\right)  }I\right\vert }\mathbf{1}_{\Gamma
_{2^{s}\eta,t}\left(  I\right)  }}\right\Vert _{L^{p}}\\
& \lesssim\sum_{t=0}^{\beta-s}2^{-2t}\left\Vert \delta\sum_{I\in\mathcal{D}%
}\frac{\left\vert \left\langle f,h_{I;\kappa}\right\rangle \right\vert ^{2}%
}{\left\vert I\right\vert }\mathbf{1}_{\Gamma_{2^{s}\eta,t}\left(  I\right)
}+\frac{1}{\delta}\sum_{I\in\mathcal{D}}\frac{\left\vert \left\langle
f,h_{\left(  \pi^{\left(  s\right)  }I\right)  ;\kappa}\right\rangle
\right\vert ^{2}}{\left\vert \pi^{\left(  s\right)  }I\right\vert }%
\mathbf{1}_{\Gamma_{2^{s}\eta,t}\left(  I\right)  }\right\Vert _{L^{p}},
\end{align*}
for every choice of $\delta\in\left(  0,1\right)  $. Thus it remains to
estimate each of the terms%
\[
\delta\sum_{t=0}^{\beta-s}2^{-2t}\left\Vert \sum_{I\in\mathcal{D}}%
\frac{\left\vert \left\langle f,h_{I;\kappa}\right\rangle \right\vert ^{2}%
}{\left\vert I\right\vert }\mathbf{1}_{\Gamma_{2^{s}\eta,t}\left(  I\right)
}\right\Vert _{L^{p}}\text{ and }\frac{1}{\delta}\sum_{t=0}^{\beta-s}%
2^{-2t}\left\Vert \sum_{I\in\mathcal{D}}\frac{\left\vert \left\langle
f,h_{\left(  \pi^{\left(  s\right)  }I\right)  ;\kappa}\right\rangle
\right\vert ^{2}}{\left\vert \pi^{\left(  s\right)  }I\right\vert }%
\mathbf{1}_{\Gamma_{2^{s}\eta,t}\left(  I\right)  }\right\Vert _{L^{p}},
\]
and then minimize the sum over $0<\delta<1$. But from (\ref{est term B}), we
have%
\begin{align*}
\sum_{t=0}^{\beta-s}2^{-2t}\left\Vert \sum_{I\in\mathcal{D}}\frac{\left\vert
\left\langle f,h_{I;\kappa}\right\rangle \right\vert ^{2}}{\left\vert
I\right\vert }\mathbf{1}_{\Gamma_{2^{s}\eta,t}\left(  I\right)  }\right\Vert
_{L^{p}}  & \lesssim C_{p,n}\left(  2^{s}\eta\right)  ^{\frac{1}{2\left(
p-1\right)  }}\left\Vert f\right\Vert _{L^{p}}\ ,\\
\sum_{t=0}^{\beta-s}2^{-2t}\left\Vert \sum_{I\in\mathcal{D}}\frac{\left\vert
\left\langle f,h_{\left(  \pi^{\left(  s\right)  }I\right)  ;\kappa
}\right\rangle \right\vert ^{2}}{\left\vert \pi^{\left(  s\right)
}I\right\vert }\mathbf{1}_{\Gamma_{2^{s}\eta,t}\left(  I\right)  }\right\Vert
_{L^{p}}  & \lesssim\sum_{t=0}^{\beta-s}2^{-2t}\left\Vert \sum_{I^{\prime}%
\in\mathcal{D}}\frac{\left\vert \left\langle f,h_{I^{\prime};\kappa
}\right\rangle \right\vert ^{2}}{\left\vert I^{\prime}\right\vert }%
\mathbf{1}_{\Gamma_{\eta,t}\left(  I^{\prime}\right)  }\right\Vert _{L^{p}%
}\lesssim C_{p,n}\eta^{\frac{1}{2\left(  p-1\right)  }}\left\Vert f\right\Vert
_{L^{p}}\ ,
\end{align*}
since
\begin{align*}
\Gamma_{\eta,t}\left(  I^{\prime}\right)  \Gamma_{2^{s}\eta,t}\left(
I\right)   & =\left\{  x\in I:\operatorname*{dist}\left(  x,\mathcal{H}%
_{2^{s}\eta}\left(  I\right)  \right)  \approx2^{t}2^{s}\eta\ell\left(
I\right)  \right\} \\
& \subset\left\{  x\in I^{\prime}:\operatorname*{dist}\left(  x,\mathcal{H}%
_{\eta}\left(  I^{\prime}\right)  \right)  \approx2^{t}\eta\ell\left(
I^{\prime}\right)  \right\}  =\Gamma_{\eta,t}\left(  I^{\prime}\right)  .
\end{align*}
Thus with $\delta=2^{-\frac{s}{4\left(  p-1\right)  }}$, we obtain%
\begin{align*}
& \left\Vert \sum_{I\in\mathcal{D}}\frac{\left\vert \left\langle
f,h_{I;\kappa}\right\rangle \right\vert }{\sqrt{\left\vert I\right\vert }%
}\frac{\left\vert \left\langle f,h_{\left(  \pi^{\left(  s\right)  }I\right)
;\kappa}\right\rangle \right\vert }{\sqrt{\left\vert \pi^{\left(  s\right)
}I\right\vert }}\left(  \frac{1}{1+\frac{\operatorname*{dist}\left(
x,\mathcal{H}_{2^{s}\eta}\left(  I\right)  \right)  }{2^{s}\eta\ell\left(
I\right)  }}\right)  ^{2}\mathbf{1}_{I}\right\Vert _{L^{p}}\\
& \lesssim\delta C_{p,n}\left(  2^{s}\eta\right)  ^{\frac{1}{2\left(
p-1\right)  }}\left\Vert f\right\Vert _{L^{p}}+\frac{1}{\delta}C_{p,n}%
\eta^{\frac{1}{2\left(  p-1\right)  }}\left\Vert f\right\Vert _{L^{p}}\\
& =\left[  \delta2^{\frac{s}{2\left(  p-1\right)  }}+\frac{1}{\delta}\right]
C_{p,n}\eta^{\frac{1}{2\left(  p-1\right)  }}\left\Vert f\right\Vert _{L^{p}%
}=2C_{p,n}2^{\frac{s}{4\left(  p-1\right)  }}2^{-\frac{\beta}{2\left(
p-1\right)  }}\left\Vert f\right\Vert _{L^{p}}\\
& \leq2C_{p,n}2^{-\frac{\beta}{4\left(  p-1\right)  }}\left\Vert f\right\Vert
_{L^{p}}=2C_{p,n}\eta^{\frac{1}{4\left(  p-1\right)  }}\left\Vert f\right\Vert
_{L^{p}}\ ,
\end{align*}
since $0\leq s\leq\beta$. Finally we sum in $s$ from $0$ to $\beta=\log
_{2}\frac{1}{\eta}$ to conclude that,%
\[
\left\Vert \left(  \sum_{Q\in\mathcal{D}}\frac{1}{\left\vert Q\right\vert
}C\left(  Q\right)  ^{2}\mathbf{1}_{Q}\right)  ^{\frac{1}{2}}\right\Vert
_{L^{p}}\lesssim\eta^{\frac{1}{4\left(  p-1\right)  }}\log_{2}\frac{1}{\eta
}\left\Vert f\right\Vert _{L^{p}}\ .
\]

This finishes the proof of (\ref{claim that}) and hence the proof of
Proposition \ref{invert}.
\end{proof}

\subsubsection{Surjectivity}

The proof of Proposition \ref{invert dual} is very similar to that of the
previous proposition in light of\ the following equivalences. Using
$\left\vert \bigtriangleup_{I;\kappa}^{\eta}f\right\vert \leq
M^{\operatorname{dy}}\left(  \bigtriangleup_{I;\kappa}^{\eta}f\right)  $,
together with the Fefferman-Stein vector-valued maximal inequalities
\cite{FeSt}\ and the Alpert square function equivalence (\ref{square}), shows
that%
\[
\left\Vert \left(  \sum_{I\in\mathcal{D}}\left\vert \bigtriangleup_{I;\kappa
}^{\eta}f\right\vert ^{2}\right)  ^{\frac{1}{2}}\right\Vert _{L^{p}}%
\approx\left\Vert \left(  \sum_{I\in\mathcal{D}}\left\vert \bigtriangleup
_{I;\kappa}f\right\vert ^{2}\right)  ^{\frac{1}{2}}\right\Vert _{L^{p}}%
\approx\left\Vert \sum_{I\in\mathcal{D}}\bigtriangleup_{I;\kappa}f\right\Vert
_{L^{p}}=\left\Vert f\right\Vert _{L^{p}}\ .
\]
We also have from the Alpert square function equivalence that
\begin{equation}
\left\Vert \left(  \sum_{I\in\mathcal{D}}\left\vert \left(  \bigtriangleup
_{I;\kappa}^{\eta}\right)  ^{\operatorname*{tr}}f\right\vert ^{2}\right)
^{\frac{1}{2}}\right\Vert _{L^{p}}=\left\Vert \left(  \sum_{I\in\mathcal{D}%
}\left\vert \left\langle f,h_{I;\kappa}^{\eta}\right\rangle h_{I;\kappa
}\right\vert ^{2}\right)  ^{\frac{1}{2}}\right\Vert _{L^{p}}\approx\left\Vert
\sum_{I\in\mathcal{D}}\left\langle f,h_{I;\kappa}^{\eta}\right\rangle
h_{I;\kappa}\right\Vert _{L^{p}}=\left\Vert \sum_{I\in\mathcal{D}}\left(
\bigtriangleup_{I;\kappa}^{\eta}\right)  ^{\operatorname*{tr}}f\right\Vert
_{L^{p}}\ .\label{eta equiv}%
\end{equation}
Furthermore, from the definition $\left(  S_{\eta}^{\mathcal{D}}\right)
^{\operatorname*{tr}}f=\sum_{I\in\mathcal{D}}\left\langle f,h_{I;\kappa}%
^{\eta}\right\rangle h_{I;\kappa}$, we then obtain%
\begin{align}
& \left\Vert \left(  S_{\eta}^{\mathcal{D}}\right)  ^{\operatorname*{tr}%
}f\right\Vert _{L^{p}}\approx\left\Vert \left(  \sum_{Q\in\mathcal{D}%
}\left\vert \bigtriangleup_{Q;\kappa}\left(  S_{\eta}^{\mathcal{D}}\right)
^{\operatorname*{tr}}f\right\vert ^{2}\right)  ^{\frac{1}{2}}\right\Vert
_{L^{p}}=\left\Vert \left(  \sum_{Q\in\mathcal{D}}\left\vert \left\langle
\left(  S_{\eta}^{\mathcal{D}}\right)  ^{\ast}f,h_{Q;\kappa}\right\rangle
h_{Q;\kappa}\right\vert ^{2}\right)  ^{\frac{1}{2}}\right\Vert _{L^{p}%
}\label{then obtain}\\
& =\left\Vert \left(  \sum_{Q\in\mathcal{D}}\frac{1}{\left\vert Q\right\vert
}\left\vert \left\langle \sum_{I\in\mathcal{D}}\left\langle f,h_{I;\kappa
}^{\eta}\right\rangle h_{I;\kappa},h_{Q;\kappa}\right\rangle \right\vert
^{2}\right)  ^{\frac{1}{2}}\right\Vert _{L^{p}}=\left\Vert \left(  \sum
_{Q\in\mathcal{D}}\frac{1}{\left\vert Q\right\vert }\left\vert \left\langle
f,h_{Q;\kappa}^{\eta}\right\rangle \right\vert ^{2}\right)  ^{\frac{1}{2}%
}\right\Vert _{L^{p}}.\nonumber
\end{align}

\begin{proof}
[Proof of Proposition \ref{invert dual}]From (\ref{then obtain}) we have,%
\[
\left\Vert \left(  S_{\eta}^{\mathcal{D}}\right)  ^{\operatorname*{tr}%
}f\right\Vert _{L^{p}}\approx\left\Vert \left(  \sum_{Q\in\mathcal{D}}\frac
{1}{\left\vert Q\right\vert }\left\vert \left\langle f,h_{Q;\kappa}^{\eta
}\right\rangle \right\vert ^{2}\right)  ^{\frac{1}{2}}\right\Vert _{L^{p}%
}=\left\Vert \left(  \sum_{Q\in\mathcal{D}}\frac{1}{\left\vert Q\right\vert
}\left\vert \sum_{I\in\mathcal{D}}\left\langle f,h_{I;\kappa}\right\rangle
\left\langle h_{I;\kappa},h_{Q;\kappa}^{\eta}\right\rangle \right\vert
^{2}\right)  ^{\frac{1}{2}}\right\Vert _{L^{p}},
\]
which we now compare to
\[
\left\Vert S_{\eta}^{\mathcal{D}}f\right\Vert _{L^{p}}\approx\left\Vert
\left(  \sum_{Q\in\mathcal{D}}\left\vert \left\langle S_{\eta}f,h_{Q;\kappa
}\right\rangle h_{Q;\kappa}\right\vert ^{2}\right)  ^{\frac{1}{2}}\right\Vert
_{L^{p}}=\left\Vert \left(  \sum_{Q\in\mathcal{D}}\frac{1}{\left\vert
Q\right\vert }\left\vert \sum_{I\in\mathcal{D}}\left\langle f,h_{I;\kappa
}\right\rangle \left\langle h_{I;\kappa}^{\eta},h_{Q;\kappa}\right\rangle
\right\vert ^{2}\right)  ^{\frac{1}{2}}\right\Vert _{L^{p}},
\]
that was shown to be comparable to $\left\Vert f\right\Vert _{L^{p}}$ in
Proposition \ref{invert} above. The only difference between the two right hand
sides is that the convolution appears with $h_{Q;\kappa}^{\eta}$ in the first
norm, and with $h_{I;\kappa}^{\eta}$ in the second norm. We now use the
estimates in Lemma \ref{inner est} just as in the proof of Proposition
\ref{invert} above. Here is a sketch of the details that is virtually verbatim
that of those in the proof of Proposition \ref{invert}. Recall that
$\mathcal{H}_{\eta}\left(  I\right)  $ is defined in (\ref{def halo}).

For convenience we first rewrite the estimates in Lemma \ref{inner est} so as
to apply directly to the inner product $\left\langle h_{I;\kappa},h_{Q;\kappa
}^{\eta}\right\rangle $ instead of $\left\langle h_{I;\kappa}^{\eta
},h_{Q;\kappa}\right\rangle $. This is accomplished by simply interchanging
$Q$ and $I$ throughout:%
\begin{align}
\left\vert \left\langle h_{Q;\kappa}^{\eta},h_{Q;\kappa}\right\rangle
\right\vert  & \approx1\text{ and }\left\vert \left\langle h_{Q;\kappa}^{\eta
},h_{Q^{\prime};\kappa}\right\rangle \right\vert \lesssim\eta
,\ \ \ \ \ \text{for }Q\text{ and }Q^{\prime}\text{ siblings},\label{cases}\\
\left\vert \left\langle h_{Q;\kappa}^{\eta},h_{I;\kappa}\right\rangle
\right\vert  & \lesssim\eta\left(  \frac{\ell\left(  Q\right)  }{\ell\left(
I\right)  }\right)  ^{\frac{n}{2}},\ \ \ \ \ \text{for }Q\in
\operatorname*{Car}\left(  I\right)  ,\nonumber\\
\left\vert \left\langle h_{Q;\kappa}^{\eta},h_{I;\kappa}\right\rangle
\right\vert  & \lesssim\eta\left(  \frac{\ell\left(  I\right)  }{\ell\left(
Q\right)  }\right)  ^{\frac{n}{2}-1},\ \ \ \ \ \text{for }I\in
\operatorname*{Car}\left(  Q\right)  \text{ and }\ell\left(  I\right)
\geq\eta\ell\left(  Q\right)  ,\nonumber\\
\left\vert \left\langle h_{Q;\kappa}^{\eta},h_{I;\kappa}\right\rangle
\right\vert  & \lesssim\frac{1}{\eta^{\kappa}}\left(  \frac{\ell\left(
I\right)  }{\ell\left(  Q\right)  }\right)  ^{\kappa+\frac{n}{2}%
},\ \ \ \ \ \text{for }\ell\left(  I\right)  \leq\eta\ell\left(  Q\right)
\text{ and }I\cap\mathcal{H}_{\frac{\eta}{2}}\left(  I\right)  \neq
\emptyset,\nonumber\\
\left\langle h_{Q;\kappa}^{\eta},h_{I;\kappa}\right\rangle  &
=0,\ \ \ \ \ \text{in all other cases}.\nonumber
\end{align}

Now we have by the Alpert square function estimate (\ref{square}),%
\begin{align*}
& \left\Vert \left(  S_{\eta}^{\mathcal{D}}\right)  ^{\operatorname*{tr}%
}f\right\Vert _{L^{p}}\approx\left\Vert \left(  \sum_{Q\in\mathcal{D}%
}\left\vert \sum_{I\in\mathcal{D}}\left\langle f,h_{I;\kappa}\right\rangle
\left\langle h_{I;\kappa},h_{Q;\kappa}^{\eta}\right\rangle h_{Q;\kappa
}\right\vert ^{2}\right)  ^{\frac{1}{2}}\right\Vert _{L^{p}}\\
& \approx\left\Vert \left(  \sum_{Q\in\mathcal{D}}\left\vert \left\langle
f,h_{Q;\kappa}\right\rangle \left\langle h_{Q;\kappa},h_{Q;\kappa}^{\eta
}\right\rangle \right\vert ^{2}\left\vert h_{Q;\kappa}\right\vert ^{2}\right)
^{\frac{1}{2}}\right\Vert _{L^{p}}+O\left(  \left\Vert \left(  \sum
_{Q\in\mathcal{D}}\left\vert \sum_{I\in\mathcal{D}:\ I\neq Q}\left\langle
f,h_{I;\kappa}\right\rangle \left\langle h_{I;\kappa},h_{Q;\kappa}^{\eta
}\right\rangle \right\vert ^{2}\left\vert h_{Q}\right\vert ^{2}\right)
^{\frac{1}{2}}\right\Vert _{L^{p}}\right) \\
& \approx\left\Vert \left(  \sum_{Q\in\mathcal{D}}\left\vert \left\langle
f,h_{Q;\kappa}\right\rangle \right\vert ^{2}\frac{1}{\left\vert Q\right\vert
}\mathbf{1}_{Q}\right)  ^{\frac{1}{2}}\right\Vert _{L^{p}}^{p}+O\left(
\left\Vert \left(  \sum_{Q\in\mathcal{D}}\frac{1}{\left\vert Q\right\vert
}\left\vert \sum_{I\in\mathcal{D}:\ I\neq Q}\left\langle f,h_{I;\kappa
}\right\rangle \left\langle h_{I;\kappa},h_{Q;\kappa}^{\eta}\right\rangle
\right\vert ^{2}\mathbf{1}_{Q}\right)  ^{\frac{1}{2}}\right\Vert _{L^{p}%
}\right)  ,
\end{align*}
where for some $C_{p},c_{p}>0$,
\[
C_{p}\left\Vert f\right\Vert _{L^{p}}^{p}\geq\left\Vert \left(  \sum
_{Q\in\mathcal{D}}\left\vert \left\langle f,h_{Q;\kappa}\right\rangle
\right\vert ^{2}\frac{1}{\left\vert Q\right\vert }\mathbf{1}_{Q}\right)
^{\frac{1}{2}}\right\Vert _{L^{p}}^{p}=\left\Vert \left(  \sum_{Q\in
\mathcal{D}}\left\vert \bigtriangleup_{Q;\kappa}f\right\vert ^{2}\right)
^{\frac{1}{2}}\right\Vert _{L^{p}}^{p}\geq c_{p}\left\Vert f\right\Vert
_{L^{p}}^{p}\ .
\]

Thus we have for each $Q\in\mathcal{D}$,
\begin{align*}
& \sum_{I\in\mathcal{D}:\ I\neq Q}\left\langle f,h_{I;\kappa}\right\rangle
\left\langle h_{I;\kappa},h_{Q;\kappa}^{\eta}\right\rangle =\sum
_{\substack{I\in\mathcal{D}:\ \ell\left(  I\right)  <\ell\left(  Q\right)
\\I\in\operatorname*{Car}\left(  Q\right)  }}\left\langle f,h_{I;\kappa
}\right\rangle \left\langle h_{I;\kappa},h_{Q;\kappa}^{\eta}\right\rangle
+\sum_{\substack{I\in\mathcal{D}:\ \ell\left(  I\right)  >\ell\left(
Q\right)  \\Q\cap\mathcal{H}_{\frac{\eta}{2}}\left(  I\right)  \neq\emptyset
}}\left\langle f,h_{I;\kappa}\right\rangle \left\langle h_{I;\kappa
},h_{Q;\kappa}^{\eta}\right\rangle \\
&
\ \ \ \ \ \ \ \ \ \ \ \ \ \ \ \ \ \ \ \ \ \ \ \ \ \ \ \ \ \ \ \ \ \ \ \ \ \ \ \ \ \ \ \ \ \ \ +\sum
_{\substack{I\in\mathcal{D}:\ \ell\left(  I\right)  \geq\ell\left(  Q\right)
\geq\eta\ell\left(  I\right)  \\Q\in\operatorname*{Car}\left(  I\right)
}}\left\langle f,h_{I;\kappa}\right\rangle \left\langle h_{I;\kappa
},h_{Q;\kappa}^{\eta}\right\rangle .
\end{align*}
As a consequence of the estimates in (\ref{cases}), we have for each
$Q\in\mathcal{D}$,%
\begin{align*}
\left\vert \sum_{I\in\mathcal{D}:\ I\neq Q}\left\langle f,h_{I;\kappa
}\right\rangle \left\langle h_{I;\kappa},h_{Q;\kappa}^{\eta}\right\rangle
\right\vert  & \lesssim\left\vert \sum_{\substack{I\in\mathcal{D}%
:\ \ell\left(  I\right)  <\ell\left(  Q\right)  \\I\in\operatorname*{Car}%
\left(  Q\right)  }}\left\langle f,h_{I;\kappa}\right\rangle \left\langle
h_{I;\kappa},h_{Q;\kappa}^{\eta}\right\rangle \right\vert +\sum
_{\substack{I\in\mathcal{D}:\ \ell\left(  Q\right)  \leq\eta\ell\left(
I\right)  \\Q\cap\mathcal{H}_{\frac{\eta}{2}}\left(  I\right)  \neq\emptyset
}}\left\vert \left\langle f,h_{I;\kappa}\right\rangle \right\vert \frac
{1}{\eta^{\kappa}}\left(  \frac{\ell\left(  Q\right)  }{\ell\left(  I\right)
}\right)  ^{\kappa+\frac{n}{2}}\\
& \ \ \ \ \ \ \ \ \ \ \ \ \ \ \ \ \ \ \ \ +\eta\sum_{\substack{I\in
\mathcal{D}:\ \ell\left(  I\right)  \geq\ell\left(  Q\right)  \geq\eta
\ell\left(  I\right)  \\Q\in\operatorname*{Car}\left(  I\right)  }}\left\vert
\left\langle f,h_{I;\kappa}\right\rangle \right\vert \left(  \frac{\ell\left(
Q\right)  }{\ell\left(  I\right)  }\right)  ^{\frac{n}{2}}\\
& \equiv A\left(  Q\right)  +B\left(  Q\right)  +C\left(  Q\right)  .
\end{align*}

Altogether we have%
\begin{align}
& \left\Vert \left(  \sum_{Q\in\mathcal{D}}\frac{1}{\left\vert Q\right\vert
}\left\vert \sum_{I\in\mathcal{D}:\ I\neq Q}\left\langle f,h_{I;\kappa
}\right\rangle \left\langle h_{I;\kappa},h_{Q;\kappa}^{\eta}\right\rangle
\right\vert ^{2}\mathbf{1}_{Q}\right)  ^{\frac{1}{2}}\right\Vert _{L^{p}%
}\lesssim\left\Vert \left(  \sum_{Q\in\mathcal{D}}\frac{1}{\left\vert
Q\right\vert }A\left(  Q\right)  ^{2}\mathbf{1}_{Q}\right)  ^{\frac{1}{2}%
}\right\Vert _{L^{p}}\label{two norms dual}\\
& \ \ \ \ \ \ \ \ \ \ \ \ \ \ \ +\left\Vert \left(  \sum_{Q\in\mathcal{D}%
}\frac{1}{\left\vert Q\right\vert }B\left(  Q\right)  ^{2}\mathbf{1}%
_{Q}\right)  ^{\frac{1}{2}}\right\Vert _{L^{p}}+\left\Vert \left(  \sum
_{Q\in\mathcal{D}}\frac{1}{\left\vert Q\right\vert }C\left(  Q\right)
^{2}\mathbf{1}_{Q}\right)  ^{\frac{1}{2}}\right\Vert _{L^{p}}.\nonumber
\end{align}

We now claim that
\begin{equation}
\left\Vert \left(  \sum_{Q\in\mathcal{D}}\frac{1}{\left\vert Q\right\vert
}\left\vert \sum_{I\in\mathcal{D}:\ I\neq Q}\left\langle f,h_{I;\kappa
}\right\rangle \left\langle h_{I;\kappa},h_{Q;\kappa}^{\eta}\right\rangle
\right\vert ^{2}\mathbf{1}_{Q}\right)  ^{\frac{1}{2}}\right\Vert _{L^{p}%
}\lesssim\eta^{\frac{1}{2}\gamma_{p}}\left(  \log_{2}\frac{1}{\eta}\right)
\left\Vert f\right\Vert _{L^{p}}.\label{claim that dual}%
\end{equation}
With this established, and taking $\kappa>\frac{n}{2}$, we obtain just as in
the proof of Proposition \ref{invert},$\ $%
\[
\left\Vert \left(  \sum_{Q\in\mathcal{D}}\frac{1}{\left\vert Q\right\vert
}\left\vert \sum_{I\in\mathcal{D}:\ I\neq Q}\left\langle f,h_{I;\kappa
}\right\rangle \left\langle h_{I;\kappa},h_{Q;\kappa}^{\eta}\right\rangle
\right\vert ^{2}\right)  ^{\frac{1}{2}}\right\Vert _{L^{p}}\leq C\eta
^{\frac{1}{2}\gamma_{p}}\left(  \log_{2}\frac{1}{\eta}\right)  \left\Vert
f\right\Vert _{L^{p}}<\frac{c_{p}}{2}\left\Vert f\right\Vert _{L^{p}}\ ,
\]
with$\ \eta>0$ sufficiently small. This then gives%
\[
C_{p}\left\Vert f\right\Vert _{L^{p}}\geq\left\Vert \left(  S_{\eta
}^{\mathcal{D}}\right)  ^{\operatorname*{tr}}f\right\Vert _{L^{p}}\geq
c_{p}\left\Vert f\right\Vert _{L^{p}}-\frac{c_{p}}{2}\left\Vert f\right\Vert
_{L^{p}}=\frac{c_{p}}{2}\left\Vert f\right\Vert _{L^{p}}\ ,
\]
which completes the proof of Proposition \ref{invert dual} modulo
(\ref{claim that dual}).

We prove (\ref{claim that dual}) by estimating each of the three terms on the
right hand side of (\ref{two norms dual}) separately. These three terms are
handled exactly as in Proposition \ref{invert} except that the arguments for
handling terms $A$ and $C$ are switched, with term $B$ handled the same as
before. We leave the routine verifications to the reader, and this finishes
our proof of Proposition \ref{invert dual}.
\end{proof}

\subsubsection{Representation}

Combining the two propositions above immediately gives the proof of Theorem
\ref{reproducing}, as we now show.

\begin{proof}
[Proof of Theorem \ref{reproducing}]Fix a grid $\mathcal{D}$ in $\mathbb{R}%
^{n}$. Combining the two propositions shows that $S_{\eta}^{\mathcal{D}}$ is a
bounded invertible linear map on $L^{p}$. Indeed, Proposition \ref{invert}
shows that $S_{\eta}^{\mathcal{D}}$ is one-to-one and Proposition
\ref{invert dual} shows that $S_{\eta}^{\mathcal{D}}$ is onto. The boundedness
of $S_{\eta}^{\mathcal{D}}$ is immediate from Proposition \ref{invert} and the
boundedness of $\left(  S_{\eta}^{\mathcal{D}}\right)  ^{-1}$ now follows from
the Open Mapping Theorem.

Thus dropping the superscript $\mathcal{D}$\ we have%
\[
f=S_{\eta}\left(  S_{\eta}\right)  ^{-1}f=\sum_{I\in\mathcal{D}}\left\langle
\left(  S_{\eta}\right)  ^{-1}f,h_{I;\kappa}\right\rangle h_{I;\kappa}^{\eta
}\ .
\]
If we set
\[
\widetilde{\bigtriangleup}_{I}^{\eta}f\equiv\left\langle S_{\eta}%
^{-1}f,h_{I;\kappa}\right\rangle h_{I;\kappa}^{\eta}=\bigtriangleup_{I}^{\eta
}\left(  S_{\eta}^{-1}f\right)  =\left\langle S_{\eta}^{-1}f,h_{I;\kappa
}\right\rangle \left(  \phi_{\eta\ell\left(  I\right)  }\ast h_{I;\kappa
}\right)  ,
\]
then we have%
\begin{align*}
f  & =\sum_{I\in\mathcal{D}}\widetilde{\bigtriangleup}_{I}^{\eta}f=\sum
_{I\in\mathcal{D}}\left\langle S_{\eta}^{-1}f,h_{I;\kappa}\right\rangle
h_{I;\kappa}^{\eta},\ \ \ \ \ \text{for }f\in L^{p},\\
\left\Vert \left(  \sum_{I\in\mathcal{D}}\left\vert \widetilde{\bigtriangleup
}_{I}^{\eta}f\right\vert ^{2}\right)  ^{\frac{1}{2}}\right\Vert _{L^{p}\left(
\sigma\right)  }  & \approx\left\Vert \left(  \sum_{I\in\mathcal{D}}\left\vert
\left\langle S_{\eta}^{-1}f,h_{I;\kappa}\right\rangle \right\vert ^{2}\frac
{1}{\left\vert I\right\vert _{\sigma}}\mathbf{1}_{I}\right)  ^{\frac{1}{2}%
}\right\Vert _{L^{p}\left(  \sigma\right)  }\approx\left\Vert S_{\eta}%
^{-1}f\right\Vert _{L^{p}\left(  \sigma\right)  }\approx\left\Vert
f\right\Vert _{L^{p}\left(  \sigma\right)  }\ ,\\
\left\Vert \left(  \sum_{I\in\mathcal{D}}\left\vert \bigtriangleup_{I}^{\eta
}f\right\vert ^{2}\right)  ^{\frac{1}{2}}\right\Vert _{L^{p}\left(
\sigma\right)  }  & \approx\left\Vert \left(  \sum_{I\in\mathcal{D}}\left\vert
\left\langle f,h_{I;\kappa}\right\rangle \right\vert ^{2}\frac{1}{\left\vert
I\right\vert _{\sigma}}\mathbf{1}_{I}\right)  ^{\frac{1}{2}}\right\Vert
_{L^{p}\left(  \sigma\right)  }\approx\left\Vert f\right\Vert _{L^{p}\left(
\sigma\right)  }\ ,
\end{align*}
which shows in particular that $\left\{  \widetilde{\bigtriangleup}_{I;\kappa
}^{\eta}\right\}  _{I\in\mathcal{D}}$ is a frame for $L^{p}$.
\end{proof}

\begin{notation}
Since the frame $\left\{  \widetilde{\bigtriangleup}_{I;\kappa}^{\eta
}\right\}  _{I\in\mathcal{D}}$ will be used extensively in what follows, we
drop the tilde and write $\bigtriangleup_{I;\kappa}^{\eta}$ instead of
$\widetilde{\bigtriangleup}_{I;\kappa}^{\eta}$, i.e. we \emph{redefine}
$\bigtriangleup_{I;\kappa}^{\eta}f$ to be
\[
\bigtriangleup_{I}^{\eta}f\equiv\sum_{I\in\mathcal{D}}\left\langle S_{\eta
}^{-1}f,h_{I;\kappa}\right\rangle h_{I;\kappa}^{\eta},
\]
as was done in the Introduction. Thus we have inserted the bounded invertible
operator $S_{\eta}^{-1}$ into the inner product above.
\end{notation}

\subsubsection{The smoothed pseudoprojections}

The smoothed operators $\bigtriangleup_{I;\kappa}^{\eta}$ are neither
self-adjoint, projections nor orthogonal, but come close as we now show.
Recall that%
\[
\bigtriangleup_{I;\kappa}^{\eta}f=\left\langle \left(  S_{\kappa,\eta}\right)
^{-1}f,h_{I;\kappa}\right\rangle h_{I;\kappa}^{\eta}\ ,\ \ \ \ \ \text{where
}h_{I;\kappa}^{\eta}=\phi_{\eta}\ast h_{I;\kappa}\ .
\]

\begin{lemma}
\label{smoothed proj}With notation as above and $\phi=\phi_{0}\ast\phi_{0}$,
we have%
\[
\left(  \bigtriangleup_{I;\kappa}^{\eta}\right)  ^{\operatorname*{tr}%
}g=\left\langle g,h_{I;\kappa}^{\eta}\right\rangle \left(  \left(
S_{\kappa,\eta}\right)  ^{-1}\right)  ^{\operatorname*{tr}}h_{I;\kappa}\ ,
\]
and%
\begin{align*}
& \left(  \bigtriangleup_{I;\kappa}^{\eta}\right)  ^{2}=a_{I;\kappa}^{\eta
}\bigtriangleup_{I;\kappa}^{\eta}\text{ and }\left[  \left(  \bigtriangleup
_{I;\kappa}^{\eta}\right)  ^{\operatorname*{tr}}\right]  ^{2}=a_{I;\kappa
}^{\eta}\left(  \bigtriangleup_{I;\kappa}^{\eta}\right)  ^{\operatorname*{tr}%
}\text{ and }\left(  \bigtriangleup_{I;\kappa}^{\eta}\right)  \left(
\bigtriangleup_{I;\kappa}^{\eta}\right)  ^{\operatorname*{tr}}=b_{I;\kappa
}^{\eta}\widetilde{\bigtriangleup}_{I;\kappa}^{\eta}=b_{I;\kappa}^{\eta
}\widetilde{\bigtriangleup}_{I;\kappa}^{\eta},\\
& \text{ where }\widetilde{\bigtriangleup}_{I;\kappa}^{\eta}f=\left\langle
f,h_{I;\kappa}^{\eta}\right\rangle h_{I;\kappa}^{\eta}\ ,\text{ and }\\
& \text{where }a_{I;\kappa}^{\eta}\equiv\left\langle \left(  S_{\kappa,\eta
}\right)  ^{-1}h_{I;\kappa}^{\eta},h_{I;\kappa}\right\rangle \approx1\text{
and }b_{I;\kappa}^{\eta}\equiv\left\langle \left(  S_{\kappa,\eta}\right)
^{-2}h_{I;\kappa},h_{I;\kappa}\right\rangle \approx1.
\end{align*}
In particular we have%
\[
\left\Vert f\right\Vert _{L^{p}}\approx\left\Vert \left(  \sum_{I\in
\mathcal{D}}\frac{\left\vert \left\langle f,h_{I;\kappa}^{\eta}\right\rangle
\right\vert ^{2}}{\left\vert I\right\vert }\mathbf{1}_{I}\right)  ^{\frac
{1}{2}}\right\Vert _{L^{p}}.
\]

\end{lemma}

\begin{proof}
The adjoint property follows from%
\begin{align*}
\left\langle \bigtriangleup_{I;\kappa}^{\eta}f,g\right\rangle  & =\left\langle
\left\langle \left(  S_{\kappa,\eta}\right)  ^{-1}f,h_{I;\kappa}\right\rangle
h_{I;\kappa}^{\eta},g\right\rangle =\left\langle h_{I;\kappa}^{\eta
},g\right\rangle \int\left(  S_{\kappa,\eta}\right)  ^{-1}f\left(  x\right)
h_{I;\kappa}\left(  x\right)  dx\\
& =\left\langle h_{I;\kappa}^{\eta},g\right\rangle \int f\left(  x\right)
\left(  \left(  S_{\kappa,\eta}\right)  ^{-1}\right)  ^{\operatorname*{tr}%
}h_{I;\kappa}\left(  x\right)  dx\\
& =\int f\left(  x\right)  \left\{  \left(  \left(  S_{\kappa,\eta}\right)
^{-1}\right)  ^{\operatorname*{tr}}h_{I;\kappa}\left(  x\right)  \left\langle
h_{I;\kappa}^{\eta},g\right\rangle \right\}  dx=\left\langle f,\left(
\bigtriangleup_{I;\kappa}^{\eta}\right)  ^{\operatorname*{tr}}g\right\rangle .
\end{align*}
The pseudoprojection property follows from%
\begin{align*}
& \left(  \bigtriangleup_{I;\kappa}^{\eta}\right)  ^{2}f=\bigtriangleup
_{I;\kappa}^{\eta}\left(  \bigtriangleup_{I;\kappa}^{\eta}f\right)
=\left\langle \left(  S_{\kappa,\eta}\right)  ^{-1}\left(  \bigtriangleup
_{I;\kappa}^{\eta}f\right)  ,h_{I;\kappa}\right\rangle h_{I;\kappa}^{\eta}\\
& =\left\langle \left(  S_{\kappa,\eta}\right)  ^{-1}\left\{  \left\langle
\left(  S_{\kappa,\eta}\right)  ^{-1}f,h_{I;\kappa}\right\rangle h_{I;\kappa
}^{\eta}\right\}  ,h_{I;\kappa}\right\rangle h_{I;\kappa}^{\eta}=\left\langle
\left(  S_{\kappa,\eta}\right)  ^{-1}f,h_{I;\kappa}\right\rangle \left\langle
\left(  S_{\kappa,\eta}\right)  ^{-1}h_{I;\kappa}^{\eta},h_{I;\kappa
}\right\rangle h_{I;\kappa}^{\eta}\\
& =\left\langle \left(  S_{\kappa,\eta}\right)  ^{-1}h_{I;\kappa}^{\eta
},h_{I;\kappa}\right\rangle \left\langle \left(  S_{\kappa,\eta}\right)
^{-1}f,h_{I;\kappa}\right\rangle h_{I;\kappa}^{\eta}=\left\langle \left(
S_{\kappa,\eta}\right)  ^{-1}h_{I;\kappa}^{\eta},h_{I;\kappa}\right\rangle
_{\mu}\bigtriangleup_{I;\kappa}^{\eta}f=a_{I;\kappa}^{\eta}\bigtriangleup
_{I;\kappa}^{\eta}f.
\end{align*}
However, $\left(  S_{\kappa,\eta}\right)  ^{-1}$ is close to the identity map
by (\ref{little oh}), so that using $\phi_{\eta}=\phi_{\eta_{0}}\ast\phi
_{\eta_{0}}$, we obtain%
\begin{align*}
a_{I;\kappa}^{\eta}  & =\left\langle \left(  S_{\kappa,\eta}\right)
^{-1}h_{I;\kappa}^{\eta},h_{I;\kappa}\right\rangle \approx\left\langle
h_{I;\kappa}^{\eta},h_{I;\kappa}\right\rangle +o\left(  1\right)
=\left\langle \phi_{\eta\ell\left(  I\right)  }\ast h_{I;\kappa},h_{I;\kappa
}\right\rangle +o\left(  1\right) \\
& =\left\langle \phi_{\eta_{0}\ell\left(  I\right)  }\ast h_{I;\kappa}%
,\phi_{\eta_{0}\ell\left(  I\right)  }\ast h_{I;\kappa}\right\rangle +o\left(
1\right)  =\left\Vert h_{I;\kappa}^{\eta_{0}}\right\Vert _{L^{2}}^{2}+o\left(
1\right)  \approx\left\Vert h_{I;\kappa}\right\Vert _{L^{2}}^{2}+o\left(
1\right)  \approx1.
\end{align*}
We also compute%
\begin{align*}
& \left(  \bigtriangleup_{I;\kappa}^{\eta}\right)  \left(  \bigtriangleup
_{I;\kappa}^{\eta}\right)  ^{\operatorname*{tr}}f=\left\langle \left(
S_{\kappa,\eta}\right)  ^{-1}\left(  \bigtriangleup_{I;\kappa}^{\eta}\right)
^{\operatorname*{tr}}f,h_{I;\kappa}\right\rangle h_{I;\kappa}^{\eta}\\
& =\left\langle \left(  S_{\kappa,\eta}\right)  ^{-1}\left\{  \left\langle
f,h_{I;\kappa}^{\eta}\right\rangle \left(  S_{\kappa,\eta}\right)
^{-1}h_{I;\kappa}\right\}  ,h_{I;\kappa}\right\rangle h_{I;\kappa}^{\eta
}=\left\langle f,h_{I;\kappa}^{\eta}\right\rangle \left\langle \left(
S_{\kappa,\eta}\right)  ^{-2}h_{I;\kappa},h_{I;\kappa}\right\rangle
h_{I;\kappa}^{\eta}\\
& =\left\langle \left(  S_{\kappa,\eta}\right)  ^{-2}h_{I;\kappa},h_{I;\kappa
}\right\rangle \left\langle f,h_{I;\kappa}^{\eta}\right\rangle h_{I;\kappa
}^{\eta}=\left\langle \left(  S_{\kappa,\eta}\right)  ^{-2}h_{I;\kappa
},h_{I;\kappa}\right\rangle \widetilde{\bigtriangleup}_{I;\kappa}^{\eta}f.
\end{align*}
Finally,%
\[
f=\sum_{I\in\mathcal{D}}\left(  \bigtriangleup_{I;\kappa}^{\eta}\right)
^{\operatorname*{tr}}f=\sum_{I\in\mathcal{D}}\left\langle f,h_{I;\kappa}%
^{\eta}\right\rangle \left[  \left(  S_{\kappa,\eta}^{\operatorname*{tr}%
}\right)  ^{-1}\right]  ^{\operatorname*{tr}}h_{I;\kappa}=\left[  \left(
S_{\kappa,\eta}^{\operatorname*{tr}}\right)  ^{-1}\right]
^{\operatorname*{tr}}\sum_{I\in\mathcal{D}}\left\langle f,h_{I;\kappa}^{\eta
}\right\rangle h_{I;\kappa}%
\]
shows that%
\[
\left\Vert f\right\Vert _{L^{p}}=\left\Vert \left[  \left(  S_{\kappa,\eta
}^{\operatorname*{tr}}\right)  ^{-1}\right]  ^{\operatorname*{tr}}\sum
_{I\in\mathcal{D}}\left\langle f,h_{I;\kappa}^{\eta}\right\rangle h_{I;\kappa
}\right\Vert _{L^{p}}\approx\left\Vert \sum_{I\in\mathcal{D}}\left\langle
f,h_{I;\kappa}^{\eta}\right\rangle h_{I;\kappa}\right\Vert _{L^{p}}%
\approx\left\Vert \left(  \sum_{I\in\mathcal{D}}\frac{\left\vert \left\langle
f,h_{I;\kappa}^{\eta}\right\rangle \right\vert ^{2}}{\left\vert I\right\vert
}\mathbf{1}_{I}\right)  ^{\frac{1}{2}}\right\Vert _{L^{p}}%
\]

\end{proof}

\section{The extension operator and oscillatory inner products}

Given $f\in L^{p}\left(  \sigma_{n-1}\right)  $, we define the extension
operator $E_{\chi}$ localized to a cutoff function $\chi\left(  x\right)  $
by
\[
E_{\chi}f\left(  \xi\right)  =\mathcal{F}\left(  f\sigma_{n-1}\right)  \left(
\xi\right)  =\int_{\mathbb{S}^{n-1}}f\left(  z\right)  e^{-iz\cdot\xi}%
\chi\left(  z\right)  d\sigma_{n-1}\left(  z\right)  .
\]
If we use a one-to-one onto coordinate patch $\Phi:U\rightarrow\mathbb{P}$
such that $\operatorname*{Supp}\chi\subset\mathbb{P}$ and $U$ is a cube
centered at the origin in $\mathbb{R}^{n-1}$ with dyadic side length, then we
can write%
\begin{align*}
E_{\chi}f\left(  \xi\right)   & =\int_{\mathbb{P}}f\left(  y\right)
e^{-iy\cdot\xi}\chi\left(  y\right)  d\sigma_{n-1}\left(  y\right)  =\int
_{U}f\left(  \Phi\left(  x\right)  \right)  e^{-i\Phi\left(  x\right)
\cdot\xi}\chi\left(  \Phi\left(  x\right)  \right)  \frac{dx}{\left\vert
\det\nabla\Phi\left(  x\right)  \right\vert }\\
& =\int_{U}h\left(  x\right)  e^{-i\Phi\left(  x\right)  \cdot\xi}\zeta\left(
x\right)  dx
\end{align*}
where%
\[
h\left(  x\right)  =f\left(  \Phi\left(  \mathbb{P}x\right)  \right)  \text{
and }\zeta\left(  x\right)  \equiv\frac{\chi\left(  \Phi\left(  x\right)
\right)  }{\left\vert \det\nabla\Phi\left(  x\right)  \right\vert }.
\]
Since the map $\Phi:U\rightarrow\mathbb{P}$ is a diffeomorphism, we have%
\[
\left\Vert h\right\Vert _{L^{p}\left(  U\right)  }\approx\left\Vert
f\right\Vert _{L^{p}\left(  \mathbb{P}\right)  },
\]
and thus the extension operator $E_{\chi}:L^{p}\left(  \sigma_{n-1}\right)
\rightarrow L^{p}\left(  \mathbb{R}^{n}\right)  $ is bounded if and only if
the linear map $T:L^{p}\left(  U\right)  \rightarrow L^{p}\left(
\mathbb{R}^{n}\right)  $ is bounded, where $T$ is defined by%
\begin{align*}
& Tf\left(  \xi\right)  \equiv\int_{B_{n-1}\left(  0,\frac{1}{2}\right)
}K_{\Phi,\zeta}\left(  x,\xi\right)  f\left(  x\right)  dx=\int_{B_{n-1}%
\left(  0,\frac{1}{2}\right)  }f\left(  x\right)  e^{-i\Phi\left(  x\right)
\cdot\xi}dx,\ \ \ \ \ \text{\ for }f\in L^{p}\left(  B_{n-1}\left(  0,\frac
{1}{2}\right)  \right)  ,\\
& \ \ \ \ \ \ \ \ \ \ \ \ \ \ \ \text{where }K_{\Phi,\zeta}\left(
x,\xi\right)  \equiv e^{-i\Phi\left(  x\right)  \cdot\xi}.
\end{align*}

Now recall the $\left(  n-1\right)  $-dimensional Alpert wavelets $\left\{
h_{I;\kappa}^{n-1}\right\}  _{I\in\mathcal{G}}$ on $\mathbb{R}^{n-1}$ where
$\mathcal{G}$ is a translation of the standard dyadic grid on $\mathbb{R}%
^{n-1}$ so that $S\in\mathcal{G}$ and the origin is a vertex of $\pi
_{\mathcal{G}}^{\left(  2\right)  }S$ (see also Notation \ref{Notation Alpert}%
), and recall the smooth analogues $h_{I;\kappa}^{n-1,\eta}$ of these wavelets
as constructed in Theorem \ref{frame}\ above. Then expand $f$ by the smooth
Alpert reproducing formula $f=S_{\kappa,\eta}S_{\kappa,\eta}^{-1}f=\sum
_{I\in\mathcal{G}}\left\langle S_{\kappa,\eta}^{-1}f,h_{I;\kappa}%
^{n-1}\right\rangle h_{I;\kappa}^{n-1,\eta}$. In addition recall the
$n$-dimensional Alpert wavelets $\left\{  h_{J;\kappa}^{n}\right\}
_{J\in\mathcal{D}}$ on $\mathbb{R}^{n}$, where $\mathcal{D}$ is the standard
grid on $\mathbb{R}^{n}$, together with their smooth analogues $h_{J;\kappa
}^{n,\eta}$. It will be important, at least in a technical sense when
estimating part of the above form in Section \ref{Sec above}, to use the
standard grid $\mathcal{D}$ on $\mathbb{R}^{n}$ which enjoys the property that
the distance from the origin to a cube $J\in\mathcal{D}$ is at least the side
length of $J$, if the origin is not a vertex of $J$.

To estimate the left hand side $\left\Vert T\sum_{I\in\mathcal{G}\left[
U\right]  }\bigtriangleup_{I;\kappa}^{\eta}f\right\Vert _{L^{p}\left(
\lambda_{n}\right)  }$ of the truncated extension inequality (\ref{T_S trunc})
when $p=q$, we will use in particular the vanishing moments up to order
$\kappa-1$ of the wavelets $h_{I;\kappa}^{n-1,\eta}$ and $h_{J;\kappa}%
^{n,\eta}$,%
\begin{align*}
\int_{\mathbb{R}^{n-1}}h_{I;\kappa}^{n-1,\eta}\left(  x\right)  x^{\alpha}dx
& =0,\ \ \ \ \ \text{for }0\leq\left\vert \alpha\right\vert <\kappa,\\
\int_{\mathbb{R}^{n}}h_{J;\kappa}^{n,\eta}\left(  \xi\right)  \xi^{\alpha}d\xi
& =0,\ \ \ \ \ \text{for }0\leq\left\vert \alpha\right\vert <\kappa,
\end{align*}
along with estimates for oscillatory integrals in which the amplitudes involve
smooth Alpert wavelets.

We will now estimate the oscillatory inner product
\begin{equation}
\left\langle Th_{I;\kappa}^{n-1,\eta},h_{J;\kappa}^{n,\eta}\right\rangle
=\int_{\mathbb{R}^{n}}\left(  \int_{\mathbb{R}^{n-1}}e^{-i\Phi\left(
x\right)  \cdot\xi}h_{I;\kappa}^{n-1,\eta}\left(  x\right)  dx\right)
h_{J;\kappa}^{n,\eta}\left(  \xi\right)  d\xi,\label{osc inn pdt}%
\end{equation}
for $\left(  I,J\right)  \in\mathcal{G}\left[  U\right]  \times\mathcal{D}$
and plug the\ resulting estimates into the decomposition of the pairs $\left(
I,J\right)  $ of dyadic cubes in $\mathcal{P}$ given in (\ref{decomp}) of the
introduction, namely%
\[
\mathcal{P}=\mathcal{P}_{0}\ \cup\ \bigcup_{m=0}^{\infty}\mathcal{P}_{m}%
\ \cup\ \mathcal{R}\ \cup\ \mathcal{X}\ ,
\]
where for the convenience of the reader we recall the definitions,%
\begin{align*}
\mathcal{P}_{0}  & \equiv\left\{  \left(  I,J\right)  \in\mathcal{G}\left[
U\right]  \times\mathcal{D}:\pi_{\tan}\left(  J\right)  \subset\Phi\left(
C_{\operatorname{pseudo}}I\right)  \right\}  \ ,\\
\mathcal{P}_{m}  & \equiv\left\{  \left(  I,J\right)  \in\mathcal{G}\left[
U\right]  \times\mathcal{D}:2^{m+1}I\subset2U\text{, }\pi_{\tan}\left(
J\right)  \subset\Phi\left(  4U\cap2^{m+1}C_{\operatorname{pseudo}}I\right)
\setminus\Phi\left(  2^{m}\frac{1}{C_{\operatorname{pseudo}}}I\right)
\right\}  ,\ \ \ \ \ m\in\mathbb{N}\ ,\\
\mathcal{R}  & \equiv\left\{  \left(  I,J\right)  \in\mathcal{G}\left[
U\right]  \times\mathcal{D}:\Phi\left(  I\right)  \subset\pi_{\tan}\left(
C_{\operatorname{pseudo}}J\right)  \right\}  \ ,\\
\mathcal{X}  & \equiv\left\{  \left(  I,J\right)  \in\mathcal{G}\left[
U\right]  \times\mathcal{D}:J\subset\mathbb{R}_{+}^{n}\text{ and }\pi_{\tan
}\left(  C_{\operatorname{pseudo}}J\right)  \cap\Phi\left(  2U\right)
=\emptyset\right\}  \ .
\end{align*}
Thus $\mathcal{P}_{0}$ consists of pairs that are aligned radially away from
the origin, $\mathcal{P}_{m}$ consists of pairs that are radially staggered at
angle roughly $2^{-m}$, $\mathcal{R}$ consists of pairs where $I$ is `close'
to the larger $J$, and $\mathcal{X}$ consists of pairs in which the spherical
projection of $J$ is disjoint from $\Phi\left(  2U\right)  $.

Regarding $\mathcal{P}_{0}$, intuition tells us that when the approximate
wavelength $\frac{1}{\left\vert \xi\right\vert }$ of the exponential
$e^{-ix\cdot\xi}$ does not exceed the depth $\frac{1}{\ell\left(  I\right)
^{2}}$ of the spherical `cap' $\Phi\left(  I\right)  $, and the side length
$\ell\left(  J\right)  $ of the cube $J$ supporting $h_{J;\kappa}^{n,\eta}$ is
approximately the distance of the sphere from the origin, namely $1$, then we
should \emph{not} expect to derive any cancellation from the presence of the
exponential $e^{-i\Phi\left(  x\right)  \cdot\xi}$. Thus the only estimate on
the inner product in this case should be the trivial one, in which the
oscillatory factor $e^{-i\Phi\left(  x\right)  \cdot\xi}$ is discarded,%
\begin{equation}
\left\vert \left\langle Th_{I;\kappa}^{n-1,\eta},h_{J;\kappa}^{n,\eta
}\right\rangle \right\vert \leq\left\Vert h_{I;\kappa}^{n-1,\eta}\right\Vert
_{L^{1}}\left\Vert h_{J;\kappa}^{n,\eta}\right\Vert _{L^{1}}.\label{crude'}%
\end{equation}
While this crude estimate will ultimately prove adequate in the case when
$\ell\left(  J\right)  \approx1$, $\frac{1}{\ell\left(  I\right)  }%
\lesssim\frac{1}{\operatorname*{dist}\left(  0,J\right)  }\approx\frac
{1}{\left\vert \xi\right\vert }\lesssim\frac{1}{\ell\left(  I\right)  ^{2}}$
and $I$ and $J$ are suitably aligned in the same direction, we must obtain
improvements with geometric decay in parameters $\left\vert k\right\vert $ and
$d\geq0$ when
\[
\ell\left(  J\right)  =2^{k}\text{ and }\frac{2^{d-1}}{\ell\left(  I\right)
^{2}}\leq\operatorname*{dist}\left(  0,J\right)  \leq\frac{2^{d+1}}%
{\ell\left(  I\right)  ^{2}}\ell\left(  J\right)  \lesssim1.
\]
Moreover, when $I$ and $J$ are not suitably aligned, and there is insufficient
oscillation within the inner product, we will need to invoke interpolation
arguments with $L^{2}$ and average $L^{4}$ estimates when acting on certain
Alpert pseudoprojections.

When $k>0$, we will gain geometically if we integrate by parts radially in
$\xi$ using the smoothness of the wavelets $h_{J;\kappa}^{n,\eta}$, and when
$k<0$, we will gain geometrically in $\left\vert k\right\vert $ using the
large number of vanishing moments of $h_{J;\kappa}^{n,\eta}$. When $d>0 $, we
will use the classical asymptotic formula for the smooth surface carried
measure $h_{I;\kappa}^{n-1,\eta}$ with sharp bounds on the derivatives of
$h_{I;\kappa}^{n-1,\eta}$ to derive gain. Regarding $\mathcal{P}_{m}$, we will
use in addition a tangential integration by parts decay principle since the
critical point of the phase\ no longer lies in the support of the amplitude
(hence stationary phase is not needed here). This suggests that we further
decompose the index set $\mathcal{P}_{0}$ as%
\begin{align}
\mathcal{P}_{0}  & =\bigcup_{k\in Z}\bigcup_{d=1}^{\infty}\mathcal{P}%
_{0}^{k,d},\text{ where}\label{def PKD}\\
\mathcal{P}_{0}^{k,d}  & \equiv\left\{  \left(  I,J\right)  \in\mathcal{P}%
:J\subset\mathcal{K}\left(  I\right)  \text{, }\ell\left(  J\right)
=2^{k}\text{, and }\frac{2^{d-1}}{\ell\left(  I\right)  ^{2}}\leq
\operatorname*{dist}\left(  0,J\right)  =\frac{2^{d+1}}{\ell\left(  I\right)
^{2}}\right\}  ,\nonumber
\end{align}
for $k,d\in\mathbb{Z}$, and the index set $\mathcal{P}_{m}$ of pairs as%
\begin{align}
& \mathcal{P}_{m}=\bigcup_{k\in\mathbb{Z}}\bigcup_{d\in\mathbb{Z}}^{\infty
}\mathcal{P}_{m}^{k,d},\text{ where }\label{def PMKD}\\
& \mathcal{P}_{m}^{k,d}\equiv\left\{  \left(  I,J\right)  \in\mathcal{P}%
_{m}:2^{m+1}I\subset U\text{, }\ell\left(  J\right)  =2^{k}\text{, and }%
2^{d}\leq\ell\left(  I\right)  ^{2}\operatorname*{dist}\left(  0,J\right)
\leq2^{d+1}\right\}  ,\nonumber
\end{align}
for $k,d\in\mathbb{Z}$ and $m\in\mathbb{N}$.

Next we introduce a standard change of variable that simplifies calculations,
and then derive the well-known asymptotic formula we will use with estimates
on the remainder term\footnote{These estimates are undoubtedly in the
literature, but since the author was unable to find the precise form used
here, we include the classical arguments below.}.

\subsection{A change of variables}

Write $z=\left(  z^{\prime},z_{n}\right)  $ for $z\in\mathbb{R}^{n}$, and set
\begin{equation}
\phi\left(  x,y\right)  =\Phi\left(  x\right)  \cdot\Phi\left(  y\right)
,\ \ \ \ \ \text{where }\Phi\left(  x\right)  =\left(  x,\sqrt{1-\left\vert
x\right\vert ^{2}}\right)  \text{ and }x\in\mathbb{R}^{n-1},\label{choice}%
\end{equation}
and define the variables $\left(  y,\lambda\right)  $ by%
\begin{equation}
y=\Phi^{-1}\left(  \frac{\xi}{\left\vert \xi\right\vert }\right)  =\frac
{\xi^{\prime}}{\left\vert \xi\right\vert }\text{ and }\lambda=\left\vert
\xi\right\vert ,\ \ \ \ \ \text{i.e. }\left(  \xi^{\prime},\xi_{n}\right)
=\xi=\lambda\Phi\left(  y\right)  =\left(  \lambda y,\lambda\sqrt{1-\left\vert
y\right\vert ^{2}}\right)  ,\label{para}%
\end{equation}
since then
\[
\lambda\phi\left(  x,y\right)  =\left\vert \xi\right\vert \Phi\left(
x\right)  \cdot\Phi\left(  y\right)  =\left\vert \xi\right\vert \Phi\left(
x\right)  \cdot\frac{\xi}{\left\vert \xi\right\vert }=\Phi\left(  x\right)
\cdot\xi\ .
\]

We claim that%
\[
\det\frac{\partial\left(  \xi^{\prime},\xi_{n}\right)  }{\partial\left(
y,\lambda\right)  }=\frac{\left\vert \xi\right\vert ^{n}}{\xi_{n}}.
\]
Indeed, we have$\ \left(  y,\lambda\right)  =\left(  \frac{\xi^{\prime}%
}{\left\vert \xi\right\vert },\left\vert \xi\right\vert \right)  $ and
$\xi=\lambda\left(  y,\sqrt{1-\left\vert y\right\vert ^{2}}\right)  $, and so%
\begin{align*}
& \frac{\partial\left(  y_{1},...,y_{n-1},\lambda\right)  }{\partial\left(
\xi_{1},...,\xi_{n-1},\xi_{n}\right)  }=\left[
\begin{array}
[c]{cccc}%
\frac{\partial}{\partial\xi_{1}}\frac{\xi_{1}}{\left\vert \xi\right\vert } &
\cdots & \frac{\partial}{\partial\xi_{n-1}}\frac{\xi_{1}}{\left\vert
\xi\right\vert } & \frac{\partial}{\partial\xi_{n}}\frac{\xi_{1}}{\left\vert
\xi\right\vert }\\
\vdots & \ddots & \vdots & \vdots\\
\frac{\partial}{\partial\xi_{1}}\frac{\xi_{n-1}}{\left\vert \xi\right\vert } &
\cdots & \frac{\partial}{\partial\xi_{n-1}}\frac{\xi_{n-1}}{\left\vert
\xi\right\vert } & \frac{\partial}{\partial\xi_{n}}\frac{\xi_{n-1}}{\left\vert
\xi\right\vert }\\
\frac{\partial}{\partial\xi_{1}}\left\vert \xi\right\vert  & \cdots &
\frac{\partial}{\partial\xi_{n-1}}\left\vert \xi\right\vert  & \frac{\partial
}{\partial\xi_{n}}\left\vert \xi\right\vert
\end{array}
\right] \\
& =\left[
\begin{array}
[c]{cccc}%
\frac{1}{\left\vert \xi\right\vert }-\frac{\xi_{1}^{2}}{\left\vert
\xi\right\vert ^{3}} & \cdots & -\frac{\xi_{1}\xi_{n-1}}{\left\vert
\xi\right\vert ^{3}} & -\frac{\xi_{1}\xi_{n}}{\left\vert \xi\right\vert ^{3}%
}\\
\vdots & \ddots & \vdots & \vdots\\
-\frac{\xi_{1}\xi_{n-1}}{\left\vert \xi\right\vert ^{3}} & \cdots & \frac
{1}{\left\vert \xi\right\vert }-\frac{\xi_{n-1}^{2}}{\left\vert \xi\right\vert
^{3}} & -\frac{\xi_{n-1}\xi_{n}}{\left\vert \xi\right\vert ^{3}}\\
\frac{\xi_{1}}{\left\vert \xi\right\vert } & \cdots & \frac{\xi_{n-1}%
}{\left\vert \xi\right\vert } & \frac{\xi_{n}}{\left\vert \xi\right\vert }%
\end{array}
\right]  =\frac{1}{\left\vert \xi\right\vert ^{3}}\left[
\begin{array}
[c]{cccc}%
\left\vert \xi\right\vert ^{2}-\xi_{1}^{2} & \cdots & -\xi_{1}\xi_{n-1} &
-\xi_{1}\xi_{n}\\
\vdots & \ddots & \vdots & \vdots\\
-\xi_{1}\xi_{n-1} & \cdots & \left\vert \xi\right\vert ^{2}-\xi_{n-1}^{2} &
-\xi_{n-1}\xi_{n}\\
\xi_{1}\left\vert \xi\right\vert ^{2} & \cdots & \xi_{n-1}\left\vert
\xi\right\vert ^{2} & \xi_{n}\left\vert \xi\right\vert ^{2}%
\end{array}
\right]
\end{align*}
where%
\begin{align*}
& \det\left[
\begin{array}
[c]{cccc}%
\left\vert \xi\right\vert ^{2}-\xi_{1}^{2} & \cdots & -\xi_{1}\xi_{n-1} &
-\xi_{1}\xi_{n}\\
\vdots & \ddots & \vdots & \vdots\\
-\xi_{1}\xi_{n-1} & \cdots & \left\vert \xi\right\vert ^{2}-\xi_{n-1}^{2} &
-\xi_{n-1}\xi_{n}\\
\xi_{1}\left\vert \xi\right\vert ^{2} & \cdots & \xi_{n-1}\left\vert
\xi\right\vert ^{2} & \xi_{n}\left\vert \xi\right\vert ^{2}%
\end{array}
\right] \\
& =\left\vert \xi\right\vert ^{2}\det\left[
\begin{array}
[c]{cccc}%
\left\vert \xi\right\vert ^{2}-\xi_{1}^{2} & \cdots & -\xi_{1}\xi_{n-1} &
-\xi_{1}\xi_{n}\\
\vdots & \ddots & \vdots & \vdots\\
-\xi_{1}\xi_{n-1} & \cdots & \left\vert \xi\right\vert ^{2}-\xi_{n-1}^{2} &
-\xi_{n-1}\xi_{n}\\
\xi_{1} & \cdots & \xi_{n-1} & \xi_{n}%
\end{array}
\right]  =\left\vert \xi\right\vert ^{2}\xi_{n}\left\vert \xi\right\vert
^{2\left(  n-1\right)  }=\xi_{n}\left\vert \xi\right\vert ^{2n},
\end{align*}
by an induction on $n\in\mathbb{N}$.

Thus we have%
\begin{align*}
\det\frac{\partial\left(  y_{1},...,y_{n-1},\lambda\right)  }{\partial\left(
\xi_{1},...,\xi_{n-1},\xi_{n}\right)  }  & =\frac{1}{\left\vert \xi\right\vert
^{3n}}\det\left[
\begin{array}
[c]{cccc}%
\left\vert \xi\right\vert ^{2}-\xi_{1}^{2} & \cdots & -\xi_{1}\xi_{n-1} &
-\xi_{1}\xi_{n}\\
\vdots & \ddots & \vdots & \vdots\\
-\xi_{1}\xi_{n-1} & \cdots & \left\vert \xi\right\vert ^{2}-\xi_{n-1}^{2} &
-\xi_{n-1}\xi_{n}\\
\xi_{1}\left\vert \xi\right\vert ^{2} & \cdots & \xi_{n-1}\left\vert
\xi\right\vert ^{2} & \xi_{n}\left\vert \xi\right\vert ^{2}%
\end{array}
\right] \\
& =\frac{1}{\left\vert \xi\right\vert ^{3n}}\xi_{n}\left\vert \xi\right\vert
^{2n}=\frac{\xi_{n}}{\left\vert \xi\right\vert ^{n}},
\end{align*}
as claimed. Hence%
\[
\det\frac{\partial\left(  \xi_{1},...,\xi_{n-1},\xi_{n}\right)  }%
{\partial\left(  y_{1},...,y_{n-1},\lambda\right)  }=\frac{\left\vert
\xi\right\vert ^{n}}{\xi_{n}}=\frac{\lambda^{n}}{\lambda\sqrt{1-\left\vert
y\right\vert ^{2}}}=\frac{\lambda^{n-1}}{\sqrt{1-\left\vert y\right\vert ^{2}%
}},
\]
and the change of variable $\xi\rightarrow\left(  y,\lambda\right)  $ gives,%
\begin{align*}
& \left\langle Th_{I;\kappa}^{n-1,\eta},h_{J;\kappa}^{n,\eta}\right\rangle
=\int_{\mathbb{R}^{n}}\int_{B_{n-1}\left(  0,\frac{1}{2}\right)  }%
e^{i\Phi\left(  x\right)  \cdot\xi}h_{I;\kappa}^{n-1,\eta}\left(  x\right)
h_{J;\kappa}^{n,\eta}\left(  \xi\right)  dxd\xi\\
& =\int_{\mathbb{R}^{n}}\int_{B_{n-1}\left(  0,\frac{1}{2}\right)  }%
e^{i\Phi\left(  x\right)  \cdot\lambda\left(  y,\sqrt{1-\left\vert
y\right\vert ^{2}}\right)  }h_{I;\kappa}^{n-1,\eta}\left(  x\right)
h_{J;\kappa}^{n,\eta}\left(  \lambda\left(  y,\sqrt{1-\left\vert y\right\vert
^{2}}\right)  \right)  \det\frac{\partial\left(  \xi_{1},...,\xi_{n-1},\xi
_{n}\right)  }{\partial\left(  y_{1},...,y_{n-1},\lambda\right)  }%
dxdyd\lambda\\
& =\int_{\mathbb{R}}\int_{B_{n-1}\left(  0,\frac{1}{2}\right)  }\int
_{B_{n-1}\left(  0,\frac{1}{2}\right)  }e^{i\lambda\Phi\left(  x\right)
\cdot\Phi\left(  y\right)  }h_{I;\kappa}^{n-1,\eta}\left(  x\right)
h_{J;\kappa}^{n,\eta}\left(  \lambda y,\lambda\sqrt{1-\left\vert y\right\vert
^{2}}\right)  \frac{\lambda^{n}}{\lambda\sqrt{1-\left\vert y\right\vert ^{2}}%
}dxdyd\lambda\\
& =\int_{\mathbb{R}}\int_{B_{n-1}\left(  0,\frac{1}{2}\right)  }\int
_{B_{n-1}\left(  0,\frac{1}{2}\right)  }e^{i\lambda\phi\left(  x,y\right)
}\varphi_{I}^{\eta}\left(  x\right)  \widetilde{\psi}_{J}^{\eta}\left(
y,\lambda\right)  dxdyd\lambda,
\end{align*}
where we are now using the convenient notation,%
\begin{align}
\phi\left(  x,y\right)   & \equiv\Phi\left(  x\right)  \cdot\Phi\left(
y\right)  ,\label{phi and psi notation}\\
\varphi_{I}^{\eta}\left(  x\right)   & \equiv h_{I;\kappa}^{n-1,\eta}\left(
x\right)  \text{ and }\psi_{J}^{\eta}\left(  \xi\right)  =h_{J;\kappa}%
^{n,\eta}\left(  \xi\right)  ,\nonumber\\
\widetilde{\psi}_{J}^{\eta}\left(  y,\lambda\right)   & \equiv h_{J;\kappa
}^{n,\eta}\left(  \lambda y,\lambda\sqrt{1-\left\vert y\right\vert ^{2}%
}\right)  \frac{\lambda^{n-1}}{\sqrt{1-\left\vert y\right\vert ^{2}}%
}.\nonumber
\end{align}
Note that if $\xi\in J$, then $\left(  y,\lambda\right)  \in\pi_{\tan}%
J\times\pi_{\operatorname{rad}}J$.

\subsection{Bounds for oscillatory integrals}

Here we review the well known asymptotics for oscillatory integrals, see e.g.
\cite[Chapter VIII]{Ste2}, paying close attention to the constants involved.
We emphasize that the results in this subsection are well known, but as we
could not find in the literature the exact form of the estimate for the
remainder term that we use here, we reproduce many familiar arguments below.

We consider the oscillatory function $\mathcal{I}_{a_{\lambda},\phi
}:\mathbb{R}^{d}\times\left(  0,\infty\right)  \rightarrow\mathbb{C}$ given by%
\[
\mathcal{I}_{a_{\lambda},\phi}\left(  y,\lambda\right)  \equiv\int
_{\mathbb{R}^{n}}e^{i\lambda\phi\left(  x,y\right)  }a_{\lambda}\left(
x,y\right)  dx,
\]
defined for $\lambda>0$ and $y\in U$ where $U$ is an open subset of
$\mathbb{R}^{d}$, and we call $\phi\left(  x,y\right)  $ the phase and
$a_{\lambda}\left(  x,y\right)  $ the amplitude of $\mathcal{I}_{a_{\lambda
},\phi}$. We will follow a treatment of asymptotics for such oscillatory
integrals given in a Rice University blog \cite{blogs.rice.edu}, but we will
obtain a sharp estimate for amplitudes of the type that arise in the smooth
Alpert expansions.

We use three familiar preparatory lemmas. The first of these is the Morse
Lemma, which will be applied to the phase function $\phi\left(  x,y\right)  $,
in order to transform $\phi$ into a nonsingular quadratic form in $x$ at a
nondegenerate critical point in $x$. The second lemma gives high order decay
bounds in the special case when there are no critical points in $x$ of the
phase function that lie in the support of the amplitude, and the third
calculates the oscillatory integral for a quadratic form.

\begin{lemma}
[Morse Lemma]Suppose $y_{0}\in U\subset\mathbb{R}^{d}$ and $x_{0}$ is a
nondegenerate stationary point for $\phi\left(  \cdot_{x},y_{0}\right)  $.
Then there exists a neighbourhood $V\subset U$ of $y_{0}$, a neighbourhood $W$
of $x_{0} $, a smooth function%
\[
X:V\rightarrow W,
\]
and a smooth function
\[
\Psi:V\rightarrow W\times V\rightarrow\mathbb{R}^{n},
\]
such that

\begin{enumerate}
\item For every $y\in V$, $X\left(  y\right)  $ is the unique stationary
point, which is also nondegenerate, for $\phi\left(  \cdot_{x},y_{0}\right)  $
in $W$.

\item For every $y\in V$, the map $W\rightarrow\mathbb{R}^{n}$ defined by
$x\rightarrow\Psi\left(  x,y\right)  $ is a diffeomorphism onto its image and%
\begin{equation}
\phi\left(  x,y\right)  =\phi\left(  X\left(  y\right)  ,y\right)  +\frac
{1}{2}\Psi\left(  x,y\right)  ^{\operatorname*{tr}}\ \left[  \partial_{x}%
^{2}\phi\left(  X\left(  y\right)  ,y\right)  \right]  \ \Psi\left(
x,y\right)  .\label{diff exp}%
\end{equation}
Furthermore,
\begin{equation}
\Psi\left(  X\left(  y\right)  ,y\right)  =0\text{ and }\partial_{x}%
\Psi\left(  X\left(  y\right)  ,y\right)  =\operatorname{Id}_{n}%
.\label{Psi and d Psi}%
\end{equation}

\item Finally, we may take $W=B\left(  x_{0},a\gamma\right)  $ for some small
positive constant
\[
a=\frac{c_{n}}{\max_{\left\vert \alpha\right\vert \leq3}\sup_{\left(
x,y\right)  \in\left(  \operatorname*{Supp}a\right)  \times U}\left\vert
\partial_{x}^{\alpha}\phi\left(  x,y\right)  \right\vert },
\]
where $\gamma>0$ satisfies $\inf_{y}\left[  \partial_{x}^{2}\phi\left(
X\left(  y\right)  ,y\right)  \right]  \succcurlyeq\gamma\operatorname{Id}%
_{n}$.
\end{enumerate}
\end{lemma}

\begin{proof}
For any $y$, the stationary points are the solutions of the equation
$0=\partial_{x}\phi\left(  x,y\right)  $, and by the nondegeneracy of the
critical point, and the Implicit Function Theorem, this equation uniquely
defines $x$ as a function of $y$ in some neighbourhood $\mathcal{N}$ of
$\left(  x_{0},y_{0}\right)  $. Since in our application, $\phi\left(
x,y\right)  $ is homogeneous of degree zero in $y$, we may assume this here as
well. Then $\left[  \partial_{x}^{2}\phi\left(  X\left(  y\right)  ,y\right)
\right]  \succcurlyeq\gamma\operatorname{Id}_{n-1}$ for some $\gamma>0$
depending only on $\phi$, and so we may take $\mathcal{N}=B\left(  \left(
x_{0},y_{0}\right)  ,a^{\prime}\gamma\right)  $ where $a^{\prime}=\frac
{c_{n}^{\prime}}{\max_{\left\vert \alpha\right\vert \leq3}\sup_{\left(
x,y\right)  }\left\vert \partial_{x}^{\alpha}\phi\left(  x,y\right)
\right\vert }$ for some small positive constant $c_{n}^{\prime}$, depending
only on the dimension $n$.

Now we take the Taylor expansion of $\phi\left(  x,y\right)  $ in $x$\ about
$X\left(  y\right)  $ to obtain, upon noting that the first derivatives in the
Taylor expansion vanish at the critical point $X\left(  y\right)  $,%
\begin{align*}
& \phi\left(  x,y\right)  =\phi\left(  X\left(  y\right)  ,y\right)  +\frac
{1}{2}\left(  x-X\left(  y\right)  \right)  ^{\operatorname*{tr}}B\left(
x,y\right)  \left(  x-X\left(  y\right)  \right)  ,\\
& \text{where }B\left(  x,y\right)  \equiv\int_{0}^{1}\left(  1-s\right)
\partial_{x}^{2}\phi\left(  sx+\left(  1-s\right)  X\left(  y\right)
,y\right)  ds.
\end{align*}
We now construct a matrix-valued function $R\left(  x,y\right)  $ such that
\[
\Psi\left(  x,y\right)  \equiv R\left(  x,y\right)  \left(  x-X\left(
y\right)  \right)
\]
has the properties listed in (2) above. Indeed, this $\Psi$ will satisfy
(\ref{diff exp}) provided%
\begin{equation}
R\left(  x,y\right)  ^{\operatorname*{tr}}\partial_{x}^{2}\phi\left(  X\left(
y\right)  ,y\right)  R\left(  x,y\right)  -B\left(  x,y\right)
=0,\ \ \ \ \ \text{for }\left(  x,y\right)  \in\mathcal{N}.\label{def R(x,y)}%
\end{equation}
We interpret the left hand side of (\ref{def R(x,y)}) as a mapping from
$\mathcal{M}_{n}\left(  \mathbb{R}\right)  _{R}\times\mathbb{R}_{x}^{n}\times
V_{y}$ to $\mathcal{S}_{n}\left(  \mathbb{R}\right)  $, where $\mathcal{M}%
_{n}\left(  \mathbb{R}\right)  $ is the set of $n\times n$ matrices and
$\mathcal{S}_{n}\left(  \mathbb{R}\right)  $ is the subset of symmetric
matrices. Taking the differential of the left hand side of (\ref{def R(x,y)})
with respect to the variable $R$ and evaluated at the identity matrix
$\operatorname{Id}_{n}$, we obtain that the derivative map,%
\[
dR\rightarrow\left(  dR\right)  ^{\operatorname*{tr}}\partial_{x}^{2}%
\phi\left(  X\left(  y\right)  ,y\right)  +\partial_{x}^{2}\phi\left(
X\left(  y\right)  ,y\right)  \left(  dR\right)  ,
\]
is surjective, since whenever $C\in\mathcal{S}_{n}\left(  \mathbb{R}\right)  $
is symmetric,%
\begin{align*}
& \left(  \frac{1}{2}\left[  \partial_{x}^{2}\phi\left(  X\left(  y\right)
,y\right)  \right]  ^{-1}C\right)  ^{\operatorname*{tr}}\partial_{x}^{2}%
\phi\left(  X\left(  y\right)  ,y\right)  +\partial_{x}^{2}\phi\left(
X\left(  y\right)  ,y\right)  \left(  \frac{1}{2}\left[  \partial_{x}^{2}%
\phi\left(  X\left(  y\right)  ,y\right)  \right]  ^{-1}C\right) \\
& \ \ \ \ \ \ \ \ \ \ \ \ \ \ \ =\frac{1}{2}C+\frac{1}{2}C=C.
\end{align*}
Thus by the Implicit Function Theorem again, there exists a smooth
$\mathcal{M}_{n}\left(  \mathbb{R}\right)  $-valued function $R\left(
x,y\right)  $ defined on some neighbourhood $\mathcal{N}_{0}\subset
\mathcal{N}$ of $\left(  x_{0},y_{0}\right)  $ that satisfies
(\ref{def R(x,y)}) everywhere that it is defined. Note that we may take
$\mathcal{N}_{0}=B\left(  \left(  x_{0},y_{0}\right)  ,a^{\prime\prime}%
\gamma\right)  $ where where $a^{\prime\prime}=\frac{c_{n}^{\prime\prime}%
}{\max_{\left\vert \alpha\right\vert \leq3}\sup_{\left(  x,y\right)
}\left\vert \partial_{x}^{\alpha}\phi\left(  x,y\right)  \right\vert }$.
Possibly shrinking\ even more the neighbourhood $\mathcal{N}_{0}$ to
$\mathcal{N}_{1}$, completes the proof that there is a neighbourhood $W$ of
$x_{0}$ such that $x\rightarrow\Psi\left(  x,y\right)  $ is a diffeomorphism
from $W$ onto its image, and that (\ref{diff exp}) holds, and that
$\Psi\left(  X\left(  y\right)  ,y\right)  =0$. Note that we may take
$W=B\left(  x_{0},a\gamma\right)  $ where $a=\frac{c_{n}}{\max_{\left\vert
\alpha\right\vert \leq3}\sup_{\left(  x,y\right)  }\left\vert \partial
_{x}^{\alpha}\phi\left(  x,y\right)  \right\vert }$. The remaining assertion
$\partial_{x}\Psi\left(  X\left(  y\right)  ,y\right)  =\operatorname{Id}_{n}$
is straightforward since,%
\[
\partial_{x}\mid_{x=X\left(  y\right)  }\Psi\left(  X\left(  y\right)
,y\right)  =\left[  \partial_{x}R\left(  x,y\right)  \left(  x-X\left(
y\right)  \right)  +R\left(  x,y\right)  \right]  \mid_{x=X\left(  y\right)
}=R\left(  X\left(  y\right)  ,y\right)  =\operatorname{Id}_{n}\ ,
\]
because we evaluated the differential in $R$ of the left hand side of
(\ref{def R(x,y)}) at the identity matrix $\operatorname{Id}_{n}$.
\end{proof}

Recall that%
\[
\mathcal{I}_{a_{\lambda},\phi}\left(  y,\lambda\right)  \equiv\int
_{\mathbb{R}^{n}}e^{i\lambda\phi\left(  x,y\right)  }a_{\lambda}\left(
x,y\right)  dx,
\]
where $\phi\in C^{\infty}\left(  \mathbb{R}_{x}^{n}\times U_{y}\right)  $ and
$a_{\lambda}\in C^{\infty}\left(  \mathbb{R}_{x}^{n}\times U_{y}\right)  $. We
will need the following estimate in the absence of critical points for
$x\rightarrow\phi\left(  x,y\right)  $.

\begin{lemma}
\label{no crit}Suppose that the $\mathbb{R}^{n}$-valued function $\partial
_{x}\phi\left(  x,y\right)  $ is nonvanishing on $\left(  \operatorname*{Supp}%
a\right)  \times U$. Then for every $N\in\mathbb{N}$ and compact $K\Subset U$
we have%
\[
\sup_{y\in K}\left\vert \mathcal{I}_{a,\phi}\left(  y,\lambda\right)
\right\vert \leq C_{N,K}\frac{1}{\lambda^{N}}\sum_{\left\vert \alpha
\right\vert \leq N}\sup_{y\in K}\left\Vert \partial_{x}^{\alpha}a_{\lambda
}\right\Vert _{L^{1}\left(  \mathbb{R}^{n}\right)  },\text{\ \ \ \ \ for
}\left(  y,\lambda\right)  \in\left(  \operatorname*{Supp}a\right)  \times U.
\]

\end{lemma}

\begin{proof}
For any $M\in\mathbb{N}$ we have%
\[
\mathcal{I}_{a_{\lambda},\phi}\left(  y,\lambda\right)  =\int_{\mathbb{R}^{n}%
}\frac{\left\langle \partial_{x}\phi\left(  x,y\right)  ,\partial
_{x}\right\rangle ^{M}e^{i\lambda\phi\left(  x,y\right)  }}{\left(
i\lambda\left\vert \partial_{x}\phi\left(  x,y\right)  \right\vert
^{2}\right)  ^{M}}a_{\lambda}\left(  x,y\right)  dx,
\]
and integrating by parts gives%
\begin{align*}
\sup_{y\in K}\left\vert \mathcal{I}_{a_{\lambda},\phi}\left(  y,\lambda
\right)  \right\vert  & \leq\sup_{y\in K}\frac{1}{\lambda^{N}}\int
_{\mathbb{R}^{n}}\left\vert \left\langle \partial_{x},\frac{\partial_{x}%
\phi\left(  x,y\right)  }{\left\vert \partial_{x}\phi\left(  x,y\right)
\right\vert ^{2}}\right\rangle ^{N}a_{\lambda}\left(  x,y\right)  \right\vert
dx\\
& \leq C_{N,K}\frac{1}{\lambda^{N}}\sum_{\left\vert \alpha\right\vert \leq
N}\sup_{y\in K}\int_{\mathbb{R}^{n}}\left\vert \partial_{x}^{\alpha}%
a_{\lambda}\left(  x,y\right)  \right\vert dx\\
& =C_{N,K}\frac{1}{\lambda^{N}}\sum_{\left\vert \alpha\right\vert \leq N}%
\sup_{y\in K}\left\Vert \partial_{x}^{\alpha}a_{\lambda}\right\Vert
_{L^{1}\left(  \mathbb{R}^{n}\right)  \times L^{\infty}\left(  \mathbb{R}%
^{n}\right)  }.
\end{align*}

\end{proof}

The final preparatory lemma is the calculation of an oscillatory integral for
a quadratic form.

\begin{definition}
For a tempered distribution $u\in\mathcal{S}\left(  \mathbb{R}^{n}\right)  $,
we have%
\[
\widehat{u}\left(  \xi\right)  =\mathcal{F}\left(  u\right)  \left(
\xi\right)  =\int_{\mathbb{R}^{n}}e^{-ix\cdot\xi}u\left(  x\right)  d\left(
x\right)  .
\]

\end{definition}

\begin{lemma}
\label{quad}Let $A\in\mathcal{M}_{n}\left(  \mathbb{R}^{n}\right)  $ be
symmetric and nondegenerate with signature $\operatorname{sgn}\left(
A\right)  $. Then the tempered distribution $e^{ix^{\operatorname*{tr}}Ax}$
has Fourier transform given by,
\begin{equation}
\mathcal{F}\left(  e^{ix^{\operatorname*{tr}}Ax}\right)  \left(  \xi\right)
=\pi^{\frac{n}{2}}e^{i\operatorname{sgn}\left(  A\right)  \frac{\pi}{4}}%
\frac{e^{-i\frac{\xi^{\operatorname*{tr}}A^{-1}\xi}{4}}}{\sqrt{\det\left(
A\right)  }}.\label{temp}%
\end{equation}

\begin{proof}
The Fourier transform of a Gaussian function $e^{-t\left\vert x\right\vert
^{2}}$ is given by%
\[
\mathcal{F}\left(  e^{-t\left\vert x\right\vert ^{2}}\right)  \left(
\xi\right)  =\pi^{\frac{n}{2}}\frac{e^{-\frac{\left\vert \xi\right\vert ^{2}%
}{4t}}}{t^{\frac{n}{2}}},\ \ \ \ \ \text{for all }t>0\text{.}%
\]
Now note that both sides of the above identity extend to analytic functions of
$t$ in the right half plane $\left\{  t\in\mathbb{C}:\operatorname{Re}%
t>0\right\}  $. A standard limiting argument and orthogonal change of
variables gives the formula (\ref{temp}).
\end{proof}
\end{lemma}

\subsection{The main oscillatory integral bound}

Here is the main oscillatory integral bound.

\begin{remark}
In the application of stationary phase to bound the below form in Section 6,
we won't actually use the oscillatory term $e^{i\lambda\phi\left(  X\left(
y\right)  ,y\right)  }$ in the asymptotic formula below, and instead we only
need the estimates of the modulus of $\mathcal{I}_{a_{\lambda},\phi}\left(
y,\lambda\right)  $ that follow from the asymptotic formula using $\left\vert
e^{i\lambda\phi\left(  X\left(  y\right)  ,y\right)  }\right\vert =1 $. The
reason for this is that when dealing with the below subform $\mathsf{B}%
_{\operatorname{below}}^{k,d}\left(  f,g\right)  $ with $k,d\geq0$ large, we
can \emph{first} apply\ radial integration by parts in the inner product, and
\emph{second} apply stationary phase to the resulting inner product with a new
amplitude. This way the geometric gain in $k$ has been achieved without using
the oscillatory term $e^{i\lambda\phi\left(  X\left(  y\right)  ,y\right)  }$.
If we were to instead apply stationary phase first, then we would need
$e^{i\lambda\phi\left(  X\left(  y\right)  ,y\right)  }$ for integration by
parts afterward.
\end{remark}

\begin{remark}
We will only use the case $M=0$ of Theorem \ref{osc int} in the proof of the
probabilistic extension conjecture in Theorem \ref{FEC}, which corresponds to
the classical asymptotic formula with just the principal term and remainder,
but with a sharp estimate here on the remainder term when the amplitude is a
smooth Alpert wavelet.
\end{remark}

We now give a more general treatment of stationary phase than we need, which
might be of use elsewhere.

\begin{theorem}
\label{osc int}Suppose that $a_{\lambda}\left(  x,y\right)  \in C_{c}^{\infty
}\left(  \mathbb{R}_{x}^{n}\times\mathbb{R}_{y}^{d}\right)  $, $y_{0}\in
U\subset\mathbb{R}^{d}$, and that $\phi\left(  \cdot_{x},y_{0}\right)  $ has
exactly one nondegenerate stationary point on the support of $a$ at $x_{0}$.
Take $V$, $W$, $X$ and $\Psi$ as in the Morse Lemma. Then for every
$M\in\mathbb{N}$, there is a positive constant $C_{M}$ depending on $M$ and
$\phi$ such that,%
\[
\mathcal{I}_{a_{\lambda},\phi}\left(  y,\lambda\right)  =\mathfrak{P}%
_{a_{\lambda},\phi}\left(  y,\lambda\right)  +\sum_{\ell=1}^{M}\mathfrak{P}%
_{a_{\lambda},\phi}^{\left(  \ell\right)  }\left(  y,\lambda\right)
+\mathfrak{R}_{a_{\lambda},\phi}^{\left(  M+1\right)  }\left(  y,\lambda
\right)  ,
\]
where%
\begin{align*}
& \mathfrak{P}_{a_{\lambda},\phi}\left(  y,\lambda\right)  =\left(  \frac
{2\pi}{\lambda}\right)  ^{\frac{n}{2}}\frac{e^{i\left[  \operatorname{sgn}%
\left[  \partial_{x}^{2}\phi\left(  X\left(  y\right)  ,y\right)  \right]
\frac{\pi}{4}+\lambda\phi\left(  X\left(  y\right)  ,y\right)  \right]  }%
}{\sqrt{\left\vert \partial_{x}^{2}\phi\left(  X\left(  y\right)  ,y\right)
\right\vert }}a_{\lambda}\left(  X\left(  y\right)  ,y\right)  ,\\
& \mathfrak{P}_{a_{\lambda},\phi}^{\left(  \ell\right)  }\left(
y,\lambda\right)  =\frac{i^{\ell}}{\left(  2\lambda\right)  ^{\ell}\ell
!}\left(  \frac{2\pi}{\lambda}\right)  ^{\frac{n}{2}}\frac{e^{i\left[
\operatorname{sgn}B\left(  y\right)  \frac{\pi}{4}+\lambda\phi\left(  X\left(
y\right)  ,y\right)  \right]  }}{\sqrt{\det B\left(  y\right)  }}\\
& \ \ \ \ \ \ \ \ \ \ \times\left\{  \left[  \partial_{x}\frac{1}{\det
\partial_{x}\Psi\left(  x,y\right)  }\right]  B\left(  y\right)  ^{-1}\frac
{1}{\det\partial_{x}\Psi\left(  x,y\right)  }\partial_{x}\right\}  ^{\ell
}\frac{a_{\lambda}\left(  x,y\right)  }{\det\left[  \partial_{x}\Psi\left(
x,y\right)  \right]  }\mid_{x=X\left(  y\right)  },
\end{align*}
and%
\begin{align}
\mathfrak{R}_{a_{\lambda},\phi}^{\left(  M+1\right)  }\left(  y,\lambda
\right)   & =\left(  \frac{2\pi}{\lambda}\right)  ^{\frac{n}{2}}%
\frac{e^{i\left[  \operatorname{sgn}B\left(  y\right)  \frac{\pi}{4}%
+\lambda\phi\left(  X\left(  y\right)  ,y\right)  \right]  }}{\sqrt{\det
B\left(  y\right)  }}\nonumber\\
& \times\int\mathcal{F}_{z}^{-1}\left(  \left[  \frac{\left\langle
i\partial_{z},B\left(  y\right)  ^{-1}\partial_{z}\right\rangle }{2\lambda
}\right]  ^{M+1}f\right)  \left(  \zeta\right)  R_{M+1}\left(  -i\frac
{\zeta^{\operatorname*{tr}}B\left(  y\right)  ^{-1}\zeta}{2\lambda}\right)
d\zeta,\nonumber
\end{align}
where%
\[
f\left(  z,y,\lambda\right)  \equiv\frac{a_{\lambda}\left(  \Psi_{y}%
^{-1}\left(  z\right)  ,y\right)  }{\det\left[  \left(  \partial_{x}%
\Psi\right)  \left(  \Psi_{y}^{-1}\left(  z\right)  \right)  \right]  },
\]
and $B\left(  y\right)  =\partial_{x}^{2}\phi\left(  X\left(  y\right)
,y\right)  $, and $X\left(  y\right)  $ is the unique stationary point of
$\phi\left(  \cdot_{x},y\right)  $ in the support of $a$,\ as given in the
Morse Lemma, and finally,
\[
R_{M+1}\left(  ib\right)  =\int_{0}^{1}e^{itb}\left(  ib\right)  ^{M+1}%
\frac{\left(  1-t\right)  ^{M+1}}{\left(  M+1\right)  !}dt,\ \ \ \ \ \text{
for }b\in\mathbb{R}.
\]
The remainder term satisfies the estimate,%
\begin{equation}
\sup_{y\in V}\left\vert \mathfrak{R}_{a_{\lambda},\phi}^{\left(  M+1\right)
}\left(  y,\lambda\right)  \right\vert \leq C_{M}\lambda^{-\frac{n}{2}-\left(
M+1\right)  }\sum_{\left\vert \alpha\right\vert \leq\rho+2\left(  M+1\right)
}\left\Vert \partial_{x}^{\alpha}a_{\lambda}\right\Vert _{L^{2}\left(
\mathbb{R}_{x}^{n}\right)  \times L^{\infty}\left(  \mathbb{R}_{y,\lambda
}^{d+1}\right)  },\label{main bound}%
\end{equation}
where $\rho=\left\lceil \frac{n}{2}\right\rceil $ is the smallest integer
greater than $\frac{n}{2}$, and if $N>M+1+\frac{n}{2}$, then we also have the
alternate bound,%
\begin{equation}
\sup_{y\in V}\left\vert \mathfrak{R}_{a_{\lambda},\phi}^{\left(  M+1\right)
}\left(  y,\lambda\right)  \right\vert \leq C_{M}\lambda^{-\frac{n}{2}%
-M-1}\left\Vert \left(  \operatorname{Id}-\bigtriangleup_{x}\right)
^{N}a_{\lambda}\right\Vert _{L^{1}\left(  \mathbb{R}_{x}^{n}\right)  \times
L^{\infty}\left(  \mathbb{R}_{y}^{n}\right)  }.\label{special bound}%
\end{equation}

\end{theorem}

\begin{proof}
Take $V$, $W$, $X$ and $\Psi$ as in the Morse Lemma, so that%
\[
\phi\left(  x,y\right)  =\phi\left(  X\left(  y\right)  ,y\right)  +\frac
{1}{2}\Psi\left(  x,y\right)  ^{\operatorname*{tr}}\ \left[  \partial_{x}%
^{2}\phi\left(  X\left(  y\right)  ,y\right)  \right]  \ \Psi\left(
x,y\right)  ,\ \ \ \ \ y\in V.
\]
Using Lemma \ref{no crit} together with a partition of unity shows that we may
assume $a_{\lambda}\left(  x,y\right)  $ is supported in $W$ for all $y\in V$.
Thus a change of variables
\[
z=\Psi\left(  x,y\right)  =\Psi_{y}\left(  x\right)  ,
\]
gives,%
\begin{align*}
\mathcal{I}_{a_{\lambda},\phi}\left(  y,\lambda\right)   & =\int
_{\mathbb{R}^{n}}e^{i\lambda\phi\left(  x,y\right)  }a_{\lambda}\left(
x,y\right)  dx=\int_{\mathbb{R}^{n}}e^{i\lambda\phi\left(  \Psi_{y}%
^{-1}z,y\right)  }\frac{a_{\lambda}\left(  \Psi_{y}^{-1}z,y\right)  }%
{\det\left[  \left(  \partial_{x}\Psi\right)  \left(  \Psi_{y}^{-1}\left(
z\right)  ,y\right)  \right]  }dz\\
& =\int_{\mathbb{R}^{n}}e^{i\lambda\left[  \phi\left(  x_{0},y_{0}\right)
+\Psi\left(  \Psi_{y}^{-1}z,y\right)  ^{\operatorname*{tr}}\frac{\partial
_{x}^{2}\phi\left(  X\left(  y\right)  ,y\right)  }{2}\Psi\left(  \Psi
_{y}^{-1}z,y\right)  \right]  }\frac{a_{\lambda}\left(  \Psi_{y}^{-1}\left(
z\right)  ,y\right)  }{\det\left[  \left(  \partial_{x}\Psi\right)  \left(
\Psi_{y}^{-1}\left(  z\right)  ,y\right)  \right]  }dz\\
& =\int_{\mathbb{R}^{n}}e^{i\lambda\left[  \phi\left(  x_{0},y_{0}\right)
+z^{\operatorname*{tr}}\frac{\partial_{x}^{2}\phi\left(  X\left(  y\right)
,y\right)  }{2}z\right]  }\frac{a_{\lambda}\left(  \Psi_{y}^{-1}\left(
z\right)  ,y\right)  }{\det\left[  \left(  \partial_{x}\Psi\right)  \left(
\Psi_{y}^{-1}\left(  z\right)  ,y\right)  \right]  }dz\\
& =e^{i\lambda\phi\left(  x_{0},y_{0}\right)  }\int_{\mathbb{R}^{n}%
}e^{i\lambda z^{\operatorname*{tr}}\frac{\partial_{x}^{2}\phi\left(  X\left(
y\right)  ,y\right)  }{2}z}f\left(  z,y,\lambda\right)  dz,
\end{align*}
where%
\[
f\left(  z,y,\lambda\right)  \equiv\frac{a_{\lambda}\left(  \Psi_{y}%
^{-1}\left(  z\right)  ,y\right)  }{\det\left[  \left(  \partial_{x}%
\Psi\right)  \left(  \Psi_{y}^{-1}\left(  z\right)  \right)  \right]  }.
\]
Now write%
\begin{equation}
B\left(  y\right)  =\left(  \partial_{x}^{2}\phi\right)  \left(  X\left(
y\right)  ,y\right)  ,\label{def B(y)}%
\end{equation}
and apply the Fourier transform $\mathcal{F}$ and its inverse $\mathcal{F}%
^{-1}$ in the variable $z$ and its dual variable $\zeta$ to obtain%
\[
\mathcal{I}_{a_{\lambda},\phi}\left(  y,\lambda\right)  =e^{i\lambda
\phi\left(  x_{0},y_{0}\right)  }\int_{\mathbb{R}^{n}}\mathcal{F}_{z}\left(
e^{i\lambda z^{\operatorname*{tr}}\frac{B\left(  y\right)  }{2}z}\right)
\left(  \zeta\right)  \ \mathcal{F}_{z}^{-1}\left(  f\left(  z,y\right)
\right)  \left(  \zeta\right)  d\zeta.
\]
Using Lemma \ref{quad} with $A=\frac{\lambda}{2}B\left(  y\right)  $, we have,%
\begin{align*}
\mathcal{I}_{a_{\lambda},\phi}\left(  y,\lambda\right)   & =e^{i\lambda
\phi\left(  x_{0},y_{0}\right)  }\frac{\pi^{\frac{n}{2}}e^{i\operatorname{sgn}%
B\left(  y\right)  \frac{\pi}{4}}}{\sqrt{\det\frac{\lambda}{2}B\left(
y\right)  }}\int_{\mathbb{R}^{n}}e^{-i\frac{\zeta^{\operatorname*{tr}}B\left(
y\right)  ^{-1}\zeta}{2\lambda}}\mathcal{F}_{z}^{-1}\left(  f\left(
z,y\right)  \right)  \left(  \zeta\right)  d\zeta\\
& =\left(  \frac{2\pi}{\lambda}\right)  ^{\frac{n}{2}}\frac
{e^{i\operatorname{sgn}B\left(  y\right)  \frac{\pi}{4}}e^{i\lambda\phi\left(
x_{0},y_{0}\right)  }}{\sqrt{\det B\left(  y\right)  }}\int_{\mathbb{R}^{n}%
}e^{-i\frac{\zeta^{\operatorname*{tr}}B\left(  y\right)  ^{-1}\zeta}{2\lambda
}}\mathcal{F}_{z}^{-1}\left(  f\left(  z,y\right)  \right)  \left(
\zeta\right)  d\zeta.
\end{align*}

Next we use Taylor's formula\ with integral remainder to obtain that for
any$\,M>0$,%
\[
e^{ib}=\sum_{\ell=0}^{M}\frac{\left(  ib\right)  ^{\ell}}{\ell!}%
+R_{M+1}\left(  ib\right)  ,
\]
where%
\[
R_{M+1}\left(  ib\right)  =\int_{0}^{1}e^{itb}\left(  ib\right)  ^{M+1}%
\frac{\left(  1-t\right)  ^{M+1}}{\left(  M+1\right)  !}dt\text{ and
}\left\vert R_{M+1}\left(  ib\right)  \right\vert \leq\frac{\left\vert
b\right\vert ^{M+1}}{\left(  M+2\right)  !}%
\]
and so with
\[
b=-\frac{\zeta^{\operatorname*{tr}}B\left(  y\right)  ^{-1}\zeta}{2\lambda},
\]
we have%
\begin{align}
& \mathcal{I}_{a_{\lambda},\phi}\left(  y,\lambda\right)  -\left(  \frac{2\pi
}{\lambda}\right)  ^{\frac{d}{2}}\frac{e^{i\left[  \operatorname{sgn}B\left(
y\right)  \frac{\pi}{4}+\lambda\phi\left(  X\left(  y\right)  ,y\right)
\right]  }}{\sqrt{\det B\left(  y\right)  }}\int_{\mathbb{R}^{n}}\sum_{\ell
=0}^{M}\frac{i^{\ell}}{\left(  2\lambda\right)  ^{\ell}\ell!}\mathcal{F}%
_{z}^{-1}\left(  \left\langle \partial_{z}^{\operatorname*{tr}}B\left(
y\right)  ^{-1}\partial_{z}\right\rangle ^{\ell}f\right)  \left(
\zeta\right)  d\zeta\label{ident with b}\\
& =\left(  \frac{2\pi}{\lambda}\right)  ^{\frac{n}{2}}\frac{e^{i\left[
\operatorname{sgn}B\left(  y\right)  \frac{\pi}{4}+\lambda\phi\left(  X\left(
y\right)  ,y\right)  \right]  }}{\sqrt{\det B\left(  y\right)  }}\nonumber\\
& \ \ \ \ \ \ \ \ \ \ \ \ \ \ \ \ \ \ \ \ \ \ \ \ \ \times\int_{\mathbb{R}%
^{n}}\mathcal{F}_{z}^{-1}\left(  \left[  \frac{\left\langle i\partial
_{z}^{\operatorname*{tr}}B\left(  y\right)  ^{-1}\partial_{z}\right\rangle
}{2\lambda}\right]  ^{M+1}f\right)  \left(  \zeta\right)  R_{M+1}\left(
-i\frac{\zeta^{\operatorname*{tr}}B\left(  y\right)  ^{-1}\zeta}{2\lambda
}\right)  d\zeta.\nonumber
\end{align}
Finally, using the Fourier inversion formula $\int_{\mathbb{R}^{n}}%
\mathcal{F}^{-1}\left(  g\right)  \left(  z\right)  dz=g\left(  0\right)  $,
together with the identities%
\begin{align*}
\Psi_{y}\left(  X\left(  y\right)  \right)   & =\Psi\left(  X\left(  y\right)
,y\right)  =0,\\
\Psi_{y}^{-1}\left(  0\right)   & =X\left(  y\right)  ,\\
\det\partial_{x}\Psi\left(  X\left(  y\right)  ,y\right)   & =\det
\operatorname{Id}_{n}=1,
\end{align*}
from part (2) of the Morse Lemma, we obtain%
\[
\int_{\mathbb{R}^{n}}\mathcal{F}_{z}^{-1}\left(  \left\langle \partial
_{z}^{\operatorname*{tr}}B\left(  y\right)  ^{-1}\partial_{z}\right\rangle
^{\ell}f\right)  \left(  \zeta\right)  d\zeta=\left\langle \partial
_{z}^{\operatorname*{tr}}B\left(  y\right)  ^{-1}\partial_{z}\right\rangle
^{\ell}f\left(  0\right)  ,\ \ \ \ \ 0\leq\ell\leq M.
\]
Now when $\ell=0$ we have%
\[
f\left(  0\right)  =\frac{a_{\lambda}\left(  \Psi_{y}^{-1}\left(  0\right)
,y\right)  }{\det\left[  \partial_{x}\Psi\left(  \Psi_{y}^{-1}\left(
0\right)  ,y\right)  \right]  }=\frac{a_{\lambda}\left(  X\left(  y\right)
,y\right)  }{\det\left[  \partial_{x}\Psi\left(  X\left(  y\right)  ,y\right)
\right]  }=a_{\lambda}\left(  X\left(  y\right)  ,y\right)  .
\]
From the change of variable $\left(  x,y\right)  \rightarrow\left(
z,w\right)  $ where $z=\Psi\left(  x,y\right)  $ and $w=y$, the Jacobian
matrix in block form is,%
\[
\frac{\partial\left(  z,w\right)  }{\partial\left(  x,y\right)  }=\left[
\begin{array}
[c]{cc}%
\partial_{x}z & \partial_{y}z\\
\partial_{x}w & \partial_{y}w
\end{array}
\right]  =\left[
\begin{array}
[c]{cc}%
\partial_{x}\Psi\left(  x,y\right)  & \partial_{y}\Psi\left(  x,y\right) \\
\operatorname{0}_{n} & \operatorname{Id}_{n}%
\end{array}
\right]  ,
\]
and so%
\[
\left[
\begin{array}
[c]{cc}%
\partial_{z}x & \partial_{w}x\\
\partial_{z}y & \partial_{w}y
\end{array}
\right]  =\frac{\partial\left(  x,y\right)  }{\partial\left(  z,w\right)
}=\left[
\begin{array}
[c]{cc}%
\partial_{x}\Psi\left(  x,y\right)  & \partial_{y}\Psi\left(  x,y\right) \\
\operatorname{0}_{n} & \operatorname{Id}_{n}%
\end{array}
\right]  ^{-1}=\frac{1}{\det\partial_{x}\Psi\left(  x,y\right)  }\left[
\begin{array}
[c]{cc}%
\operatorname{Id}_{n} & -\partial_{y}\Psi\left(  x,y\right) \\
\operatorname{0}_{n} & \partial_{x}\Psi\left(  x,y\right)
\end{array}
\right]  .
\]
Thus we have by the chain rule,%
\begin{align*}
\left(
\begin{array}
[c]{c}%
\partial_{z}\\
\partial_{w}%
\end{array}
\right)   & =\left[
\begin{array}
[c]{cc}%
\partial_{z}x & \partial_{z}y\\
\partial_{w}x & \partial_{w}y
\end{array}
\right]  \left(
\begin{array}
[c]{c}%
\partial_{x}\\
\partial_{y}%
\end{array}
\right)  =\frac{1}{\det\partial_{x}\Psi\left(  x,y\right)  }\left[
\begin{array}
[c]{cc}%
\operatorname{Id}_{n} & -\partial_{y}\Psi\left(  x,y\right) \\
\operatorname{0}_{n} & \partial_{x}\Psi\left(  x,y\right)
\end{array}
\right]  ^{\operatorname*{tr}}\left(
\begin{array}
[c]{c}%
\partial_{x}\\
\partial_{y}%
\end{array}
\right) \\
& =\frac{1}{\det\partial_{x}\Psi\left(  x,y\right)  }\left[
\begin{array}
[c]{cc}%
\operatorname{Id}_{n} & \operatorname{0}_{n}\\
-\partial_{y}\Psi\left(  x,y\right)  & \partial_{x}\Psi\left(  x,y\right)
\end{array}
\right]  \left(
\begin{array}
[c]{c}%
\partial_{x}\\
\partial_{y}%
\end{array}
\right) \\
& =\frac{1}{\det\partial_{x}\Psi\left(  x,y\right)  }\left(
\begin{array}
[c]{c}%
\partial_{x}\\
-\partial_{y}\Psi\left(  x,y\right)  \partial_{x}+\partial_{x}\Psi\left(
x,y\right)  \partial_{y}%
\end{array}
\right)  ,
\end{align*}
i.e.,%
\begin{equation}
\partial_{z}=\frac{1}{\det\partial_{x}\Psi\left(  x,y\right)  }\partial
_{x}.\label{d tilda}%
\end{equation}
Thus when $\ell=1$ we have
\begin{align*}
\left\langle \partial_{z}^{\operatorname*{tr}}B\left(  y\right)  ^{-1}%
\partial_{z}\right\rangle f\left(  0\right)   & =\left(  \partial
_{z}^{\operatorname*{tr}}B\left(  y\right)  ^{-1}\partial_{z}\frac{a_{\lambda
}\left(  \Psi_{y}^{-1}\left(  z\right)  ,y\right)  }{\det\left[  \left(
\partial_{x}\Psi\right)  \left(  \Psi_{y}^{-1}\left(  z\right)  ,y\right)
\right]  }\right)  \left(  0\right) \\
& =\left(  \left\{  \left[  \partial_{x}\frac{1}{\det\partial_{x}\Psi\left(
x,y\right)  }\right]  ^{\operatorname*{tr}}B\left(  y\right)  ^{-1}\frac
{1}{\det\partial_{x}\Psi\left(  x,y\right)  }\partial_{x}\right\}
\frac{a_{\lambda}\left(  x,y\right)  }{\det\left[  \partial_{x}\Psi\left(
x,y\right)  \right]  }\right)  \mid_{x=X\left(  y\right)  }\\
& =L\left(  y,\partial_{x}\right)  \frac{a_{\lambda}\left(  x,y\right)  }%
{\det\left[  \partial_{x}\Psi\left(  x,y\right)  \right]  }\mid_{x=X\left(
y\right)  },
\end{align*}
where
\[
L\left(  y,\partial_{x}\right)  \equiv\left[  \partial_{x}\frac{1}%
{\det\partial_{x}\Psi\left(  x,y\right)  }\right]  ^{\operatorname*{tr}%
}B\left(  y\right)  ^{-1}\frac{1}{\det\partial_{x}\Psi\left(  x,y\right)
}\partial_{x}%
\]
is a second order differential operator in $x$ with coefficients depending on
both $x$ and $y$. More generally, the same calculation shows that for
$0\leq\ell\leq M$, we have,%
\begin{align*}
\left\langle \partial_{z},B\left(  y\right)  ^{-1}\partial_{z}\right\rangle
^{\ell}f\left(  0\right)   & =\left(  \left\{  \left[  \partial_{x}\frac
{1}{\det\partial_{x}\Psi\left(  x,y\right)  }\right]  ^{\operatorname*{tr}%
}B\left(  y\right)  ^{-1}\frac{1}{\det\partial_{x}\Psi\left(  x,y\right)
}\partial_{x}\right\}  ^{\ell}\frac{a_{\lambda}\left(  x,y\right)  }%
{\det\left[  \partial_{x}\Psi\left(  x,y\right)  \right]  }\right)
\mid_{x=X\left(  y\right)  }\\
& =L\left(  y,\partial_{x}\right)  ^{\ell}\frac{a_{\lambda}\left(  x,y\right)
}{\det\left[  \partial_{x}\Psi\left(  x,y\right)  \right]  }\mid_{x=X\left(
y\right)  }.
\end{align*}

Thus the identity (\ref{ident with b}), together with the bound $\left\vert
g_{M+1}\left(  -i\frac{\xi^{\operatorname*{tr}}B\left(  y\right)  ^{-1}\xi
}{2\lambda}\right)  \right\vert \leq\frac{1}{\left(  M+1\right)  !}$, implies
that,%
\begin{align}
& \left\vert \mathcal{R}_{a_{\lambda},\phi}^{\left(  M+1\right)  }\left(
y,\lambda\right)  \right\vert \leq C_{M}\lambda^{-\frac{n}{2}-\left(
M+1\right)  }\left\Vert \mathcal{F}_{z}^{-1}\left(  \left\langle \partial
_{z},B\left(  y\right)  ^{-1}\partial_{z}\right\rangle ^{M+1}f\right)
R_{M+1}\right\Vert _{L^{1}\left(  \mathbb{R}_{\zeta}^{n}\right)
}\label{fin est}\\
& \leq C_{M,n}\lambda^{-\frac{n}{2}-\left(  M+1\right)  }\sum_{\left\vert
\alpha\right\vert \leq\rho+2\left(  M+1\right)  }\left\Vert \partial
_{x}^{\alpha}a_{\lambda}\right\Vert _{L^{2}\left(  \mathbb{R}_{x}^{n}\right)
\times L^{\infty}\left(  \mathbb{R}_{y}^{n}\right)  }\ ,\nonumber
\end{align}
where in the last line we have used Cauchy-Schwarz, the derivative identities
for $\mathcal{F}$, and Plancherel's theorem with the smallest integer
$\rho=\left\lceil \frac{n}{2}\right\rceil $ greater than $\frac{n}{2}$.
Indeed,%
\begin{align*}
\int_{\mathbb{R}^{n}}\left\vert \widehat{h}\left(  \xi\right)  \right\vert
d\xi & =\int_{\mathbb{R}^{n}}\left\vert \widehat{h}\left(  \xi\right)
\right\vert \left(  1+\left\vert \xi\right\vert ^{2}\right)  ^{\rho}\left(
1+\left\vert \xi\right\vert ^{2}\right)  ^{-\rho}d\xi\\
& \leq\left(  \int_{\mathbb{R}^{n}}\left\vert \left(  1+\left\vert
\xi\right\vert ^{2}\right)  ^{\rho}\widehat{h}\left(  \xi\right)  \right\vert
^{2}d\xi\right)  ^{\frac{1}{2}}\left(  \int_{\mathbb{R}^{n}}\left(
1+\left\vert \xi\right\vert ^{2}\right)  ^{-2\rho}d\xi\right)  ^{\frac{1}{2}%
}\\
& \leq C_{m}\left(  \int_{\mathbb{R}^{n}}\left\vert \left(  \operatorname{Id}%
_{n}-\bigtriangleup_{x}\right)  ^{\rho}h\left(  x\right)  \right\vert
^{2}dx\right)  ^{\frac{1}{2}},
\end{align*}
for the function
\begin{align*}
h\left(  x\right)   & =\left\langle \partial_{z},B\left(  y\right)
^{-1}\partial_{z}\right\rangle ^{M+1}f\\
& =\left\{  \left[  \partial_{x}\frac{1}{\det\partial_{x}\Psi\left(
x,y\right)  }\right]  B\left(  y\right)  ^{-1}\frac{1}{\det\partial_{x}%
\Psi\left(  x,y\right)  }\partial_{x}\right\}  ^{M+1}\frac{a_{\lambda}\left(
x,y\right)  }{\det\left[  \partial_{x}\Psi\left(  x,y\right)  \right]  }.
\end{align*}

To prove the alternate bound (\ref{special bound}), we use the estimate
$\left\vert R_{M+1}\left(  -i\frac{\zeta^{\operatorname*{tr}}B\left(
y\right)  ^{-1}\zeta}{2\lambda}\right)  \right\vert \lesssim\left\vert
\frac{\zeta^{\operatorname*{tr}}B\left(  y\right)  ^{-1}\zeta}{2\lambda
}\right\vert ^{M+1}$ to obtain,
\begin{align*}
& \left\Vert \mathcal{F}_{z}^{-1}\left(  \left\langle \partial_{z},B\left(
y\right)  ^{-1}\partial_{z}\right\rangle ^{M+1}f\right)  R_{M+1}\right\Vert
_{L^{1}\left(  \mathbb{R}_{\zeta}^{n}\right)  }\\
& \leq\frac{1}{\left(  M+1\right)  !}\left\Vert \mathcal{F}_{z}^{-1}\left(
\left\langle \partial_{z},B\left(  y\right)  ^{-1}\partial_{z}\right\rangle
^{M+1}f\right)  \right\Vert _{L^{1}\left(  \mathbb{R}_{\zeta}^{n}\right)
}\lesssim\left(  \frac{1}{\lambda}\right)  ^{M+1}\int_{\mathbb{R}^{n}%
}\left\vert \zeta\right\vert ^{2\left(  M+1\right)  }\left\vert \left(
\mathcal{F}_{z}f\right)  \left(  \zeta\right)  \right\vert d\zeta,
\end{align*}
where%
\begin{align*}
\left(  \mathcal{F}_{z}f\right)  \left(  \zeta\right)   & =\left(
\mathcal{F}_{z}\frac{a_{\lambda}\left(  \Psi_{y}^{-1}\left(  z\right)
,y\right)  }{\det\left[  \left(  \partial_{x}\Psi\right)  \left(  \Psi
_{y}^{-1}\left(  z\right)  \right)  \right]  }\right)  \left(  \zeta\right)
=\mathcal{F}_{z}\varphi_{y}\left(  \zeta\right)  ,\\
\varphi_{y}\left(  z\right)   & \equiv\frac{a_{\lambda}\left(  \Psi_{y}%
^{-1}\left(  z\right)  ,y\right)  }{\det\left[  \left(  \partial_{x}%
\Psi\right)  \left(  \Psi_{y}^{-1}\left(  z\right)  \right)  \right]  }.
\end{align*}
From the estimate%
\begin{align*}
\left\vert \mathcal{F}_{z}\varphi_{y}\left(  \zeta\right)  \right\vert  &
=\left\vert \int_{\mathbb{R}^{n}}e^{ix\cdot\zeta}\varphi_{y}\left(  x\right)
dx\right\vert =\left\vert \int_{\mathbb{R}^{n}}\left[  \left(  \frac
{\operatorname{Id}-\bigtriangleup_{x}}{1+\left\vert \zeta\right\vert ^{2}%
}\right)  ^{N}e^{ix\cdot\zeta}\right]  \varphi_{y}\left(  x\right)
dx\right\vert \\
& =\frac{1}{\left(  1+\left\vert \zeta\right\vert ^{2}\right)  ^{N}}\left\vert
\int_{\mathbb{R}^{n}}e^{ix\cdot\zeta}\left(  \operatorname{Id}-\bigtriangleup
_{x}\right)  ^{N}\varphi_{y}\left(  x\right)  dx\right\vert \leq\left\Vert
\left(  \operatorname{Id}-\bigtriangleup_{x}\right)  ^{N}\varphi
_{y}\right\Vert _{L^{1}}\frac{1}{\left(  1+\left\vert \zeta\right\vert
^{2}\right)  ^{N}},
\end{align*}
we have for $N>M+1+\frac{n}{2}$ that%
\begin{align*}
& \left(  \frac{1}{\lambda}\right)  ^{M+1}\int_{\mathbb{R}^{n}}\left\vert
\zeta\right\vert ^{2M+2}\left\vert \left(  \mathcal{F}_{z}f\right)  \left(
\zeta\right)  \right\vert d\zeta\lesssim\left(  \frac{1}{\lambda}\right)
^{M+1}\left\Vert \left(  \operatorname{Id}-\bigtriangleup_{x}\right)
^{N}\varphi_{y}\right\Vert _{L^{1}\left(  \mathbb{R}_{x}^{n}\right)  \times
L^{\infty}\left(  \mathbb{R}_{y}^{n}\right)  }\int_{\mathbb{R}^{n}}%
\frac{\left\vert \zeta\right\vert ^{2M+2}}{\left(  1+\left\vert \zeta
\right\vert ^{2}\right)  ^{N}}d\zeta\\
& \ \ \ \ \ \ \ \ \ \ \lesssim\left(  \frac{1}{\lambda}\right)  ^{M+1}%
\left\Vert \left(  \operatorname{Id}-\bigtriangleup_{x}\right)  ^{N}%
\varphi_{y}\right\Vert _{L^{1}\left(  \mathbb{R}_{x}^{n}\right)  \times
L^{\infty}\left(  \mathbb{R}_{y}^{n}\right)  }\lesssim\left(  \frac{1}%
{\lambda}\right)  ^{M+1}\left\Vert \left(  \operatorname{Id}-\bigtriangleup
_{x}\right)  ^{N}a_{\lambda}\right\Vert _{L^{1}\left(  \mathbb{R}_{x}%
^{n}\right)  \times L^{\infty}\left(  \mathbb{R}_{y}^{n}\right)  }\text{ }.
\end{align*}
We conclude that,%
\begin{align*}
\left\vert \mathcal{R}_{a_{\lambda},\phi}^{\left(  M+1\right)  }\left(
y,\lambda\right)  \right\vert  & \leq C_{M}\lambda^{-\frac{n}{2}-\left(
M+1\right)  }\left\Vert \mathcal{F}_{z}^{-1}\left(  \left\langle \partial
_{z},B\left(  y\right)  ^{-1}\partial_{z}\right\rangle ^{M+1}f\right)
g_{M+1}\right\Vert _{L^{1}\left(  \mathbb{R}_{\zeta}^{n}\right)  }\\
& \leq C_{M}\lambda^{-\left(  M+1+\frac{n}{2}\right)  }\left\Vert \left(
\operatorname{Id}-\bigtriangleup_{x}\right)  ^{N}a_{\lambda}\right\Vert
_{L^{1}\left(  \mathbb{R}_{x}^{n}\right)  \times L^{\infty}\left(
\mathbb{R}_{y}^{n}\right)  },\ \ \ \ \ \text{for }N>M+1+\frac{n}{2}.
\end{align*}

\end{proof}

\begin{remark}
\label{calc}The identity $\partial_{x}\Psi\left(  X\left(  y\right)
,y\right)  =\operatorname{Id}_{n}$ implies that $\det\left[  \partial_{x}%
\Psi\left(  X\left(  y\right)  ,y\right)  \right]  =1$. Thus for $\ell=1$ we
have%
\begin{align*}
& \partial_{x}\left\{  \frac{1}{\det\partial_{x}\Psi\left(  x,y\right)
}B\left(  y\right)  ^{-1}\frac{1}{\det\partial_{x}\Psi\left(  x,y\right)
}\partial_{x}\frac{a_{\lambda}\left(  x,y\right)  }{\det\left[  \partial
_{x}\Psi\left(  x,y\right)  \right]  }\right\} \\
& =B\left(  y\right)  ^{-1}\left\{  -2\left(  \det\partial_{x}\Psi\left(
x,y\right)  \right)  ^{-3}\partial_{x}\det\partial_{x}\Psi\left(  x,y\right)
+\partial_{x}^{2}\left[  \left(  \det\left[  \partial_{x}\Psi\left(
x,y\right)  \right]  \right)  ^{-1}a_{\lambda}\left(  x,y\right)  \right]
\right\}  ,
\end{align*}
where $\partial_{x}^{2}\left[  \left(  \det\left[  \partial_{x}\Psi\left(
x,y\right)  \right]  \right)  ^{-1}a_{\lambda}\left(  x\right)  \right]  $ is%
\begin{align*}
& 2\left(  \det\left[  \partial_{x}\Psi\left(  x,y\right)  \right]  \right)
\left\vert ^{-3}\partial_{x}\det\partial_{x}\Psi\left(  x,y\right)
\right\vert ^{2}a_{\lambda}\left(  x,y\right) \\
& -\left(  \det\left[  \partial_{x}\Psi\left(  x,y\right)  \right]  \right)
^{-2}\partial_{x}^{2}\det\partial_{x}\Psi\left(  x,y\right)  a_{\lambda
}\left(  x,y\right)  -\left(  \det\left[  \partial_{x}\Psi\left(  x\right)
\right]  \right)  ^{-2}\partial_{x}\det\partial_{x}\Psi\left(  x,y\right)
\partial_{x}a_{\lambda}\left(  x,y\right) \\
& -\left(  \det\left[  \partial_{x}\Psi\left(  x,y\right)  \right]  \right)
^{-2}\partial_{x}\det\partial_{x}\Psi\left(  x,y\right)  \partial
_{x}a_{\lambda}\left(  x,y\right)  +\left(  \det\left[  \partial_{x}%
\Psi\left(  x,y\right)  \right]  \right)  ^{-1}\partial_{x}^{2}a_{\lambda
}\left(  x,y\right)  ,
\end{align*}
and so when we evaluate at $x=X\left(  y\right)  $, we obtain that $\left(
\det\left[  \partial_{x}\Psi\left(  x,y\right)  \right]  \right)
^{-1}\partial_{x}^{2}a\left(  x,y\right)  $ equals $\partial_{x}^{2}a\left(
X\left(  y\right)  ,y\right)  $, and hence,
\[
\mathcal{P}_{a,\phi}^{\left(  1\right)  }\left(  y,\lambda\right)
=\frac{i^{\ell}}{\left(  2\lambda\right)  ^{\ell}\ell!}\left(  \frac{2\pi
}{\lambda}\right)  ^{\frac{n}{2}}\frac{e^{i\left[  \operatorname{sgn}B\left(
y\right)  \frac{\pi}{4}+\lambda\phi\left(  X\left(  y\right)  ,y\right)
\right]  }}{\sqrt{\det B\left(  y\right)  }}\left\{  \partial_{x}^{2}a\left(
X\left(  y\right)  ,y\right)  +O\left(  \left\Vert \partial_{x}a_{\lambda
}\right\Vert _{L^{\infty}\left(  \mathbb{R}_{x}^{n}\right)  }+\left\Vert
a_{\lambda}\right\Vert _{L^{\infty}\left(  \mathbb{R}_{x}^{n}\right)
}\right)  \right\}  .
\]
Thus every gain of $\frac{1}{\lambda}$\ costs \textbf{two} derivatives of
$a_{\lambda}$\ in\ $x$ (ignoring the contribution from $\left\Vert
\partial_{x}a_{\lambda}\right\Vert _{L^{\infty}\left(  \mathbb{R}_{x}%
^{n}\right)  }+\left\Vert a_{\lambda}\right\Vert _{L^{\infty}\left(
\mathbb{R}_{x}^{n}\right)  }$),\ which dictates our definition of the
parameter $d$ in the subform (\ref{def B^k,d}) below.
\end{remark}

Note that we can write the formula for $\mathfrak{P}_{a_{\lambda},\phi
}^{\left(  \ell\right)  }\left(  y,\lambda\right)  $ compactly as%
\begin{equation}
\mathfrak{P}_{a_{\lambda},\phi}^{\left(  \ell\right)  }\left(  y,\lambda
\right)  =\left(  \frac{2\pi}{\lambda}\right)  ^{\frac{n}{2}}\frac{i^{\ell}%
}{\left(  2\lambda\right)  ^{\ell}\ell!}\frac{e^{i\left[  \operatorname{sgn}%
B\left(  y\right)  \frac{\pi}{4}+\lambda\phi\left(  X\left(  y\right)
,y\right)  \right]  }}{\sqrt{\det B}}\left(  \left\{  L^{-1}\partial
_{x}BL^{-1}\partial_{x}\right\}  ^{\ell}\frac{a\left(  x,y\right)  }{\det
L}\right)  \mid_{x=X\left(  y\right)  },\label{def P^k}%
\end{equation}
where%
\begin{equation}
L\equiv\partial_{x}\Psi\left(  x,y\right)  \text{ and }B\equiv B\left(
y\right)  =\left(  \partial_{x}^{2}\phi\right)  \left(  X\left(  y\right)
,y\right)  .\label{def L and B}%
\end{equation}

\section{Starting the proof of the probabilistic extension
conjecture\label{Section start}}

We must prove the truncated probabilistic extension inequality
(\ref{prob ext''}),%
\[
\mathbb{E}_{2^{\mathcal{G}}}^{\mu}\left\Vert T\sum_{I\in\mathcal{G}\left[
U\right]  }a_{I}\bigtriangleup_{I;\kappa}^{\eta}f\right\Vert _{L^{p}\left(
\lambda_{n}\right)  }\leq C\left\Vert f\right\Vert _{L^{p}\left(  B\left(
0,\frac{1}{2}\right)  \right)  },\ \ \ \ \ p>\frac{2n}{n-1}.
\]
However, we will instead begin by setting out to prove the much stronger
truncated \emph{deterministic} extension inequality (\ref{T_S trunc}),%
\[
\left\Vert T\sum_{I\in\mathcal{G}\left[  U\right]  }\bigtriangleup_{I;\kappa
}^{\eta}f\right\Vert _{L^{p}\left(  \lambda_{n}\right)  }\leq C\left\Vert
f\right\Vert _{L^{p}\left(  B\left(  0,\frac{1}{2}\right)  \right)  },
\]
and only when we run into difficulty proving this, will we resort to using
expectation. Thus we begin by considering its equivalent bilinear inequality%
\[
\left\vert \left\langle T\sum_{I\in\mathcal{G}\left[  U\right]  }%
\bigtriangleup_{I;\kappa}^{\eta}f,g\right\rangle \right\vert \lesssim
\left\Vert f\right\Vert _{L^{p}}\left\Vert g\right\Vert _{L^{p^{\prime}}}\ .
\]

Our initial splitting of the above bilinear form is modeled after that in two
weight testing theory using (\ref{decomp}),%
\begin{align}
\left\langle T\sum_{I\in\mathcal{G}\left[  U\right]  }\bigtriangleup
_{I;\kappa}^{\eta}f,g\right\rangle  & =\sum_{\left(  I,J\right)
\in\mathcal{G}\left[  U\right]  \times\mathcal{D}}\left\langle T\bigtriangleup
_{I;\kappa}^{n-1,\eta}f,\bigtriangleup_{J;\kappa}^{n,\eta}g\right\rangle
\label{init split}\\
& =\left\{  \sum_{\left(  I,J\right)  \in\mathcal{P}_{0}}+\sum_{\left(
I,J\right)  \in\mathcal{R}}+\sum_{m=1}^{\infty}\sum_{\left(  I,J\right)
\in\mathcal{P}_{m}}+\sum_{\left(  I,J\right)  \in\mathcal{X}}\right\}
\left\langle T\bigtriangleup_{I;\kappa}^{n-1,\eta}f,\bigtriangleup_{J;\kappa
}^{n,\eta}g\right\rangle \nonumber\\
& \equiv\mathsf{B}_{\operatorname{below}}\left(  f,g\right)  +\mathsf{B}%
_{\operatorname{above}}\left(  f,g\right)  +\mathsf{B}%
_{\operatorname*{disjoint}}\left(  f,g\right)  +\mathsf{B}%
_{\operatorname*{distal}}\left(  f,g\right)  .\nonumber
\end{align}
We further decomposed the pairs $\mathcal{P}_{0}$ and $\mathcal{P}_{m}$ in
(\ref{def PKD}) and (\ref{def PMKD}) according to the oscillation properties
of the inner product%
\[
\left\langle T\bigtriangleup_{I;\kappa}^{n-1,\eta}f,\bigtriangleup_{J;\kappa
}^{n,\eta}g\right\rangle =\left\langle Th_{I;\kappa}^{n-1,\eta},h_{J;\kappa
}^{n,\eta}\right\rangle \left\langle \left(  S_{\kappa,\eta}\right)
^{-1}f,h_{I;\kappa}^{n-1}\right\rangle \left\langle \left(  S_{\kappa,\eta
}\right)  ^{-1}g,h_{J;\kappa}^{n}\right\rangle ,
\]
namely%
\begin{align*}
\mathcal{P}_{0}  & =\bigcup_{k\in Z}\bigcup_{d=1}^{\infty}\mathcal{P}%
_{0}^{k,d},\text{ where}\\
\mathcal{P}_{0}^{k,d}  & \equiv\left\{  \left(  I,J\right)  \in\mathcal{P}%
:J\subset\mathcal{K}\left(  I\right)  \text{, }\ell\left(  J\right)
=2^{k}\text{, and }\frac{2^{d-1}}{\ell\left(  I\right)  ^{2}}\leq
\operatorname*{dist}\left(  0,J\right)  =\frac{2^{d+1}}{\ell\left(  I\right)
^{2}}\right\}  ,\\
\mathcal{P}_{m}  & =\bigcup_{k\in\mathbb{Z}}\bigcup_{d\in\mathbb{Z}}^{\infty
}\mathcal{P}_{m}^{k,d},\text{ where }\\
\mathcal{P}_{m}^{k,d}  & \equiv\left\{  \left(  I,J\right)  \in\mathcal{P}%
_{m}:2^{m+1}I\subset U\text{, }\ell\left(  J\right)  =2^{k}\text{, and }%
2^{d}\leq\ell\left(  I\right)  ^{2}\operatorname*{dist}\left(  0,J\right)
\leq2^{d+1}\right\}  .
\end{align*}

We now decompose the disjoint form $\mathsf{B}_{\operatorname*{disjoint}%
}\left(  f,g\right)  $ into upper and lower components determined by $d$
nonnegative and negative respectively,%
\begin{align}
\mathsf{B}_{\operatorname*{disjoint}}\left(  f,g\right)   & =\mathsf{B}%
_{\operatorname*{disjoint}}^{\operatorname*{upper}}\left(  f,g\right)
+\mathsf{B}_{\operatorname*{disjoint}}^{\operatorname*{lower}}\left(
f,g\right)  ,\label{def up disj}\\
\mathsf{B}_{\operatorname*{disjoint}}^{\operatorname*{upper}}\left(
f,g\right)   & \equiv\sum_{m=1}^{\infty}\sum_{k\in\mathbb{Z}}\sum_{d\geq
0}\mathsf{B}_{\operatorname*{disjoint}}^{k,d,m}\left(  f,g\right)  \text{ and
}\mathsf{B}_{\operatorname*{disjoint}}^{\operatorname*{lower}}\left(
f,g\right)  \equiv\sum_{m=1}^{\infty}\sum_{k\in\mathbb{Z}}\sum_{d<0}%
\mathsf{B}_{\operatorname*{disjoint}}^{k,d,m}\left(  f,g\right)  .\nonumber
\end{align}
For the distal form $\mathsf{B}_{\operatorname*{distal}}\left(  f,g\right)  $
we write,%
\begin{align*}
\mathsf{B}_{\operatorname*{distal}}^{k,d}\left(  f,g\right)   & \equiv
\sum_{\left(  I,J\right)  \in\mathcal{X}^{k,d}}\left\langle T\bigtriangleup
_{I;\kappa}^{n-1,\eta}f,\bigtriangleup_{J;\kappa}^{n,\eta}g\right\rangle ,\\
\text{where }\mathcal{X}^{k,d}  & \equiv\left\{  \left(  I,J\right)
\in\mathcal{X}:\ell\left(  J\right)  =2^{k}\text{, and }2^{d}\leq\ell\left(
I\right)  ^{2}\operatorname*{dist}\left(  0,J\right)  \leq2^{d+1}\right\}  ,\\
\text{and }\mathcal{X}  & \equiv\left\{  \left(  I,J\right)  \in
\mathcal{G}\left[  U\right]  \times\mathcal{D}:2^{m+1}I\subset S\text{ and
}\pi_{\tan}\left(  J\right)  \cap\Phi\left(  2U\right)  =\emptyset\right\}  ,
\end{align*}
and decompose it into upper and lower subforms in the analogous way,%
\begin{align}
\mathsf{B}_{\operatorname*{distal}}\left(  f,g\right)   & =\mathsf{B}%
_{\operatorname*{distal}}^{\operatorname*{upper}}\left(  f,g\right)
+\mathsf{B}_{\operatorname*{distal}}^{\operatorname*{lower}}\left(
f,g\right)  ,\label{def up dist}\\
\mathsf{B}_{\operatorname*{distal}}^{\operatorname*{upper}}\left(  f,g\right)
& \equiv\sum_{k\in\mathbb{Z}}\sum_{d\geq0}\mathsf{B}_{\operatorname*{distal}%
}^{k,d}\left(  f,g\right)  \text{ and }\mathsf{B}_{\operatorname*{distal}%
}^{\operatorname*{lower}}\left(  f,g\right)  \equiv\sum_{k\in\mathbb{Z}}%
\sum_{d<0}\mathsf{B}_{\operatorname*{distal}}^{k,d}\left(  f,g\right)
.\nonumber
\end{align}

For $m\in\mathbb{N}$ and $d\leq0$, a different pigeonholing that respects
resonance is required, which we defer until needed in Section 9. Similarly, we
defer further pigeonholing of $\mathcal{R}$\ until needed in Section 7. In all
of these index sets, the cubes $I$ are restricted to $\mathcal{G}\left[
U\right]  $.

\begin{enumerate}
\item The below form $\mathsf{B}_{\operatorname{below}}\left(  f,g\right)  $
combines stationary phase with either integration by parts or moment
vanishing, and only its subform $\mathsf{B}_{\operatorname{below}}%
^{k,d}\left(  f,g\right)  $ for $k,d\geq0$ requires the strict inequality
$p>\frac{2n}{n-1}$. Moreover, the subforms with $d\leq0$ can be controlled by
relatively simple arguments when $p>\frac{2n}{n-1}$.

\item The above form $\mathsf{B}_{\operatorname{above}}\left(  f,g\right)  $
is less critical and easier to handle in that it doesn't use stationary phase,
and is in fact bounded for all $1<p<\infty$.

\item The disjoint form $\mathsf{B}_{\operatorname*{disjoint}}\left(
f,g\right)  $ is handled similarly in some places, and made easier in those
places due to the fact that stationary phase is not needed, because the
critical point of the phase lies outside the support of the amplitude.
However, in those difficult places where large numbers of inner products are
resonant, i.e. without either appropriate oscillation or smoothness,
\emph{probability} is used in conjunction with an interpolation argument
between $L^{2}$ and $L^{4} $ estimates.

\item The upper distal form $\mathsf{B}_{\operatorname*{distal}}%
^{\operatorname*{upper}}\left(  f,g\right)  $ is handled as an extreme case of
the upper disjoint form $\mathsf{B}_{\operatorname*{disjoint}}%
^{\operatorname*{upper}}\left(  f,g\right)  $ in Section 8, and the lower
distal form $\mathsf{B}_{\operatorname*{distal}}^{\operatorname*{lower}%
}\left(  f,g\right)  $ is bundled together with the lower disjoint form
$\mathsf{B}_{\operatorname*{disjoint}}^{\operatorname*{lower}}\left(
f,g\right)  $ and controlled using probability in Section 9.
\end{enumerate}

We have%
\begin{align}
& \left\vert \left\langle T\bigtriangleup_{I;\kappa}^{n-1,\eta}%
f,\bigtriangleup_{J;\kappa}^{n,\eta}g\right\rangle _{\omega}\right\vert
=\left\vert \left\langle Th_{I;\kappa}^{n-1,\eta},h_{J;\kappa}^{n,\eta
}\right\rangle \right\vert \left\vert \left\langle \left(  S_{\kappa,\eta
}^{\sigma}\right)  ^{-1}f,h_{I;\kappa}^{n-1}\right\rangle \right\vert
\left\vert \left\langle \left(  S_{\kappa,\eta}^{\omega}\right)
^{-1}g,h_{J;\kappa}^{n}\right\rangle \right\vert \label{inner Delta}\\
& \approx\left\{  \frac{\left\vert \left\langle Th_{I;\kappa}^{n-1,\eta
},h_{J;\kappa}^{n,\eta}\right\rangle _{\omega}\right\vert }{\sqrt{\left\vert
I\right\vert \left\vert J\right\vert }}\right\}  \left\{  \int_{\mathbb{R}%
^{n-1}}\left\vert \bigtriangleup_{I;\kappa}^{n-1,\eta}f\left(  x\right)
\right\vert d\sigma\left(  x\right)  \right\}  \left\{  \int_{\mathbb{R}^{n}%
}\left\vert \bigtriangleup_{J;\kappa}^{n,\eta}g\left(  \xi\right)  \right\vert
d\omega\left(  \xi\right)  \right\}  ,\nonumber
\end{align}
since%
\begin{align}
\int_{\mathbb{R}^{n-1}}\left\vert \bigtriangleup_{I;\kappa}^{n-1,\eta}f\left(
x\right)  \right\vert d\sigma\left(  x\right)   & =\int_{\mathbb{R}^{n-1}%
}\left\vert \left\langle \left(  S_{\kappa,\eta}\right)  ^{-1}f,h_{I;\kappa
}^{n-1}\right\rangle h_{I;\kappa}^{n-1,\eta}\right\vert d\sigma\left(
x\right) \label{inner Delta'}\\
& \approx\left\vert \left\langle \left(  S_{\kappa,\eta}^{\sigma}\right)
^{-1}f,h_{I;\kappa}^{n-1}\right\rangle \right\vert \left\Vert h_{I;\kappa
}^{n-1,\eta}\right\Vert _{L^{1}\left(  \sigma\right)  }\approx\left\vert
\left\langle \left(  S_{\kappa,\eta}\right)  ^{-1}f,h_{I;\kappa}%
^{n-1}\right\rangle \right\vert \sqrt{\left\vert I\right\vert },\nonumber\\
\int_{\mathbb{R}^{n}}\left\vert \bigtriangleup_{J;\kappa}^{n,\eta}g\left(
\xi\right)  \right\vert d\omega\left(  \xi\right)   & \approx\left\vert
\left\langle \left(  S_{\kappa,\eta}\right)  ^{-1}g,h_{J;\kappa}%
^{n}\right\rangle \right\vert \sqrt{\left\vert J\right\vert }.\nonumber
\end{align}

Our strategy is to estimate the inner product%
\[
\left\langle Th_{I;\kappa}^{n-1},h_{J;\kappa}^{n}\right\rangle =\int
_{\mathbb{R}^{n}}\left\{  \int_{\mathbb{R}^{n-1}}e^{i\Phi\left(  x\right)
\cdot\xi}h_{I;\kappa}^{n-1,\eta}\left(  x\right)  dx\right\}  h_{J;\kappa
}^{n,\eta}\left(  \xi\right)  d\xi,
\]
and then using these inner product estimates, we will bound the two bilinear
forms $\mathsf{B}_{\operatorname{below}}\left(  f,g\right)  $ and
$\mathsf{B}_{\operatorname{above}}\left(  f,g\right)  $, along with some of
the subforms of $\mathsf{B}_{\operatorname*{disjoint}}\left(  f,g\right)  $
and $\mathsf{B}_{\operatorname*{distal}}\left(  f,g\right)  $, namely those
comprising the upper disjoint and distal forms $\mathsf{B}%
_{\operatorname*{disjoint}}^{\operatorname*{upper}}\left(  f,g\right)  $ and
$\mathsf{B}_{\operatorname*{distal}}^{\operatorname*{upper}}\left(
f,g\right)  $.

In fact, if we denote by $\left\vert \mathsf{B}_{\operatorname{below}%
}\right\vert \left(  f,g\right)  $, $\left\vert \mathsf{B}%
_{\operatorname{above}}\right\vert \left(  f,g\right)  $, $\left\vert
\mathsf{B}_{\operatorname*{disjoint}}^{\operatorname*{upper}}\right\vert
\left(  f,g\right)  $ and $\left\vert \mathsf{B}_{\operatorname*{distal}%
}^{\operatorname*{upper}}\right\vert \left(  f,g\right)  $ the forms
$\mathsf{B}_{\operatorname{below}}\left(  f,g\right)  $, $\mathsf{B}%
_{\operatorname{above}}\left(  f,g\right)  $, $\mathsf{B}%
_{\operatorname*{disjoint}}^{\operatorname*{upper}}\left(  f,g\right)  $ and
$\mathsf{B}_{\operatorname*{distal}}^{\operatorname*{upper}}\left(
f,g\right)  $ with absolute values taken inside the sum of inner products,
then we will prove the following `deterministic' estimate in which probability
plays no role.

\begin{proposition}
\label{det control}For $p>\frac{2n}{n-1}$ we have%
\[
\left\vert \mathsf{B}_{\operatorname{below}}\right\vert \left(  f,g\right)
+\left\vert \mathsf{B}_{\operatorname{above}}\right\vert \left(  f,g\right)
+\left\vert \mathsf{B}_{\operatorname*{disjoint}}^{\operatorname*{upper}%
}\right\vert \left(  f,g\right)  +\left\vert \mathsf{B}%
_{\operatorname*{distal}}^{\operatorname*{upper}}\right\vert \left(
f,g\right)  \lesssim\left\Vert f\right\Vert _{L^{p}\left(  \mathbb{R}%
^{n-1}\right)  }\left\Vert g\right\Vert _{L^{p^{\prime}}\left(  \mathbb{R}%
^{n}\right)  }\ .
\]

\end{proposition}

\begin{proof}
This follows immediately from (\ref{strong below}), (\ref{strong above}),
(\ref{strong upper disjoint}) and (\ref{strong upper distal}) below.
\end{proof}

\begin{remark}
Proposition \ref{det control} shows that the Fourier extension conjecture
(\ref{extension}) with $p=q$ is equivalent to boundedness of the lower
disjoint and distal forms,%
\[
\left\vert \mathsf{B}_{\operatorname*{disjoint}}^{\operatorname*{lower}%
}\left(  f,g\right)  +\mathsf{B}_{\operatorname*{distal}}%
^{\operatorname*{lower}}\left(  f,g\right)  \right\vert \lesssim\left\Vert
f\right\Vert _{L^{p}\left(  \mathbb{R}^{n-1}\right)  }\left\Vert g\right\Vert
_{L^{p^{\prime}}\left(  \mathbb{R}^{n}\right)  }\ .
\]

\end{remark}

Note that the small positive constant $\eta$ in the construction of the smooth
Alpert wavelets is fixed throughout the estimates below, and so powers of
$\frac{1}{\eta}$ depending on $n$ and $\kappa$ will often be absorbed into the
notation of approximate inequality $\lesssim$.

\begin{notation}
In an inner product of the form $\left\langle T\varphi,\psi\right\rangle $, we
refer to $\varphi$ as the \emph{amplitude} function, and to $\psi$ as the
\emph{pairing} function.
\end{notation}

\subsection{Pigeonholing into bilinear subforms}

Recall the decomposition (with bounded overlap) of the pairs $\left(
I,J\right)  \in\mathcal{G}\left[  U\right]  \times\mathcal{D}$ of dyadic cubes
introduced in (\ref{decomp}),
\[
\mathcal{G}\left[  U\right]  \times\mathcal{D}=\mathcal{P}_{0}\ \cup
\ \bigcup_{m=0}^{\infty}\mathcal{P}_{m}\ \cup\ \mathcal{R}\ \cup
\ \mathcal{X}\ ,
\]
where%
\begin{align*}
\mathcal{P}_{0}  & \equiv\left\{  \left(  I,J\right)  \in\mathcal{G}\left[
U\right]  \times\mathcal{D}:\pi_{\tan}\left(  J\right)  \subset\Phi\left(
C_{\operatorname{pseudo}}I\right)  \right\}  \ ,\\
\mathcal{P}_{m}  & \equiv\left\{  \left(  I,J\right)  \in\mathcal{G}\left[
U\right]  \times\mathcal{D}:2^{m+1}I\subset U\text{ and }\pi_{\tan}\left(
J\right)  \subset\Phi\left(  2^{m+1}C_{\operatorname{pseudo}}I\right)
\setminus\Phi\left(  2^{m}C_{\operatorname{pseudo}}I\right)  \right\}
,\ \ \ \ \ m\in\mathbb{N}\ ,\\
\mathcal{R}  & \equiv\left\{  \left(  I,J\right)  \in\mathcal{G}\left[
U\right]  \times\mathcal{D}:\Phi\left(  I\right)  \subset\pi_{\tan}\left(
C_{\operatorname{pseudo}}J\right)  \right\}  \ .
\end{align*}

In treating the below form $\mathsf{B}_{\operatorname{below}}\left(
f,g\right)  $, we will consider the inner products%
\begin{align*}
\left\langle T_{\sigma}\bigtriangleup_{I;\kappa}^{n-1,\eta}f,\bigtriangleup
_{J;\kappa}^{n,\eta}g\right\rangle  & =\int_{\mathbb{R}^{n}}\int
_{\mathbb{R}^{n-1}}\bigtriangleup_{I;\kappa}^{n-1,\eta}f\left(  x\right)
e^{-i\Phi\left(  x\right)  \cdot\xi}dx\bigtriangleup_{J;\kappa}^{n,\eta
}g\left(  \xi\right)  d\xi=\left\langle T_{\sigma}h_{I;\kappa}^{n-1,\eta
},h_{J;\kappa}^{\eta,\omega}\right\rangle \left\langle f,h_{I;\kappa
}^{n-1,\eta}\right\rangle \left\langle g,h_{J;\kappa}^{n,\eta}\right\rangle
\ ,\\
\left\langle T_{\sigma}h_{I;\kappa}^{n-1,\eta},h_{J;\kappa}^{n,\eta
}\right\rangle  & =\int_{\mathbb{R}^{n}}\int_{\mathbb{R}^{n-1}}h_{I;\kappa
}^{n-1,\eta}\left(  x\right)  e^{-i\Phi\left(  x\right)  \cdot\xi
}dxh_{J;\kappa}^{n,\eta}\left(  \xi\right)  d\xi,
\end{align*}
for $\left(  I,J\right)  \in\mathcal{P}_{0}\subset\mathcal{G}\left[  U\right]
\times\mathcal{D}$, and as in (\ref{def PKD}), we further decompose the index
set $\mathcal{P}_{0}$ of pairs by pigeonholing the side length of $J$ and its
distance from the origin relative to $\frac{1}{\ell\left(  I\right)  ^{2}}$,
the reciprocal of the `depth' of the spherical `cap' $\Phi\left(  I\right)  $:%
\begin{align*}
& \mathcal{P}_{0}=\bigcup_{k\in\mathbb{Z}}\bigcup_{d\in\mathbb{Z}}^{\infty
}\mathcal{P}_{0}^{k,d},\text{ where }\\
& \mathcal{P}_{0}^{k,d}\equiv\left\{  \left(  I,J\right)  \in\mathcal{P}%
_{0}:\ell\left(  J\right)  =2^{k}\text{, and }2^{d}\leq\ell\left(  I\right)
^{2}\operatorname*{dist}\left(  0,J\right)  \leq2^{d+1}\right\}  \text{,}\\
& \ \ \ \ \ \ \ \ \ \ \ \ \ \ \ \ \ \ \ \ \ \ \ \ \ \ \ \ \ \ \text{for
}k,d\in\mathbb{Z}.
\end{align*}
Then we define the associated subforms,%
\begin{equation}
\mathsf{B}_{\operatorname{below}}^{k,d}\left(  f,g\right)  \equiv\sum_{\left(
I,J\right)  \in\mathcal{P}_{0}^{k,d}}\left\langle T_{S}h_{I;\kappa}^{n-1,\eta
},h_{J;\kappa}^{n,\eta}\right\rangle .\label{def  B^k,d}%
\end{equation}

We decompose the disjoint form $\mathsf{B}_{\operatorname*{disjoint}}\left(
f,g\right)  $ into subforms $\mathsf{B}_{\operatorname*{disjoint}}%
^{k,d,m}\left(  f,g\right)  $ similar\ to that done for the below form
$\mathsf{B}_{\operatorname{below}}\left(  f,g\right)  $. Recall that in
(\ref{def PMKD}), for each $m\geq0$, we decomposed the index set
\[
\mathcal{P}_{m}\equiv\left\{  \left(  I,J\right)  \in\mathcal{G}\left[
U\right]  \times\mathcal{D}:2^{m+1}I\subset U\text{ and }\pi_{\tan}\left(
J\right)  \subset\Phi\left(  2^{m+1}C_{\operatorname{pseudo}}I\right)
\setminus\Phi\left(  2^{m}C_{\operatorname{pseudo}}I\right)  \right\}
,\ \ \ \ \ 1\leq m\leq cs\ ,
\]
of pairs by pigeonholing the side length of $J$ and its distance from the
origin relative to $\frac{1}{\ell\left(  I\right)  ^{2}}$, the reciprocal of
the `depth' of the spherical set $\Phi\left(  I\right)  $:%
\begin{align*}
& \mathcal{P}_{m}=\bigcup_{k\in\mathbb{Z}}\bigcup_{d\in\mathbb{Z}}^{\infty
}\mathcal{P}_{m}^{k,d},\text{ where }\\
& \mathcal{P}_{m}^{k,d}\equiv\left\{  \left(  I,J\right)  \in\mathcal{P}%
_{m}:\ell\left(  J\right)  =2^{k}\text{, and }2^{d}\leq\ell\left(  I\right)
^{2}\operatorname*{dist}\left(  0,J\right)  \leq2^{d+1}\right\}  \text{,}\\
& \ \ \ \ \ \ \ \ \ \ \ \ \ \ \ \ \ \ \ \ \ \ \ \ \ \ \ \ \ \ \text{for
}k,d\in\mathbb{Z},
\end{align*}
and now we define the disjoint subforms,%
\begin{equation}
\mathsf{B}_{\operatorname*{disjoint}}^{k,d,m}\left(  f,g\right)  \equiv
\sum_{\left(  I,J\right)  \in\mathcal{P}_{m}^{k,d}}\left\langle
T\bigtriangleup_{I;\kappa}^{n-1,\eta}f,\bigtriangleup_{J;\kappa}^{n,\eta
}g\right\rangle .\label{def A^k,d_m}%
\end{equation}
We point out that in those inner products in the disjoint form with resonance,
such as when $k=0$ and $m=-d$, we need analogues for smooth Alpert wavelets of
the traditional $L^{2}$ and $L^{4}$ estimates averaged over involutive smooth
Alpert multipliers. We then write%
\[
\mathsf{B}_{\operatorname*{disjoint}}^{\operatorname*{upper}}\left(
f,g\right)  \equiv\sum_{k\in\mathbb{Z}}\sum_{d\geq0}\sum_{m\in\mathbb{N}%
}\mathsf{B}_{\operatorname*{disjoint}}^{k,d,m}\left(  f,g\right)  \text{ and
}\mathsf{B}_{\operatorname*{disjoint}}^{\operatorname*{lower}}\left(
f,g\right)  \equiv\sum_{k\in\mathbb{Z}}\sum_{d<0}\sum_{m\in\mathbb{N}%
}\mathsf{B}_{\operatorname*{disjoint}}^{k,d,m}\left(  f,g\right)  .
\]

We defer the analogous pigeonholed decompositions for the above form
$\mathsf{B}_{\operatorname{above}}\left(  f,g\right)  $ and the distal form
$\mathsf{B}_{\operatorname*{distal}}\left(  f,g\right)  $ until needed. Now we
turn to the four \emph{principles of decay} used on the smooth Alpert inner
products $\left\langle Th_{I;\kappa}^{n-1,\eta},h_{J;\kappa}^{n,\eta
}\right\rangle $, followed in the next subsection with the interpolation estimates.

\subsection{Decay principles}

We introduce four different principles of decay in the oscillatory kernel of
the Fourier transform, namely

\begin{enumerate}
\item radial integration by parts,

\item moment vanishing of smooth Alpert wavelets (for both $h_{I;\kappa
}^{n-1,\eta}$\ and $h_{J;\kappa}^{n,\eta}$),

\item stationary phase of oscillatory integrals,

\item and tangential integration by parts.
\end{enumerate}

These four principles of decay will be used as building blocks for compound
principles of decay, which are obtained by iterating the exact formulas for
each principle, \emph{before} taking absolute values inside the resulting
integrals, in order to obtain estimates. These estimates are then used with
Alpert square function techniques as in \cite{SaWi} to bound the three forms
$\mathsf{B}_{\operatorname{below}}\left(  f,g\right)  $, $\mathsf{B}%
_{\operatorname*{disjoint}}\left(  f,g\right)  $ and $\mathsf{B}%
_{\operatorname{above}}\left(  f,g\right)  $. However, in order to handle
\emph{resonant} subforms of $\mathsf{B}_{\operatorname*{disjoint}}\left(
f,g\right)  $, we require an additional decay principle involving
interpolation of $L^{2}$ and $L^{4}$ estimates for smooth Alpert
pseudoprojections, that is described in the next subsection.

Our baseline is the following rather trivial $L^{1}$ estimate, which we refer
to as the \emph{crude} estimate,%
\begin{align}
\left\vert \left\langle Th_{I;\kappa}^{n-1,\eta},h_{J;\kappa}^{n,\eta
}\right\rangle \right\vert  & \leq\left\Vert h_{I;\kappa}^{n-1,\eta
}\right\Vert _{L^{1}\left(  \sigma\right)  }\left\Vert h_{J;\kappa}^{n,\eta
}\right\Vert _{L^{1}}\approx\sqrt{\left\vert I\right\vert \left\vert
J\right\vert }\ ,\label{crude''}\\
\left\vert \left\langle T\bigtriangleup_{I;\kappa}^{n-1,\eta}f,\bigtriangleup
_{J;\kappa}^{n,\eta}g\right\rangle _{\omega}\right\vert  & \leq\left\Vert
\bigtriangleup_{I;\kappa}^{n-1,\eta}f\right\Vert _{L^{1}}\left\Vert
\bigtriangleup_{J;\kappa}^{n,\eta}g\right\Vert _{L^{1}}\approx\sqrt{\left\vert
I\right\vert \left\vert J\right\vert }\left\vert \left\langle f,h_{I;\kappa
}^{n-1,\eta}\right\rangle \left\langle g,h_{J;\kappa}^{n,\eta}\right\rangle
\right\vert \ ,\nonumber
\end{align}
where we have used (\ref{inner Delta'})\ at the end of the second line.

\subsubsection{Radial integration by parts\label{Sub int}}

First we improve upon the crude estimate (\ref{crude''}) when $\left(
I,J\right)  \in P_{0}^{k,0}$ with $k>0$, i.e. $\ell\left(  J\right)  =2^{k}$,
namely we show that%
\begin{align}
& \ \ \ \ \ \ \ \ \ \ \ \ \ \ \ \left\vert \left\langle Th_{I;\kappa
}^{n-1,\eta},h_{J;\kappa}^{n,\eta}\right\rangle \right\vert \leq C_{N}%
2^{-kN}\left\Vert h_{I;\kappa}^{n-1,\eta}\right\Vert _{L^{1}}\left\Vert
h_{J;\kappa}^{n,\eta}\right\Vert _{L^{1}}\approx2^{-kN}\sqrt{\left\vert
I\right\vert \left\vert J\right\vert }\ ,\label{crude k}\\
& \left\vert \left\langle T\bigtriangleup_{I;\kappa}^{n-1,\eta}%
f,\bigtriangleup_{J;\kappa}^{n,\eta}g\right\rangle \right\vert \leq
C_{N}2^{-kN}\left\Vert \bigtriangleup_{I;\kappa}^{n-1,\eta}f\right\Vert
_{L^{1}}\left\Vert \bigtriangleup_{J;\kappa}^{n,\eta}g\right\Vert _{L^{1}%
}\approx2^{-kN}\sqrt{\left\vert I\right\vert \left\vert J\right\vert
}\left\vert \left\langle f,h_{I;\kappa}^{n-1,\eta}\right\rangle \left\langle
g,h_{J;\kappa}^{n,\eta}\right\rangle \right\vert \ .\nonumber
\end{align}
To see this, recall the change of variables (\ref{para}) made earlier,%
\begin{align*}
& \left\langle Th_{I;\kappa}^{n-1,\eta},h_{J;\kappa}^{n,\eta}\right\rangle
=\int_{\mathbb{R}^{n}}\int_{\mathbb{R}^{n-1}}e^{i\Phi\left(  x\right)
\cdot\xi}h_{I;\kappa}^{n-1,\eta}\left(  x\right)  h_{J;\kappa}^{n,\eta}\left(
\xi\right)  dxd\xi\\
& =\int_{\mathbb{R}}\int_{\mathbb{R}^{n-1}}\int_{\mathbb{R}^{n-1}}%
e^{i\lambda\phi\left(  x,y\right)  }\varphi_{I}^{\eta}\left(  x\right)
\widetilde{\psi}_{J}^{\eta}\left(  y,\lambda\right)  dxdyd\lambda,
\end{align*}
where%
\begin{align*}
\phi\left(  x,y\right)   & \equiv\Phi\left(  x\right)  \cdot\Phi\left(
y\right)  ,\\
\varphi_{I}^{\eta}\left(  x\right)   & \equiv h_{I;\kappa}^{n-1,\eta}\left(
x\right)  \text{ and }\psi_{J}^{\eta}\left(  \xi\right)  =h_{J;\kappa}%
^{n,\eta}\left(  \xi\right)  ,\\
\widetilde{\psi}_{J}^{\eta}\left(  y,\lambda\right)   & \equiv h_{J;\kappa
}^{n,\eta}\left(  \lambda y,\lambda\sqrt{1-\left\vert y\right\vert ^{2}%
}\right)  \frac{\lambda^{n-1}}{\sqrt{1-\left\vert y\right\vert ^{2}}}.
\end{align*}
We use the formula%
\[
\left(  \frac{1}{\phi\left(  x,y\right)  }\partial_{\lambda}\right)
^{N}e^{i\lambda\phi\left(  x,y\right)  }=e^{i\lambda\phi\left(  x,y\right)  },
\]
to obtain the equality,%
\begin{equation}
\left\langle Th_{I;\kappa}^{n-1,\eta},h_{J;\kappa}^{n,\eta}\right\rangle
=\int_{\mathbb{R}}\int_{\mathbb{R}^{n-1}}\int_{\mathbb{R}^{n-1}}%
\frac{e^{i\lambda\phi\left(  x,y\right)  }}{\phi\left(  x,y\right)  ^{N}%
}\varphi_{I}^{\eta}\left(  x\right)  \partial_{\lambda}^{N}\widetilde{\psi
}_{J}^{\eta}\left(  y,\lambda\right)  dxdyd\lambda
,\label{int by parts formula}%
\end{equation}
which can then be estimated by%
\begin{align}
& \left\vert \left\langle Th_{I;\kappa}^{n-1,\eta},h_{J;\kappa}^{n,\eta
}\right\rangle \right\vert \lesssim\left\Vert \varphi_{I}^{\eta}\right\Vert
_{L^{1}}\int_{\mathbb{R}}\int_{\mathbb{R}^{n-1}}\left\vert \partial_{\lambda
}^{N}\widetilde{\psi}_{J}^{\eta}\left(  y,\lambda\right)  \right\vert
dyd\lambda\label{int est}\\
& \lesssim\left\Vert \varphi_{I}^{\eta}\right\Vert _{L^{1}}\int_{\mathbb{R}%
}\int_{\mathbb{R}^{n-1}}\left\vert \partial_{\xi}^{N}\widetilde{\psi}%
_{J}^{\eta}\left(  y,\lambda\right)  \right\vert \left(  \min\left\{  \frac
{1}{\eta\ell\left(  J\right)  },\frac{1}{\lambda}\right\}  \right)
^{N}dyd\lambda\nonumber\\
& \approx\left(  \frac{1}{\eta\ell\left(  J\right)  }\right)  ^{N}\left\Vert
\varphi_{I}^{\eta}\right\Vert _{L^{1}}\left\Vert \partial_{\xi}^{N}%
\widetilde{\psi}_{J}^{\eta}\right\Vert _{L^{1}}\approx2^{-kN}\left\Vert
\varphi_{I}^{\eta}\right\Vert _{L^{1}}\left\Vert \partial_{\xi}^{N}%
\widetilde{\psi}_{J}^{\eta}\right\Vert _{L^{1}}\approx2^{-kN}\sqrt{\left\vert
I\right\vert \left\vert J\right\vert },\nonumber
\end{align}
which gives both lines in (\ref{crude k}).

\subsubsection{Vanishing moments of smooth Alpert wavelets\label{Sub van}}

Now we improve upon the crude estimate (\ref{crude''}) when $\left(
I,J\right)  \in P_{0}^{k,0}$ with $k<0$, i.e. $\ell\left(  J\right)  =2^{k}$,
namely we show that%
\begin{align}
& \ \ \ \ \ \ \ \ \ \ \ \ \ \ \ \left\vert \left\langle Th_{I;\kappa
}^{n-1,\eta},h_{J;\kappa}^{n,\eta}\right\rangle \right\vert \leq C_{\kappa
}2^{-\left\vert k\right\vert \kappa}\left\Vert h_{I;\kappa}^{n-1,\eta
}\right\Vert _{L^{1}}\left\Vert h_{J;\kappa}^{n,\eta}\right\Vert _{L^{1}%
}\approx2^{-\left\vert k\right\vert \kappa}\sqrt{\left\vert I\right\vert
\left\vert J\right\vert }\ ,\label{crude k<0}\\
& \left\vert \left\langle T\bigtriangleup_{I;\kappa}^{n-1,\eta}%
f,\bigtriangleup_{J;\kappa}^{n,\eta}g\right\rangle _{\omega}\right\vert \leq
C_{\kappa}2^{-\left\vert k\right\vert \kappa}\left\Vert \bigtriangleup
_{I;\kappa}^{n-1,\eta}f\right\Vert _{L^{1}}\left\Vert \bigtriangleup
_{J;\kappa}^{n,\eta}g\right\Vert _{L^{1}}\approx2^{-\left\vert k\right\vert
\kappa}\sqrt{\left\vert I\right\vert \left\vert J\right\vert }\left\vert
\left\langle f,h_{I;\kappa}^{n-1,\eta}\right\rangle \left\langle
g,h_{J;\kappa}^{n,\eta}\right\rangle \right\vert \ .\nonumber
\end{align}
For any entire function $f$, Taylor's formula\ with integral remainder applied
to $t\rightarrow f\left(  tz\right)  $ gives,%
\begin{align*}
f\left(  z\right)   & =\sum_{\ell=0}^{\kappa-1}\frac{1}{\ell!}\frac{d^{\ell}%
}{dt^{\ell}}f\left(  tz\right)  \mid_{t=0}+\int_{0}^{1}\left(  \frac
{d^{\kappa}}{dt^{\kappa}}f\left(  tz\right)  \right)  \frac{\left(
1-t\right)  ^{\kappa}}{\kappa!}dt\\
& =\sum_{\ell=0}^{\kappa-1}\frac{1}{\ell!}f^{\left(  \ell\right)  }\left(
0\right)  z^{\ell}+\int_{0}^{1}f^{\left(  \kappa\right)  }\left(  tz\right)
z^{\kappa}\frac{\left(  1-t\right)  ^{\kappa}}{\kappa!}dt,
\end{align*}
which shows that for any $\kappa\in\mathbb{N}$ and $b\in\mathbb{R}$, we have%
\begin{equation}
e^{ib}=\sum_{\ell=0}^{\kappa-1}\frac{\left(  ib\right)  ^{\ell}}{\ell
!}+R_{\kappa}\left(  ib\right)  ,\label{exp Tay}%
\end{equation}
where%
\begin{equation}
R_{\kappa}\left(  ib\right)  =\int_{0}^{1}e^{itb}\left(  ib\right)  ^{\kappa
}\frac{\left(  1-t\right)  ^{\kappa}}{\kappa!}dt\text{ and }\left\vert
R_{\kappa}\left(  ib\right)  \right\vert \leq\frac{\left\vert b\right\vert
^{\kappa}}{\left(  \kappa+1\right)  !}.\label{g kappa bound}%
\end{equation}
We also have%
\begin{align}
\left\vert \partial_{b}^{\ell}R_{\kappa}\left(  ib\right)  \right\vert  &
\lesssim\frac{\left\vert b\right\vert ^{\kappa-\ell}}{\left(  \kappa+1\right)
!},\ \ \ \ \ \text{for }0\leq\ell<\kappa,\label{g kappa bound'}\\
\partial_{b}^{\ell}R_{\kappa}\left(  ib\right)   & =\partial_{b}^{\ell}%
e^{ib}=i^{\ell}e^{ib},\ \ \ \ \ \text{for }\ell\geq\kappa.\nonumber
\end{align}

Now let $c_{J}$ denote the center of the cube $J$ and write,
\[
e^{-i\Phi\left(  x\right)  \cdot\xi}=e^{-i\Phi\left(  x\right)  \cdot c_{J}%
}e^{-i\Phi\left(  x\right)  \cdot\left(  \xi-c_{J}\right)  }=e^{-i\Phi\left(
x\right)  \cdot c_{J}}\left\{  \sum_{\ell=0}^{\kappa-1}\frac{\left(
-i\Phi\left(  x\right)  \cdot\left(  \xi-c_{J}\right)  \right)  ^{\ell}}%
{\ell!}+R_{\kappa}\left(  -i\Phi\left(  x\right)  \cdot\left(  \xi
-c_{J}\right)  \right)  \right\}  .
\]
Note that%
\[
e^{-i\Phi\left(  x\right)  \cdot c_{J}}R_{\kappa}\left(  -i\Phi\left(
x\right)  \cdot\left(  \xi-c_{J}\right)  \right)  =\int_{0}^{1}e^{-i\Phi
\left(  x\right)  \cdot c_{J}}e^{-it\Phi\left(  x\right)  \cdot\left(
\xi-c_{J}\right)  }\left(  -i\Phi\left(  x\right)  \cdot\left(  \xi
-c_{J}\right)  \right)  ^{\kappa}\frac{\left(  1-t\right)  ^{\kappa}}{\kappa
!}dt
\]
Since $h_{J;\kappa}^{n,\eta}$ has vanishing moments up to order less than
$\kappa$, we obtain
\begin{align}
& \ \ \ \ \ \left\langle Th_{I;\kappa}^{n-1,\eta},h_{J;\kappa}^{n,\eta
}\right\rangle =\int_{\mathbb{R}^{n}}\int_{\mathbb{R}^{n-1}}e^{-i\Phi\left(
x\right)  \cdot\xi}h_{I;\kappa}^{n-1,\eta}\left(  x\right)  dxh_{J;\kappa
}^{n,\eta}\left(  \xi\right)  d\xi\label{van mom formula}\\
& =\int_{\mathbb{R}^{n-1}}e^{-i\Phi\left(  x\right)  \cdot c_{J}}h_{I;\kappa
}^{n-1,\eta}\left(  x\right)  \left\{  \int_{\mathbb{R}^{n}}\left[  \sum
_{\ell=0}^{\kappa-1}\frac{\left(  -i\Phi\left(  x\right)  \cdot\left(
\xi-c_{J}\right)  \right)  ^{\ell}}{\ell!}+R_{\kappa}\left(  -i\Phi\left(
x\right)  \cdot\left(  \xi-c_{J}\right)  \right)  \right]  h_{J;\kappa
}^{n,\eta}\left(  \xi\right)  d\xi\right\}  dx\nonumber\\
& =\int_{\mathbb{R}^{n-1}}e^{-i\Phi\left(  x\right)  \cdot c_{J}}h_{I;\kappa
}^{n-1,\eta}\left(  x\right)  \left\{  \int_{\mathbb{R}^{n}}R_{\kappa}\left(
-i\Phi\left(  x\right)  \cdot\left(  \xi-c_{J}\right)  \right)  h_{J;\kappa
}^{n,\eta}\left(  \xi\right)  d\xi\right\}  dx.\nonumber
\end{align}
From the bound for $R_{\kappa}$ in (\ref{g kappa bound}) with $b=-\Phi\left(
x\right)  \cdot\left(  \xi-c_{J}\right)  $, we have%
\begin{align}
\left\vert \left\langle Th_{I;\kappa}^{n-1,\eta},h_{J;\kappa}^{n,\eta
}\right\rangle \right\vert  & \leq\int\left\vert h_{I;\kappa}^{n-1,\eta
}\left(  x\right)  \right\vert \int_{\mathbb{R}^{n}}\frac{\left\vert
\Phi\left(  x\right)  \cdot\left(  \xi-c_{J}\right)  \right\vert ^{\kappa}%
}{\left(  \kappa+1\right)  !}\left\vert h_{J;\kappa}^{n,\eta}\left(
\xi\right)  \right\vert d\xi dx\label{van mom est}\\
& \lesssim\ell\left(  J\right)  ^{\kappa}\left\Vert \varphi_{I}^{\eta
}\right\Vert _{L^{1}}\left\Vert \psi_{J}^{\eta}\right\Vert _{L^{1}}%
\approx2^{-\left\vert k\right\vert \kappa}\sqrt{\left\vert I\right\vert
\left\vert J\right\vert }.\nonumber
\end{align}

\subsubsection{Stationary phase with bounds\label{Sub asym}}

Now we improve upon the crude estimate (\ref{crude''}) when $\left(
I,J\right)  \in P_{0}^{0,d}$ with $d\geq0$, i.e. $J\subset\mathcal{K}\left(
I\right)  $, $\ell\left(  J\right)  =1$, and $\ell\left(  I\right)
^{2}\operatorname*{dist}\left(  0,J\right)  \approx2^{d}$, namely we show,%
\begin{align}
& \ \ \ \ \ \ \ \ \ \ \ \ \ \ \ \left\vert \left\langle Th_{I;\kappa
}^{n-1,\eta},h_{J;\kappa}^{n,\eta}\right\rangle \right\vert \lesssim
2^{-d\frac{n-1}{2}}\left(  1+2^{-d}\left(  \frac{1}{\ell\left(  I\right)
^{2}}\right)  ^{\tau}\right)  \sqrt{\left\vert I\right\vert \left\vert
J\right\vert }\ ,\label{crude asym}\\
& \left\vert \left\langle T\bigtriangleup_{I;\kappa}^{n-1,\eta}%
f,\bigtriangleup_{J;\kappa}^{n,\eta}g\right\rangle _{\omega}\right\vert
\lesssim2^{-d\frac{n-1}{2}}\left(  1+2^{-d}\left(  \frac{1}{\ell\left(
I\right)  ^{2}}\right)  ^{\tau}\right)  \sqrt{\left\vert I\right\vert
\left\vert J\right\vert }\left\vert \left\langle f,h_{I;\kappa}^{n-1,\eta
}\right\rangle \left\langle g,h_{J;\kappa}^{n,\eta}\right\rangle \right\vert
\ ,\nonumber
\end{align}
where $0<\tau\leq1$. For this, recall the change of variables in (\ref{para})
and (\ref{phi and psi notation}),%
\begin{align*}
& \left\langle Th_{I;\kappa}^{n-1,\eta},h_{J;\kappa}^{n,\eta}\right\rangle
=\int_{\mathbb{R}^{n}}\int_{\mathbb{R}^{n-1}}e^{i\Phi\left(  x\right)
\cdot\xi}h_{I;\kappa}^{n-1,\eta}\left(  x\right)  h_{J;\kappa}^{n,\eta}\left(
\xi\right)  dxd\xi\\
& =\int_{\mathbb{R}}\int_{\mathbb{R}^{n-1}}\left\{  \int_{\mathbb{R}^{n-1}%
}e^{i\lambda\phi\left(  x,y\right)  }\varphi_{I}^{\eta}\left(  x\right)
dx\right\}  \widetilde{\psi}_{J}^{\eta}\left(  y,\lambda\right)  dyd\lambda,
\end{align*}
where%
\begin{align*}
\phi\left(  x,y\right)   & \equiv\Phi\left(  x\right)  \cdot\Phi\left(
y\right)  ,\\
\varphi_{I}^{\eta}\left(  x\right)   & \equiv h_{I;\kappa}^{n-1,\eta}\left(
x\right)  \text{ and }\psi_{J}^{\eta}\left(  \xi\right)  =h_{J;\kappa}%
^{n,\eta}\left(  \xi\right)  ,\\
\widetilde{\psi}_{J}^{\eta}\left(  y,\lambda\right)   & \equiv\psi_{J}^{\eta
}\left(  \lambda y,\lambda\sqrt{1-\left\vert y\right\vert ^{2}}\right)
\frac{\lambda^{n-1}}{\sqrt{1-\left\vert y\right\vert ^{2}}}.
\end{align*}
Applying Theorem \ref{osc int} with $n$ replaced by $n-1$ and $a_{\lambda
}\left(  x,y\right)  $ equal to $\varphi_{I}^{\eta}\left(  x\right)  $, shows
that the oscillatory integral
\[
\mathcal{I}_{\varphi_{I}^{\eta},\phi}\left(  y,\lambda\right)  \equiv
\int_{\mathbb{R}^{n-1}}e^{i\lambda\phi\left(  x,y\right)  }\varphi_{I}^{\eta
}\left(  x\right)  dx,
\]
satisfies%
\[
\mathcal{I}_{\varphi_{I}^{\eta},\phi}\left(  y,\lambda\right)  =\mathfrak{P}%
_{\varphi_{I}^{\eta},\phi}\left(  y,\lambda\right)  +\sum_{\ell=1}%
^{M}\mathfrak{P}_{\varphi_{I}^{\eta},\phi}^{\left(  \ell\right)  }\left(
y,\lambda\right)  +\mathfrak{R}_{\varphi_{I}^{\eta},\phi}^{\left(  M+1\right)
}\left(  y,\lambda\right)  ,
\]
where%
\begin{equation}
\mathfrak{P}_{\varphi_{I}^{\eta},\phi}\left(  y,\lambda\right)  =\left(
\frac{2\pi}{\lambda}\right)  ^{\frac{n-1}{2}}\frac{e^{i\left[
\operatorname{sgn}\left[  \partial_{x}^{2}\phi\left(  X\left(  y\right)
,y\right)  \right]  \frac{\pi}{4}+\lambda\phi\left(  X\left(  y\right)
,y\right)  \right]  }}{\sqrt{\left\vert \det B\left(  y\right)  \right\vert }%
}\varphi_{I}^{\eta}\left(  X\left(  y\right)  \right)  ,\label{main term}%
\end{equation}
and for $1\leq\ell\leq M$,%
\begin{align*}
& \mathfrak{P}_{\varphi_{I}^{\eta},\phi}^{\left(  \ell\right)  }\left(
y,\lambda\right)  =\frac{i^{\ell}}{\left(  2\lambda\right)  ^{\ell}\ell
!}\left(  \frac{2\pi}{\lambda}\right)  ^{\frac{n}{2}}\frac{e^{i\left[
\operatorname{sgn}B\left(  y\right)  \frac{\pi}{4}+\lambda\phi\left(  X\left(
y\right)  ,y\right)  \right]  }}{\sqrt{\det B\left(  y\right)  }}\\
& \ \ \ \ \ \ \ \ \ \ \times\left\{  \left[  \partial_{x}\frac{1}{\det
\partial_{x}\Psi\left(  X\left(  y\right)  ,y\right)  }\right]  B\left(
y\right)  ^{-1}\frac{1}{\det\partial_{x}\Psi\left(  X\left(  y\right)
,y\right)  }\partial_{x}\right\}  ^{\ell}\frac{\varphi_{I}^{\eta}\left(
X\left(  y\right)  \right)  }{\det\left[  \partial_{x}\Psi\left(  X\left(
y\right)  ,y\right)  \right]  },
\end{align*}
and%
\begin{align}
\mathfrak{R}_{\varphi_{I}^{\eta},\phi}^{\left(  M+1\right)  }\left(
y,\lambda\right)   & =\left(  \frac{2\pi}{\lambda}\right)  ^{\frac{n-1}{2}%
}\frac{e^{i\left[  \operatorname{sgn}B\left(  y\right)  \frac{\pi}{4}%
+\lambda\phi\left(  X\left(  y\right)  ,y\right)  \right]  }}{\sqrt{\left\vert
\det B\left(  y\right)  \right\vert }}\nonumber\\
& \times\int\mathcal{F}_{z}^{-1}\left(  \left[  \frac{\left\langle
i\partial_{z},B\left(  y\right)  ^{-1}\partial_{z}\right\rangle }{2\lambda
}\right]  ^{M+1}f\right)  \left(  \zeta\right)  g_{M+1}\left(  -i\frac
{\zeta^{\operatorname*{tr}}B\left(  y\right)  ^{-1}\zeta}{2\lambda}\right)
d\zeta,\nonumber
\end{align}
and where $B\left(  y\right)  =\partial_{x}^{2}\phi\left(  X\left(  y\right)
,y\right)  $, and $X\left(  y\right)  $ is the unique stationary point of
$\phi\left(  \cdot_{x},y\right)  $ in the support of $a$,\ as given in the
Morse Lemma, and $\rho=\left\lceil \frac{n}{2}\right\rceil $ is the smallest
integer greater than $\frac{n}{2}$, and finally $g_{M+1}\left(  b\right)
=\frac{1}{M!}\int_{0}^{b}e^{t}\left(  b-t\right)  ^{M}dt$ for $b\in\mathbb{C}%
$. Thus at this point we have the formula,%
\begin{align}
\left\langle Th_{I;\kappa}^{n-1,\eta},h_{J;\kappa}^{n,\eta}\right\rangle  &
=\int_{\mathbb{R}}\int_{\mathbb{R}^{n-1}}\left\{  \int_{\mathbb{R}^{n-1}%
}e^{i\lambda\phi\left(  x,y\right)  }h_{I;\kappa}^{n-1,\eta}\left(  x\right)
dx\right\}  \widehat{\psi}_{J}^{\eta}\left(  y,\lambda\right)  dyd\lambda
\label{asym formula}\\
& =\int_{\mathbb{R}}\int_{\mathbb{R}^{n-1}}\mathcal{I}_{\varphi_{I}^{\eta
},\phi}\left(  y,\lambda\right)  \widetilde{\psi}_{J}^{\eta}\left(
y,\lambda\right)  dyd\lambda\nonumber
\end{align}

In the case $\phi\left(  x,y\right)  \equiv\Phi\left(  x\right)  \cdot
\Phi\left(  y\right)  $ we have $X\left(  y\right)  =y$ and
\begin{align*}
B\left(  y\right)   & =\partial_{x}^{2}\Phi\left(  x\right)  \cdot\Phi\left(
y\right)  \mid_{x=y}=\partial_{x}^{2}\sqrt{1-\left\vert x\right\vert ^{2}}%
\mid_{x=y}\sqrt{1-\left\vert y\right\vert ^{2}}\\
& =\left(  -\frac{1}{\sqrt{1-\left\vert x\right\vert ^{2}}}\operatorname{Id}%
_{n-1}-\frac{xx^{\operatorname*{tr}}}{\left(  1-\left\vert x\right\vert
^{2}\right)  ^{\frac{3}{2}}}\mid_{x=y}\right)  \sqrt{1-\left\vert y\right\vert
^{2}}\\
& =-\operatorname{Id}_{n-1}-\frac{yy^{\operatorname*{tr}}}{1-\left\vert
y\right\vert ^{2}},
\end{align*}
so that $\operatorname{sgn}B\left(  y\right)  =-\left(  n-1\right)  $ and%

\begin{align*}
\det B\left(  y\right)   & =\det\left[
\begin{array}
[c]{cccc}%
-1-\frac{y_{1}^{2}}{1-\left\vert y\right\vert ^{2}} & -\frac{y_{1}y_{2}%
}{1-\left\vert y\right\vert ^{2}} & \cdots & -\frac{y_{1}y_{n-1}}{1-\left\vert
y\right\vert ^{2}}\\
-\frac{y_{2}y_{1}}{1-\left\vert y\right\vert ^{2}} & -1-\frac{y_{2}^{2}%
}{1-\left\vert y\right\vert ^{2}} &  & -\frac{y_{2}y_{n-1}}{1-\left\vert
y\right\vert ^{2}}\\
\vdots &  & \ddots & \vdots\\
-\frac{y_{n-1}y_{1}}{1-\left\vert y\right\vert ^{2}} & -\frac{y_{n-1}y_{1}%
}{1-\left\vert y\right\vert ^{2}} & \cdots & -1-\frac{y_{n-1}^{2}%
}{1-\left\vert y\right\vert ^{2}}%
\end{array}
\right] \\
& =\det\frac{1}{1-\left\vert y\right\vert ^{2}}\left[
\begin{array}
[c]{cccc}%
-1+\left\vert y\right\vert ^{2}-y_{1}^{2} & -\frac{y_{1}y_{2}}{1-\left\vert
y\right\vert ^{2}} & \cdots & -y_{1}y_{n-1}\\
-y_{2}y_{1} & -1+\left\vert y\right\vert ^{2}-y_{2}^{2} &  & -y_{2}y_{n-1}\\
\vdots &  & \ddots & \vdots\\
-y_{n-1}y_{1} & -y_{n-1}y_{1} & \cdots & -1+\left\vert y\right\vert
^{2}-y_{n-1}^{2}%
\end{array}
\right]  =\frac{\left(  -1\right)  ^{n-1}}{1-\left\vert y\right\vert ^{2}},
\end{align*}
by induction on $n$.

In particular then, from (\ref{Psi and d Psi}) and the above calculation, we
have $\Psi\left(  X\left(  y\right)  ,y\right)  =0$, $\phi\left(  X\left(
y\right)  ,y\right)  $ and $\partial_{x}\Psi\left(  X\left(  y\right)
,y\right)  =\operatorname{Id}_{n}$ and so%
\[
\mathfrak{P}_{h_{I;\kappa}^{n-1,\eta},\phi}\left(  y,\lambda\right)  =\left(
\frac{2\pi}{\lambda}\right)  ^{\frac{n-1}{2}}e^{i\left[  -\frac{\left(
n-1\right)  \pi}{4}+\lambda\right]  }\sqrt{1-\left\vert y\right\vert ^{2}%
}\varphi_{I}^{\eta}\left(  y\right)  ,
\]
which can be written in the variable $\xi=\left(  \lambda y,\lambda
\sqrt{1-\left\vert y\right\vert ^{2}}\right)  $ as%
\[
\mathfrak{P}_{h_{I;\kappa}^{n-1,\eta},\phi}\left(  \xi\right)  =\left(
\frac{2\pi}{\left\vert \xi\right\vert }\right)  ^{\frac{n-1}{2}}\frac{\xi_{n}%
}{\left\vert \xi\right\vert }e^{i\left(  \left\vert \xi\right\vert
-\frac{\left(  n-1\right)  \pi}{4}\right)  }h_{I;\kappa}^{n-1,\eta}\left(
\frac{\xi^{\prime}}{\left\vert \xi\right\vert }\right)  ,\ \ \ \ \ \xi
^{\prime}=\left(  \xi_{1},...,\xi_{n-1}\right)  .
\]
We compute that for $J\in\mathcal{K}\left(  I\right)  $ and $\ell\left(
I\right)  ^{2}\operatorname*{dist}\left(  0,J\right)  \approx2^{d}$,
\begin{align*}
& \left\vert \left\langle \mathfrak{P}_{h_{I;\kappa}^{n-1,\eta},\phi
},h_{J;\kappa}^{n,\eta}\right\rangle \right\vert =\left\vert \int
_{\mathbb{R}^{n}}\mathfrak{P}_{h_{I;\kappa}^{n-1,\eta},\phi}\left(
\xi\right)  h_{J;\kappa}^{n,\eta}\left(  \xi\right)  d\xi\right\vert
=\left\vert \int_{\mathbb{R}^{n}}\left(  \frac{2\pi}{\left\vert \xi\right\vert
}\right)  ^{\frac{n-1}{2}}\frac{\xi_{n}}{\left\vert \xi\right\vert
}e^{i\left(  \left\vert \xi\right\vert -\frac{\left(  n-2\right)  \pi}%
{4}\right)  }h_{I;\kappa}^{n-1,\eta}\left(  \frac{\xi^{\prime}}{\left\vert
\xi\right\vert }\right)  h_{J;\kappa}^{n,\eta}\left(  \xi\right)
d\xi\right\vert \\
& \lesssim\int_{\mathbb{R}^{n}}\left(  \frac{1}{\operatorname*{dist}\left(
0,J\right)  }\right)  ^{\frac{n-1}{2}}\left\vert h_{I;\kappa}^{n-1,\eta
}\left(  \frac{\xi^{\prime}}{\left\vert \xi\right\vert }\right)  \right\vert
\left\vert h_{J;\kappa}^{n,\eta}\left(  \xi\right)  \right\vert d\xi
\approx\left(  \frac{1}{\operatorname*{dist}\left(  0,J\right)  }\right)
^{\frac{n-1}{2}}\int_{\mathbb{R}^{n}}\frac{1}{\sqrt{\left\vert I\right\vert }%
}\mathbf{1}_{I}\left(  \frac{\xi^{\prime}}{\left\vert \xi\right\vert }\right)
\frac{1}{\sqrt{\left\vert J\right\vert }}\mathbf{1}_{J}\left(  \xi\right)
d\xi\\
& \approx\left(  \frac{1}{\operatorname*{dist}\left(  0,J\right)  }\right)
^{\frac{n-1}{2}}\frac{1}{\sqrt{\left\vert I\right\vert \left\vert J\right\vert
}}\left\vert J\right\vert =\left(  \frac{1}{\ell\left(  I\right)
^{2}\operatorname*{dist}\left(  0,J\right)  }\right)  ^{\frac{n-1}{2}}%
\sqrt{\left\vert I\right\vert \left\vert J\right\vert }\lesssim2^{-d\frac
{n-1}{2}}\sqrt{\left\vert I\right\vert \left\vert J\right\vert }.
\end{align*}
The intermediate terms $\mathfrak{P}_{\varphi_{I}^{\eta},\phi}^{\left(
\ell\right)  }\left(  y,\lambda\right)  $ can be estimated in a similar way.

Next we estimate the inner product with the error term $\mathfrak{R}%
_{h_{I;\kappa}^{n-1,\eta},\phi}^{\left(  M+1\right)  }$ using the bound
(\ref{special bound}),%
\[
\left\vert \mathfrak{R}_{h_{I;\kappa}^{n-1,\eta},\phi}^{\left(  M+1\right)
}\left(  y,\lambda\right)  \right\vert \leq C_{M}\lambda^{-\frac{n-1}{2}%
-M-1}\left\Vert \left(  \operatorname{Id}-\bigtriangleup_{x}\right)
^{N}h_{I;\kappa}^{n-1,\eta}\right\Vert _{L^{1}\left(  \mathbb{R}_{x}%
^{n-1}\right)  \times L^{\infty}\left(  \mathbb{R}_{y}^{n}-1\right)  }\leq
C_{M}\lambda^{-\frac{n-1}{2}-M-1}\frac{1}{\ell\left(  I\right)  ^{2N}}%
\sqrt{\left\vert I\right\vert },
\]
for $N>1+\frac{n-1}{2}$, to obtain%
\begin{align}
& \ \ \ \ \ \ \ \ \ \ \left\vert \left\langle \mathfrak{R}_{h_{I;\kappa
}^{n-1,\eta},\phi}^{\left(  M+1\right)  },h_{J;\kappa}^{n,\eta}\right\rangle
\right\vert =\left\vert \int_{\mathbb{R}^{n}}\mathfrak{R}_{h_{I;\kappa
}^{n-1,\eta},\phi}^{\left(  M+1\right)  }\left(  \xi\right)  h_{J;\kappa
}^{n,\eta}\left(  \xi\right)  d\xi\right\vert \label{calcul}\\
& \lesssim\left(  \frac{1}{\operatorname*{dist}\left(  0,J\right)  }\right)
^{\frac{n-1}{2}+1}\left(  \frac{1}{\ell\left(  I\right)  ^{2}}\right)
^{N}\sqrt{\left\vert I\right\vert \left\vert J\right\vert }\approx2^{-d\left(
\frac{n-1}{2}+1\right)  }\left(  \frac{1}{\ell\left(  I\right)  ^{2}}\right)
^{\tau}\sqrt{\left\vert I\right\vert \left\vert J\right\vert },\nonumber
\end{align}
where $\tau=N-\frac{n+1}{2}>0$.

Adding these estimates gives,%
\[
\left\vert \left\langle Th_{I;\kappa}^{n-1,\eta},h_{J;\kappa}^{n,\eta
}\right\rangle \right\vert \lesssim\left\{  \sum_{\ell=0}^{M}2^{-d\left(
\frac{n-1}{2}+\ell\right)  }+2^{-d\left(  \frac{n-1}{2}+M+1\right)  }\left(
\frac{1}{\ell\left(  I\right)  ^{2}}\right)  ^{\tau}\right\}  \sqrt{\left\vert
I\right\vert \left\vert J\right\vert },
\]
which completes the proof of (\ref{crude asym}). Since $N-\frac{n+1}{2}%
\in\frac{1}{2}\mathbb{Z}$, we may assume $0<\tau\leq1$.

\subsubsection{Tangential integration by parts}

Finally, we improve on the crude estimate (\ref{crude''}) in the case $k=0$,
$d\geq0$ and $m\in\mathbb{N}$ using a tangential integration by parts as our
last principle of decay, where the supports of $I$ and $\Phi^{-1}\left(
\pi_{\tan}J\right)  $ are separated by at least $\ell\left(  I\right)  $. Let
$\left(  I,J\right)  \in\mathcal{P}_{m}^{0,d}$ with $d\geq0$, i.e.
\[
\operatorname*{dist}\left(  \pi_{\tan}J,I\right)  \approx2^{m}\ell\left(
I\right)  ,\ \ \ \ell\left(  J\right)  =1,\ \ \ \text{and }\frac{2^{d-1}}%
{\ell\left(  I\right)  ^{2}}\leq\operatorname*{dist}\left(  0,J\right)
\leq\frac{2^{d+1}}{\ell\left(  I\right)  ^{2}}.
\]
Recall again the change of variable in (\ref{para}) and
(\ref{phi and psi notation}),%
\begin{align*}
& \left\langle Th_{I;\kappa}^{n-1,\eta},h_{J;\kappa}^{n,\eta}\right\rangle
=\int_{\mathbb{R}^{n}}\int_{\mathbb{R}^{n-1}}e^{-i\Phi\left(  x\right)
\cdot\xi}h_{I;\kappa}^{n-1,\eta}\left(  x\right)  h_{J;\kappa}^{n,\eta}\left(
\xi\right)  dxd\xi\\
& =\int_{\mathbb{R}}\int_{\mathbb{R}^{n-1}}\int_{\mathbb{R}^{n-1}}%
e^{-i\lambda\phi\left(  x,y\right)  }\varphi_{I}^{\eta}\left(  x\right)
\widetilde{\psi}_{J}^{\eta}\left(  y,\lambda\right)  dxdyd\lambda,
\end{align*}
where%
\begin{align*}
\phi\left(  x,y\right)   & \equiv\Phi\left(  x\right)  \cdot\Phi\left(
y\right)  ,\\
\varphi_{I}^{\eta}\left(  x\right)   & \equiv h_{I;\kappa}^{n-1,\eta}\left(
x\right)  \text{ and }\psi_{J}^{\eta}\left(  \xi\right)  =h_{J;\kappa}%
^{n,\eta}\left(  \xi\right)  ,\\
\widetilde{\psi}_{J}^{\eta}\left(  y,\lambda\right)   & \equiv\psi_{J}^{\eta
}\left(  \lambda y,\lambda\sqrt{1-\left\vert y\right\vert ^{2}}\right)
\frac{\lambda^{n-1}}{\sqrt{1-\left\vert y\right\vert ^{2}}}.
\end{align*}
Here the supports of $\pi_{\tan}J$ and $I$ are separated by a distance of
approximately $2^{m}\ell\left(  I\right)  $, and $\ell\left(  \pi_{\tan
}J\right)  \lesssim\ell\left(  I\right)  $, and this suggests we should
integrate by parts in the variables $x$ and $y$.

So let $y_{J}=\Phi^{-1}\left(  \pi_{\tan}c_{J}\right)  $ and $\mathbf{v}%
=\frac{y_{J}-c_{I}}{\left\vert y_{J}-c_{I}\right\vert }\in\mathbb{S}^{n-2}$ be
the unit vector in the direction of $y_{J}-c_{I}$, which is close to the
direction of $y-x$ for $x\in I$ and $y=\Phi^{-1}\left(  \pi_{\tan}\xi\right)
$ with $\xi\in J$. Consider the directional partial derivative $D_{\mathbf{v}%
}^{x}=\mathbf{v}\cdot\frac{\partial}{\partial x}$, and note that
\[
D_{\mathbf{v}}^{x}\phi\left(  x,y\right)  =\left(  D_{\mathbf{v}}\Phi\right)
\left(  x\right)  \cdot\Phi\left(  y\right)  .
\]
Since $\left(  D_{\mathbf{v}}\Phi\right)  \left(  x\right)  $ is perpendicular
to $\Phi\left(  x\right)  $ in $\mathbb{R}^{n}$, we have the estimate%
\[
\left\vert D_{\mathbf{v}}^{x}\phi\left(  x,y\right)  \right\vert
\approx\left\vert x-y\right\vert \ ,\ \ \ \ \ x\in I,\xi\in J.
\]

Now we compute%
\[
D_{\mathbf{v}}^{x}e^{-i\lambda\phi\left(  x,y\right)  }=-i\lambda
e^{-i\lambda\phi\left(  x,y\right)  }D_{\mathbf{v}}^{x}\phi\left(  x,y\right)
=-i\lambda e^{-i\lambda\phi\left(  x,y\right)  }\left(  D_{\mathbf{v}}%
\Phi\right)  \left(  x\right)  \cdot\Phi\left(  y\right)  ,
\]
and so%
\[
\left(  \frac{1}{-i\lambda\left(  D_{\mathbf{v}}\Phi\right)  \left(  x\right)
\cdot\Phi\left(  y\right)  }D_{\mathbf{v}}^{x}\right)  ^{N}e^{-i\lambda
\phi\left(  x,y\right)  }=e^{-i\lambda\phi\left(  x,y\right)  },
\]
which gives,%
\begin{align}
& \ \ \ \ \ \ \ \ \ \ \left\langle Th_{I;\kappa}^{n-1,\eta},h_{J;\kappa
}^{n,\eta}\right\rangle \label{tan int by parts}\\
& =\int_{\mathbb{R}}\int_{\mathbb{R}^{n-1}}\int_{\mathbb{R}^{n-1}}\left\{
\left(  \frac{1}{-i\lambda\left(  D_{\mathbf{v}}\Phi\right)  \left(  x\right)
\cdot\Phi\left(  y\right)  }D_{\mathbf{v}}^{x}\right)  ^{N}e^{i\lambda
\phi\left(  x,y\right)  }\right\}  \varphi_{I}^{\eta}\left(  x\right)
\widetilde{\psi}_{J}^{\eta}\left(  y,\lambda\right)  dxdyd\lambda\nonumber\\
& =i^{N}\int_{\mathbb{R}}\int_{\mathbb{R}^{n-1}}\int_{\mathbb{R}^{n-1}%
}e^{i\lambda\phi\left(  x,y\right)  }\left\{  \left(  D_{\mathbf{v}}^{x}%
\frac{1}{\left(  D_{\mathbf{v}}\Phi\right)  \left(  x\right)  \cdot\Phi\left(
y\right)  }\right)  ^{N}\right\}  \varphi_{I}^{\eta}\left(  x\right)
\widetilde{\psi}_{J}^{\eta}\left(  y,\lambda\right)  dxdy\frac{d\lambda
}{\lambda^{N}}.\nonumber
\end{align}
This integral can be estimated by%
\[
\left\vert \left\langle Th_{I;\kappa}^{n-1,\eta},h_{J;\kappa}^{n,\eta
}\right\rangle \right\vert \lesssim\int_{\mathbb{R}}\int_{\mathbb{R}^{n-1}%
}\int_{\mathbb{R}^{n-1}}\left\vert \left(  D_{\mathbf{v}}^{x}\frac{1}{\left(
D_{\mathbf{v}}\Phi\right)  \left(  x\right)  \cdot\Phi\left(  y\right)
}\right)  ^{N}\varphi_{I}^{\eta}\left(  x\right)  \right\vert \frac{1}%
{\lambda^{N}}\left\vert \widetilde{\psi}_{J}^{\eta}\left(  y,\lambda\right)
\right\vert dxdyd\lambda,
\]
where we have the following pointwise estimates for $N=0$ and $N=1$,%
\begin{align*}
& \left\vert \varphi_{I}^{\eta}\left(  x\right)  \right\vert \lesssim\frac
{1}{\sqrt{\left\vert I\right\vert }},\\
\text{and}  & \left\vert D_{\mathbf{v}}^{x}\frac{1}{\left(  D_{\mathbf{v}}%
\Phi\right)  \left(  x\right)  \cdot\Phi\left(  y\right)  }\varphi_{I}^{\eta
}\left(  x\right)  \right\vert \frac{1}{\lambda}\lesssim\frac{\left\vert
\partial_{x}\varphi_{I}^{\eta}\left(  x\right)  \right\vert }{\lambda
\left\vert \left(  D_{\mathbf{v}}\Phi\right)  \left(  x\right)  \cdot
\Phi\left(  y\right)  \right\vert }+\frac{\left\vert \varphi_{I}^{\eta}\left(
x\right)  \right\vert \left\vert \left(  D_{\mathbf{v}}^{2}\Phi\right)
\left(  x\right)  \cdot\Phi\left(  y\right)  \right\vert }{\lambda\left\vert
\left(  D_{\mathbf{v}}\Phi\right)  \left(  x\right)  \cdot\Phi\left(
y\right)  \right\vert ^{2}}\\
& \lesssim\frac{\frac{1}{\eta\ell\left(  I\right)  }\frac{1}{\sqrt{\left\vert
I\right\vert }}}{\lambda\left\vert x-y\right\vert }+\frac{\frac{1}%
{\sqrt{\left\vert I\right\vert }}}{\lambda\left\vert x-y\right\vert ^{2}%
}\lesssim\frac{\frac{1}{\eta}\frac{1}{\sqrt{\left\vert I\right\vert }}%
}{\lambda2^{m}\ell\left(  I\right)  \ell\left(  I\right)  }+\frac{\frac
{1}{\sqrt{\left\vert I\right\vert }}}{\lambda\left(  2^{m}\ell\left(
I\right)  \right)  ^{2}}\\
& \lesssim\frac{1}{\lambda2^{m}\ell\left(  I\right)  ^{2}}\frac{1}%
{\sqrt{\left\vert I\right\vert }}=2^{-m}\frac{1}{\operatorname*{dist}\left(
0,J\right)  \ell\left(  I\right)  ^{2}}\frac{1}{\sqrt{\left\vert I\right\vert
}}.
\end{align*}

We claim that by induction on $N$ we have%
\begin{equation}
\frac{1}{\lambda^{N}}\left\vert \left(  D_{\mathbf{v}}^{x}\frac{1}{\left(
D_{\mathbf{v}}\Phi\right)  \left(  x\right)  \cdot\Phi\left(  y\right)
}\right)  ^{N}\varphi_{I}^{\eta}\left(  x\right)  \right\vert \lesssim
2^{-Nm}\left(  \frac{1}{\operatorname*{dist}\left(  0,J\right)  \ell\left(
I\right)  ^{2}}\right)  ^{N}\frac{1}{\sqrt{\left\vert I\right\vert }%
}.\label{inductive conc}%
\end{equation}
For simplicity, we illustrate the inductive step in the case $N=2$, and
compute%
\begin{align*}
& D_{\mathbf{v}}^{x}\frac{1}{\left(  D_{\mathbf{v}}\Phi\right)  \left(
x\right)  \cdot\Phi\left(  y\right)  }D_{\mathbf{v}}^{x}\frac{1}{\left(
D_{\mathbf{v}}\Phi\right)  \left(  x\right)  \cdot\Phi\left(  y\right)
}\varphi_{I}^{\eta}\left(  x\right) \\
& =D_{\mathbf{v}}^{x}\left(  \frac{D_{\mathbf{v}}^{x}\varphi_{I}^{\eta}\left(
x\right)  }{\left[  \left(  D_{\mathbf{v}}\Phi\right)  \left(  x\right)
\cdot\Phi\left(  y\right)  \right]  ^{2}}-\frac{\varphi_{I}^{\eta}\left(
x\right)  \left(  D_{\mathbf{v}}^{2}\Phi\right)  \left(  x\right)  \cdot
\Phi\left(  y\right)  }{\left[  \left(  D_{\mathbf{v}}\Phi\right)  \left(
x\right)  \cdot\Phi\left(  y\right)  \right]  ^{3}}\right) \\
& =\frac{\left(  D_{\mathbf{v}}^{x}\right)  ^{2}\varphi_{I}^{\eta}\left(
x\right)  }{\left[  \left(  D_{\mathbf{v}}\Phi\right)  \left(  x\right)
\cdot\Phi\left(  y\right)  \right]  ^{2}}-3\frac{D_{\mathbf{v}}^{x}\varphi
_{I}^{\eta}\left(  x\right)  \left(  D_{\mathbf{v}}^{2}\Phi\right)  \left(
x\right)  \cdot\Phi\left(  y\right)  }{\left[  \left(  D_{\mathbf{v}}%
\Phi\right)  \left(  x\right)  \cdot\Phi\left(  y\right)  \right]  ^{3}}\\
& -\frac{\varphi_{I}^{\eta}\left(  x\right)  \left(  D_{\mathbf{v}}^{3}%
\Phi\right)  \left(  x\right)  \cdot\Phi\left(  y\right)  }{\left[  \left(
D_{\mathbf{v}}\Phi\right)  \left(  x\right)  \cdot\Phi\left(  y\right)
\right]  ^{3}}+3\frac{\varphi_{I}^{\eta}\left(  x\right)  \left[  \left(
D_{\mathbf{v}}^{2}\Phi\right)  \left(  x\right)  \cdot\Phi\left(  y\right)
\right]  ^{2}}{\left[  \left(  D_{\mathbf{v}}\Phi\right)  \left(  x\right)
\cdot\Phi\left(  y\right)  \right]  ^{4}},
\end{align*}
which gives,%
\begin{align*}
\frac{1}{\lambda^{2}}\left\vert \left(  D_{\mathbf{v}}^{x}\frac{1}{\left(
D_{\mathbf{v}}\Phi\right)  \left(  x\right)  \cdot\Phi\left(  y\right)
}\right)  ^{2}\varphi_{I}^{\eta}\left(  x\right)  \right\vert  & \lesssim
\frac{1}{\lambda^{2}}\left(  \frac{\left(  \frac{1}{\eta\ell\left(  I\right)
}\right)  ^{2}\frac{1}{\sqrt{\left\vert I\right\vert }}}{\left\vert
x-y\right\vert ^{2}}+\frac{\left(  \frac{1}{\eta\ell\left(  I\right)
}\right)  \frac{1}{\sqrt{\left\vert I\right\vert }}}{\left\vert x-y\right\vert
^{3}}+\frac{\frac{1}{\sqrt{\left\vert I\right\vert }}\left\vert x-y\right\vert
}{\left\vert x-y\right\vert ^{3}}+\frac{\frac{1}{\sqrt{\left\vert I\right\vert
}}}{\left\vert x-y\right\vert ^{4}}\right) \\
& \lesssim\frac{1}{\lambda^{2}}\left(  \frac{1}{2^{2m}\ell\left(  I\right)
^{4}}+\frac{1}{2^{3m}\ell\left(  I\right)  ^{4}}+\frac{1}{2^{4m}\ell\left(
I\right)  ^{4}}\right)  \frac{1}{\sqrt{\left\vert I\right\vert }}\\
& \lesssim\frac{1}{\lambda^{2}}\frac{1}{2^{2m}\ell\left(  I\right)  ^{4}}%
\frac{1}{\sqrt{\left\vert I\right\vert }}=2^{-2m}\left(  \frac{1}%
{\operatorname*{dist}\left(  0,J\right)  \ell\left(  I\right)  ^{2}}\right)
^{2}\frac{1}{\sqrt{\left\vert I\right\vert }},
\end{align*}
which is the case $N=2$ of (\ref{inductive conc}). The general case is similar.

The estimate (\ref{inductive conc}) leads to the inner product estimate,%
\begin{align}
& \ \ \ \ \ \ \ \ \ \ \ \ \ \ \ \left\vert \left\langle Th_{I;\kappa
}^{n-1,\eta},h_{J;\kappa}^{n,\eta}\right\rangle \right\vert
\label{tan int est}\\
& \leq\int_{\mathbb{R}}\int_{\mathbb{R}^{n-1}}\int_{\mathbb{R}^{n-1}}\frac
{1}{\lambda^{N}}\left\vert \left\{  \left(  D_{\mathbf{v}}^{x}\frac{1}{\left(
D_{\mathbf{v}}\Phi\right)  \left(  x\right)  \cdot\Phi\left(  y\right)
}\right)  ^{N}\right\}  \varphi_{I}^{\eta}\left(  x\right)  \right\vert
\left\vert \widehat{\psi}_{J}^{\eta}\left(  y,\lambda\right)  \right\vert
dxdyd\lambda\nonumber\\
& \leq\int_{\mathbb{R}}\int_{\mathbb{R}^{n-1}}\int_{\mathbb{R}^{n-1}}%
2^{-Nm}\left(  \frac{1}{\operatorname*{dist}\left(  0,J\right)  \ell\left(
I\right)  ^{2}}\right)  ^{N}\frac{1}{\sqrt{\left\vert I\right\vert }%
}\left\vert \widehat{\psi}_{J}^{\eta}\left(  y,\lambda\right)  \right\vert
dxdyd\lambda\nonumber\\
& \leq2^{-Nm}\left(  \frac{1}{\operatorname*{dist}\left(  0,J\right)
\ell\left(  I\right)  ^{2}}\right)  ^{N}\frac{1}{\sqrt{\left\vert I\right\vert
}}\left\vert I\right\vert \left\Vert \widehat{\psi}_{J}^{\eta}\left(
y,\lambda\right)  \right\Vert _{L^{1}}\approx2^{-N\left(  m+d\right)  }%
\sqrt{\left\vert I\right\vert \left\vert J\right\vert },\nonumber
\end{align}
since $\operatorname*{dist}\left(  0,J\right)  \ell\left(  I\right)
^{2}\approx2^{d} $ for $\left(  I,J\right)  \in\mathcal{P}_{m}^{0,d}$,
$d\geq0$.

\section{Interpolation estimates\label{Sub interp}}

Here we describe the decay principle needed to handle sums of resonant inner
products by probability. In fact the probabilistic estimates here rely only on
the \emph{transversality} induced by\ the curvature of the sphere, and not on
stationary phase estimates. Throughout this subsection we will use the
familiar notation $\widehat{\varphi}$ for the Fourier transform of $\varphi$,
and we will use the parameter $s\in\mathbb{N}$ to pigeonhole the side length
$2^{-s}$ of a cube $I\in\mathcal{G}$. Let
\[
\mathsf{Q}_{U}^{s}\equiv\sum_{I\in\mathcal{G}_{s}\left[  U\right]
}\bigtriangleup_{I;\kappa}^{n-1},\ \ \ \ \ \text{where }\mathcal{G}_{s}\left[
U\right]  =\left\{  I\in\mathcal{G}:I\subset U\text{ and }\ell\left(
I\right)  =2^{-s}\right\}  ,
\]
be the Alpert projection onto $\mathcal{G}_{s}\left[  U\right]  $, i.e.
$\bigtriangleup_{I;\kappa}$ and $\bigtriangleup_{I;\kappa}^{\eta}$ are
restricted to dyadic subcubes $I$ of $U$ at depth $s$ in the grid
$\mathcal{G}$. Then we have%
\begin{align*}
\left(  \mathsf{Q}_{U}^{s}\right)  ^{\spadesuit}f  & =S_{\kappa,\eta
}\mathsf{Q}_{U}^{s}\left(  S_{\kappa,\eta}\right)  ^{-1}f=S_{\kappa,\eta}%
\sum_{I\in\mathcal{G}_{s}\left[  U\right]  }\left\langle \left(
S_{\kappa,\eta}\right)  ^{-1}f,h_{I;\kappa}^{n-1}\right\rangle h_{I;\kappa
}^{n-1}\\
& =\sum_{I\in\mathcal{G}_{s}\left[  U\right]  }\left\langle \left(
S_{\kappa,\eta}\right)  ^{-1}f,h_{I;\kappa}^{n-1}\right\rangle h_{I;\kappa
}^{n-1,\eta}=\sum_{I\in\mathcal{G}_{s}\left[  U\right]  }\bigtriangleup
_{I;\kappa}^{n-1,\eta}f.
\end{align*}

Let $\varphi\in C^{\infty}\left(  \mathbb{R}^{n}\right)  $ be a smooth
nonnegative function satisfying
\begin{equation}
\varphi\left(  \xi\right)  =\left\{
\begin{array}
[c]{ccc}%
1 & \ \text{if } & \xi\in B_{\mathbb{R}^{n}}\left(  0,1\right)  \\
0 & \ \text{if } & \xi\notin B_{\mathbb{R}^{n}}\left(  0,2\right)
\end{array}
\right.  ,\label{def phi}%
\end{equation}
and set
\[
\varphi_{t}\left(  \xi\right)  =2^{-tn}\varphi\left(  2^{-t}\xi\right)
,\ \ \ \ \ \text{for }t\geq0,
\]
where we note that the scaling is with respect to $2^{-t}$ instead of the
usual scaling $t$. Recall that $\Phi\left(  x\right)  =\left(  x,\sqrt
{1-\left\vert x\right\vert ^{2}}\right)  \in\mathbb{S}^{n-1}$ for $x\in S$.
Define the spherical measure $f_{\Phi}^{I}$ by%
\begin{align}
f_{\Phi}^{I}\left(  z\right)    & \equiv\Phi_{\ast}\bigtriangleup_{I;\kappa
}^{n-1,\eta}f=\bigtriangleup_{I;\kappa}^{n-1,\eta}f\left(  \Phi^{-1}\left(
z\right)  \right)  \det\partial\Phi^{-1}\left(  z\right)  d\sigma_{n-1}\left(
z\right)  \label{def f_Phi^I}\\
& =\left\langle \left(  S_{\kappa,\eta}\right)  ^{-1}f,h_{I;\kappa}%
^{n-1}\right\rangle h_{I;\kappa}^{n-1,\eta}\left(  \Phi^{-1}\left(  z\right)
\right)  \det\partial\Phi^{-1}\left(  z\right)  d\sigma_{n-1}\left(  z\right)
,\nonumber
\end{align}
and set
\[
f_{\Phi}^{s}\left(  z\right)  \equiv\sum_{I\in\mathcal{G}_{s}\left[  U\right]
}f_{\Phi}^{I}\left(  z\right)  =\Phi_{\ast}\sum_{I\in\mathcal{G}_{s}\left[
U\right]  }f^{I}\left(  z\right)  =\Phi_{\ast}\left(  \mathsf{Q}_{U}%
^{s}\right)  ^{\spadesuit}f\ .
\]
Note that the spherical measure $f_{\Phi}^{I}$ has mass roughly $\left\vert
\left\langle \left(  S_{\kappa,\eta}\right)  ^{-1}f,h_{I;\kappa}%
^{n-1}\right\rangle \right\vert 2^{-s\left(  n-1\right)  }$ for $I\in
\mathcal{G}_{s}\left[  U\right]  $ and is supported in $\mathbb{S}^{n-1}$.

Here is the model result of this subsection, where we recall that%
\[
\left(  \mathcal{A}_{\mathbf{a}}\mathsf{Q}_{U}^{s}\right)  ^{\spadesuit
}f=S_{\kappa,\eta}\mathcal{A}_{\mathbf{a}}\mathsf{Q}_{U}^{s}\left(
S_{\kappa,\eta}\right)  ^{-1}f=\sum_{I\in\mathcal{G}_{s}\left[  U\right]
}a_{I}\bigtriangleup_{I;\kappa}^{n-1,\eta}f.
\]

\begin{proposition}
\label{prop interp}Let $n\geq2$. Then for $p>\frac{2n}{n-1}$, there is
$\varepsilon_{p,n}>0$ such that for every $s\in\mathbb{N}$, and every $f\in
L^{p}\left(  \mathbb{R}^{n-1}\right)  $, we have,
\begin{equation}
\mathbb{E}_{2^{\mathcal{G}}}^{\mu}\left\Vert T\left[  \left(  \mathcal{A}%
_{\mathbf{a}}\mathsf{Q}_{U}^{s}\right)  ^{\spadesuit}f\right]  \right\Vert
_{L^{p}\left(  B_{n}\left(  0,2^{s}\right)  \right)  }\lesssim2^{-s\varepsilon
_{p,n}}\left\Vert f\right\Vert _{L^{p}\left(  \mathbb{R}^{n-1}\right)
},\label{eps}%
\end{equation}
where the implied constant depends on $n$, $p$ and $U$, but is independent of
$s\in\mathbb{N}$.
\end{proposition}

This estimate is a building block toward controlling the resonant portion of
the disjoint form, which however requires a much larger localization to a ball
of radius $2^{2s}$.

We prove Proposition \ref{prop interp}\ in three steps, beginning with
Plancherel's theorem in the form of a lemma that allows improvement of the
traditional $L^{2}$ and $L^{4}$ curvature estimates in the presence of
probability and Alpert wavelets. Then we use the scaled Marcinkiewicz
interpolation theorem to obtain the desired conclusion if certain $L^{2}$ and
$L^{4}$ estimates hold. Finally we establish these $L^{2}$ and $L^{4}$
estimates to complete the proof of Proposition \ref{prop interp}.

Recall that%
\begin{equation}
f_{\Phi}^{s}\equiv\Phi_{\ast}\left(  \mathsf{Q}_{U}^{s}\right)  ^{\spadesuit
}f\text{ and }f_{\Phi}^{I}\equiv\left(  \bigtriangleup_{I;\kappa}^{n-1,\eta
}f\right)  _{\Phi}\ .\label{def Phi and I}%
\end{equation}
For $s\leq r\leq2s$, define a fattened $n$-dimensional measure $f_{\Phi,r}%
^{s}$ by%
\begin{equation}
f_{\Phi,r}^{s}\equiv f_{\Phi}^{s}\ast\varphi_{r}=\sum_{I\in\mathcal{G}%
_{s}\left[  U\right]  }f_{\Phi}^{I}\ast\varphi_{r}=\sum_{I\in\mathcal{G}%
_{s}\left[  U\right]  }f_{\Phi,r}^{I},\ \ \ \ \ \text{where }f_{\Phi,r}%
^{I}\equiv f_{\Phi}^{I}\ast\varphi_{r}\ .\label{def maj}%
\end{equation}
We will use the upper majorant properties of $L^{2}$ and $L^{4}$ (we use this
latter phrase loosely to denote that convolution is a positive operation) to
obtain Lemma \ref{fattened red} below in order to significantly reduce the
norm $\left\Vert T\left(  \mathsf{Q}_{U}^{s}\right)  ^{\spadesuit}f\right\Vert
_{L^{p}\left(  \left\vert \widehat{\varphi_{s}}\right\vert ^{4}\lambda
_{n}\right)  }^{p}$ when averaged over involutive Alpert multipliers of $f$.

\begin{description}
\item[Note] \label{fat}The $n$-dimensional measure $f_{\Phi,r}^{I}=f_{\Phi
}^{I}\ast\varphi_{r}$ is supported in the fattened spherical cap
\[
\mathcal{I}_{2^{-r}}\equiv\left\{  z\in\mathbb{R}^{n}:\operatorname*{dist}%
\left(  z,\operatorname*{Supp}f_{\Phi}^{I}\right)  \lesssim2^{-r}\right\}  ,
\]
which for $r=2s$ is roughly a rectangular block of side lengths $2^{-2s}%
\times2^{-s}$ oriented perpendicular to a normal of the spherical cap
$\operatorname*{Supp}f_{\Phi}^{I}$. We have the estimate,%
\begin{equation}
\left\vert f_{\Phi,r}^{I}\left(  z\right)  \right\vert \lesssim\left\vert
\left\langle S_{\kappa,\eta}^{-1}f,h_{I;\kappa}^{n-1}\right\rangle \right\vert
2^{r}2^{s\frac{n-1}{2}}\mathbf{1}_{\mathcal{I}_{2^{-r}}}\left(  z\right)
.\label{A app}%
\end{equation}

\end{description}

\begin{lemma}
\label{fattened red}Suppose $s\in\mathbb{N}$, and $\varphi$ is as in
(\ref{def phi}) above, so that $\left\vert \widehat{\varphi_{s}}\right\vert
\approx1$ on $B\left(  0,C2^{s}\right)  $. Then for $s\leq r\leq2s$, we have%
\begin{align*}
\int_{\mathbb{R}^{n}}\left\vert \widehat{f_{\Phi}^{s}}\left(  \xi\right)
\right\vert ^{2}\left\vert \widehat{\varphi_{2s}}\left(  \xi\right)
\right\vert ^{2}\left\vert \widehat{\varphi_{r}}\left(  \xi\right)
\right\vert ^{2}d\xi & =\int_{\mathbb{R}^{n}}\left\vert \widehat{f_{\Phi
,2s}^{s}}\left(  \xi\right)  \right\vert ^{2}\left\vert \widehat{\varphi_{r}%
}\left(  \xi\right)  \right\vert ^{2}d\xi,\\
\int_{\mathbb{R}^{n}}\left\vert \widehat{f_{\Phi}^{s}}\left(  \xi\right)
\right\vert ^{4}\left\vert \widehat{\varphi_{r}}\left(  \xi\right)
\right\vert ^{4}d\xi & =\int_{\mathbb{R}^{n}}\left\vert \widehat{f_{\Phi
,r}^{s}}\left(  \xi\right)  \right\vert ^{4}d\xi,
\end{align*}

\end{lemma}

\begin{proof}
From Plancherel's formula, we have%
\[
\int_{\mathbb{R}^{n}}\left\vert \widehat{f_{\Phi}^{s}}\left(  \xi\right)
\right\vert ^{2}\left\vert \widehat{\varphi_{2s}}\left(  \xi\right)
\right\vert ^{2}\left\vert \widehat{\varphi_{r}}\left(  \xi\right)
\right\vert ^{4}d\xi=\int_{\mathbb{R}^{n}}\left\vert \widehat{f_{\Phi}^{s}%
\ast\varphi_{2s}}\left(  \xi\right)  \right\vert ^{2}\left\vert \widehat
{\varphi_{r}}\left(  \xi\right)  \right\vert ^{2}d\xi=\int_{\mathbb{R}^{n}%
}\left\vert \widehat{f_{\Phi,2s}^{s}}\left(  \xi\right)  \right\vert
^{2}\left\vert \widehat{\varphi_{r}}\left(  \xi\right)  \right\vert ^{2}d\xi,
\]

and using Plancherel's formula again with the convolution identity
$\widehat{F\ast G}=\widehat{F}\widehat{G}$, gives%
\begin{align*}
& \int_{\mathbb{R}^{n}}\left\vert \widehat{f_{\Phi}^{s}}\left(  \xi\right)
\right\vert ^{4}\left\vert \widehat{\varphi_{r}}\left(  \xi\right)
\right\vert ^{4}d\xi=\int_{\mathbb{R}^{n}}\left\vert \widehat{f_{\Phi}^{s}\ast
f_{\Phi}^{s}\ast\varphi_{r}\ast\varphi_{r}}\left(  \xi\right)  \right\vert
^{2}d\xi\\
& =\int_{\mathbb{R}^{n}}\overline{\widehat{f_{\Phi}^{s}\ast f_{\Phi}^{s}%
\ast\varphi_{r}\ast\varphi_{r}}\left(  \xi\right)  }\left(  \xi\right)
\ \widehat{f_{\Phi}^{s}\ast f_{\Phi}^{s}\ast\varphi_{r}\ast\varphi_{r}}\left(
\xi\right)  d\xi\\
& =\int_{S}\overline{f_{\Phi,r}^{s}\ast f_{\Phi,r}^{s}\left(  x\right)
}\ f_{\Phi,r}^{s}\ast f_{\Phi,r}^{s}\left(  x\right)  \ dx=\int_{\mathbb{R}%
^{n}}\left\vert \widehat{f_{\Phi,r}^{s}}\left(  \xi\right)  \right\vert
^{4}d\xi.
\end{align*}

\end{proof}

Here is the lemma that obtains the required $L^{p}$ bounds from improved
$L^{2}$ and $L^{4}$ bounds.

\begin{lemma}
\label{interp red}Let $n\geq2$ and $s\in\mathbb{N}$. Assume that%
\begin{align}
\left\Vert \widehat{f_{\Phi,2s}^{s}}\right\Vert _{L^{2}\left(  \left\vert
\widehat{\varphi_{s}}\right\vert ^{2}\lambda_{n}\right)  }  & \lesssim
2^{\frac{s}{2}}\left\Vert f\right\Vert _{L^{2}\left(  S\right)  }%
,\label{assumed}\\
\mathbb{E}_{2^{\mathcal{G}}}^{\mu}\left\Vert \widehat{\left[  \left(
\mathcal{A}_{\mathbf{a}}\mathsf{Q}_{U}^{s}\right)  ^{\spadesuit}\right]
_{\Phi,2s}}\right\Vert _{L^{4}\left(  \lambda_{n}\right)  }  & \lesssim
2^{-s\frac{n-2}{4}}\left\Vert f\right\Vert _{L^{4}\left(  S\right)
}.\nonumber
\end{align}
Then for $p>\frac{2n}{n-1}$, there is $\varepsilon_{p,n}>0$ such that%
\[
\mathbb{E}_{2^{\mathcal{G}}}^{\mu}\left\Vert \widehat{\left[  \left(
\mathcal{A}_{\mathbf{a}}\mathsf{Q}_{U}^{s}\right)  ^{\spadesuit}\right]
_{\Phi,2s}}\right\Vert _{L^{p}\left(  \left\vert \widehat{\varphi_{s}%
}\right\vert ^{2}\left\vert \widehat{\varphi_{2s}}\right\vert ^{4}\lambda
_{n}\right)  }\lesssim2^{-s\varepsilon_{p,n}}\left\Vert f\right\Vert
_{L^{p}\left(  \mathbb{R}^{n-1}\right)  },
\]
holds for every $s\in\mathbb{N}$ with implied constant independent of $\Psi
$\ and $s$.
\end{lemma}

Note in particular that Lemma \ref{interp red} implies (\ref{eps}) in
Proposition \ref{prop interp} .

\begin{proof}
Combining Lemma \ref{fattened red} with the assumptions (\ref{assumed}) gives
the pair of inequalities,%
\begin{align*}
\left\Vert T\left(  \mathsf{Q}_{U}^{s}\right)  ^{\spadesuit}f\right\Vert
_{L^{2}\left(  \left\vert \widehat{\varphi_{s}}\right\vert ^{2}\left\vert
\widehat{\varphi_{2s}}\right\vert ^{4}\lambda_{n}\right)  }  & \lesssim
2^{\frac{s}{2}}\left\Vert f\right\Vert _{L^{2}\left(  S\right)  },\\
\left(  \mathbb{E}_{2^{\mathcal{G}}}^{\mu}\left\Vert T\left(  \mathcal{A}%
_{\mathbf{a}}\mathsf{Q}_{U}^{s}\right)  ^{\spadesuit}f\right\Vert
_{L^{4}\left(  \left\vert \widehat{\varphi_{s}}\right\vert ^{2}\left\vert
\widehat{\varphi_{2s}}\right\vert ^{4}\lambda_{n}\right)  }^{4}\right)
^{\frac{1}{4}}  & \lesssim2^{-s\frac{n-2}{4}}\left\Vert f\right\Vert
_{L^{4}\left(  S\right)  }.
\end{align*}
Indeed,%
\begin{align*}
& \left\Vert T\left(  \mathsf{Q}_{U}^{s}\right)  ^{\spadesuit}f\right\Vert
_{L^{2}\left(  \left\vert \widehat{\varphi_{s}}\right\vert ^{2}\left\vert
\widehat{\varphi_{2s}}\right\vert ^{4}\lambda_{n}\right)  }^{2}\leq\left\Vert
T\left(  \mathsf{Q}_{U}^{s}\right)  ^{\spadesuit}f\right\Vert _{L^{2}\left(
\left\vert \widehat{\varphi_{s}}\right\vert ^{2}\left\vert \widehat
{\varphi_{2s}}\right\vert ^{2}\lambda_{n}\right)  }^{2}\\
& =\int_{\mathbb{R}^{n}}\left\vert T\left(  \mathsf{Q}_{U}^{s}\right)
^{\spadesuit}f\left(  \xi\right)  \right\vert ^{2}\left\vert \widehat
{\varphi_{s}}\left(  \xi\right)  \right\vert ^{2}\left\vert \widehat
{\varphi_{2s}}\left(  \xi\right)  \right\vert ^{2}d\xi\\
& =\int_{\mathbb{R}^{n}}\left\vert \widehat{\left[  \left(  \mathsf{Q}_{U}%
^{s}\right)  ^{\spadesuit}f\right]  _{\Phi}}\left(  \xi\right)  \right\vert
^{2}\left\vert \widehat{\varphi_{2s}}\left(  \xi\right)  \right\vert
^{2}\left\vert \widehat{\varphi_{s}}\left(  \xi\right)  \right\vert ^{2}%
d\xi=\int_{\mathbb{R}^{n}}\left\vert \widehat{\left[  \left(  \mathsf{Q}%
_{U}^{s}\right)  ^{\spadesuit}f\right]  _{\Phi,2s}}\left(  \xi\right)
\right\vert ^{2}\left\vert \widehat{\varphi_{s}}\left(  \xi\right)
\right\vert ^{2}d\xi\\
& =\left\Vert \widehat{\left[  \left(  \mathsf{Q}_{U}^{s}\right)
^{\spadesuit}f\right]  _{\Phi,2s}}\right\Vert _{L^{2}\left(  \left\vert
\widehat{\varphi_{s}}\right\vert ^{2}\lambda_{n}\right)  }^{2}\lesssim
2^{s}\left\Vert \left(  \mathsf{Q}_{U}^{s}\right)  ^{\spadesuit}f\right\Vert
_{L^{2}\left(  S\right)  }^{2}\lesssim2^{s}\left\Vert f\right\Vert
_{L^{2}\left(  S\right)  }^{2},
\end{align*}
and%
\begin{align*}
& \mathbb{E}_{2^{\mathcal{G}}}^{\mu}\left\Vert T\left(  \mathcal{A}%
_{\mathbf{a}}\mathsf{Q}_{U}^{s}\right)  ^{\spadesuit}f\right\Vert
_{L^{4}\left(  \left\vert \widehat{\varphi_{s}}\right\vert ^{2}\left\vert
\widehat{\varphi_{2s}}\right\vert ^{4}\lambda_{n}\right)  }^{4}\leq
\mathbb{E}_{2^{\mathcal{G}}}^{\mu}\left\Vert T\left(  \mathcal{A}_{\mathbf{a}%
}\mathsf{Q}_{U}^{s}\right)  ^{\spadesuit}f\right\Vert _{L^{4}\left(
\left\vert \widehat{\varphi_{s}}\right\vert ^{2}\left\vert \widehat
{\varphi_{2s}}\right\vert ^{4}\lambda_{n}\right)  }^{4}\\
& \leq\mathbb{E}_{2^{\mathcal{G}}}^{\mu}\int_{\mathbb{R}^{n}}\left\vert
\widehat{\left(  \left(  \mathcal{A}_{\mathbf{a}}\mathsf{Q}_{U}^{s}\right)
^{\spadesuit}f\right)  _{\Phi}}\left(  \xi\right)  \right\vert ^{4}\left\vert
\widehat{\varphi_{2s}}\left(  \xi\right)  \right\vert ^{4}d\xi=\mathbb{E}%
_{2^{\mathcal{G}}}^{\mu}\int_{\mathbb{R}^{n}}\left\vert \widehat{\sum
_{I\in\mathcal{G}_{s}\left[  U\right]  }a_{I}\left(  \bigtriangleup_{I;\kappa
}^{n-1,\eta}f\right)  _{\Phi}}\left(  \xi\right)  \right\vert ^{4}\left\vert
\widehat{\varphi_{2s}}\left(  \xi\right)  \right\vert ^{4}d\xi\\
& =\mathbb{E}_{2^{\mathcal{G}}}^{\mu}\int_{\mathbb{R}^{n}}\left\vert
\widehat{\left(  \sum_{I\in\mathcal{G}_{s}\left[  U\right]  }a_{I}%
\bigtriangleup_{I;\kappa}^{n-1,\eta}f\right)  _{\Phi,2s}}\left(  \xi\right)
\right\vert ^{4}d\xi=\mathbb{E}_{2^{\mathcal{G}_{s}\left[  U\right]  }}^{\mu
}\int_{\mathbb{R}^{n}}\left\vert \widehat{\left(  \sum_{I\in\mathcal{G}%
_{s}\left[  U\right]  }a_{I}\left(  \bigtriangleup_{I;\kappa}^{n-1,\eta
}f\right)  \right)  _{\Phi,2s}}\left(  \xi\right)  \right\vert ^{4}d\xi\\
& =\mathbb{E}_{2^{\mathcal{G}_{s}\left[  U\right]  }}^{\mu_{m}}\int
_{\mathbb{R}^{n}}\left\vert \widehat{\sum_{I\in\mathcal{G}_{s}\left[
U\right]  }a_{I}\left(  \bigtriangleup_{I;\kappa}^{n-1,\eta}f\right)
_{\Phi,2s}}\left(  \xi\right)  \right\vert ^{4}d\xi=\mathbb{E}_{2^{\mathcal{G}%
_{s}\left[  U\right]  }}^{\mu_{m}}\int_{\mathbb{R}^{n}}\left\vert \left[
\left(  \mathcal{A}_{\mathbf{a}}\mathsf{Q}_{U}^{s}\right)  ^{\spadesuit
}f\right]  _{\Phi,2s}\left(  \xi\right)  \right\vert ^{44}d\xi\\
& =\mathbb{E}_{2^{\mathcal{G}_{s}\left[  U\right]  }}^{\mu_{m}}\left\Vert
\left[  \left(  \mathcal{A}_{\mathbf{a}}\mathsf{Q}_{U}^{s}\right)
^{\spadesuit}f\right]  _{\Phi,2s}\right\Vert _{L^{4}\left(  \lambda
_{n}\right)  }^{4}\lesssim2^{-s\left(  n-2\right)  }\left\Vert \left(
\mathsf{Q}_{U}^{s}\right)  ^{\spadesuit}f\right\Vert _{L^{4}\left(
\lambda_{n-1}\right)  }^{4}\lesssim2^{-s\left(  n-2\right)  }\left\Vert
f\right\Vert _{L^{4}\left(  \lambda_{n-1}\right)  }^{4}\ ,
\end{align*}
since all three operators in the factorization $\left(  \mathsf{Q}_{U}%
^{s}\right)  ^{\spadesuit}=S_{\kappa,\eta}\mathsf{Q}_{U}^{s}\left(
S_{\kappa,\eta}\right)  ^{-1}$ are \ bounded on $L^{4}\left(  \lambda
_{n-1}\right)  $.

These $L^{2}$ and $L^{4}$ estimates can be recast in terms of Fourier square
functions by Khintchine's inequalities, and we will now show that the scaled
Marcinkiewicz interpolation theorem applies to obtain (\ref{eps}).

Indeed, by Khinchine's inequalities, the above bounds are equivalent to%
\begin{align*}
\left\Vert \mathcal{S}_{T,s}f\right\Vert _{L^{2}\left(  \lambda_{n}\right)  }
& \lesssim2^{\frac{s}{2}}\left\Vert f\right\Vert _{L^{2}\left(  \sigma
_{n-1}\right)  }\ ,\\
\left\Vert \mathcal{S}_{T,s}f\right\Vert _{L^{4}\left(  \mathbf{1}_{B\left(
0,2^{s}\right)  }\lambda_{n}\right)  }  & \lesssim2^{-s\frac{n-2}{4}%
}\left\Vert f\right\Vert _{L^{4}\left(  \sigma_{n-1}\right)  }\ ,
\end{align*}
where $\mathcal{S}_{T,s}$ is the Fourier square function defined by%
\[
\mathcal{S}_{T,s}f\equiv\left(  \sum_{I\in\mathcal{G}_{s}\left[  U\right]
}\left\vert T_{S}\bigtriangleup_{I;\kappa}^{n-1,\eta}f\right\vert ^{2}\right)
^{\frac{1}{2}}.
\]
The sublinear operator $\mathcal{S}_{T,s}$ is actually \emph{linearizable}
since it is the supremum of the linear operators $L_{\mathbf{u}}f\equiv
T_{S}\sum_{I\in\mathcal{G}_{s}\left[  U\right]  }u_{I}\bigtriangleup
_{I;\kappa}^{n-1,\eta}f$ taken over all vectors $\mathbf{u}=\left(
u_{I}\right)  _{I\in\mathcal{G}_{s}\left[  S\right]  }$ with $\left\vert
\mathbf{u}\right\vert _{\ell^{2}}=1$. Then by the scaled Marcinkiewicz theorem
applied to $\mathcal{S}_{T,s}$, see e.g. \cite[Remark 29]{Tao2}, we have
\[
\left\Vert \mathcal{S}_{T,s}f\right\Vert _{L^{p}}\leq C_{n,p}2^{\frac{s}%
{2}\left(  1-\theta\right)  }2^{-s\frac{n-2}{4}\theta}=C_{n,p}2^{\frac{s}%
{2}\left(  1-\left(  2-\frac{4}{p}\right)  \right)  }2^{-s\frac{n-2}{4}\left(
2-\frac{4}{p}\right)  }=C_{n,p}2^{-s\varepsilon_{n,p}},
\]
where%
\[
\varepsilon_{n,p}=\frac{n-2}{4}\left(  2-\frac{4}{p}\right)  -\frac{1}%
{2}\left(  1-\left(  2-\frac{4}{p}\right)  \right)  =\frac{n-1}{2p}\left(
p-\frac{2n}{n-1}\right)  >0,
\]
for $p>\frac{2n}{n-1}$. Another application of Khintchine's inquality converts
this bound back to the expectation bound,%
\[
\mathbb{E}_{2^{\mathcal{G}_{s}\left[  U\right]  }}^{\mu_{s}}\left\Vert
T\left(  \mathcal{A}_{\mathbf{a}}\mathsf{Q}_{U}^{s}\right)  ^{\spadesuit
}f\right\Vert _{L^{p}\left(  B_{n}\left(  0,2^{s}\right)  \right)  }\lesssim
C_{n,p}2^{-s\varepsilon_{n,p}}\left\Vert f\right\Vert _{L^{p}\left(
\mathbb{R}^{n-1}\right)  },
\]

which completes the proof of Lemma \ref{interp red}.
\end{proof}

It remains to establish the improved bounds in (\ref{assumed}), which we
accomplish in the next two subsections. Once this is done, the proof of
Proposition \ref{prop interp} is complete.

\subsection{The $L^{2}$ estimate\label{subsubsection L2}}

We first compute the norm of $\Lambda_{\mathsf{Q}_{U}^{s}}^{2s}$ from
$L^{2}\left(  \lambda_{n-1}\right)  $ to $L^{2}\left(  \left\vert
\widehat{\varphi_{s}}\right\vert ^{2}\lambda_{n}\right)  $, where
\[
\Lambda_{\mathsf{Q}_{U}^{s}}^{2s}f\equiv\widehat{\left(  \left(
\mathsf{Q}_{U}^{s}\right)  ^{\spadesuit}f\right)  _{\Phi,2s}}.
\]
We write $f_{U}^{s}\equiv\left(  \mathsf{Q}_{U}^{s}\right)  ^{\spadesuit}f$
for convenience in notation so that we have,%
\begin{align*}
& \left\Vert \Lambda_{\mathsf{Q}_{U}^{s}}^{2s}f\right\Vert _{L^{2}\left(
\left\vert \widehat{\varphi_{s}}\right\vert ^{2}\lambda_{n}\right)  }^{2}%
=\int_{\mathbb{R}^{n}}\left\vert \widehat{\left(  f_{U}^{s}\right)  _{\Phi
,2s}}\left(  \xi\right)  \right\vert ^{2}\left\vert \widehat{\varphi_{s}%
}\left(  \xi\right)  \right\vert ^{2}d\xi\\
& =\int_{\mathbb{R}^{n}}\overline{\widehat{\left(  f_{U}^{s}\right)
_{\Phi,2s}\ast\varphi_{s}}\left(  \xi\right)  }\ \widehat{\left(  f_{U}%
^{s}\right)  _{\Phi,2s}\ast\varphi_{s}}\left(  \xi\right)  d\xi\\
& =\sum_{I,K\in\mathcal{G}_{s}\left[  U\right]  }\int_{\mathbb{R}^{n}%
}\overline{\widehat{f_{\Phi,2s}^{I}\ast\varphi_{s}}\left(  \xi\right)
}\ \widehat{f_{\Phi,2s}^{K}\ast\varphi_{s}}\left(  \xi\right)  d\xi
=\sum_{I,K\in\mathcal{G}_{s}\left[  U\right]  }\int_{S}\overline{f_{\Phi
,2s}^{I}\ast\varphi_{s}\left(  x\right)  }\ \left(  f_{\Phi,2s}^{K}\ast
\varphi_{s}\right)  \left(  x\right)  dx.
\end{align*}
Noting that the supports of $f_{\Phi,2s}^{I}\ast\varphi_{s}$ and $f_{\Phi
,2s}^{K}\ast\varphi_{s}$ are essentially disjoint unless $I\sim K$, and
recalling the definition of $\mathcal{I}_{2^{-s}}$ in Note \ref{fat}, we can
use (\ref{A app}),%
\[
\left\vert f_{\Phi,r}^{I}\left(  z\right)  \right\vert \lesssim\left\vert
\left\langle S_{\kappa,\eta}^{-1}f,h_{I;\kappa}^{n-1}\right\rangle \right\vert
2^{r}2^{s\frac{n-1}{2}}\mathbf{1}_{\mathcal{I}_{2^{-r}}}\left(  z\right)  ,
\]
with $r=s$ to estimate the above expression by%
\begin{align}
\left\Vert \Lambda_{\mathsf{Q}_{U}^{s}}^{2s}f\right\Vert _{L^{2}\left(
\left\vert \widehat{\varphi_{s}}\right\vert ^{2}\lambda_{n}\right)  }^{2}  &
\lesssim\sum_{I\in\mathcal{G}_{s}\left[  U\right]  }\int_{\mathbb{R}^{n}%
}\left\vert f_{\Phi,2s}^{I}\ast\varphi_{s}\left(  \xi\right)  \right\vert
^{2}d\xi\label{est above exp}\\
& \lesssim\sum_{I\in\mathcal{G}_{s}\left[  U\right]  }\int_{\mathbb{R}^{n}%
}\left\vert \left\vert \left\langle S_{\kappa,\eta}^{-1}f,h_{I;\kappa}%
^{n-1}\right\rangle \right\vert 2^{s}2^{s\frac{n-1}{2}}\mathbf{1}%
_{\mathcal{I}_{2^{-s}}}\ast\varphi_{s}\left(  \xi\right)  \right\vert ^{2}%
d\xi\nonumber\\
& \lesssim\sum_{I\in\mathcal{G}_{s}\left[  U\right]  }\left\vert \left\langle
S_{\kappa,\eta}^{-1}f,h_{I;\kappa}^{n-1}\right\rangle \right\vert ^{2}%
\int_{\mathbb{R}^{n}}\left\vert 2^{s}2^{s\frac{n-1}{2}}\mathbf{1}%
_{\mathcal{I}_{2^{-s}}}\left(  \xi\right)  \right\vert ^{2}d\xi.\nonumber
\end{align}
Then we continue with%
\begin{align*}
\left\Vert \Lambda_{\mathsf{Q}_{U}^{s}}^{2s}f\right\Vert _{L^{2}\left(
\left\vert \widehat{\varphi_{s}}\right\vert ^{4}\lambda_{n}\right)  }^{2}  &
\lesssim\sum_{I\in\mathcal{G}_{s}\left[  U\right]  }\left\vert \left\langle
S_{\kappa,\eta}^{-1}f,h_{I;\kappa}^{n-1}\right\rangle \right\vert ^{2}\left(
2^{s}2^{s\frac{n-1}{2}}\right)  ^{2}\left\vert \mathcal{I}_{2^{-s}}\right\vert
\\
& =2^{s}\sum_{I\in\mathcal{G}_{s}\left[  U\right]  }\left\vert \left\langle
S_{\kappa,\eta}^{-1}f,h_{I;\kappa}^{n-1}\right\rangle \right\vert ^{2}%
\lesssim2^{s}\left\Vert S_{\kappa,\eta}^{-1}f\right\Vert _{L^{2}\left(
\mathbb{R}^{n-1}\right)  }^{2}\lesssim2^{s}\left\Vert f\right\Vert
_{L^{2}\left(  S\right)  }^{2}.
\end{align*}
This proves the first line in (\ref{assumed}).

\subsection{The probabilistic $L^{4}$ estimate\label{subsubsection L4}}

Now we turn to computing the norm of $\Lambda_{2s}$ from $L^{4}\left(
\lambda_{n-1}\right)  $ to $L^{4}\left(  \mathbb{R}^{n}\right)  $. We have
using $f_{U}^{s}\equiv\left(  \mathsf{Q}_{U}^{s}\right)  ^{\spadesuit}f$ that%
\begin{align*}
& \left\Vert f_{U}^{s}\right\Vert _{L^{4}\left(  \lambda_{n-1}\right)  }%
^{4}=\int_{\mathbb{R}^{n-1}}\left(  \sum_{I\in\mathcal{G}_{s}\left[  U\right]
}\left\langle \left(  S_{\kappa,\eta}\right)  ^{-1}f,h_{I;\kappa}%
^{n-1}\right\rangle h_{I;\kappa}^{n-1,\eta}\left(  x\right)  \right)  ^{4}dx\\
& \approx\int_{\mathbb{R}^{n-1}}\sum_{I\in\mathcal{G}_{s}\left[  U\right]
}\left(  \left\langle \left(  S_{\kappa,\eta}\right)  ^{-1}f,h_{I;\kappa
}^{n-1}\right\rangle h_{I;\kappa}^{n-1,\eta}\left(  x\right)  \right)
^{4}dx\\
& =\sum_{I\in\mathcal{G}_{s}\left[  U\right]  }\left\vert \left\langle \left(
S_{\kappa,\eta}\right)  ^{-1}f,h_{I;\kappa}^{n-1}\right\rangle \right\vert
^{4}\int_{\mathbb{R}^{n-1}}\left\vert h_{I;\kappa}^{n-1,\eta}\left(  x\right)
\right\vert ^{4}dx\\
& \approx\sum_{I\in\mathcal{G}_{s}\left[  U\right]  }\left\vert \left\langle
\left(  S_{\kappa,\eta}\right)  ^{-1}f,h_{I;\kappa}^{n-1}\right\rangle
\right\vert ^{4}\left(  \frac{1}{\sqrt{\left\vert I\right\vert }}\right)
^{4}\left\vert I\right\vert =\sum_{I\in\mathcal{G}_{s}\left[  U\right]
}\left\vert \left\langle \left(  S_{\kappa,\eta}\right)  ^{-1}f,h_{I;\kappa
}^{n-1}\right\rangle \right\vert ^{4}\frac{1}{\left\vert I\right\vert }\\
& =2^{s\left(  n-1\right)  }\sum_{I\in\mathcal{G}_{s}\left[  U\right]
}\left\vert \left\langle \left(  S_{\kappa,\eta}\right)  ^{-1}f,h_{I;\kappa
}^{n-1}\right\rangle \right\vert ^{4}=2^{s\left(  n-1\right)  }\left\vert
\breve{f}\right\vert _{\ell^{4}\left(  \mathcal{G}_{s}\left[  U\right]
\right)  }^{4}\ ,
\end{align*}
where $\breve{f}\equiv\left\{  \left\langle \left(  S_{\kappa,\eta}\right)
^{-1}f,h_{I;\kappa}^{n-1}\right\rangle \right\}  _{I\in\mathcal{G}_{s}\left[
S\right]  }$ is the sequence of Alpert coefficients of $\left(  S_{\kappa
,\eta}\right)  ^{-1}f$ restricted to $\mathcal{G}_{s}\left[  S\right]  $.
Recall that $\left\Vert \left(  S_{\kappa,\eta}\right)  ^{-1}f\right\Vert
_{L^{p}\left(  \mathbb{R}^{n-1}\right)  }\approx\left\Vert f\right\Vert
_{L^{p}\left(  \mathbb{R}^{n-1}\right)  }$ by Theorem \ref{reproducing}.

Next we calculate the $L^{4}\left(  \lambda_{n}\right)  $ norm of
$\Lambda_{\mathsf{Q}_{U}^{s}}^{2s}f\equiv\widehat{\left(  \left(
\mathsf{Q}_{U}^{s}\right)  ^{\spadesuit}f\right)  _{\Phi,2s}}=\widehat{\left(
f_{U}^{s}\right)  _{\Phi,2s}}$:%
\begin{align*}
& \left\Vert \Lambda_{\mathsf{Q}_{U}^{s}}^{2s}f\right\Vert _{L^{4}\left(
\lambda_{n}\right)  }^{4}=\int_{\mathbb{R}^{n}}\left\vert \widehat{\left(
f_{U}^{s}\right)  _{\Phi,2s}}\left(  \xi\right)  \right\vert ^{4}d\xi
=\int_{\mathbb{R}^{n}}\left\vert \sum_{I\in\mathcal{G}_{s}\left[  U\right]
}\widehat{f_{\Phi,2s}^{I}}\left(  \xi\right)  \right\vert ^{4}d\xi\\
& =\int_{\mathbb{R}^{n}}\left\vert \sum_{I,J\in\mathcal{G}_{s}\left[
U\right]  }\widehat{f_{\Phi,2s}^{I}}\left(  \xi\right)  \widehat{f_{\Phi
,2s}^{J}}\left(  \xi\right)  \right\vert ^{2}d\xi=\int_{\mathbb{R}^{n}%
}\left\vert \sum_{I,J\in\mathcal{G}_{s}\left[  U\right]  }\widehat{f_{\Phi
,2s}^{I}\ast f_{\Phi,2s}^{J}}\left(  \xi\right)  \right\vert ^{2}d\xi,
\end{align*}
by the Fourier convolution formula, and then by Plancherel's theorem,%
\[
\left\Vert \Lambda_{\mathsf{Q}_{U}^{s}}^{2s}f\right\Vert _{L^{4}\left(
\lambda_{n}\right)  }^{4}=\int_{\mathbb{R}^{n}}\left\vert \sum_{I,J\in
\mathcal{G}_{s}\left[  U\right]  }f_{\Phi,2s}^{I}\ast f_{\Phi,2s}^{J}\left(
z\right)  \right\vert ^{2}dz=\sum_{I,J,I^{\prime},J^{\prime}\in\mathcal{G}%
_{s}\left[  U\right]  }\int f_{\Phi,2s}^{I}\ast f_{\Phi,2s}^{J}\left(
z\right)  \ f_{\Phi,2s}^{I^{\prime}}\ast f_{\Phi,2s}^{J^{\prime}}\left(
z\right)  \ dz.
\]

Now we compute the average $\mathbb{E}_{2^{\mathcal{G}}}^{\mu}\left\Vert
\Lambda_{\mathcal{A}_{\mathbf{a}}\mathsf{Q}_{U}^{s}}^{2s}f\right\Vert
_{L^{4}\left(  \lambda_{n}\right)  }^{4}$ over all involutive smooth Alpert
multipliers $\left(  \mathcal{A}_{\mathbf{a}}\mathsf{Q}_{U}^{s}\right)
^{\spadesuit}$, where remembering that the functions $f_{\Phi,2s}^{I}$ have
the $\eta$-smoothness built into their definition,%
\begin{align*}
& \mathbb{E}_{2^{\mathcal{G}}}^{\mu}\left\Vert \Lambda_{\mathcal{A}%
_{\mathbf{a}}\mathsf{Q}_{U}^{s}}^{2s}f\right\Vert _{L^{4}\left(  \lambda
_{n}\right)  }^{4}\\
& =\mathbb{E}_{2^{\mathcal{G}}}^{\mu}\sum_{I,J,I^{\prime},J^{\prime}%
\in\mathcal{G}_{s}\left[  S\right]  }\sum_{\left(  a_{I},a_{J},a_{I^{\prime}%
},a_{J^{\prime}}\right)  \in\left\{  -1,1\right\}  ^{\mathcal{G}_{s}\left[
U\right]  }}\mathbb{E}_{2^{\mathcal{G}}}^{\mu}\int\left(  a_{I}f_{\Phi,2s}%
^{I}\right)  \ast\left(  a_{J}f_{\Phi,2s}^{J}\right)  \left(  z\right)
\ \left(  a_{I^{\prime}}f_{\Phi,2s}^{I^{\prime}}\right)  \ast\left(
a_{J^{\prime}}f_{\Phi,2s}^{J^{\prime}}\right)  \left(  z\right)  \ dz\\
& =2\left\{  \sum_{\substack{I,J,I^{\prime},J^{\prime}\in\mathcal{G}%
_{s}\left[  U\right]  \\I=J\text{ and }I^{\prime}=J^{\prime}}}+\sum
_{\substack{I,J,I^{\prime},J^{\prime}\in\mathcal{G}_{s}\left[  U\right]
\\I=I^{\prime}\text{ and }J=J^{\prime}}}\right\}  \int f_{\Phi,2s}^{I}\ast
f_{\Phi,2s}^{J}\left(  z\right)  \ f_{\Phi,2s}^{I^{\prime}}\ast f_{\Phi
,2s}^{J^{\prime}}\left(  z\right)  \ dz\equiv\mathcal{E}_{1}+\mathcal{E}_{2},
\end{align*}
since the only summands that survive expectation are those for which
$a_{I}a_{J}a_{I^{\prime}}a_{J^{\prime}}$ is a product of squares, i.e. the
factors occur in pairs of equal sign $\pm1$.

\begin{remark}
This is the key consequence of taking expectation, and is the only place in
the paper where it arises. Note also that in $n=2$ dimensions, Fefferman made
the critical observation that the supports of the convolutions $f_{\Phi
,2s}^{I}\ast f_{\Phi,2s}^{J}$ are essentially pairwise disjoint, so that the
$L^{2}$ norm squared of the sum is the sum of the $L^{2}$ norms squared. This
then led to the resolution of the extension problem in dimension $n=2$.
However, in higher dimensions this observation doesn't generalize in a simple
way, since there is an $\left(  n-2\right)  $-dimension sphere contained
inside $\mathbb{S}^{n-1}$ whose pairs of `antipodal cubes' support functions
whose convolutions all occupy the same space. The products of distinct pairs
of antipodal cubes vanish under expectation, which leads to a favourable
$L^{4}$ estimate.
\end{remark}

We have%
\[
\mathcal{E}_{2}=2\sum_{I,J\in\mathcal{G}_{s}\left[  U\right]  }\int
f_{\Phi,2s}^{I}\ast f_{\Phi,2s}^{J}\left(  z\right)  \ f_{\Phi,2s}^{I}\ast
f_{\Phi,2s}^{J}\left(  z\right)  \ dz=2\sum_{I,J\in\mathcal{G}_{s}\left[
U\right]  }\int\left\vert f_{\Phi,2s}^{I}\ast f_{\Phi,2s}^{J}\left(  z\right)
\right\vert ^{2}dz.
\]
Since the supports of $f_{\Phi,2s}^{I}\ast f_{\Phi,2s}^{I}$ and $f_{\Phi
,2s}^{I^{\prime}}\ast f_{\Phi,2s}^{I^{\prime}}$ are disjoint unless
$\operatorname*{dist}\left(  I,I^{\prime}\right)  \lesssim1$, we also have%
\[
\mathcal{E}_{1}=2\sum_{I,I^{\prime}\in\mathcal{G}_{s}\left[  U\right]  }\int
f_{\Phi,2s}^{I}\ast f_{\Phi,2s}^{I}\left(  z\right)  \ f_{\Phi,2s}^{I^{\prime
}}\ast f_{\Phi,2s}^{I^{\prime}}\left(  z\right)  \ dz\lesssim\sum
_{I\in\mathcal{G}_{s}\left[  U\right]  }\int\left\vert f_{\Phi,2s}^{I}\ast
f_{\Phi,2s}^{I}\left(  z\right)  \right\vert ^{2}dz.
\]
Altogether we obtain%
\begin{align*}
& \mathbb{E}_{2^{\mathcal{G}}}^{\mu}\left\Vert \Lambda_{\mathcal{A}%
_{\mathbf{a}}\mathsf{Q}_{U}^{s}}^{2s}f\right\Vert _{L^{4}\left(  \lambda
_{n}\right)  }^{4}\lesssim\sum_{I,J\in\mathcal{G}_{s}\left[  U\right]  }%
\int\left\vert f_{\Phi,2s}^{I}\ast f_{\Phi,2s}^{J}\left(  z\right)
\right\vert ^{2}dz\\
& =\sum_{I,J\in\mathcal{G}_{s}\left[  U\right]  :\ \operatorname*{dist}\left(
I,J\right)  \lesssim2^{-s}}\int\left\vert f_{\Phi,2s}^{I}\ast f_{\Phi,2s}%
^{J}\left(  z\right)  \right\vert ^{2}dz+\sum_{t=0}^{s}\sum_{I,J\in
\mathcal{G}_{s}\left[  U\right]  :\ \operatorname*{dist}\left(  I,J\right)
\approx2^{-t}}\int\left\vert f_{\Phi,2s}^{I}\ast f_{\Phi,2s}^{J}\left(
z\right)  \right\vert ^{2}dz\\
& \equiv\Psi+\sum_{t=0}^{s}\Psi_{t}.
\end{align*}

Now note that the $L^{1}$ norm of $f_{\Phi,2s}^{I}\ast f_{\Phi,2s}^{J}$ is
essentially
\begin{align*}
\left\Vert f_{\Phi,2s}^{I}\right\Vert _{L^{1}}\left\Vert f_{\Phi,2s}%
^{J}\right\Vert _{L^{1}}  & \approx\left\vert \left\langle \left(
S_{\kappa,\eta}\right)  ^{-1}f,h_{I;\kappa}\right\rangle \left\langle \left(
S_{\kappa,\eta}\right)  ^{-1}f,h_{J;\kappa}\right\rangle \right\vert
\left\Vert h_{I}\right\Vert _{L^{1}}\left\Vert h_{J}\right\Vert _{L^{1}}\\
& =\left\vert \left\langle \left(  S_{\kappa,\eta}\right)  ^{-1}f,h_{I;\kappa
}\right\rangle \left\langle \left(  S_{\kappa,\eta}\right)  ^{-1}%
f,h_{J;\kappa}\right\rangle \right\vert 2^{-s\left(  n-1\right)  },
\end{align*}
and since the volume of $R_{2s}\left(  I,J\right)  =\mathcal{I}_{2^{-2s}%
}+\mathcal{J}_{2^{-2s}}$ is essentially $2^{-sn}\operatorname*{dist}\left(
I,J\right)  $, we have%
\[
\left\vert R_{s+t}\left(  I,J\right)  \right\vert \approx\left\vert
R_{2s}\left(  I,J\right)  \right\vert \approx2^{-sn}\operatorname*{dist}%
\left(  I,J\right)  =2^{-sn-t},\ \ \ \ \ \text{for }\operatorname*{dist}%
\left(  I,J\right)  \approx2^{-t},
\]
where the first equivalence is a simple consequence of the geometry of the
situation. Thus we conclude that for $\operatorname*{dist}\left(  I,J\right)
\approx2^{-t}$,
\begin{align*}
\left\Vert f_{\Phi,2s}^{I}\ast f_{\Phi,2s}^{J}\right\Vert _{L^{1}}  &
\lesssim\left\vert \left\langle \left(  S_{\kappa,\eta}\right)  ^{-1}%
f,h_{I;\kappa}\right\rangle \left\langle \left(  S_{\kappa,\eta}\right)
^{-1}f,h_{J;\kappa}\right\rangle \right\vert 2^{-s\left(  n-1\right)  }\\
& \approx\left\Vert \left\vert \left\langle \left(  S_{\kappa,\eta}\right)
^{-1}f,h_{I;\kappa}\right\rangle \left\langle \left(  S_{\kappa,\eta}\right)
^{-1}f,h_{J;\kappa}\right\rangle \right\vert 2^{-s\left(  n-1\right)  }%
\frac{1}{2^{-sn}\operatorname*{dist}\left(  I,J\right)  }\mathbf{1}%
_{R_{2s}\left(  I,J\right)  }\right\Vert _{L^{1}}.
\end{align*}
Since there is $\lambda>0$ and a rectangle $R_{I}$ such that $\left\vert
f_{\Phi,2s}^{I}\right\vert \leq\lambda\mathbf{1}_{R_{I}}$ and $\left\Vert
f_{\Phi,2s}^{I}\right\Vert _{L^{1}}\approx\left\Vert \lambda\mathbf{1}_{R_{I}%
}\right\Vert _{L^{1}}$, which again is a simple consequence of geometry, we
then deduce the comparability of the integrands for $\operatorname*{dist}%
\left(  I,J\right)  \approx2^{-t}$,
\begin{align*}
& f_{\Phi,2s}^{I}\ast f_{\Phi,2s}^{J}\left(  z\right)  \approx\left\vert
\left\langle \left(  S_{\kappa,\eta}\right)  ^{-1}f,h_{I;\kappa}\right\rangle
\left\langle \left(  S_{\kappa,\eta}\right)  ^{-1}f,h_{J;\kappa}\right\rangle
\right\vert 2^{-s\left(  n-1\right)  }\frac{1}{2^{-sn}\operatorname*{dist}%
\left(  I,J\right)  }\mathbf{1}_{R_{2s}\left(  I,J\right)  }\left(  z\right)
\\
& =\left\vert \left\langle \left(  S_{\kappa,\eta}\right)  ^{-1}f,h_{I;\kappa
}\right\rangle \left\langle \left(  S_{\kappa,\eta}\right)  ^{-1}%
f,h_{J;\kappa}\right\rangle \right\vert \frac{2^{s}}{\operatorname*{dist}%
\left(  I,J\right)  }\mathbf{1}_{R_{2s}\left(  I,J\right)  }\left(  z\right)
\\
& =2^{s+t}\left\vert \left\langle \left(  S_{\kappa,\eta}\right)
^{-1}f,h_{I;\kappa}\right\rangle \left\langle \left(  S_{\kappa,\eta}\right)
^{-1}f,h_{J;\kappa}\right\rangle \right\vert \mathbf{1}_{R_{2s}\left(
I,J\right)  }\left(  z\right)  .
\end{align*}
Thus we have%
\begin{align*}
\sum_{t=0}^{s}\Psi_{t}  & \lesssim\sum_{t=0}^{s}\sum_{I,J\in\mathcal{G}%
_{s}\left[  S\right]  :\ \operatorname*{dist}\left(  I,J\right)  \approx
2^{-t}}\int_{\mathbb{R}^{n}}\left\vert f_{\Phi,2s}^{I}\ast f_{\Phi,2s}%
^{J}\left(  z\right)  \right\vert ^{2}dz\\
& \lesssim\sum_{t=0}^{s}\sum_{I,J\in\mathcal{G}_{s}\left[  S\right]
:\ \operatorname*{dist}\left(  I,J\right)  \approx2^{-t}}\int_{\mathbb{R}^{n}%
}\left\vert 2^{s+t}\left\vert \left\langle \left(  S_{\kappa,\eta}\right)
^{-1}f,h_{I;\kappa}\right\rangle \left\langle \left(  S_{\kappa,\eta}\right)
^{-1}f,h_{J;\kappa}\right\rangle \right\vert \mathbf{1}_{R_{2s}\left(
I,J\right)  }\left(  z\right)  \right\vert ^{2}dz\\
& \lesssim\sum_{t=0}^{s}\sum_{I,J\in\mathcal{G}_{s}\left[  S\right]
:\ \operatorname*{dist}\left(  I,J\right)  \approx2^{-t}}2^{2s+2t}\left\vert
\left\langle \left(  S_{\kappa,\eta}\right)  ^{-1}f,h_{I;\kappa}\right\rangle
\left\langle \left(  S_{\kappa,\eta}\right)  ^{-1}f,h_{J;\kappa}\right\rangle
\right\vert ^{2}\left\vert R_{2s}\left(  I,J\right)  \right\vert \\
& \lesssim\sum_{t=0}^{s}\sum_{I,J\in\mathcal{G}_{s}\left[  S\right]
:\ \operatorname*{dist}\left(  I,J\right)  \approx2^{-t}}2^{-s\left(
n-2\right)  }2^{t}\left\vert \left\langle \left(  S_{\kappa,\eta}\right)
^{-1}f,h_{I;\kappa}\right\rangle \left\langle \left(  S_{\kappa,\eta}\right)
^{-1}f,h_{J;\kappa}\right\rangle \right\vert ^{2}\equiv\sum_{t=0}^{s}%
\Omega_{t},
\end{align*}
where we have defined $\Omega_{t}$ to be the bound for $\Psi_{t}$ obtained above.

Now recall that
\[
\left\Vert \left(  \mathsf{Q}_{U}^{s}\right)  ^{\spadesuit}f\right\Vert
_{L^{4}\left(  \lambda_{n-1}\right)  }^{4}\approx2^{s\left(  n-1\right)  }%
\sum_{I\in\mathcal{G}_{s}\left[  U\right]  }\left\langle \left(
S_{\kappa,\eta}\right)  ^{-1}f,h_{I;\kappa}^{n-1}\right\rangle ^{4}.
\]
Thus for $0<t<s$ we have%
\begin{align*}
& \Omega_{t}\lesssim\sum_{I,J\in\mathcal{G}_{s}\left[  U\right]
:\ \operatorname*{dist}\left(  I,J\right)  \approx2^{-t}}2^{-s\left(
n-2\right)  }2^{t}\left\vert \left\langle \left(  S_{\kappa,\eta}\right)
^{-1}f,h_{I;\kappa}\right\rangle \left\langle \left(  S_{\kappa,\eta}\right)
^{-1}f,h_{J;\kappa}\right\rangle \right\vert ^{2}\\
& \lesssim2^{-s\left(  n-2\right)  }2^{t}\sum_{I,J\in\mathcal{G}_{s}\left[
U\right]  :\ \operatorname*{dist}\left(  I,J\right)  \approx2^{-t}}\left\vert
\left\langle \left(  S_{\kappa,\eta}\right)  ^{-1}f,h_{I;\kappa}\right\rangle
\right\vert ^{4}\\
& \lesssim2^{-s\left(  n-2\right)  }2^{t}2^{\left(  s-t\right)  \left(
n-1\right)  }\sum_{I\in\mathcal{G}_{s}\left[  U\right]  }\left\vert
\left\langle \left(  S_{\kappa,\eta}\right)  ^{-1}f,h_{I;\kappa}\right\rangle
\right\vert ^{4}=2^{-t\left(  n-2\right)  }2^{-s\left(  n-2\right)
}\left\Vert \left(  \mathsf{Q}_{U}^{s}\right)  ^{\spadesuit}f\right\Vert
_{L^{4}\left(  S\right)  }^{4}\ ,
\end{align*}
since%
\[
\#\left\{  J\in\mathcal{G}_{s}\left[  S\right]  :\operatorname*{dist}\left(
I,J\right)  \approx2^{-t}\right\}  \approx\frac{\text{volume of annulus}%
}{\text{volume of cube}}\approx\frac{2^{-t\left(  n-1\right)  }}{2^{-s\left(
n-1\right)  }},
\]
which then gives%
\[
\sum_{t=0}^{s}\Psi_{t}\lesssim\sum_{t=0}^{s}\Omega_{t}\lesssim\sum_{t=0}%
^{s}2^{-t\left(  n-2\right)  }2^{-s\left(  n-2\right)  }\left\Vert \left(
\mathsf{Q}_{U}^{s}\right)  ^{\spadesuit}f\right\Vert _{L^{4}\left(  S\right)
}^{4}\approx2^{-s\left(  n-2\right)  }\left\Vert \left(  \mathsf{Q}_{U}%
^{s}\right)  ^{\spadesuit}f\right\Vert _{L^{4}\left(  S\right)  }^{4}.
\]
Similarly we obtain%
\[
\Psi\lesssim2^{-s\left(  n-2\right)  }\left\Vert \left(  \mathsf{Q}_{U}%
^{s}\right)  ^{\spadesuit}f\right\Vert _{L^{4}\left(  S\right)  }^{4},
\]
and adding these results gives
\[
\mathbb{E}_{2^{\mathcal{G}}}^{\mu}\left\Vert \Lambda_{\mathcal{A}_{\mathbf{a}%
}\mathsf{Q}_{U}^{s}}^{2s}f\right\Vert _{L^{4}\left(  \lambda_{n}\right)  }%
^{4}\lesssim2^{-s\left(  n-2\right)  }\left\Vert \left(  \mathsf{Q}_{U}%
^{s}\right)  ^{S_{\kappa,\eta}}f\right\Vert _{L^{4}\left(  \mathbb{R}%
^{n-1}\right)  }^{4}\lesssim2^{-s\left(  n-2\right)  }\left\Vert f\right\Vert
_{L^{4}\left(  \mathbb{R}^{n-1}\right)  }^{4},
\]
which implies the second line in (\ref{assumed}) since $\mu$ is a probability measure.

\section{Control of the $\operatorname{below}$ form}

Combining the principles of decay in Subsection 4.2, and staying the
introduction of absolute values until the very end, we will be able to obtain
estimates on the inner products $\left\langle Th_{I;\kappa}^{n-1,\eta
},h_{J;\kappa}^{n,\eta}\right\rangle $, which will lead to the following form
bounds for some fixed $\delta>0$ depending only on $n$ and $p $,%
\[
\left\vert \mathsf{B}_{\operatorname{below}}^{k,d}\left(  f,g\right)
\right\vert \lesssim2^{-\delta\left(  \left\vert d\right\vert +\left\vert
k\right\vert \right)  }\left\Vert f\right\Vert _{L^{p}}\left\Vert g\right\Vert
_{L^{p^{\prime}}}\ ,\ \ \ \ \ \text{for }p>\frac{2n}{n-1}.
\]
In fact we obtain stronger bounds in which the absolute values are inside the
sum. Indeed, if we define%
\[
\left\vert \mathsf{B}_{\operatorname{below}}\right\vert \left(  f,g\right)
\equiv\sum_{\left(  I,J\right)  \in\mathcal{P}_{0}}\left\vert \left\langle
T\bigtriangleup_{I;\kappa}^{n-1,\eta}f,\bigtriangleup_{J;\kappa}^{n,\eta
}g\right\rangle \right\vert ,
\]
we prove in this section that%
\begin{equation}
\left\vert \mathsf{B}_{\operatorname{below}}\right\vert \left(  f,g\right)
\lesssim\left\Vert f\right\Vert _{L^{p}}\left\Vert g\right\Vert _{L^{p^{\prime
}}}\ ,\ \ \ \ \ \text{for }p>\frac{2n}{n-1}.\label{strong below}%
\end{equation}
We will begin with the two easier cases involving $d\leq0$, since each of
these cases requires just one of the decay principles described above.

Later we turn to the subforms involving $d\geq0$, which are harder to control
as each of them requires combining two of the decay principles described above.

\begin{remark}
The next result shows in particular that the basic form $\mathsf{B}%
_{\operatorname{below}}^{0,0}\left(  f,g\right)  $ is bounded using only the
crude estimate (\ref{crude''}), and the strict restriction to $p>\frac
{2n}{n-1}$. See also the Direct Argument in Subsubsection \ref{subsub direct}
for a much shorter proof of essentially the same result.
\end{remark}

\subsection{Subforms with $k\geq0,d\leq0$}

Here is the conclusion of this first subsection.

\begin{lemma}
\label{Kak below'}Fix $s\in\mathbb{N}$. Then%
\begin{equation}
\sum_{k\geq0}\sum_{d\leq0}\left\vert \mathsf{B}_{\operatorname{below}}%
^{k,d}\left(  f,g\right)  \right\vert \leq\sum_{k\geq0}\sum_{d\leq0}%
\sum_{\left(  I,J\right)  \in\mathcal{P}_{0}^{k,d}}\left\vert \left\langle
T\bigtriangleup_{I;\kappa}^{n-1,\eta}f,\bigtriangleup_{J;\kappa}^{n,\eta
}g\right\rangle \right\vert \lesssim\left\Vert f\right\Vert _{L^{p}}\left\Vert
g\right\Vert _{L^{p^{\prime}}}\ ,\ \ \ \ \ \text{for }p\geq\frac{2n}%
{n-1}.\label{Kak below}%
\end{equation}

\end{lemma}

To prove Lemma \ref{Kak below'}, we just need the estimate (\ref{crude k})
that used radial integration by parts, namely,%
\[
\left\vert \left\langle Th_{I;\kappa}^{n-1,\eta},h_{J;\kappa}^{n,\eta
}\right\rangle \right\vert \leq C_{N}2^{-kN}\left\Vert h_{I;\kappa}^{n-1,\eta
}\right\Vert _{L^{1}}\left\Vert h_{J;\kappa}^{n,\eta}\right\Vert _{L^{1}%
}\approx2^{-kN}\sqrt{\left\vert I\right\vert \left\vert J\right\vert
},\ \ \ \ \ k\geq0.
\]
Let $I_{\eta}\equiv\left(  1+\eta\right)  I$ so that $\operatorname*{Supp}%
\bigtriangleup_{I;\kappa}^{n-1,\eta}f\subset I_{\eta}$. Note also that
$\left\vert I_{\eta}\right\vert \approx\left\vert I\right\vert $. Then we have
from (\ref{crude k}),%
\begin{align*}
& \left\vert \mathsf{B}_{\operatorname{below}}^{k,d}\left(  f,g\right)
\right\vert \leq\sum_{\left(  I,J\right)  \in\mathcal{P}_{0}^{k,d}}\left\vert
\left\langle T\bigtriangleup_{I;\kappa}^{n-1,\eta}f,\bigtriangleup_{J;\kappa
}^{n,\eta}g\right\rangle \right\vert \leq\sum_{\left(  I,J\right)
\in\mathcal{P}_{0}^{k,d}}2^{-kN}\left(  \int_{I_{\eta}}\left\vert
\bigtriangleup_{I;\kappa}^{n-1,\eta}f\left(  x\right)  \right\vert dx\right)
\left(  \int_{J_{\eta}}\left\vert \bigtriangleup_{J;\kappa}^{n,\eta}g\left(
\xi\right)  \right\vert d\xi\right) \\
& =2^{-kN}\int_{\mathbb{R}^{n}}\sum_{\left(  I,J\right)  \in\mathcal{P}%
_{0}^{k,d}}\left(  \int_{I_{\eta}}\left\vert \bigtriangleup_{I;\kappa
}^{n-1,\eta}f\left(  x\right)  \right\vert dx\right)  \mathbf{1}_{J_{\eta}%
}\left(  \xi\right)  \left\vert \bigtriangleup_{J;\kappa}^{n,\eta}g\left(
\xi\right)  \right\vert d\xi\\
& \leq2^{-kN}\int_{\mathbb{R}^{n}}\sqrt{\sum_{\left(  I,J\right)
\in\mathcal{P}_{0}^{k,d}}\left(  \int_{I_{\eta}}\left\vert \bigtriangleup
_{I;\kappa}^{n-1,\eta}f\left(  x\right)  \right\vert dx\mathbf{1}_{J_{\eta}%
}\left(  \xi\right)  \right)  ^{2}}\sqrt{\sum_{\left(  I,J\right)
\in\mathcal{P}_{0}^{k,d}}\left\vert \bigtriangleup_{J;\kappa}^{n,\eta}g\left(
\xi\right)  \right\vert ^{2}}d\xi\\
& \lesssim2^{-kN}\left(  \int_{\mathbb{R}^{n}}\left(  \sum_{\left(
I,J\right)  \in\mathcal{P}_{0}^{k,d}}\left(  \int_{I}\left\vert \bigtriangleup
_{I;\kappa}^{n-1,\eta}f\left(  x\right)  \right\vert dx\mathbf{1}_{J_{\eta}%
}\left(  \xi\right)  \right)  ^{2}\right)  ^{\frac{p}{2}}d\xi\right)
^{\frac{1}{p}}\left(  \int_{\mathbb{R}^{n}}\left(  \sum_{\left(  I,J\right)
\in\mathcal{P}_{0}^{k,d}}\left\vert \bigtriangleup_{J;\kappa}^{n,\eta}g\left(
\xi\right)  \right\vert ^{2}\right)  ^{\frac{p^{\prime}}{2}}d\xi\right)
^{\frac{1}{p^{\prime}}}\\
& \equiv2^{-kN}\Gamma_{1}\Gamma_{2}\ ,
\end{align*}
where%
\[
\Gamma_{2}^{p^{\prime}}=\int_{\mathbb{R}^{n}}\left(  \sum_{\left(  I,J\right)
\in\mathcal{P}_{0}^{k,d}}\left\vert \bigtriangleup_{J;\kappa}^{n,\eta}g\left(
\xi\right)  \right\vert ^{2}\right)  ^{\frac{p^{\prime}}{2}}d\xi
=\int_{\mathbb{R}^{n}}\left(  \sum_{J\in\mathcal{D}}\left(  \sum
_{I\in\mathcal{G}:\ \left(  I,J\right)  \in\mathcal{P}_{0}^{k,d}}1\right)
\left\vert \bigtriangleup_{J;\kappa}^{n,\eta}g\left(  \xi\right)  \right\vert
^{2}\right)  ^{\frac{p^{\prime}}{2}}d\xi.
\]

We now choose a dyadic cube $I_{J}\in\mathcal{G}$ that approximates the
spherical projection $\pi_{\tan}\left(  J\right)  $ of $J$. So fix
$J\in\mathcal{D}$ and let $I_{J}\in\mathcal{G}$ satisfy%
\[
c_{n}\ell\left(  \pi_{\tan}\left(  J\right)  \right)  \leq\ell\left(
I_{J}\right)  \leq\ell\left(  \pi_{\tan}\left(  J\right)  \right)  \text{ and
}I_{J}\subset\pi_{\tan}\left(  J\right)  ,
\]
where $\pi_{\tan}\left(  J\right)  $\ is the spherical projection $J$ onto
$\mathbb{S}^{n-1}$, and where $c_{n}>0$ is chosen small enough that such a
cube $I_{J}$ exists.

Now $\left(  I,J\right)  \in\mathcal{P}_{0}^{k,d}$ if and only if%
\[
\pi_{\tan}J\subset\Phi\left(  C_{\operatorname{pseudo}}I\right)  \text{ and
}\frac{2^{d-1}}{\ell\left(  I\right)  ^{2}}\leq\operatorname*{dist}\left(
0,J\right)  \leq\frac{2^{d+1}}{\ell\left(  I\right)  ^{2}},
\]
which is essentially equivalent to%
\[
I\supset\pi_{\tan}J\supset I_{J}\text{ and }\sqrt{\frac{2^{d-1}}%
{2\operatorname*{dist}\left(  0,J\right)  }}\leq\ell\left(  I\right)
\leq\sqrt{\frac{2^{d+1}}{\operatorname*{dist}\left(  0,J\right)  }}.
\]
Thus for fixed $J\in\mathcal{D}_{k}$ where%
\[
\mathcal{D}_{k}\equiv\left\{  J\in\mathcal{D}:\ell\left(  J\right)
=2^{k}\right\}  ,
\]
the set of cubes $I\in\mathcal{G}$ with $\left(  I,J\right)  \in
\mathcal{P}_{0}^{k,d}$ is contained in the finite tower of dyadic cubes
$\left\{  \pi^{\left(  k\right)  }I_{J}\right\}  _{k=d-A}^{d+A}$ for some
fixed $A\in\mathbb{N}$. It follows that $\sum_{I\in\mathcal{G}:\ \left(
I,J\right)  \in\mathcal{P}_{0}^{k,d}}1\leq2A$ and so%
\begin{equation}
\Gamma_{2}^{p^{\prime}}=\int_{\mathbb{R}^{n}}\left(  \sum_{\left(  I,J\right)
\in\mathcal{P}_{0}^{k,d}}\left\vert \bigtriangleup_{J;\kappa}^{n,\eta}g\left(
\xi\right)  \right\vert ^{2}\right)  ^{\frac{p^{\prime}}{2}}d\xi\leq
\int_{\mathbb{R}^{n}}\left(  \sum_{J\in\mathcal{D}_{k}}2A\left\vert
\bigtriangleup_{J;\kappa}^{n,\eta}g\left(  \xi\right)  \right\vert
^{2}\right)  ^{\frac{p^{\prime}}{2}}d\xi\lesssim\left\Vert g\right\Vert
_{L^{p^{\prime}}}^{p^{\prime}}\ ,\label{Gamma 2}%
\end{equation}
by the Alpert square function estimate (\ref{squ est}).

We turn now to estimating $\Gamma_{1}$. Since the cubes $J_{\eta}$ in
$\mathcal{D}_{k}$ have bounded overlap with measure roughly $2^{kn}$,%
\begin{align}
\Gamma_{1}^{p}  & =\int_{\mathbb{R}^{n}}\left(  \sum_{\left(  I,J\right)
\in\mathcal{P}_{0}^{k,d}}\left(  \int_{I_{\eta}}\left\vert \bigtriangleup
_{I;\kappa}^{n-1,\eta}f\left(  x\right)  \right\vert dx\mathbf{1}_{J_{\eta}%
}\left(  \xi\right)  \right)  ^{2}\right)  ^{\frac{p}{2}}d\xi\label{Gamma 1}\\
& =\int_{\mathbb{R}^{n}}\left(  \sum_{J\in\mathcal{D}_{k}}\left\{  \sum
_{I\in\mathcal{G}\left[  S\right]  :\ \left(  I,J\right)  \in\mathcal{P}%
_{0}^{k,d}}\left(  \int_{I_{\eta}}\left\vert \bigtriangleup_{I;\kappa
}^{n-1,\eta}f\left(  x\right)  \right\vert dx\right)  ^{2}\right\}
\mathbf{1}_{J_{\eta}}\left(  \xi\right)  \right)  ^{\frac{p}{2}}%
d\xi\nonumber\\
& \approx\int_{\mathbb{R}^{n}}\sum_{J\in\mathcal{D}_{k}}\left\{  \sum
_{I\in\mathcal{G}\left[  S\right]  :\ \left(  I,J\right)  \in\mathcal{P}%
_{0}^{k,d}}\left(  \int_{I_{\eta}}\left\vert \bigtriangleup_{I;\kappa
}^{n-1,\eta}f\left(  x\right)  \right\vert dx\right)  ^{2}\right\}  ^{\frac
{p}{2}}\mathbf{1}_{J_{\eta}}\left(  \xi\right)  d\xi\nonumber\\
& \approx2^{kn}\sum_{J\in\mathcal{D}_{k}}\left(  \sum_{I\in\mathcal{G}\left[
S\right]  :\ \left(  I,J\right)  \in\mathcal{P}_{0}^{k,d}}\left(
\int_{I_{\eta}}\left\vert \bigtriangleup_{I;\kappa}^{n-1,\eta}f\left(
x\right)  \right\vert dx\right)  ^{2}\right)  ^{\frac{p}{2}}.\nonumber
\end{align}
Now for each fixed $J\in\mathcal{D}_{k}$ and $I\in\mathcal{G}\left[  S\right]
$ with $\left(  I,J\right)  \in\mathcal{P}_{0}^{k,d}$, we have%
\begin{align*}
\ell\left(  J\right)   & =2^{k},\ \ \ \ell\left(  I\right)  ^{2}%
\operatorname*{dist}\left(  0,J\right)  \approx2^{d},\ \ \ \pi_{\tan}%
J\subset\Phi\left(  C_{\operatorname{pseudo}}I\right)  ,\\
\ell\left(  I_{J}\right)   & \approx\ell\left(  \pi_{\tan}J\right)
\approx\frac{\ell\left(  J\right)  }{\operatorname*{dist}\left(  0,J\right)
}=\frac{2^{k}}{\operatorname*{dist}\left(  0,J\right)  },
\end{align*}
which implies
\begin{align*}
\ell\left(  I\right)   & \approx\sqrt{\frac{2^{d}}{\operatorname*{dist}\left(
0,J\right)  }}\approx\sqrt{\frac{2^{d}\ell\left(  \pi_{\tan}J\right)  }{2^{k}%
}}=2^{\frac{d-k}{2}}\sqrt{\ell\left(  I_{J}\right)  },\\
\log_{2}\frac{\ell\left(  I\right)  }{\ell\left(  I_{J}\right)  }  &
\approx\log_{2}\frac{2^{\frac{d-k}{2}}}{\sqrt{\ell\left(  I_{J}\right)  }%
}\approx\frac{1}{2}\left(  d-k-\log_{2}\frac{1}{\ell\left(  I_{J}\right)
}\right)  .
\end{align*}
Thus with $d^{\ast}\equiv\frac{1}{2}\left(  d-k-\log_{2}\frac{1}{\ell\left(
I_{J}\right)  }\right)  $ and $A$ as in (\ref{Gamma 2}) above, we have for
each $J\in\mathcal{D}$,%
\begin{align*}
& \left(  \sum_{I\in\mathcal{G}\left[  S\right]  :\ \left(  I,J\right)
\in\mathcal{P}_{0}^{k,d}}\left(  \int_{I_{\eta}}\left\vert \bigtriangleup
_{I;\kappa}^{n-1,\eta}f\left(  x\right)  \right\vert dx\right)  ^{2}\right)
^{\frac{p}{2}}\leq\left(  \sum_{s=d^{\ast}-A}^{d^{\ast}+A}\left(  \int
_{\pi^{\left(  m\right)  }\left(  I_{J}\right)  _{\eta}}\left\vert
\bigtriangleup_{\pi^{\left(  s\right)  }\left(  I_{J}\right)  ;\kappa
}^{n-1,\eta}f\left(  x\right)  \right\vert dx\right)  ^{2}\right)  ^{\frac
{p}{2}}\\
& \leq\left(  2A\right)  ^{\frac{p}{2}-1}\sum_{I\in\mathcal{G}\left[
S\right]  :\ \left(  I,J\right)  \in\mathcal{P}_{0}^{k,d}}\left(
\int_{I_{\eta}}\left\vert \bigtriangleup_{I;\kappa}^{n-1,\eta}f\left(
x\right)  \right\vert dx\right)  ^{p}\approx\sum_{I\in\mathcal{G}\left[
S\right]  :\ \left(  I,J\right)  \in\mathcal{P}_{0}^{k,d}}\left(
\int_{I_{\eta}}\left\vert \bigtriangleup_{I;\kappa}^{n-1,\eta}f\left(
x\right)  \right\vert dx\right)  ^{p}.
\end{align*}

Altogether then,
\begin{align}
\Gamma_{1}^{p}  & \lesssim2^{kn}\sum_{J\in\mathcal{D}_{k}}\sum_{I\in
\mathcal{G}\left[  S\right]  :\ \left(  I,J\right)  \in\mathcal{P}_{0}^{k,d}%
}\left(  \int_{I_{\eta}}\left\vert \bigtriangleup_{I;\kappa}^{n-1,\eta
}f\left(  x\right)  \right\vert dx\right)  ^{p}\label{Gamma 1'}\\
& \leq2^{kn}\sum_{J\in\mathcal{D}_{k}}\sum_{I\in\mathcal{G}\left[  S\right]
:\ \left(  I,J\right)  \in\mathcal{P}_{0}^{k,d}}\left\vert I\right\vert
^{\frac{p}{2}}\left(  \int_{I_{\eta}}\left\vert \bigtriangleup_{I;\kappa
}^{n-1,\eta}f\left(  x\right)  \right\vert ^{2}dx\right)  ^{\frac{p}{2}%
}\nonumber\\
& \approx2^{kn}\sum_{I\in\mathcal{G}\left[  S\right]  }\left(  \sum
_{J\in\mathcal{D}_{k}:\ \left(  I,J\right)  \in\mathcal{P}_{0}^{k,d}}1\right)
\left\vert I\right\vert ^{p}\left(  \frac{1}{\left\vert I_{\eta}\right\vert
}\int_{I_{\eta}}\left\vert \bigtriangleup_{I;\kappa}^{n-1,\eta}f\left(
x\right)  \right\vert ^{2}dx\right)  ^{\frac{p}{2}}.\nonumber
\end{align}
Now recall that $\mathcal{P}_{0}\equiv\left\{  \left(  I,J\right)
\in\mathcal{G}\left[  S\right]  \times\mathcal{D}:\pi_{\tan}\left(  J\right)
\subset\Phi\left(  C_{\operatorname{pseudo}}I\right)  \right\}  $, and define%
\[
\mathcal{K}\left(  I\right)  \equiv%
%TCIMACRO{\dbigcup }%
%BeginExpansion
{\displaystyle\bigcup}
%EndExpansion
\left\{  J\in\mathcal{D}:\pi_{\tan}\left(  J\right)  \subset\Phi\left(
C_{\operatorname{pseudo}}I\right)  \right\}  .
\]
Now for fixed $I\in\mathcal{G}\left[  S\right]  $,
\begin{align}
& \#\left\{  J\in\mathcal{D}_{k}:\ \left(  I,J\right)  \in\mathcal{P}%
_{0}^{k,d}\right\} \label{card m=0}\\
& \approx2^{-kn}\left\vert \mathcal{K}_{d}\left(  I\right)  \right\vert
\approx2^{-kn}\left(  \frac{2^{d}}{\ell\left(  I\right)  ^{2}}\ell\left(
I\right)  \right)  ^{n-1}\frac{2^{d}}{\ell\left(  I\right)  ^{2}}\nonumber\\
& =2^{-kn}\frac{2^{dn}}{\ell\left(  I\right)  ^{n+1}}=2^{-kn}2^{dn}\left(
\frac{1}{\left\vert I\right\vert }\right)  ^{\frac{n+1}{n-1}};\nonumber\\
& \text{where }\mathcal{K}_{d}\left(  I\right)  \equiv%
%TCIMACRO{\dbigcup }%
%BeginExpansion
{\displaystyle\bigcup}
%EndExpansion
\left\{  J\subset\mathcal{K}\left(  I\right)  :\frac{2^{d-1}}{\ell\left(
I\right)  ^{2}}\leq\operatorname*{dist}\left(  0,J\right)  \leq\frac{2^{d+1}%
}{\ell\left(  I\right)  ^{2}}\right\}  ,\nonumber
\end{align}
and so we have%
\begin{align*}
\Gamma_{1}^{p}  & \lesssim2^{kn}\sum_{I\in\mathcal{G}\left[  U\right]
}\left(  \#\left\{  J\in\mathcal{D}_{k}:\ \left(  I,J\right)  \in
\mathcal{P}_{0}^{k,d}\right\}  \right)  \left\vert I\right\vert ^{p}\left(
\frac{1}{\left\vert I_{\eta}\right\vert }\int_{I_{\eta}}\left\vert
\bigtriangleup_{I;\kappa}^{n-1,\eta}f\right\vert ^{2}\right)  ^{\frac{p}{2}}\\
& \lesssim2^{kn}2^{-kn}2^{dn}\sum_{I\in\mathcal{G}\left[  U\right]
}\left\vert I\right\vert ^{p-\frac{n+1}{n-1}}\left(  \frac{1}{\left\vert
I_{\eta}\right\vert }\int_{I_{\eta}}\left\vert \bigtriangleup_{I;\kappa
}^{n-1,\eta}f\right\vert ^{2}\right)  ^{\frac{p}{2}}\\
& =2^{dn}\int_{S}\sum_{I\in\mathcal{G}\left[  U\right]  }\left\vert
I\right\vert ^{p-\frac{n+1}{n-1}-1}\left(  \frac{1}{\left\vert I_{\eta
}\right\vert }\int_{I_{\eta}}\left\vert \bigtriangleup_{I;\kappa}^{n-1,\eta
}f\right\vert ^{2}\right)  ^{\frac{p}{2}}\mathbf{1}_{I}\left(  x\right)  dx\\
& \leq2^{dn}\int_{S}\sum_{I\in\mathcal{G}\left[  U\right]  }\left(  \frac
{1}{\left\vert I_{\eta}\right\vert }\int_{I_{\eta}}\left\vert \bigtriangleup
_{I;\kappa}^{n-1,\eta}f\right\vert ^{2}\mathbf{1}_{I}\left(  x\right)
\right)  ^{\frac{p}{2}}dx\ ,
\end{align*}
if $p\geq\frac{2n}{n-1}$. Now using H\"{o}lder's inequality with $\frac{p}%
{2}>1$, and the Fefferman Stein vector valued maximal inequality,we can
continue with%
\begin{align}
\Gamma_{1}^{p}  & \lesssim2^{dn}\int_{S}\left(  \sum_{I\in\mathcal{G}\left[
U\right]  }\frac{1}{\left\vert I_{\eta}\right\vert }\int_{I_{\eta}}\left\vert
\bigtriangleup_{I;\kappa}^{n-1,\eta}f\right\vert ^{2}\mathbf{1}_{I}\left(
x\right)  \right)  ^{\frac{p}{2}}dx\lesssim2^{dn}\int_{S}\left(  \sum
_{I\in\mathcal{G}\left[  U\right]  }\left(  M\left\vert \bigtriangleup
_{I;\kappa}^{n-1,\eta}f\right\vert ^{2}\right)  \left(  x\right)  \right)
^{\frac{p}{2}}dx\label{Gamma 1''}\\
& \lesssim2^{dn}\int_{S}\left(  \sum_{I\in\mathcal{G}\left[  U\right]
}\left\vert \bigtriangleup_{I;\kappa}^{n-1,\eta}f\right\vert ^{2}\left(
x\right)  \right)  ^{\frac{p}{2}}dx\lesssim2^{dn}\left\Vert f\right\Vert
_{L^{p}}^{p}\ ,\nonumber
\end{align}
by the Alpert square function estimate (\ref{squ est}). Thus we have proved,%
\[
\left\vert \mathsf{B}_{\operatorname{below}}^{k,d}\left(  f,g\right)
\right\vert \lesssim2^{-kN}2^{\frac{dn}{p}}\left\Vert f\right\Vert _{L^{p}%
}\left\Vert g\right\Vert _{L^{p^{\prime}}}\ ,\ \ \ \ \ \text{for }k\geq0\text{
and }d\leq0,
\]
which gives%
\[
\sum_{k\geq0}\sum_{d\leq0}\left\vert \mathsf{B}_{\operatorname{below}}%
^{k,d}\left(  f,g\right)  \right\vert \lesssim\left\Vert f\right\Vert _{L^{p}%
}\left\Vert g\right\Vert _{L^{p^{\prime}}}\ ,\ \ \ \ \ \text{for }p\geq
\frac{2n}{n-1}.
\]

\subsection{Subforms with $k\leq0,d\leq0$}

This case also requires just one principle of decay, but this time we use the
moment vanishing decay principle instead of the radial integration by parts
decay principle. From (\ref{van mom formula}) we have
\[
\left\langle Th_{I;\kappa}^{n-1,\eta},h_{J;\kappa}^{n,\eta}\right\rangle
=\int_{S}e^{-i\Phi\left(  x\right)  \cdot c_{J}}h_{I;\kappa}^{n-1,\eta}\left(
x\right)  \left\{  \int_{\mathbb{R}^{n}}R_{\kappa}\left(  -i\Phi\left(
x\right)  \cdot\left(  \xi-c_{J}\right)  \right)  h_{J;\kappa}^{n,\eta}\left(
\xi\right)  d\xi\right\}  dx,
\]
and then from (\ref{van mom est}), we obtain the estimate,%
\begin{align*}
\left\vert \left\langle Th_{I;\kappa}^{n-1,\eta},h_{J;\kappa}^{n,\eta
}\right\rangle \right\vert  & \leq\int_{S}\left\vert h_{I;\kappa}^{n-1,\eta
}\left(  x\right)  \right\vert \int_{\mathbb{R}^{n}}\frac{\left\vert
\Phi\left(  x\right)  \cdot\left(  \xi-c_{J}\right)  \right\vert ^{\kappa}%
}{\left(  \kappa+1\right)  !}\left\vert h_{J;\kappa}^{n,\eta}\left(
\xi\right)  \right\vert d\xi dx\\
& \lesssim\ell\left(  J\right)  ^{\kappa}\left\Vert \varphi_{I}^{\eta
}\right\Vert _{L^{1}}\left\Vert \psi_{J}^{\eta}\right\Vert _{L^{1}}%
\approx2^{-\left\vert k\right\vert \kappa}\sqrt{\left\vert I\right\vert
\left\vert J\right\vert }.
\end{align*}
The proof is now virtually the same as that in the previous subsection, but
using the above estimate instead, and results in the bound,%
\[
\left\vert \mathsf{B}_{\operatorname{below}}^{k,d}\left(  f,g\right)
\right\vert \lesssim2^{-\left\vert k\right\vert \kappa}2^{\frac{dn}{p}%
}\left\Vert f\right\Vert _{L^{p}}\left\Vert g\right\Vert _{L^{p^{\prime}}%
}\ ,\ \ \ \ \ \text{for }k\leq0\text{ and }d\leq0,
\]
which gives%
\[
\sum_{k\leq0}\sum_{d\leq0}\left\vert \mathsf{B}_{\operatorname{below}}%
^{k,d}\left(  f,g\right)  \right\vert \lesssim\left\Vert f\right\Vert _{L^{p}%
}\left\Vert g\right\Vert _{L^{p^{\prime}}}\ ,\ \ \ \ \ \text{for }p\geq
\frac{2n}{n-1}.
\]

\subsection{Subforms with $k\leq0,d\geq0$}

Here we will use the vanishing moments of $h_{J;\kappa}^{n,\eta}$ together
with stationary phase. In the case $k\leq0$ and $d\geq0$, we have from
(\ref{van mom formula}), which used the vanishing moments of$\ h_{J;\kappa
}^{n,\eta}$,%
\[
\left\langle Th_{I;\kappa}^{n-1,\eta},h_{J;\kappa}^{n,\eta}\right\rangle
=\int_{S}e^{-i\Phi\left(  x\right)  \cdot c_{J}}h_{I;\kappa}^{n-1,\eta}\left(
x\right)  \left\{  \int_{\mathbb{R}^{n}}R_{\kappa}\left(  -i\Phi\left(
x\right)  \cdot\left(  \xi-c_{J}\right)  \right)  h_{J;\kappa}^{n,\eta}\left(
\xi\right)  d\xi\right\}  dx,
\]
and using the change of variable $\xi\rightarrow\left(  y,\lambda\right)  $ in
(\ref{phi and psi notation}) with $\frac{c_{J}}{\left\vert c_{J}\right\vert
}=\Phi\left(  y_{J}\right)  $, this can be written,%
\begin{align}
& \ \ \ \ \ \ \ \ \ \ \ \ \ \ \ \left\langle Th_{I;\kappa}^{n-1,\eta
},h_{J;\kappa}^{n,\eta}\right\rangle \label{rec}\\
& =\int_{\mathbb{R}^{n}}\left\{  \int_{S}e^{-i\lambda\phi\left(
x,y_{J}\right)  }h_{I;\kappa}^{n-1,\eta}\left(  x\right)  R_{\kappa}\left(
-i\lambda\Phi\left(  x\right)  \cdot\left(  \Phi\left(  y\right)
-\frac{\left\vert c_{J}\right\vert }{\lambda}\Phi\left(  y_{J}\right)
\right)  \right)  dx\right\}  h_{J;\kappa}^{n,\eta}\left(  \lambda\Phi\left(
y\right)  \right)  \frac{dy}{\sqrt{1-\left\vert y\right\vert ^{2}}}%
\lambda^{n-1}d\lambda\nonumber\\
& =\int_{\mathbb{R}^{n}}\mathcal{I}_{\overset{\frown}{\varphi_{I}^{\eta}}%
,\phi}\left(  y_{J},\lambda\right)  h_{J;\kappa}^{n,\eta}\left(  \lambda
\Phi\left(  y\right)  \right)  \frac{dy}{\sqrt{1-\left\vert y\right\vert ^{2}%
}}\lambda^{n-1}d\lambda,\nonumber
\end{align}
where%
\[
\mathcal{I}_{\overset{\frown}{\varphi_{I}^{\eta}},\phi}\left(  y_{J}%
,\lambda\right)  =\int_{S}e^{-i\lambda\phi\left(  x,y_{J}\right)  }%
\overset{\frown}{\varphi_{I}^{\eta}}\left(  x,y,y_{J}\right)  dx,
\]
and%
\[
\overset{\frown}{\varphi_{I}^{\eta}}\left(  x,y_{J},y\right)  \equiv
h_{I;\kappa}^{n-1,\eta}\left(  x\right)  R_{\kappa}\left(  -i\lambda
\Phi\left(  x\right)  \cdot\left(  \Phi\left(  y\right)  -\frac{\left\vert
c_{J}\right\vert }{\lambda}\Phi\left(  y_{J}\right)  \right)  \right)
=h_{I;\kappa}^{n-1,\eta}\left(  x\right)  R_{\kappa}\left(  -i\Phi\left(
x\right)  \cdot\left(  \xi-c_{J}\right)  \right)  ,
\]
where $R_{\kappa}$ satisfies the estimates,%
\begin{align}
\left\vert R_{\kappa}\left(  ib\right)  \right\vert  & =\left\vert \int
_{0}^{1}e^{itb}\left(  ib\right)  ^{\kappa}\frac{\left(  1-t\right)  ^{\kappa
}}{\kappa!}dt\right\vert \lesssim\frac{\left\vert b\right\vert ^{\kappa}%
}{\kappa!},\label{remd}\\
\left\vert R_{\kappa}^{\left(  \ell\right)  }\left(  b\right)  \right\vert  &
=\left\vert \int_{0}^{1}\partial_{b}^{\ell}\left[  e^{itb}\left(  ib\right)
^{\kappa}\right]  \frac{\left(  1-t\right)  ^{\kappa}}{\kappa!}dt\right\vert
\lesssim\max\left\{  \left\vert b\right\vert ^{\kappa-\ell},\left\vert
b\right\vert ^{\kappa}\right\}  ,\nonumber
\end{align}
and $y_{J}$ is the unique point in $S$ such that $\frac{c_{J}}{\left\vert
c_{J}\right\vert }=\Phi\left(  y_{J}\right)  $.

Theorem \ref{osc int} with $M=0$ gives the asymptotic expansion,%
\[
\mathcal{I}_{\overset{\frown}{\varphi_{I}^{\eta}},\phi}\left(  y_{J}%
,\lambda\right)  =\mathfrak{P}_{\overset{\frown}{\varphi_{I}^{\eta}},\phi
}\left(  y_{J},\lambda\right)  +\mathfrak{R}_{\overset{\frown}{\varphi
_{I}^{\eta}},\phi}^{\left(  1\right)  }\left(  y_{J},\lambda\right)  ,
\]
where%
\[
\mathfrak{P}_{\overset{\frown}{\varphi_{I}^{\eta}},\phi}\left(  y_{J}%
,\lambda\right)  =\left(  \frac{2\pi}{\lambda}\right)  ^{\frac{n}{2}}%
\frac{e^{i\operatorname{sgn}\left[  \partial_{x}^{2}\phi\left(  X\left(
y_{J}\right)  ,y_{J}\right)  \right]  \frac{\pi}{4}+\lambda\phi\left(
X\left(  y_{J}\right)  ,y_{J}\right)  }}{\sqrt{\left\vert \partial_{x}^{2}%
\phi\left(  X\left(  y_{J}\right)  ,y_{J}\right)  \right\vert }}%
\overset{\frown}{\varphi_{I}^{\eta}}\left(  X\left(  y_{J}\right)
,y_{J},y\right)  ,
\]
and%
\begin{align}
\mathfrak{R}_{\overset{\frown}{\varphi_{I}^{\eta}},\phi}^{\left(  1\right)
}\left(  y_{J},\lambda\right)   & =\left(  \frac{2\pi}{\lambda}\right)
^{\frac{n}{2}}\frac{e^{i\left[  \operatorname{sgn}B\left(  y_{J}\right)
\frac{\pi}{4}+\lambda\phi\left(  X\left(  y_{J}\right)  ,y_{J}\right)
\right]  }}{\sqrt{\det B\left(  y_{J}\right)  }}\nonumber\\
& \times\int\mathcal{F}_{z}^{-1}\left(  \left[  \frac{\left\langle
i\partial_{z},B\left(  y_{J}\right)  ^{-1}\partial_{z}\right\rangle }%
{2\lambda}\right]  ^{1}f\right)  \left(  \zeta\right)  R_{1}\left(
-i\frac{\zeta^{\operatorname*{tr}}B\left(  y_{J}\right)  ^{-1}\zeta}{2\lambda
}\right)  d\zeta,\nonumber
\end{align}
and where%
\[
R_{1}\left(  ib\right)  =\int_{0}^{1}e^{itb}\left(  ib\right)  ^{1}%
\frac{\left(  1-t\right)  ^{1}}{\left(  M+1\right)  !}dt,\ \ \ \ \ \text{ for
}b\in\mathbb{R},
\]
and%
\begin{equation}
f\left(  z,y_{J},y\right)  \equiv\frac{\overset{\frown}{\varphi_{I}^{\eta}%
}\left(  \Psi_{y}^{-1}\left(  z\right)  ,y_{J},y\right)  }{\det\left[  \left(
\partial_{x}\Psi\right)  \left(  \Psi_{y}^{-1}\left(  z\right)  \right)
\right]  }.\label{fphI}%
\end{equation}

We can rewrite the principal term as%
\begin{align*}
\mathfrak{P}_{\overset{\frown}{\varphi_{I}^{\eta}},\phi}\left(  y_{J}%
,\lambda\right)   & =\left(  \frac{2\pi}{\lambda}\right)  ^{\frac{n-1}{2}%
}\frac{e^{i\operatorname{sgn}\left[  \partial_{x}^{2}\phi\left(  X\left(
y_{J}\right)  ,y_{J}\right)  \right]  \frac{\pi}{4}+\lambda\phi\left(
X\left(  y_{J}\right)  ,y_{J}\right)  }}{\sqrt{\left\vert \det B\left(
y_{J}\right)  \right\vert }}\overset{\frown}{\varphi_{I}^{\eta}}\left(
X\left(  y\right)  ,y_{J},y\right) \\
& =\left(  \frac{2\pi}{\lambda}\right)  ^{\frac{n-1}{2}}e^{i\frac{\left(
n-1\right)  \pi}{4}+\lambda}\sqrt{1-\left\vert y_{J}\right\vert ^{2}}%
\overset{\frown}{\varphi_{I}^{\eta}}\left(  y_{J},y_{J},y\right) \\
& =e^{-\frac{\left(  n-1\right)  \pi}{4}}e^{i\left\vert \xi\right\vert
}\left(  \frac{2\pi}{\left\vert \xi\right\vert }\right)  ^{\frac{n-1}{2}}%
\frac{\xi_{n}}{\left\vert \xi\right\vert }\overset{\frown}{\varphi_{I}^{\eta}%
}\left(  \frac{c_{J}^{\prime}}{\left\vert c_{J}\right\vert },\frac
{c_{J}^{\prime}}{\left\vert c_{J}\right\vert },\frac{\xi^{\prime}}{\left\vert
\xi\right\vert }\right)  ,
\end{align*}
and the remainder term as%
\begin{align}
\mathfrak{R}_{\overset{\frown}{\varphi_{I}^{\eta}},\phi}^{\left(  1\right)
}\left(  y_{J},\lambda\right)   & =\left(  \frac{2\pi}{\lambda}\right)
^{\frac{n-1}{2}}\frac{e^{i\left[  \operatorname{sgn}B\left(  y_{J}\right)
\frac{\pi}{4}+\lambda\phi\left(  X\left(  y_{J}\right)  ,y_{J}\right)
\right]  }}{\sqrt{\left\vert \det B\left(  y\right)  \right\vert }%
}\label{remt}\\
& \times\int\mathcal{F}_{z}^{-1}\left(  \left[  \frac{\left\langle
i\partial_{z},B\left(  y_{J}\right)  ^{-1}\partial_{z}\right\rangle }%
{2\lambda}\right]  f\right)  \left(  \zeta\right)  R_{1}\left(  -i\frac
{\zeta^{\operatorname*{tr}}B\left(  y_{J}\right)  ^{-1}\zeta}{2\lambda
}\right)  d\zeta.\nonumber
\end{align}

Now we compute that for $x\in I$ and $y\in\pi_{\tan}J$,
\begin{align}
\left\vert \lambda~\Phi\left(  x\right)  \cdot\left(  \Phi\left(  y\right)
-\frac{\left\vert c_{J}\right\vert }{\lambda}\Phi\left(  y_{J}\right)
\right)  \right\vert  & \lesssim\lambda\left\vert \Phi\left(  y\right)
-\frac{\left\vert c_{J}\right\vert }{\lambda}\Phi\left(  y_{J}\right)
\right\vert \lesssim\ell\left(  J\right)  ,\label{y in pi tan J}\\
\text{and }\left\vert \lambda~\partial_{x}^{N}\Phi\left(  x\right)
\cdot\left(  \Phi\left(  y\right)  -\frac{\left\vert c_{J}\right\vert
}{\lambda}\Phi\left(  y_{J}\right)  \right)  \right\vert  & \lesssim
\lambda\left\vert \Phi\left(  y\right)  -\frac{\left\vert c_{J}\right\vert
}{\lambda}\Phi\left(  y_{J}\right)  \right\vert \lesssim\ell\left(  J\right)
,\ \ \ \ \ \text{for }N\geq1.\nonumber
\end{align}

Since $\left\vert \lambda\Phi\left(  x\right)  \cdot\left(  \Phi\left(
y\right)  -\frac{\left\vert c_{J}\right\vert }{\lambda}\Phi\left(
y_{J}\right)  \right)  \right\vert \lesssim\ell\left(  J\right)  \lesssim1$,
the modulus of the inner product $\left\langle \mathfrak{P}_{\overset{\frown
}{\varphi_{I}^{\eta}},\phi},h_{J;\kappa}^{n,\eta}\right\rangle $ is thus
bounded by,%
\begin{align*}
& \left\vert \left\langle \mathfrak{P}_{\overset{\frown}{\varphi_{I}^{\eta}%
},\phi},h_{J;\kappa}^{n,\eta}\right\rangle \right\vert \leq\int_{\mathbb{R}%
^{n}}\left\vert \mathfrak{P}_{\overset{\frown}{\varphi_{I}^{\eta}},\phi
}\left(  \xi\right)  h_{J;\kappa}^{n,\eta}\left(  \xi\right)  \right\vert
d\xi\leq\left\Vert \mathfrak{P}_{\overset{\frown}{\varphi_{I}^{\eta}},\phi
}\right\Vert _{L^{\infty}}\left\Vert h_{J;\kappa}^{n,\eta}\right\Vert
_{L^{\infty}}\left\vert J\right\vert \\
& \lesssim\left(  \frac{1}{\operatorname*{dist}0,J}\right)  ^{\frac{n-1}{2}%
}\left\Vert \overset{\frown}{\varphi_{I}^{\eta}}\right\Vert _{L^{\infty}}%
\sqrt{\left\vert J\right\vert }\lesssim\left(  \frac{1}{\operatorname*{dist}%
0,J}\right)  ^{\frac{n-1}{2}}\frac{1}{\sqrt{\left\vert I\right\vert }}%
\sup_{y\in\pi_{\tan}J}\left\vert R_{\kappa}\left(  -i\lambda\Phi\left(
x\right)  \cdot\left(  \Phi\left(  y\right)  -\frac{\left\vert c_{J}%
\right\vert }{\lambda}\Phi\left(  y_{J}\right)  \right)  \right)  \right\vert
\sqrt{\left\vert J\right\vert }\\
& \lesssim\left(  \frac{1}{\operatorname*{dist}\left(  0,J\right)  }\right)
^{\frac{n-1}{2}}\frac{1}{\sqrt{\left\vert I\right\vert }}\ell\left(  J\right)
^{\kappa}\sqrt{\left\vert J\right\vert }=\left(  \frac{1}{\ell\left(
I\right)  ^{2}\operatorname*{dist}\left(  0,J\right)  }\right)  ^{\frac
{n-1}{2}}\ell\left(  J\right)  ^{\kappa}\sqrt{\left\vert I\right\vert
\left\vert J\right\vert }\\
& =\left(  \frac{1}{\ell\left(  I\right)  ^{2}\operatorname*{dist}\left(
0,J\right)  }\right)  ^{\frac{n-1}{2}}\ell\left(  J\right)  ^{\kappa}%
\sqrt{\left\vert I\right\vert \left\vert J\right\vert }\approx2^{-d\frac
{n-1}{2}}2^{-\left\vert k\right\vert \kappa}\sqrt{\left\vert I\right\vert
\left\vert J\right\vert }\\
& \lesssim2^{-d\frac{n-1}{2}}2^{-\left\vert k\right\vert \kappa}%
\sqrt{\left\vert I\right\vert \left\vert J\right\vert }.
\end{align*}

To estimate the remainder term (\ref{remt}), we thank Cristian Rios\ for the
following argument, which corrects and simplifies an earlier one in a previous
version of this paper. We first need to estimate derivatives of $f$ in
(\ref{fphI}). From the identity
\begin{equation}
\frac{\partial}{\partial x^{\alpha}}R_{\kappa}\left(  ib\left(  x\right)
\right)  =\sum_{0\neq\beta\leq\alpha}\left(
\begin{array}
[c]{c}%
\alpha\\
\beta
\end{array}
\right)  \frac{d^{\left\vert \beta\right\vert }R_{\kappa}}{db^{\left\vert
\beta\right\vert }}\left(  ib\right)  \frac{\partial}{\partial x^{\alpha
-\beta}}\prod_{\ell=1}^{n-1}\left(  \partial_{x_{\ell}}\left(  ib\left(
x\right)  \right)  \right)  ^{\beta_{\ell}}\label{chainrule}%
\end{equation}
With $R=R_{\kappa}$, $b\left(  x\right)  =-\lambda\Phi\left(  x\right)
\cdot\left(  \Phi\left(  y\right)  -\frac{\left\vert c_{J}\right\vert
}{\lambda}\Phi\left(  y_{J}\right)  \right)  $, by (\ref{remd}) and the fact
that $\left\vert \frac{\partial}{\partial x^{\beta}}ib\left(  x\right)
\right\vert \lesssim\ell\left(  J\right)  $, we have that%
\[
\left\vert \frac{\partial}{\partial x^{\alpha}}R_{\kappa}\left(  -i\lambda
\Phi\left(  x\right)  \cdot\left(  \Phi\left(  y\right)  -\frac{\left\vert
c_{J}\right\vert }{\lambda}\Phi\left(  y_{J}\right)  \right)  \right)
\right\vert \lesssim\sum_{j=1}^{\left\vert \alpha\right\vert }\left\vert
\frac{d^{\left\vert \beta\right\vert }R_{\kappa}}{db^{\left\vert
\beta\right\vert }}\left(  ib\right)  \right\vert \ell\left(  J\right)
^{\left\vert \beta\right\vert }\lesssim\ell\left(  J\right)  ^{\kappa}.
\]
Then, whever $\kappa\geq\left\vert \alpha\right\vert $ we have%
\begin{align}
\left\vert \partial_{x^{\alpha}}\overset{\frown}{\varphi_{I}^{\eta}}\left(
x\right)  \right\vert  & \leq\left\vert \partial_{x^{\alpha}}\overset{\frown
}{\varphi_{I}^{\eta}}\left(  x,y_{J},y\right)  \right\vert \nonumber\\
& =\left\vert \sum_{\beta\leq\alpha}\left(
\begin{array}
[c]{c}%
\alpha\\
\beta
\end{array}
\right)  \left(  \frac{\partial}{\partial x^{\beta-\alpha}}h_{I;\kappa
}^{n-1,\eta}\left(  x\right)  \right)  \left(  \frac{\partial}{\partial
x^{\beta-\alpha}}R_{\kappa}\left(  ib\left(  x\right)  \right)  \right)
\right\vert \nonumber\\
& \lesssim\sum_{j=0}^{\left\vert \alpha\right\vert }\frac{\mathbf{1}_{I_{\eta
}}\left(  x\right)  }{\sqrt{\left\vert I\right\vert }}\frac{\ell\left(
J\right)  ^{\kappa}}{\ell\left(  I\right)  ^{j}}\lesssim\frac{\mathbf{1}%
_{I_{\eta}}\left(  x\right)  }{\sqrt{\left\vert I\right\vert }}\frac
{\ell\left(  J\right)  ^{\kappa}}{\ell\left(  I\right)  ^{\left\vert
\alpha\right\vert }}.\label{mphIder}%
\end{align}
Now we estimate the first factor in the integral in (\ref{remt})
\[
\mathcal{F}_{z}^{-1}\left(  \left[  \frac{\left\langle i\partial_{z},B\left(
y_{J}\right)  ^{-1}\partial_{z}\right\rangle }{2\lambda}\right]  f\right)
\left(  \zeta\right)  =\int_{\Psi\left(  I_{\eta}\right)  }\left(  \left[
\frac{\left\langle i\partial_{z},B\left(  y_{J}\right)  ^{-1}\partial
_{z}\right\rangle }{2\lambda}\right]  \frac{\overset{\frown}{\varphi_{I}%
^{\eta}}\left(  \Psi_{y}^{-1}\left(  z\right)  ,y_{J},y\right)  }{\det\left[
\left(  \partial_{x}\Psi\right)  \left(  \Psi_{y}^{-1}\left(  z\right)
\right)  \right]  }\right)  e^{iz\cdot\zeta}~dz
\]
Since $\Psi_{y}$ is a diffeomorphism we have that $\left\vert \det\left[
\left(  \partial_{x}\Psi\right)  \left(  \Psi_{y}^{-1}\left(  z\right)
\right)  \right]  \right\vert \approx1$, and $\left\vert \partial_{z}^{j}%
\det\left[  \left(  \partial_{x}\Psi\right)  \left(  \Psi_{y}^{-1}\left(
z\right)  \right)  \right]  \right\vert \lesssim C_{j}$ for $j\geq1$. Then by
the worst case $\left\vert \alpha\right\vert =2$ in (\ref{mphIder}) we obtain%
\[
\left\vert \left[  \frac{\left\langle i\partial_{z},B\left(  y_{J}\right)
^{-1}\partial_{z}\right\rangle }{2\lambda}\right]  \frac{\overset{\frown
}{\varphi_{I}^{\eta}}\left(  \Psi_{y}^{-1}\left(  z\right)  ,y_{J},y\right)
}{\det\left[  \left(  \partial_{x}\Psi\right)  \left(  \Psi_{y}^{-1}\left(
z\right)  \right)  \right]  }\right\vert \lesssim\frac{1}{\lambda}\left\vert
\partial_{x}^{2}\overset{\frown}{\varphi_{I}^{\eta}}\left(  x,y_{J},y\right)
\right\vert \lesssim\frac{1}{\lambda}\frac{\mathbf{1}_{I_{\eta}}\left(
x\right)  }{\sqrt{\left\vert I\right\vert }}\frac{\ell\left(  J\right)
^{\kappa}}{\ell\left(  I\right)  ^{2}}.
\]
Hence,%
\begin{equation}
\left\vert \mathcal{F}_{z}^{-1}\left(  \left[  \frac{\left\langle
i\partial_{z},B\left(  y_{J}\right)  ^{-1}\partial_{z}\right\rangle }%
{2\lambda}\right]  f\right)  \left(  \zeta\right)  \right\vert \lesssim
\int_{\Psi\left(  I_{\eta}\right)  }\frac{1}{\lambda}\frac{\mathbf{1}%
_{I_{\eta}}\left(  x\right)  }{\sqrt{\left\vert I\right\vert }}\frac
{\ell\left(  J\right)  ^{\kappa}}{\ell\left(  I\right)  ^{2}}~d\zeta
\lesssim\frac{1}{\lambda}\frac{\ell\left(  J\right)  ^{\kappa}}{\ell\left(
I\right)  ^{2}}\sqrt{\left\vert I\right\vert }.\label{fo1}%
\end{equation}
From the identity $e^{iz\cdot\zeta}=\left\vert \zeta\right\vert ^{-2N}\left(
-i\sum_{j=1}^{n-1}\zeta_{j}\partial_{z_{j}}\right)  ^{N}e^{iz\cdot\zeta}$, we
can also write%
\[
\left\vert \mathcal{F}_{z}^{-1}\left(  \left[  \frac{\left\langle
i\partial_{z},B\left(  y_{J}\right)  ^{-1}\partial_{z}\right\rangle }%
{2\lambda}\right]  f\right)  \left(  \zeta\right)  \right\vert =\left\vert
\zeta\right\vert ^{-2N}\left\vert \int_{\Psi\left(  I_{\eta}\right)  }\left(
\left(  i\sum_{j=1}^{n-1}\zeta_{j}\partial_{z_{j}}\right)  ^{Ne^{iz\cdot\zeta
}}\left[  \frac{\left\langle i\partial_{z},B\left(  y_{J}\right)
^{-1}\partial_{z}\right\rangle }{2\lambda}\right]  f\right)  e^{iz\cdot\zeta
}~dz\right\vert
\]
and since, as before, we have the bounds%
\[
\left\vert \left(  i\sum_{j=1}^{n-1}\zeta_{j}\partial_{z_{j}}\right)  ^{N}
\left[  \frac{\left\langle i\partial_{z},B\left(  y_{J}\right)  ^{-1}%
\partial_{z}\right\rangle }{2\lambda}\right]  f\right\vert \lesssim
\frac{\left\vert \zeta\right\vert ^{N}}{\lambda}\left\vert \partial_{x}%
^{N+2}\overset{\frown}{\varphi_{I}^{\eta}}\left(  x,y_{J},y\right)
\right\vert \lesssim\frac{\left\vert \zeta\right\vert ^{N}}{\lambda}%
\frac{\mathbf{1}_{I_{\eta}}\left(  x\right)  }{\sqrt{\left\vert I\right\vert
}}\frac{\ell\left(  J\right)  ^{\kappa}}{\ell\left(  I\right)  ^{N+2}},
\]
hence,%
\[
\left\vert \mathcal{F}_{z}^{-1}\left(  \left[  \frac{\left\langle
i\partial_{z},B\left(  y_{J}\right)  ^{-1}\partial_{z}\right\rangle }%
{2\lambda}\right]  f\right)  \left(  \zeta\right)  \right\vert \lesssim
\left\vert \zeta\right\vert ^{-2N}\int_{\Psi\left(  I_{\eta}\right)  }%
\frac{\left\vert \zeta\right\vert ^{N}}{\lambda}\frac{\mathbf{1}_{I_{\eta}%
}\left(  x\right)  }{\sqrt{\left\vert I\right\vert }}\frac{\ell\left(
J\right)  ^{\kappa}}{\ell\left(  I\right)  ^{N+2}}~dz\lesssim\frac{1}{\lambda
}\frac{1}{\left\vert \zeta\right\vert ^{N}}\frac{\ell\left(  J\right)
^{\kappa}}{\ell\left(  I\right)  ^{N+2}}\sqrt{\left\vert I\right\vert }.
\]
Combining this with (\ref{fo1}) yields%
\[
\left\vert \mathcal{F}_{z}^{-1}\left(  \left[  \frac{\left\langle
i\partial_{z},B\left(  y_{J}\right)  ^{-1}\partial_{z}\right\rangle }%
{2\lambda}\right]  f\right)  \left(  \zeta\right)  \right\vert \lesssim
\frac{1}{\lambda}\sqrt{\left\vert I\right\vert }\frac{\ell\left(  J\right)
^{\kappa}}{\ell\left(  I\right)  ^{2}}\min\left\{  1,\frac{1}{\left\vert
\zeta\right\vert ^{N}}\frac{1}{\ell\left(  I\right)  ^{N}}\right\}  .
\]
Then, from (\ref{remt}) and the fact that $\left\vert R_{1}\left(  ib\right)
\right\vert \leq\left\vert b\right\vert $, we obtain%
\begin{align}
& \ \ \ \ \ \ \ \ \ \ \ \ \ \ \ \left\vert \mathfrak{R}_{\overset{\frown
}{\varphi_{I}^{\eta}},\phi}^{\left(  1\right)  }\left(  y_{J},\lambda\right)
\right\vert \label{R1}\\
& =\left\vert \frac{\left(  \frac{2\pi}{\lambda}\right)  ^{\frac{n-1}{2}}%
}{\sqrt{\left\vert \det B\left(  y\right)  \right\vert }}\right\vert
\left\vert \int\mathcal{F}_{z}^{-1}\left(  \left[  \frac{\left\langle
i\partial_{z},B\left(  y_{J}\right)  ^{-1}\partial_{z}\right\rangle }%
{2\lambda}\right]  f\right)  \left(  \zeta\right)  R_{1}\left(  -i\frac
{\zeta^{\operatorname*{tr}}B\left(  y_{J}\right)  ^{-1}\zeta}{2\lambda
}\right)  d\zeta\right\vert \nonumber\\
& \lesssim\frac{1}{\lambda^{\frac{n-1}{2}}}\int_{\mathbb{R}^{n-1}}\frac
{1}{\lambda}\sqrt{\left\vert I\right\vert }\frac{\ell\left(  J\right)
^{\kappa}}{\ell\left(  I\right)  ^{2}}\min\left\{  1,\frac{1}{\left\vert
\zeta\right\vert ^{N}}\frac{1}{\ell\left(  I\right)  ^{N}}\right\}
\frac{\left\vert \zeta\right\vert ^{2}}{\lambda}~d\zeta\nonumber\\
& =\frac{1}{\lambda^{\frac{n-1}{2}+2}}\frac{\ell\left(  J\right)  ^{\kappa}%
}{\ell\left(  I\right)  ^{2}}\sqrt{\left\vert I\right\vert }\int
_{\mathbb{R}^{n-1}}\min\left\{  1,\frac{1}{\left\vert \zeta\right\vert ^{N}%
}\frac{1}{\ell\left(  I\right)  ^{N}}\right\}  \left\vert \zeta\right\vert
^{2}~d\zeta\nonumber\\
& \approx\frac{1}{\lambda^{\frac{n-1}{2}+2}}\frac{\ell\left(  J\right)
^{\kappa}}{\ell\left(  I\right)  ^{2}}\sqrt{\left\vert I\right\vert }\left(
\int_{0}^{\frac{1}{\ell\left(  I\right)  }}r^{2}r^{n-2}~dr+\frac{1}%
{\ell\left(  I\right)  ^{N}}\int_{\frac{1}{\ell\left(  I\right)  }}^{\infty
}\frac{1}{r^{N}}r^{2}r^{n-2}~dr\right)  .\nonumber
\end{align}
Choosing $N=n+2$ so the second integral is finite, we get%
\begin{align*}
\left\vert \mathfrak{R}_{\overset{\frown}{\varphi_{I}^{\eta}},\phi}^{\left(
1\right)  }\left(  \xi\right)  \right\vert  & \lesssim\frac{1}{\lambda
^{\frac{n-1}{2}+2}}\frac{\ell\left(  J\right)  ^{\kappa}}{\ell\left(
I\right)  ^{2}}\sqrt{\left\vert I\right\vert }\frac{1}{\ell\left(  I\right)
^{n+1}}\approx\frac{1}{\left(  \operatorname*{dist}\left(  0,J\right)
\ell\left(  I\right)  ^{2}\right)  ^{\frac{n-1}{2}+2}}\frac{\ell\left(
I\right)  ^{n+3}}{\ell\left(  I\right)  ^{n+3}}\ell\left(  J\right)  ^{\kappa
}\sqrt{\left\vert I\right\vert }\\
& \lesssim2^{-d\left(  \frac{n-1}{2}+2\right)  }\ell\left(  J\right)
^{\kappa}\sqrt{\left\vert I\right\vert },
\end{align*}
if we take $\kappa\geq N=n+2$.

\begin{remark}
This error estimate is the same estimate as that for the main term, but with
an \emph{additional} small factor of $2^{-2d}$.
\end{remark}

Combining the two estimates for the principle term and the remainder term, we
have
\begin{align*}
& \left\vert \left\langle Th_{I;\kappa}^{n-1,\eta},h_{J;\kappa}^{n,\eta
}\right\rangle \right\vert \leq\left\vert \left\langle \mathfrak{P}%
_{\overset{\frown}{\varphi_{I}^{\eta}},\phi},h_{J;\kappa}^{n,\eta
}\right\rangle \right\vert +\left\vert \left\langle \mathfrak{R}%
_{\overset{\frown}{\varphi_{I}^{\eta}},\phi}^{\left(  1\right)  },h_{J;\kappa
}^{n,\eta}\right\rangle \right\vert \\
& \lesssim2^{-d\frac{n-1}{2}}2^{-\left\vert k\right\vert \kappa}%
\sqrt{\left\vert I\right\vert \left\vert J\right\vert }+2^{-d\frac{n+3}{2}%
}2^{-\left\vert k\right\vert \kappa}\sqrt{\left\vert I\right\vert \left\vert
J\right\vert },
\end{align*}
when $k\leq0$, $d\geq0$, and $\kappa\geq n+2$. We record this as%
\begin{equation}
\left\vert \left\langle Th_{I;\kappa}^{n-1,\eta},h_{J;\kappa}^{n,\eta
}\right\rangle \right\vert \lesssim2^{-d\frac{n-1}{2}}\ 2^{-\left\vert
k\right\vert \kappa}\ \sqrt{\left\vert I\right\vert \left\vert J\right\vert
}.\label{sharp bound}%
\end{equation}

Next, we will use the estimate (\ref{sharp bound}), in the argument we used
above to bound $\mathsf{B}_{\operatorname{below}}^{0,d}\left(  f,g\right)  $,
to show that there is $\delta>0$ such that for all $p>\frac{2n}{n-1}$,%
\[
\left\vert \mathsf{B}_{\operatorname{below}}^{k,d}\left(  f,g\right)
\right\vert \lesssim2^{-\left\vert k\right\vert \delta}2^{-\left\vert
d\right\vert \delta}\left\Vert f\right\Vert _{L^{p}}\left\Vert g\right\Vert
_{L^{p^{\prime}}}\ ,\ \ \ \ \ \text{for all }k\leq0,d\geq0.
\]
Of course we now have $d\geq0$ instead of the opposite inequality $d\leq0$
used in the previous argument, but we will see that much of the geometry of
the decomposition remains the same.

For $k\leq0$ and $d\geq0$, the estimates (\ref{sharp bound}) imply,%
\begin{align*}
& \left\vert \mathsf{B}_{\operatorname{below}}^{k,d}\left(  f,g\right)
\right\vert \equiv\left\vert \sum_{\left(  I,J\right)  \in\mathcal{P}%
_{0}^{k,d}}\left\langle T\bigtriangleup_{I;\kappa}^{n-1,\eta}f,\bigtriangleup
_{J;\kappa}^{n,\eta}g\right\rangle \right\vert =\left\vert \sum_{\left(
I,J\right)  \in\mathcal{P}_{0}^{k,d}}\left\langle Th_{I;\kappa}^{n-1,\eta
}f,h_{J;\kappa}^{n,\eta}g\right\rangle _{\omega}\left\langle f,h_{I;\kappa
}^{n-1,\eta}\right\rangle \left\langle g,h_{J;\kappa}^{n,\eta}\right\rangle
\right\vert \\
& \lesssim\sum_{\left(  I,J\right)  \in\mathcal{P}_{0}^{k,d}}\left\vert
\left\langle Th_{I;\kappa}^{n-1,\eta}f,h_{J;\kappa}^{n,\eta}g\right\rangle
\right\vert \left\{  \frac{1}{\sqrt{\left\vert I\right\vert }}\int_{I_{\eta}%
}\left\vert \bigtriangleup_{I;\kappa}^{n-1,\eta}f\left(  x\right)  \right\vert
dx\right\}  \left\{  \frac{1}{\sqrt{\left\vert J\right\vert }}\int_{J_{\eta}%
}\left\vert \bigtriangleup_{J;\kappa}^{n,\eta}g\left(  \xi\right)  \right\vert
d\xi\right\} \\
& \lesssim\int_{\mathbb{R}^{n}}\sum_{\left(  I,J\right)  \in\mathcal{P}%
_{0}^{k,d}}\frac{\left\vert \left\langle Th_{I;\kappa}^{n-1,\eta},h_{J;\kappa
}^{n,\eta}\right\rangle \right\vert }{\sqrt{\left\vert I\right\vert }%
\sqrt{\left\vert J\right\vert }}\left\{  \int_{I_{\eta}}\left\vert
\bigtriangleup_{I;\kappa}^{n-1,\eta}f\left(  x\right)  \right\vert dx\right\}
\left\vert \bigtriangleup_{J;\kappa}^{n,\eta}g\left(  \xi\right)  \right\vert
d\xi\\
& \lesssim2^{-d\frac{n-1}{2}}2^{-\left\vert k\right\vert \kappa}%
\int_{\mathbb{R}^{n}}\sum_{\left(  I,J\right)  \in\mathcal{P}_{0}^{k,d}%
}\left\{  \int_{I_{\eta}}\left\vert \bigtriangleup_{I;\kappa}^{n-1,\eta
}f\left(  x\right)  \right\vert dx\right\}  \left\vert \bigtriangleup
_{J;\kappa}^{n,\eta}g\left(  \xi\right)  \right\vert d\xi
\end{align*}
which is at most
\begin{align*}
& 2^{-d\frac{n-1}{2}}2^{-\left\vert k\right\vert \kappa}\int_{\mathbb{R}^{n}%
}\sqrt{\sum_{\left(  I,J\right)  \in\mathcal{P}_{0}^{k,d}}\left(
\int_{I_{\eta}}\left\vert \bigtriangleup_{I;\kappa}^{n-1,\eta}f\left(
x\right)  \right\vert dx\right)  ^{2}}\sqrt{\sum_{\left(  I,J\right)
\in\mathcal{P}_{0}^{k,d}}\left\vert \bigtriangleup_{J;\kappa}^{n,\eta}g\left(
\xi\right)  \right\vert ^{2}}d\xi\\
& \lesssim2^{-d\frac{n-1}{2}}2^{-\left\vert k\right\vert \kappa}\left(
\int_{\mathbb{R}^{n}}\left(  \sum_{\left(  I,J\right)  \in\mathcal{P}%
_{0}^{k,d}}\left(  \int_{I_{\eta}}\left\vert \bigtriangleup_{I;\kappa
}^{n-1,\eta}f\left(  x\right)  \right\vert dx\mathbf{1}_{J}\left(  \xi\right)
\right)  ^{2}\right)  ^{\frac{p}{2}}d\xi\right)  ^{\frac{1}{p}}\\
& \ \ \ \ \ \ \ \ \ \ \ \ \ \ \ \ \ \ \ \ \ \ \ \ \ \ \ \ \ \ \times\left(
\int_{\mathbb{R}^{n}}\left(  \sum_{\left(  I,J\right)  \in\mathcal{P}%
_{0}^{k,d}}\left\vert \bigtriangleup_{J;\kappa}^{n,\eta}g\left(  \xi\right)
\right\vert ^{2}\right)  ^{\frac{p^{\prime}}{2}}d\xi\right)  ^{\frac
{1}{p^{\prime}}}\\
& \equiv2^{-d\frac{n-1}{2}}2^{-\left\vert k\right\vert \kappa}\Gamma_{1}%
\Gamma_{2}.
\end{align*}

We have%
\[
\Gamma_{2}^{p^{\prime}}=\int_{\mathbb{R}^{n}}\left(  \sum_{\left(  I,J\right)
\in\mathcal{P}_{0}^{k,d}}\left\vert \bigtriangleup_{J;\kappa}^{n,\eta}g\left(
x\right)  \right\vert ^{2}\right)  ^{\frac{p^{\prime}}{2}}dx=\int
_{\mathbb{R}^{n}}\left(  \sum_{J\in\mathcal{D}}\left(  \sum_{I\in
\mathcal{G}:\ \left(  I,J\right)  \in\mathcal{P}_{0}^{k,d}}1\right)
\left\vert \bigtriangleup_{J;\kappa}^{n,\eta}g\left(  x\right)  \right\vert
^{2}\right)  ^{\frac{p^{\prime}}{2}}dx,
\]
and now we repeat some of the geometric constructions relating to
$\mathcal{P}_{0}^{k,d}$ from before. Fix $J\in\mathcal{D}$ and let $I_{J}%
\in\mathcal{G}$ satisfy%
\[
c_{n}\pi_{1}\left(  J\right)  \leq\ell\left(  I_{J}\right)  \leq\pi_{1}\left(
J\right)  \text{ and }I_{J}\subset\pi_{1}\left(  J\right)  ,
\]
where $\pi_{1}\left(  J\right)  $\ is the spherical projection $J$ onto
$\mathbb{S}^{n-1}$, and where $c_{n}>0$ is chosen small enough that such a
cube $I_{J}$ exists. Now $\left(  I,J\right)  \in\mathcal{P}_{0}^{k,d}$ if and
only if%
\[
J\subset\mathcal{K}\left(  I\right)  \text{ and }\frac{2^{d-1}}{\ell\left(
I\right)  ^{2}}\leq\operatorname*{dist}\left(  0,J\right)  \leq\frac{2^{d+1}%
}{\ell\left(  I\right)  ^{2}},
\]
which is essentially equivalent to%
\[
I\supset\pi_{1}J\supset I_{J}\text{ and }\sqrt{\frac{2^{d-1}}%
{2\operatorname*{dist}\left(  0,J\right)  }}\leq\ell\left(  I\right)
\leq\sqrt{\frac{2^{d+1}}{\operatorname*{dist}\left(  0,J\right)  }}.
\]
Thus\ just as in the previous argument, the set of cubes $I\in\mathcal{G}%
\left[  U\right]  $ with $\left(  I,J\right)  \in\mathcal{P}_{0}^{k,d}$ is
contained in the finite tower of dyadic cubes $\left\{  \pi^{\left(  k\right)
}I_{J}\right\}  _{k=d-A}^{d+A}$ for some fixed $A\in\mathbb{N}$. It follows
that $\sum_{I\in\mathcal{G}\left[  U\right]  :\ \left(  I,J\right)
\in\mathcal{C}_{0}^{0,0}}1\leq2A$ and so%
\[
\Gamma_{2}^{p^{\prime}}=\int_{\mathbb{R}^{n}}\left(  \sum_{\left(  I,J\right)
\in\mathcal{P}_{0}^{0,0}}\left\vert \bigtriangleup_{J;\kappa}^{n,\eta}g\left(
x\right)  \right\vert ^{2}\right)  ^{\frac{p^{\prime}}{2}}dx\leq
\int_{\mathbb{R}^{n}}\left(  \sum_{J\in\mathcal{D}}2A\left\vert \bigtriangleup
_{J;\kappa}^{n,\eta}g\left(  x\right)  \right\vert ^{2}\right)  ^{\frac
{p^{\prime}}{2}}dx\lesssim\left\Vert g\right\Vert _{L^{p^{\prime}}}%
^{p^{\prime}}\ .
\]

We see that on the other hand, since the cubes $J$ in $\mathcal{D}_{k}$ are
pairwise disjoint with measure $2^{kn}$,%
\begin{align*}
\Gamma_{1}^{p}  & =\int_{\mathbb{R}^{n}}\left(  \sum_{\left(  I,J\right)
\in\mathcal{P}_{0}^{k,d}}\left(  \int_{I_{\eta}}\left\vert \bigtriangleup
_{I;\kappa}^{n-1,\eta}f\left(  x\right)  \right\vert dx\right)  ^{2}%
\mathbf{1}_{J_{\eta}}\left(  \xi\right)  \right)  ^{\frac{p}{2}}d\xi\\
& \approx\int_{\mathbb{R}^{n}}\left(  \sum_{J\in\mathcal{D}_{k}}\left\{
\sum_{I\in\mathcal{G}\left[  U\right]  :\ \left(  I,J\right)  \in
\mathcal{P}_{0}^{k,d}}\left(  \int_{I_{\eta}}\left\vert \bigtriangleup
_{I;\kappa}^{n-1,\eta}f\left(  x\right)  \right\vert dx\right)  ^{2}\right\}
\mathbf{1}_{J_{\eta}}\left(  \xi\right)  \right)  ^{\frac{p}{2}}d\xi\\
& =\int_{\mathbb{R}^{n}}\sum_{J\in\mathcal{D}_{k}}\left\{  \sum_{I\in
\mathcal{G}\left[  U\right]  :\ \left(  I,J\right)  \in\mathcal{P}_{0}^{k,d}%
}\left(  \int_{I_{\eta}}\left\vert \bigtriangleup_{I;\kappa}^{n-1,\eta
}f\left(  x\right)  \right\vert dx\right)  ^{2}\right\}  ^{\frac{p}{2}%
}\mathbf{1}_{J_{\eta}}\left(  \xi\right)  d\xi\\
& \approx\sum_{J\in\mathcal{D}_{k}}2^{kn}\left(  \sum_{I\in\mathcal{G}\left[
U\right]  :\ \left(  I,J\right)  \in\mathcal{P}_{0}^{k,d}}\left(
\int_{I_{\eta}}\left\vert \bigtriangleup_{I;\kappa}^{n-1,\eta}f\left(
x\right)  \right\vert dx\right)  ^{2}\right)  ^{\frac{p}{2}}.
\end{align*}
Now for each fixed $J\in\mathcal{D}_{k}$ we have with $A$ as above,%
\begin{align*}
& \left(  \sum_{I\in\mathcal{G}:\ \left(  I,J\right)  \in\mathcal{P}_{0}%
^{k,d}}\left(  \int_{I_{\eta}}\left\vert \bigtriangleup_{I;\kappa}^{n-1,\eta
}f\left(  x\right)  \right\vert dx\right)  ^{2}\right)  ^{\frac{p}{2}}%
\leq\left(  \sum_{s=d-A}^{d+A}\left(  \int_{\pi^{\left(  s\right)  }\left(
I_{J}\right)  _{\eta}}\left\vert \bigtriangleup_{\pi^{\left(  s\right)
}\left(  I_{J}\right)  ;\kappa}^{n-1,\eta}f\left(  x\right)  \right\vert
dx\right)  ^{2}\right)  ^{\frac{p}{2}}\\
& \leq\left(  2A\right)  ^{\frac{p}{2}-1}\sum_{I\in\mathcal{G}\left[
U\right]  :\ \left(  I,J\right)  \in\mathcal{P}_{0}^{k,d}}\left(
\int_{I_{\eta}}\left\vert \bigtriangleup_{I;\kappa}^{n-1,\eta}f\left(
x\right)  \right\vert dx\right)  ^{p}\approx\sum_{I\in\mathcal{G}\left[
U\right]  :\ \left(  I,J\right)  \in\mathcal{P}_{0}^{k,d}}\left(
\int_{I_{\eta}}\left\vert \bigtriangleup_{I;\kappa}^{n-1,\eta}f\left(
x\right)  \right\vert dx\right)  ^{p}.
\end{align*}

Altogether then,%
\begin{align*}
\Gamma_{1}^{p}  & \lesssim\sum_{J\in\mathcal{D}_{k}}2^{kn}\sum_{I\in
\mathcal{G}\left[  U\right]  :\ \left(  I,J\right)  \in\mathcal{P}_{0}^{k,d}%
}\left(  \int_{I_{\eta}}\left\vert \bigtriangleup_{I;\kappa}^{n-1,\eta
}f\left(  x\right)  \right\vert dx\right)  ^{p}\\
& \leq\sum_{J\in\mathcal{D}_{k}}2^{kn}\sum_{I\in\mathcal{G}\left[  U\right]
:\ \left(  I,J\right)  \in\mathcal{P}_{0}^{k,d}}\left\vert I\right\vert
^{\frac{p}{2}}\left(  \int_{I_{\eta}}\left\vert \bigtriangleup_{I;\kappa
}^{n-1,\eta}f\left(  x\right)  \right\vert ^{2}dx\right)  ^{\frac{p}{2}}\\
& =2^{kn}\sum_{I\in\mathcal{G}\left[  U\right]  }\left(  \sum_{J\in
\mathcal{D}_{k}:\ \left(  I,J\right)  \in\mathcal{P}_{0}^{k,d}}1\right)
\left\vert I\right\vert ^{p}\left(  \frac{1}{\left\vert I_{\eta}\right\vert
}\int_{I_{\eta}}\left\vert \bigtriangleup_{I;\kappa}^{n-1,\eta}f\left(
x\right)  \right\vert ^{2}dx\right)  ^{\frac{p}{2}},
\end{align*}
and since
\begin{align*}
& \#\left\{  J\in\mathcal{D}_{k}:\ \left(  I,J\right)  \in\mathcal{P}%
_{0}^{k,d}\right\} \\
& \approx2^{-kn}\left\vert \mathcal{K}_{d}\left(  I\right)  \right\vert
\approx2^{-kn}\left(  \frac{2^{d}}{\ell\left(  I\right)  ^{2}}\ell\left(
I\right)  \right)  ^{n-1}\frac{2^{d}}{\ell\left(  I\right)  ^{2}}=2^{-kn}%
\frac{2^{dn}}{\ell\left(  I\right)  ^{n+1}}=2^{-kn}2^{dn}\left(  \frac
{1}{\left\vert I\right\vert }\right)  ^{\frac{n+1}{n-1}};\\
& \text{where }\mathcal{K}_{d}\left(  I\right)  \equiv\left\{  J\subset
\mathcal{K}\left(  I\right)  :\frac{2^{d-1}}{\ell\left(  I\right)  ^{2}}%
\leq\operatorname*{dist}\left(  0,J\right)  \leq\frac{2^{d+1}}{\ell\left(
I\right)  ^{2}}\right\}  ,
\end{align*}
we have that%
\begin{align*}
\Gamma_{1}^{p}  & \lesssim2^{kn}\sum_{I\in\mathcal{G}\left[  U\right]
}\left(  \#\left\{  J\in\mathcal{D}_{k}:\ \left(  I,J\right)  \in
\mathcal{P}_{0}^{k,d}\right\}  \right)  \left\vert I\right\vert ^{p}\left(
\frac{1}{\left\vert I_{\eta}\right\vert }\int_{I_{\eta}}\left\vert
\bigtriangleup_{I;\kappa}^{n-1,\eta}f\left(  x\right)  \right\vert
^{2}dx\right)  ^{\frac{p}{2}}\\
& \lesssim2^{kn}2^{-kn}2^{dn}\sum_{I\in\mathcal{G}\left[  U\right]
}\left\vert I\right\vert ^{p-\frac{n+1}{n-1}}\left(  \frac{1}{\left\vert
I_{\eta}\right\vert }\int_{I_{\eta}}\left\vert \bigtriangleup_{I;\kappa
}^{n-1,\eta}f\left(  x\right)  \right\vert ^{2}dx\right)  ^{\frac{p}{2}}\\
& =2^{dn}\int_{S}\sum_{I\in\mathcal{G}\left[  U\right]  }\left\vert
I\right\vert ^{p-\frac{n+1}{n-1}-1}\left(  \frac{1}{\left\vert I_{\eta
}\right\vert }\int_{I_{\eta}}\left\vert \bigtriangleup_{I;\kappa}^{n-1,\eta
}f\left(  x\right)  \right\vert ^{2}dx\right)  ^{\frac{p}{2}}\mathbf{1}%
_{I}\left(  z\right)  dz\lesssim2^{dn}\left\Vert f\right\Vert _{L^{p}}^{p}\ ,
\end{align*}
provided $p\geq\frac{2n}{n-1}$, using the the Alpert square function estimate
(\ref{squ est}) as in (\ref{Gamma 1''}) above. Thus we have proved,%
\begin{align*}
\left\vert \mathsf{B}_{\operatorname{below}}^{k,d}\left(  f,g\right)
\right\vert  & \lesssim2^{-d\frac{n-1}{2}}2^{-\left\vert k\right\vert \kappa
}\left(  2^{dn}\right)  ^{\frac{1}{p}}\left\Vert f\right\Vert _{L^{p}%
}\left\Vert g\right\Vert _{L^{p^{\prime}}}\\
& \lesssim2^{-d\left(  \frac{n-1}{2}-\frac{n}{p}\right)  }2^{-\left\vert
k\right\vert \kappa}\left\Vert f\right\Vert _{L^{p}}\left\Vert g\right\Vert
_{L^{p^{\prime}}}\ ,\ \ \ \ \ \text{for }k\leq0,d\geq0,
\end{align*}
and so
\[
\sum_{k\leq0}\sum_{d\geq0}\left\vert \mathsf{B}_{\operatorname{below}}%
^{k,d}\left(  f,g\right)  \right\vert \lesssim\sum_{k\leq0}2^{-\left\vert
k\right\vert \kappa}\sum_{d\geq0}2^{-d\left(  \frac{n-1}{2}-\frac{n}%
{p}\right)  }\left\Vert f\right\Vert _{L^{p}}\left\Vert g\right\Vert
_{L^{p^{\prime}}}\lesssim\left\Vert f\right\Vert _{L^{p}}\left\Vert
g\right\Vert _{L^{p^{\prime}}}\ ,
\]
provided $p>\frac{2n}{n-1}$, and $\kappa\geq1$. Note that we only needed
\emph{strict} inequality $p>\frac{2n}{n-1}$ in this last line. Moreover, the
previous lines of argument can be simplified when $p>\frac{2n}{n-1}$ - see
Subsubsection \ref{subsub direct}.

\subsection{Subforms with $k\geq0,d\geq0$}

We take both $k$ and $d$ to be nonnegative, and begin with the radial
integration by parts formula (\ref{int by parts formula}) to obtain,%
\begin{align*}
\left\langle Th_{I;\kappa}^{n-1,\eta},h_{J;\kappa}^{n,\eta}\right\rangle  &
=\int_{\left(  0,\infty\right)  }\int_{\mathbb{R}^{n-1}}\left\{
\int_{\mathbb{R}^{n-1}}e^{i\lambda\phi\left(  x,y\right)  }\frac{\varphi
_{I}^{\eta}\left(  x\right)  }{\phi\left(  x,y\right)  ^{Z}}dx\right\}
\partial_{\lambda}^{Z}\widetilde{\psi}_{J}^{\eta}\left(  y,\lambda\right)
dyd\lambda\\
& =\int_{\left(  0,\infty\right)  }\int_{\mathbb{R}^{n-1}}\mathcal{I}%
_{\widetilde{\varphi_{I}^{\eta}},\phi}\left(  y,\lambda\right)  \partial
_{\lambda}^{Z}\widetilde{\psi}_{J}^{\eta}\left(  y,\lambda\right)  dyd\lambda,
\end{align*}
where%
\[
\mathcal{I}_{\widetilde{\varphi_{I}^{\eta}},\phi}\left(  y,\lambda\right)
=\int_{\mathbb{R}^{n-1}}e^{i\lambda\phi\left(  x,y\right)  }\frac{\varphi
_{I}^{\eta}\left(  x\right)  }{\phi\left(  x,y\right)  ^{Z}}dx
\]
which is an oscillatory term having the form of (\ref{asym formula}),\ but
with amplitude
\[
\widetilde{\varphi_{I}^{\eta}}\left(  x,y\right)  =\frac{\varphi_{I}^{\eta
}\left(  x\right)  }{\phi\left(  x,y\right)  ^{Z}},
\]
in place of $\varphi_{I}^{\eta}\left(  x\right)  $, which is then paired with
the function
\[
\partial_{\lambda}^{Z}\widetilde{\psi}_{J}^{\eta}\left(  y,\lambda\right)
=\partial_{\lambda}^{Z}h_{J;\kappa}^{n,\eta}\left(  \lambda y,\lambda
\sqrt{1-\left\vert y\right\vert ^{2}}\right)  \frac{\lambda^{n-1}}%
{\sqrt{1-\left\vert y\right\vert ^{2}}}%
\]
in place of $\widetilde{\psi}_{J}^{\eta}\left(  y,\lambda\right)  $, and where
we can take $Z\in\mathbb{N}$ to be a large positive integer depending only on
$n$.

Now we proceed by treating the integral%
\[
\int_{\left(  0,\infty\right)  }\int_{\mathbb{R}^{n-1}}\mathcal{I}%
_{\widetilde{\varphi_{I}^{\eta}},\phi}\left(  y,\lambda\right)  \partial
_{\lambda}^{Z}\widetilde{\psi}_{J}^{\eta}\left(  y,\lambda\right)  dyd\lambda
\]
as in the previous case where $k\leq0$ and $d\geq0$, but with the new
amplitudes $\widetilde{\varphi_{I}^{\eta}}$ and pairing functions
$\partial_{\lambda}^{Z}\widetilde{\psi}_{J}^{\eta}\left(  y,\lambda\right)  $
as above. The end result that we will obtain below is the estimate,
\begin{equation}
\left\vert \mathsf{B}_{\operatorname{below}}^{k,d}\left(  f,g\right)
\right\vert \lesssim2^{-d\delta}2^{-k\delta}\left\Vert f\right\Vert _{L^{p}%
}\left\Vert g\right\Vert _{L^{p^{\prime}}}\ ,\ \ \ \ \ \text{for }k\geq
0,d\geq0,\label{the est}%
\end{equation}
for some $\delta>0$.

Indeed, we apply Theorem \ref{osc int} to $\mathcal{I}_{\widetilde{\varphi
_{I}^{\eta}},\phi}\left(  y,\lambda\right)  =\mathfrak{P}_{\widetilde
{\varphi_{I}^{\eta}},\phi}\left(  y,\lambda\right)  +\mathfrak{R}%
_{\widetilde{\varphi_{I}^{\eta}},\phi}^{\left(  1\right)  }\left(
y,\lambda\right)  $ and first note that
\[
\mathfrak{P}_{\widetilde{\varphi_{I}^{\eta}},\phi}\left(  y,\lambda\right)
=\left(  \frac{2\pi}{\lambda}\right)  ^{\frac{n-1}{2}}\frac
{e^{i\operatorname{sgn}\left[  \partial_{x}^{2}\phi\left(  X\left(  y\right)
,y\right)  \right]  \frac{\pi}{4}+\lambda\phi\left(  X\left(  y\right)
,y\right)  }}{\sqrt{\left\vert \det B\left(  y\right)  \right\vert }%
}\widetilde{\varphi_{I}^{\eta}}\left(  X\left(  y\right)  \right)  ,
\]
and arguing as above, we get%
\[
\left\vert \int_{\left(  0,\infty\right)  }\mathfrak{P}_{\widetilde
{\varphi_{I}^{\eta}},\phi}\left(  y,\lambda\right)  \partial_{\lambda}%
^{Z}\widetilde{\psi}_{J}^{\eta}\left(  y,\lambda\right)  dyd\lambda\right\vert
\lesssim2^{-d\frac{n-1}{2}}2^{-kZ}\sqrt{\left\vert I\right\vert \left\vert
J\right\vert }.
\]
As for the remainder term $\mathfrak{R}_{\widetilde{\varphi_{I}^{\eta}},\phi
}^{\left(  1\right)  }\left(  y,y_{J},\lambda\right)  $, we again invoke the
argument of C. Rios to obtain from (\ref{R1}) with $\kappa=0$ that
\begin{subequations}
\begin{align}
& \left\vert \left\langle \mathfrak{R}_{\widetilde{\varphi_{I}^{\eta}},\phi
}^{\left(  1\right)  }\left(  y,\lambda\right)  ,\partial_{\lambda}%
h_{J;\kappa}^{n,\eta}\right\rangle \right\vert \leq\int\left\vert
\mathfrak{R}_{\widetilde{\varphi_{I}^{\eta}},\phi}^{\left(  1\right)  }\left(
\xi\right)  \partial_{\lambda}h_{J;\kappa}^{n,\eta}\left(  \xi\right)
\right\vert d\xi\leq\left\Vert \mathfrak{R}_{\widetilde{\varphi_{I}^{\eta}%
},\phi}^{\left(  1\right)  }\right\Vert _{L^{\infty}}\left\Vert \partial
_{\lambda}h_{J;\kappa}^{n,\eta}\right\Vert _{L^{\infty}}\left\vert
J\right\vert \label{bound}\\
& \lesssim2^{-d\left(  \frac{n-1}{2}+2\right)  }2^{-kZ}\sqrt{\left\vert
I\right\vert }\sqrt{\left\vert J\right\vert }\leq2^{-d\frac{n-1}{2}}%
2^{-kZ}\sqrt{\left\vert I\right\vert \left\vert J\right\vert },\nonumber
\end{align}
where we have discarded the small factor $2^{-2d}$.

\subsubsection{The Alpert square function estimates}

From above, we have the estimate,
\end{subequations}
\[
\left\vert \int_{\left(  0,\infty\right)  }\int_{S}\mathcal{I}_{\widetilde
{\varphi_{I}^{\eta}},\phi}\left(  y_{J},\lambda\right)  \partial_{\lambda}%
^{Z}\widehat{\psi}_{J}^{\eta}\left(  y,\lambda\right)  dyd\lambda\right\vert
\lesssim2^{-d\frac{n-1}{2}}2^{-kZ}\sqrt{\left\vert I\right\vert \left\vert
J\right\vert }.
\]
Now we apply the Alpert square function arguments to obtain (\ref{the est})
for some $\delta>0$ by choosing $Z$ sufficiently large depending on $n$.
Indeed, following the argument in the above subsection, we have
\begin{align*}
& \left\vert \mathsf{B}_{\operatorname{below}}^{k,d}\left(  f,g\right)
\right\vert \lesssim2^{-d\frac{n-1}{2}}2^{-kZ}\int_{\mathbb{R}^{n}}\sqrt
{\sum_{\left(  I,J\right)  \in\mathcal{P}_{0}^{k,d}}\left(  \int_{I_{\eta}%
}\left\vert \bigtriangleup_{I;\kappa}^{n-1,\eta}f\left(  x\right)  \right\vert
dx\right)  ^{2}}\sqrt{\sum_{\left(  I,J\right)  \in\mathcal{P}_{0}^{k,d}%
}\left\vert \bigtriangleup_{J;\kappa}^{n,\eta}g\left(  \xi\right)  \right\vert
^{2}}d\xi\\
& \lesssim2^{-d\frac{n-1}{2}}2^{-kZ}\left(  \int_{\mathbb{R}^{n}}\left(
\sum_{\left(  I,J\right)  \in\mathcal{P}_{0}^{k,d}}\left(  \int_{I_{\eta}%
}\left\vert \bigtriangleup_{I;\kappa}^{n-1,\eta}f\left(  x\right)  \right\vert
dx\mathbf{1}_{J_{\eta}}\left(  \xi\right)  \right)  ^{2}\right)  ^{\frac{p}%
{2}}d\xi\right)  ^{\frac{1}{p}}\\
& \ \ \ \ \ \ \ \ \ \ \ \ \ \ \ \ \ \ \ \ \ \ \ \ \ \ \ \ \ \ \times\left(
\int_{\mathbb{R}^{n}}\left(  \sum_{\left(  I,J\right)  \in\mathcal{P}%
_{0}^{k,d}}\left\vert \bigtriangleup_{J;\kappa}^{n,\eta}g\left(  \xi\right)
\right\vert ^{2}\right)  ^{\frac{p^{\prime}}{2}}d\xi\right)  ^{\frac
{1}{p^{\prime}}}\\
& \equiv2^{-d\frac{n-1}{2}}2^{-kZ}\Gamma_{1}\Gamma_{2}.
\end{align*}
and $\sum_{I\in\mathcal{G}:\ \left(  I,J\right)  \in\mathcal{C}_{0}^{0,0}%
}1\leq2A$, which together give,%
\[
\Gamma_{2}^{p^{\prime}}=\int_{\mathbb{R}^{n}}\left(  \sum_{\left(  I,J\right)
\in\mathcal{P}_{0}^{0,0}}\left\vert \bigtriangleup_{J;\kappa}^{n,\eta}g\left(
x\right)  \right\vert ^{2}\right)  ^{\frac{p^{\prime}}{2}}dx\leq
\int_{\mathbb{R}^{n}}\left(  \sum_{J\in\mathcal{D}}2A\left\vert \bigtriangleup
_{J;\kappa}^{n,\eta}g\left(  x\right)  \right\vert ^{2}\right)  ^{\frac
{p^{\prime}}{2}}dx\lesssim\left\Vert g\right\Vert _{L^{p^{\prime}}}%
^{p^{\prime}}\ ,
\]
by the\ Alpert square function estimate (\ref{squ est}).

We also have%
\[
\Gamma_{1}^{p}=2^{kn}\sum_{J\in\mathcal{D}_{k}}\left(  \sum_{I\in
\mathcal{G}\left[  U\right]  :\ \left(  I,J\right)  \in\mathcal{P}_{0}^{k,d}%
}\left(  \int_{I_{\eta}}\left\vert \bigtriangleup_{I;\kappa}^{n-1,\eta
}f\left(  x\right)  \right\vert dx\right)  ^{2}\right)  ^{\frac{p}{2}},
\]
and since $k\geq0$, we obtain that $\#\left\{  J\in\mathcal{D}_{k}:\ \left(
I,J\right)  \in\mathcal{P}_{0}^{k,d}\right\}  \lesssim2^{-kn}$, which yields%
\begin{align*}
\Gamma_{1}^{p}  & \lesssim2^{kn}\sum_{I\in\mathcal{G}\left[  U\right]
}\left(  \#\left\{  J\in\mathcal{D}_{k}:\ \left(  I,J\right)  \in
\mathcal{P}_{0}^{k,d}\right\}  \right)  \left\vert I\right\vert ^{p}\left(
\frac{1}{\left\vert I_{\eta}\right\vert }\int_{I_{\eta}}\left\vert
\bigtriangleup_{I;\kappa}^{n-1,\eta}f\left(  x\right)  \right\vert
^{2}dx\right)  ^{\frac{p}{2}}\\
& \lesssim2^{kn}2^{-kn}2^{dn}\sum_{I\in\mathcal{G}\left[  U\right]
}\left\vert I\right\vert ^{p-\frac{n+1}{n-1}}\left(  \frac{1}{\left\vert
I_{\eta}\right\vert }\int_{I_{\eta}}\left\vert \bigtriangleup_{I;\kappa
}^{n-1,\eta}f\left(  x\right)  \right\vert ^{2}dx\right)  ^{\frac{p}{2}}\\
& =2^{dn}\int_{S}\sum_{I\in\mathcal{G}\left[  U\right]  }\left\vert
I\right\vert ^{p-\frac{n+1}{n-1}-1}\left(  \frac{1}{\left\vert I_{\eta
}\right\vert }\int_{I_{\eta}}\left\vert \bigtriangleup_{I;\kappa}^{n-1,\eta
}f\left(  x\right)  \right\vert ^{2}dx\right)  ^{\frac{p}{2}}\mathbf{1}%
_{I}\left(  z\right)  dz\lesssim2^{dn}\left\Vert f\right\Vert _{L^{p}}^{p}\ ,
\end{align*}
just as before, by the Alpert square function estimate (\ref{squ est}),
provided $p\geq\frac{2n}{n-1}$.

Altogether then we have%
\[
\left\vert \mathsf{B}_{\operatorname{below}}^{k,d}\left(  f,g\right)
\right\vert \lesssim2^{-d\frac{n-1}{2}}2^{-kZ}\Gamma_{1}\Gamma_{2}%
\lesssim2^{-d\left(  \frac{n-1}{2}-\frac{n}{p}\right)  }2^{-kZ}\left\Vert
f\right\Vert _{L^{p}}\left\Vert g\right\Vert _{L^{p^{\prime}}},
\]
which implies (\ref{the est}) with
\[
\delta\equiv\min\left\{  \frac{n-1}{2}-\frac{n}{p},Z\right\}  >0,
\]
provided $p>\frac{2n}{n-1}$ and $Z\geq1$. Finally, summing in $k,d\geq0$, we
obtain%
\[
\sum_{k\geq0}\sum_{d\geq0}\left\vert \mathsf{B}_{\operatorname{below}}%
^{k,d}\left(  f,g\right)  \right\vert \leq\sum_{k\geq0}\sum_{d\geq
0}2^{-d\delta}2^{-k\delta}\left\Vert f\right\Vert _{L^{p}}\left\Vert
g\right\Vert _{L^{p^{\prime}}}\lesssim\left\Vert f\right\Vert _{L^{p}%
}\left\Vert g\right\Vert _{L^{p^{\prime}}}\ .
\]

\subsection{Wrapup}

Combining the estimates from all four subsections above yields the desired
bound,%
\[
\left\vert \mathsf{B}_{\operatorname{below}}\left(  f,g\right)  \right\vert
\lesssim\left\Vert f\right\Vert _{L^{p}}\left\Vert g\right\Vert _{L^{p^{\prime
}}}\ ,\ \ \ \ \ p>\frac{2n}{n-1},
\]
in fact the stronger bound (\ref{strong below}).

\begin{remark}
The \emph{strict} inequality $p>\frac{2n}{n-1}$ was used only in bounding the
below form for large $d$. We will also use $p>\frac{2n}{n-1}$ for
probabilistic control of the disjoint form, but only $p>1$ for controlling the
above form $\mathsf{B}_{\operatorname{above}}\left(  f,g\right)  $, to which
we turn next.
\end{remark}

\section{Control of the $\operatorname{above}$ form\label{Sec above}}

Next we control the above form,%
\[
\mathsf{B}_{\operatorname{above}}\left(  f,g\right)  \equiv\sum_{\left(
I,J\right)  \in\mathcal{R}}\left\langle Th_{I;\kappa}^{n-1,\eta},h_{J;\kappa
}^{n,\eta}\right\rangle ,
\]
where%
\[
\mathcal{R}\equiv\left\{  \left(  I,J\right)  \in\mathcal{G}\left[  U\right]
\times\mathcal{D}:\Phi\left(  I\right)  \subset\pi_{\tan}\left(
C_{\operatorname{pseudo}}J\right)  \right\}  \ .
\]
For this form, we will use the pigeonholed parameter $k=\log_{2}\ell\left(
J\right)  $ already used in the below subforms, together with a new parameter
$r=\log_{2}\frac{\ell\left(  \pi_{\tan}J\right)  }{\ell\left(  I\right)  }$,
measuring the ratio of the side lengths of $I$ and $\pi_{\tan}J$. Note that
for fixed $k$ and $r$, and a fixed cube $I\in\mathcal{G}$, there is at most a
bounded number of cubes $J\in\mathcal{D}$ satisfying the pigeonholed
properties $\ell\left(  J\right)  =2^{k}$ and $\frac{\ell\left(  \pi_{\tan
}J\right)  }{\ell\left(  I\right)  }=2^{r}$ such that $\left(  I,J\right)
\in\mathcal{R}$. This fact dictates that we arrange our Alpert square function
decompositions relative to the cubes $I$ in the grid $\mathcal{G}$ (rather
than to cubes $J$ in $\mathcal{D}$ as as in $\mathsf{B}_{\operatorname*{below}%
}\left(  f,g\right)  $) in the arguments below.

To achieve geometric decay in both of these parameters, we will use the high
order moment vanishing principle of decay for the Alpert wavelets
$h_{I;\kappa}^{n-1,\eta}$ in $S$ for decay in $r$, an integration by parts in
the radial Fourier variable for decay in $k\geq0$, and the high order moment
vanishing principle of decay for the Alpert wavelets $h_{J;\kappa}^{n,\eta}$
for decay in $k\leq0$. The stationary phase estimate in Theorem \ref{osc int}
is not needed for the form $\mathsf{B}_{\operatorname{above}}\left(
f,g\right)  $.

In fact we will prove the stronger result that the sublinear form%
\[
\left\vert \mathsf{B}_{\operatorname{above}}\right\vert \left(  f,g\right)
\equiv\sum_{\left(  I,J\right)  \in\mathcal{R}}\left\vert \left\langle
T_{S}h_{I;\kappa}^{n-1,\eta},h_{J;\kappa}^{n,\eta}\right\rangle \right\vert
\]
satisfies%
\begin{equation}
\left\vert \mathsf{B}_{\operatorname{above}}\right\vert \left(  f,g\right)
\lesssim\left\Vert f\right\Vert _{L^{p}}\left\Vert g\right\Vert _{L^{p^{\prime
}}},\ \ \ \ \ \text{for }p>\frac{2n}{n-1}.\label{strong above}%
\end{equation}

Here is the decomposition of $\mathcal{R}$ we will use:%
\begin{align}
& \mathcal{R}=\bigcup_{k\in\mathbb{Z}}\bigcup_{r=1}^{\infty}\mathcal{R}%
^{k,r},\text{ where for all }k\in\mathbb{Z}\text{ and }r\in\mathbb{N}%
,\label{R_k,r}\\
& \mathcal{R}^{k,r}\equiv\left\{  \left(  I,J\right)  \in\mathcal{R}%
:\ell\left(  J\right)  =2^{k}\text{, and }\ell\left(  \pi_{\tan}J\right)
\approx2^{r}\ell\left(  I\right)  \right\}  .\nonumber
\end{align}
First we reduce matters to consideration of cubes $J$ that are disjoint from a
large cube $\left[  -2^{M},2^{M}\right]  ^{n}$ centered at the origin, which
will permit the manipulations used below.

\subsection{Reduction to far away dyadic cubes\label{red faraway}}

We now dispense with the first set of trivial pairs $\left(  I,J\right)
\in\mathcal{R}$, namely those for which $J\subset\left[  -2^{M},2^{M}\right]
^{n}$ for some fixed large positive integer $M$. This can be achieved by
splitting the function $g$ into
\[
g=\mathbf{1}_{\left[  -2^{M},2^{M}\right]  ^{n}}g+\mathbf{1}_{\mathbb{R}%
^{n}\setminus\left[  -2^{M},2^{M}\right]  ^{n}}g=g_{1}+g_{2},
\]
and noting that%
\[
\left\vert \left\langle Tf,g_{1}\right\rangle \right\vert \lesssim\left\Vert
f\right\Vert _{L^{1}}\left\Vert g_{1}\right\Vert _{L^{1}}\lesssim\left\Vert
f\right\Vert _{L^{p}}2^{Mnp}\left\Vert g_{1}\right\Vert _{L^{p^{\prime}}%
},\ \ \ \ \ 1<p<\infty.
\]
Then we may assume that $g$ is supported outside $\left[  -2^{M},2^{M}\right]
^{n}$, and it follows that $\bigtriangleup_{J;\kappa}^{n,\eta}f=\left\langle
f,h_{J;\kappa}^{n}\right\rangle h_{J;\kappa}^{n,\eta}$ vanishes for
$J\subset\left[  -2^{M},2^{M}\right]  ^{n}$.

Next we deal with the slightly less trivial case of dyadic cubes $J$ that have
the origin as one of their vertices. These cubes are contained in $2^{n} $
towers of dyadic cubes, and we will derive here the bound corresponding to the
tower $\left\{  J_{k}\right\}  _{k=M}^{\infty}$ where $J_{k}=\left[
0,2^{k}\right]  ^{n}$, the other cases being similar. First we note that%
\[
\left(  \frac{1}{-ix_{n}}\mathbf{e}_{n}\cdot\partial_{\xi}\right)
^{N}e^{-ix\cdot\xi}=e^{-ix\cdot\xi}\text{ for all }N\text{,}%
\]
and so integrating by parts $N$ times gives,%
\begin{align*}
\left\langle Tf,\bigtriangleup_{J_{k}}^{n,\eta}g\right\rangle  &
=\int_{\mathbb{R}^{n}}\int_{\Phi\left(  S\right)  }f\left(  z\right)
e^{-iz\cdot\xi}d\sigma_{n-1}\left(  z\right)  \bigtriangleup_{J_{k}}^{n,\eta
}g\left(  \xi\right)  d\xi\\
& =\int_{\Phi\left(  S\right)  }\left\{  \int_{\mathbb{R}^{n}}e^{-iz\cdot\xi
}h_{J_{k}}^{n,\eta}\left(  \xi\right)  \left\langle g,h_{J_{k}}^{n,\eta
}\right\rangle d\xi\right\}  d\sigma_{n-1}\left(  z\right) \\
& =i^{N}\left\langle g,h_{J_{k}}^{n,\eta}\right\rangle \int_{S}\left\{
\int_{\mathbb{R}^{n}}e^{-iz\cdot\xi}\left(  \mathbf{e}_{n}\cdot\partial_{\xi
}\right)  ^{N}h_{J_{k}}^{n,\eta}g\left(  \xi\right)  d\xi\right\}  \left(
\frac{1}{z_{n}}\right)  ^{N}f\left(  z\right)  d\sigma_{n-1}\left(  z\right)
,
\end{align*}
and then%
\begin{align*}
& \sum_{k=N}^{\infty}\left\vert \left\langle Tf,\bigtriangleup_{J_{k}}%
^{n,\eta}g\right\rangle \right\vert \lesssim\sum_{k=N}^{\infty}\left\vert
\left\langle g,h_{J_{k}}^{n,\eta}\right\rangle \right\vert \int_{S}\left(
\frac{1}{\eta\ell\left(  J_{k}\right)  }\right)  ^{N}\sqrt{\left\vert
J_{k}\right\vert }\left(  \frac{1}{x_{n}}\right)  ^{N}f\left(  x\right)  dx\\
& \leq\left(  \frac{1}{\eta}\right)  ^{N}\sum_{k=N}^{\infty}\left\vert
\left\langle g,h_{J_{k}}^{n,\eta}\right\rangle \right\vert \ell\left(
J_{k}\right)  ^{\frac{n}{2}-N}\left\Vert f\right\Vert _{L^{1}}=\left(
\frac{1}{\eta}\right)  ^{N}\left\Vert f\right\Vert _{L^{1}}\int_{\mathbb{R}%
^{n}}\sum_{k=N}^{\infty}\left(  \left\vert \left\langle g,h_{J_{k}}^{n,\eta
}\right\rangle \right\vert \frac{1}{\sqrt{\left\vert J_{k}\right\vert }%
}\right)  \ell\left(  J_{k}\right)  ^{-N}\mathbf{1}_{J_{k}}\left(  z\right)
dz\\
& \leq\left(  \frac{1}{\eta}\right)  ^{N}\left\Vert f\right\Vert _{L^{1}}%
\int_{\mathbb{R}^{n}}\left(  \sum_{k=N}^{\infty}\left(  \left\vert
\left\langle g,h_{J_{k}}^{n,\eta}\right\rangle \right\vert \frac{1}%
{\sqrt{\left\vert J_{k}\right\vert }}\right)  ^{2}\mathbf{1}_{J_{k}}\left(
z\right)  \right)  ^{\frac{1}{2}}\left(  \sum_{k=N}^{\infty}\ell\left(
J_{k}\right)  ^{-2N}\mathbf{1}_{J_{k}}\left(  z\right)  \right)  ^{\frac{1}%
{2}}dz\\
& \leq\left(  \frac{1}{\eta}\right)  ^{N}\left\Vert f\right\Vert _{L^{1}%
}\left(  \int_{\mathbb{R}^{n}}\left(  \sum_{k=N}^{\infty}\frac{\left\vert
\left\langle g,h_{J_{k}}^{n,\eta}\right\rangle \right\vert ^{2}}{\left\vert
J_{k}\right\vert }\mathbf{1}_{J_{k}}\left(  z\right)  \right)  ^{\frac
{p^{\prime}}{2}}dz\right)  ^{\frac{1}{p^{\prime}}}\left(  \int_{\mathbb{R}%
^{n}}\left(  \sum_{k=N}^{\infty}\ell\left(  J_{k}\right)  ^{-2N}%
\mathbf{1}_{J_{k}}\left(  z\right)  \right)  ^{\frac{p}{2}}dz\right)
^{\frac{1}{p}}.
\end{align*}
Thus we obtain%
\[
\sum_{k=N}^{\infty}\left\vert \left\langle T_{S}f,\bigtriangleup_{J_{k}%
;\kappa}^{n,\eta}g\right\rangle \right\vert \leq C_{p,N}\left(  \frac{1}{\eta
}\right)  ^{N}\left\Vert f\right\Vert _{L^{1}}\left\Vert g\right\Vert
_{L^{p^{\prime}}}\leq C_{p,N}\left(  \frac{1}{\eta}\right)  ^{N}\left\Vert
f\right\Vert _{L^{p}}\left\Vert g\right\Vert _{L^{p^{\prime}}}%
\ ,\ \ \ \ \ 1<p<\infty,
\]
using the equivalence (\ref{square}) of Alpert square function norms on $g$,
together with the finiteness of the final factor if $N$ is chosen sufficiently
large. Indeed, $\left\Vert g\right\Vert _{L^{p^{\prime}}}\approx\left\Vert
\mathcal{S}g\right\Vert _{L^{p^{\prime}}}$ where
\begin{align*}
\left\Vert \mathcal{S}g\right\Vert _{L^{p^{\prime}}}^{p^{\prime}}  &
=\int_{\mathbb{R}^{n}}\left(  \sum_{J\in\mathcal{D}}^{\infty}\left\vert
\bigtriangleup_{J;\kappa}^{n,\eta}g\left(  z\right)  \right\vert ^{2}\right)
^{\frac{p^{\prime}}{2}}dz=\int_{\mathbb{R}^{n}}\left(  \sum_{J\in\mathcal{D}%
}^{\infty}\left(  \left\langle g,h_{J;\kappa}^{n,\eta}\right\rangle
h_{J;\kappa}^{n,\eta}\left(  z\right)  \right)  ^{2}\right)  ^{\frac
{p^{\prime}}{2}}dz\\
& =\int_{\mathbb{R}^{n}}\left(  \sum_{J\in\mathcal{D}}^{\infty}\left(
\left\langle g,h_{J;\kappa}^{n,\eta}\right\rangle \frac{1}{\sqrt{\left\vert
J\right\vert }}\right)  ^{2}\mathbf{1}_{J_{\eta}}\left(  z\right)  \right)
^{\frac{p^{\prime}}{2}}dz=\int_{\mathbb{R}^{n}}\left(  \sum_{J\in\mathcal{D}%
}^{\infty}\frac{\left\langle g,h_{J;\kappa}^{n,\eta}\right\rangle ^{2}%
}{\left\vert J\right\vert }\mathbf{1}_{J_{\eta}}\left(  z\right)  \right)
^{\frac{p^{\prime}}{2}},
\end{align*}
and for $N>\frac{n}{p}$ we have,%
\[
\int_{\mathbb{R}^{n}}\left(  \sum_{k=N}^{\infty}\ell\left(  J_{k}\right)
^{-2N}\mathbf{1}_{\left(  J_{k}\right)  _{\eta}}\left(  z\right)  \right)
^{\frac{p}{2}}dz=\int_{\mathbb{R}^{n}}\left(  \sum_{k=N}^{\infty}%
2^{-2Nk}\mathbf{1}_{\left(  \left[  0,2^{k}\right]  ^{n}\right)  _{\eta}%
}\left(  z\right)  \right)  ^{\frac{p}{2}}dz\lesssim\int_{\mathbb{R}^{n}%
}\left(  1+\left\vert z\right\vert ^{-2N}\right)  ^{\frac{p}{2}}dz<\infty.
\]

\begin{definition}
\label{def R*}Set%
\begin{align*}
\mathcal{R}_{\ast}  & \equiv\left\{  \left(  I,J\right)  \in\mathcal{R}%
:J\cap\left[  -2^{M},2^{M}\right]  ^{n}=\emptyset\right\} \\
& =\left\{  \left(  I,J\right)  \in\mathcal{G}\left[  U\right]  \times
\mathcal{D}:\Phi\left(  I\right)  \subset\pi_{\tan}\left(
C_{\operatorname{pseudo}}J\right)  \text{ and }J\cap\left[  -2^{M}%
,2^{M}\right]  ^{n}=\emptyset\right\}  ,
\end{align*}
and with $\mathcal{R}^{k,r}$ as in (\ref{R_k,r}),%
\begin{align}
\mathcal{R}_{\ast}^{k,r}  & \equiv\left\{  \left(  I,J\right)  \in
\mathcal{R}^{k,r}:J\cap\left[  -2^{M},2^{M}\right]  ^{n}=\emptyset\right\}
,\label{def R*kr}\\
\mathcal{R}_{\ast}^{r}  & \equiv\bigcup_{k}\mathcal{R}_{\ast}^{k,r}%
\ .\nonumber
\end{align}

\end{definition}

\begin{description}
\item[Assumption] It is understood from now on that all of the cubes
$J\in\mathcal{R}$ considered below in this section satisfy $J\cap\left[
-2^{M},2^{M}\right]  ^{n}=\emptyset$, i.e. $\left(  I,J\right)  \in
\mathcal{R}_{\ast}$.
\end{description}

\subsection{Pigeonholed subforms\label{Sub pigeon}}

Using the moment vanishing of the smooth wavelets $h_{I;\kappa}^{n-1,\eta}$,
we first show the preliminary estimate that for all $r\in\mathbb{N}$,%
\begin{equation}
\left\vert \left\langle Th_{I;\kappa}^{n-1,\eta},h_{J;\kappa}^{n,\eta
}\right\rangle \right\vert \lesssim\ell\left(  I\right)  ^{\kappa}\ell\left(
J\right)  ^{\kappa}\sqrt{\left\vert I\right\vert \left\vert J\right\vert
},\ \ \ \ \ \text{for all }\left(  I,J\right)  \in\mathcal{R}_{\ast}^{r}\text{
when }r\geq1.\label{pre est}%
\end{equation}
So consider the case $\left(  I,J\right)  \in\mathcal{R}_{\ast}^{r}$, $r\geq
1$. Using (\ref{exp Tay}) and (\ref{g kappa bound}), with $c_{I}$ denoting the
center of $I$, we have%
\begin{align*}
& \left\langle Th_{I;\kappa}^{n-1,\eta},h_{J;\kappa}^{n,\eta}\right\rangle
=\int_{\mathbb{R}^{n}}\int_{\mathbb{R}^{n-1}}e^{-i\Phi\left(  x\right)
\cdot\xi}h_{I;\kappa}^{n-1,\eta}\left(  x\right)  dxh_{J;\kappa}^{n,\eta
}\left(  \xi\right)  d\xi\\
& =\int_{\mathbb{R}^{n}}e^{-i\Phi\left(  c_{I}\right)  \cdot\xi}h_{J;\kappa
}^{n,\eta}\left(  \xi\right)  \left\{  \int_{\mathbb{R}^{n-1}}e^{-i\left[
\Phi\left(  x\right)  -\Phi\left(  c_{I}\right)  \right]  \cdot\xi}%
h_{I;\kappa}^{n-1,\eta}\left(  x\right)  dx\right\}  d\xi\\
& =\int_{\mathbb{R}^{n}}e^{-i\Phi\left(  c_{I}\right)  \cdot\xi}h_{J;\kappa
}^{n,\eta}\left(  \xi\right)  \left\{  \int_{\mathbb{R}^{n-1}}\left[
\sum_{\ell=0}^{\kappa-1}\frac{\left(  -i\xi\cdot\left[  \Phi\left(  x\right)
-\Phi\left(  c_{I}\right)  \right]  \right)  ^{\ell}}{\ell!}+R_{\kappa}\left(
-i\xi\cdot\left[  \Phi\left(  x\right)  -\Phi\left(  c_{I}\right)  \right]
\right)  \right]  h_{I;\kappa}^{n-1,\eta}\left(  x\right)  dx\right\}  d\xi.
\end{align*}
In order to apply the moment vanishing properties of $h_{I;\kappa}^{n-1,\eta}%
$, we need to express $\Phi\left(  x\right)  $ by Taylor's formula as well,%
\[
\Phi\left(  x\right)  =\sum_{\ell=0}^{\kappa-1}\frac{\left(  \left(
x-c_{I}\right)  \cdot\partial_{x}\right)  ^{\ell}}{\ell!}\Phi\left(
c_{I}\right)  +\Gamma_{\kappa}\left(  x-c_{I}\right)  ,
\]
and then plug this expression into the previous Taylor formula. The result is
that all the terms with a polynomial in $x$ of order less than $\kappa$
vanish, and we are left with%
\begin{align}
\left\langle Th_{I;\kappa}^{n-1,\eta},h_{J;\kappa}^{n,\eta}\right\rangle  &
=\int_{\mathbb{R}^{n}}e^{-i\Phi\left(  c_{I}\right)  \cdot\xi}h_{J;\kappa
}^{n,\eta}\left(  \xi\right)  \left\{  \int_{\mathbb{R}^{n-1}}\Gamma\left(
\xi,x\right)  h_{I;\kappa}^{n-1,\eta}\left(  x\right)  dx\right\}
d\xi\label{first formula}\\
& =\int_{\mathbb{R}^{n}}e^{-i\Phi\left(  c_{I}\right)  \cdot\xi}h_{J;\kappa
}^{n,\eta}\left(  \xi\right)  \left\{  \int_{\mathbb{R}^{n-1}}\left[
R_{\kappa}\left(  -i\xi\cdot\left[  \Phi\left(  x\right)  -\Phi\left(
c_{I}\right)  \right]  \right)  \right]  h_{I;\kappa}^{n-1,\eta}\left(
x\right)  dx\right\}  d\xi\nonumber\\
& +\int_{\mathbb{R}^{n}}e^{-i\Phi\left(  c_{I}\right)  \cdot\xi}h_{J;\kappa
}^{n,\eta}\left(  \xi\right)  \left\{  \int_{\mathbb{R}^{n-1}}\left[
\Gamma_{\kappa}\left(  x-c_{I}\right)  \right]  h_{I;\kappa}^{n-1,\eta}\left(
x\right)  dx\right\}  d\xi\nonumber
\end{align}
where
\begin{equation}
\Gamma\left(  \xi,x\right)  =R_{\kappa}\left(  -i\xi\cdot\left[  \Phi\left(
x\right)  -\Phi\left(  c_{I}\right)  \right]  \right)  +\Gamma_{\kappa}\left(
x-c_{I}\right) \label{Gamma}%
\end{equation}
consists of the remainder term $R_{\kappa}$ and a collection of error
expressions in $\Gamma_{\kappa}\left(  \xi,x\right)  $. Because $\left\vert
x-c_{I}\right\vert \leq\left\vert \Phi\left(  x\right)  -\Phi\left(
c_{I}\right)  \right\vert $, these error expressions satisfy the same
pointwise bounds as the original remainder term $R_{\kappa}\left(  -i\xi
\cdot\left[  \Phi\left(  x\right)  -\Phi\left(  c_{I}\right)  \right]
\right)  $. Recalling from (\ref{g kappa bound}) that the remainder term
$R_{\kappa}$ satisfies $\left\vert R_{\kappa}\left(  ib\right)  \right\vert
\leq\frac{\left\vert b\right\vert ^{\kappa}}{\left(  \kappa+1\right)  !}$, and
taking absolute values inside the integral, we obtain,%
\begin{equation}
\left\vert \left\langle Th_{I;\kappa}^{n-1,\eta},h_{J;\kappa}^{n,\eta
}\right\rangle \right\vert \lesssim\left(  \operatorname*{dist}\left(
0,J\right)  \ell\left(  I\right)  \sin\theta\right)  ^{\kappa}\sqrt{\left\vert
I\right\vert \left\vert J\right\vert },\label{first prelim}%
\end{equation}
where $\theta$ is the angle between $\xi$ and $\Phi\left(  x\right)
-\Phi\left(  c_{I}\right)  $. In the case at hand where $\left(  I,J\right)
\in\mathcal{R}_{\ast}^{r}$, we have $\theta\approx\ell\left(  \pi_{\tan
}J\right)  \approx\frac{\ell\left(  J\right)  }{\operatorname*{dist}\left(
0,J\right)  }$, and so%
\[
\left\vert \left\langle Th_{I;\kappa}^{n-1,\eta},h_{J;\kappa}^{n,\eta
}\right\rangle \right\vert \lesssim\left(  \operatorname*{dist}\left(
0,J\right)  \ell\left(  I\right)  \frac{\ell\left(  J\right)  }%
{\operatorname*{dist}\left(  0,J\right)  }\right)  ^{\kappa}\sqrt{\left\vert
I\right\vert \left\vert J\right\vert }\approx\ell\left(  I\right)  ^{\kappa
}\ell\left(  J\right)  ^{\kappa}\sqrt{\left\vert I\right\vert \left\vert
J\right\vert },\ \ \ \ \ \text{for }\left(  I,J\right)  \in\mathcal{R}_{\ast
}^{r}\ ,
\]
which proves the preliminary estimate (\ref{pre est}).

The case $k\leq0$ will be handled by this last estimate alone, since for
$\left(  I,J\right)  \in\mathcal{R}_{\ast}^{r}$, it yields%
\begin{equation}
\left\vert \left\langle Th_{I;\kappa}^{n-1,\eta},h_{J;\kappa}^{n,\eta
}\right\rangle \right\vert \lesssim\ell\left(  \pi_{\tan}J\right)  ^{\kappa
}\left(  \frac{\ell\left(  I\right)  }{\ell\left(  \pi_{\tan}J\right)
}\right)  ^{\kappa}\ell\left(  J\right)  ^{\kappa}\sqrt{\left\vert
I\right\vert \left\vert J\right\vert }\leq2^{-r\kappa}2^{-\left\vert
k\right\vert \kappa},\ \ \ \ \ \text{for }k\leq0,\label{first prelim'}%
\end{equation}
upon discarding the small factor $\ell\left(  \pi_{\tan}J\right)  ^{\kappa} $.

To handle the case $k\geq0$, we introduce the radial integration by parts
principle of decay, that will deliver geometric gain in $k$. First we observe
that $\left(  I,J\right)  \in\mathcal{R}_{\ast}$ implies $I\subset\pi_{\tan
}\left(  C_{\operatorname{pseudo}}J\right)  $, and so for $\mathbf{v}%
=\pi_{\tan}c_{J}$ and for $x\in\pi_{\tan}\left(  C_{\operatorname{pseudo}%
}J\right)  $ we have
\[
\mathbf{v}\cdot\Phi\left(  x\right)  \geq c>0,
\]
and%
\[
\left(  \frac{1}{-i\mathbf{v}\cdot\Phi\left(  x\right)  }\mathbf{v}%
\cdot\partial_{\xi}\right)  ^{N}e^{-i\Phi\left(  x\right)  \cdot\xi
}=e^{-ix\cdot\xi}\text{ for all }N\text{.}%
\]
Integrating by parts $N$ times then gives,%
\begin{align}
& \left\langle Th_{I;\kappa}^{n-1,\eta},h_{J;\kappa}^{n,\eta}\right\rangle
=\int_{\mathbb{R}^{n}}\int_{\mathbb{R}^{n-1}}h_{I;\kappa}^{n-1,\eta}%
e^{-i\Phi\left(  x\right)  \cdot\xi}dxh_{J}^{n,\eta}g\left(  \xi\right)
d\xi\label{second formula}\\
& =\int_{\mathbb{R}^{n-1}}\left\{  \int_{\mathbb{R}^{n}}e^{-ix\cdot\xi}%
h_{J}^{n,\eta}g\left(  \xi\right)  d\xi\right\}  h_{I;\kappa}^{n-1,\eta
}dx\nonumber\\
& =i^{N}\int_{\mathbb{R}^{n-1}}\left\{  \int_{\mathbb{R}^{n}}e^{-ix\cdot\xi
}\left(  \mathbf{v}\cdot\partial_{\xi}\right)  ^{N}h_{J}^{n,\eta}\left(
\xi\right)  d\xi\right\}  \left(  \frac{1}{\mathbf{v}\cdot\Phi\left(
x\right)  }\right)  ^{N}h_{I;\kappa}^{n-1,\eta}\left(  x\right)  dx,\nonumber
\end{align}
and then we have the second preliminary estimate,%
\begin{equation}
\left\vert \left\langle Th_{I;\kappa}^{n-1,\eta},h_{J;\kappa}^{n,\eta
}\right\rangle \right\vert \lesssim\int_{\mathbb{R}^{n-1}}\left(  \frac
{1}{\eta\ell\left(  J\right)  }\right)  ^{N}\sqrt{\left\vert J\right\vert
}\left(  \frac{1}{c}\right)  ^{N}\left\vert h_{I;\kappa}^{n-1,\eta}\left(
x\right)  \right\vert dx\approx\ell\left(  J\right)  ^{-N}\sqrt{\left\vert
I\right\vert \left\vert J\right\vert }.\label{second prelim}%
\end{equation}

We must now combine these two preliminary estimates in the case $k\geq0$. As
usual, to achieve this we iterate the two associated formulas
(\ref{first formula}) and (\ref{second formula}) \emph{before} taking absolute
values inside the resulting integral. Thus we write,%
\begin{align*}
& \left\langle Th_{I;\kappa}^{n-1,\eta},h_{J;\kappa}^{n,\eta}\right\rangle
=i^{N}\int_{\mathbb{R}^{n-1}}\left\{  \int_{\mathbb{R}^{n}}e^{-i\Phi\left(
x\right)  \cdot\xi}\left(  \mathbf{v}\cdot\partial_{\xi}\right)
^{N}h_{J;\kappa}^{n,\eta}\left(  \xi\right)  d\xi\right\}  \left(  \frac
{1}{\mathbf{v}\cdot\Phi\left(  x\right)  }\right)  ^{N}h_{I;\kappa}^{n-1,\eta
}\left(  x\right)  dx\\
& =\int_{\mathbb{R}^{n}}e^{-i\Phi\left(  c_{I}\right)  \cdot\xi}\left\{
\int_{\mathbb{R}^{n-1}}e^{-i\left[  \Phi\left(  x\right)  -\Phi\left(
c_{I}\right)  \right]  \cdot\xi}h_{I;\kappa}^{n-1,\eta}\left(  x\right)
\left(  \frac{1}{\mathbf{v}\cdot\Phi\left(  x\right)  }\right)  ^{N}%
dx\right\}  \left(  \mathbf{v}\cdot\partial_{\xi}\right)  ^{N}h_{J;\kappa
}^{n,\eta}\left(  \xi\right)  d\xi\\
& =\int_{\mathbb{R}^{n}}e^{-i\Phi\left(  c_{I}\right)  \cdot\xi}\left\{
\int_{\mathbb{R}^{n-1}}\Gamma\left(  \xi,x\right)  h_{I;\kappa}^{n-1,\eta
}\left(  x\right)  \left(  \frac{1}{\mathbf{v}\cdot\Phi\left(  x\right)
}\right)  ^{N}dx\right\}  \left(  \mathbf{v}\cdot\partial_{\xi}\right)
^{N}h_{J;\kappa}^{n,\eta}\left(  \xi\right)  d\xi,
\end{align*}
where
\[
\Gamma\left(  \xi,x\right)  =R_{\kappa}\left(  -i\xi\cdot\left[  \Phi\left(
x\right)  -\Phi\left(  c_{I}\right)  \right]  \right)  +\Gamma_{\kappa}\left(
x-c_{I}\right)  ,
\]
is as in (\ref{Gamma}) above, and $\Gamma\left(  \xi,x\right)  $ satisfies the
estimates\ given there. Now we take absolute values inside the integral, and
using the estimates developed above, we obtain the following inequality for
$k\geq0$,%
\begin{equation}
\left\vert \left\langle Th_{I;\kappa}^{n-1,\eta},h_{J;\kappa}^{n,\eta
}\right\rangle \right\vert \lesssim\ell\left(  I\right)  ^{\kappa}\ell\left(
J\right)  ^{\kappa}\ell\left(  J\right)  ^{-N}\sqrt{\left\vert I\right\vert
\left\vert J\right\vert }\lesssim\left(  \frac{\ell\left(  I\right)  }%
{\ell\left(  \pi_{\tan}J\right)  }\right)  ^{\kappa}\ell\left(  J\right)
^{2\kappa-N}\lesssim2^{-r\kappa}2^{-k\left(  N-2\kappa\right)  }%
\sqrt{\left\vert I\right\vert \left\vert J\right\vert }.\label{emerges}%
\end{equation}

Combining (\ref{first prelim'}) and (\ref{emerges}) gives%
\begin{equation}
\left\vert \left\langle Th_{I;\kappa}^{n-1,\eta},h_{J;\kappa}^{n,\eta
}\right\rangle \right\vert \lesssim2^{-r\kappa}2^{-\left\vert k\right\vert
\min\left\{  \kappa,N-2\kappa\right\}  }\sqrt{\left\vert I\right\vert
\left\vert J\right\vert },\label{combine}%
\end{equation}
and with this estimate in hand, we will now prove that for all $N>2\kappa$ and
$r\in\mathbb{N}$,%
\begin{equation}
\left\vert \sum_{\left(  I,J\right)  \in\mathcal{R}_{\ast}^{k,r}}\left\langle
T\bigtriangleup_{I;\kappa}^{n-1,\eta}f,\bigtriangleup_{J;\kappa}^{n,\eta
}g\right\rangle \right\vert \lesssim2^{-r\left(  \kappa-\frac{n-1}{2}\right)
}2^{-\left\vert k\right\vert \min\left\{  \kappa,N-2\kappa\right\}
}\left\Vert f\right\Vert _{L^{p}}\left\Vert g\right\Vert _{L^{p^{\prime}}%
},\label{will now prove}%
\end{equation}
where $\mathcal{R}_{\ast}^{k,r}$ is defined in (\ref{def R*kr}). Indeed, we
have from (\ref{combine}) that%
\begin{align*}
& \sum_{\left(  I,J\right)  \in\mathcal{R}_{\ast}^{k,r}}\left\vert
\left\langle T\bigtriangleup_{I;\kappa}^{n-1,\eta}f,\bigtriangleup_{J;\kappa
}^{n,\eta}g\right\rangle \right\vert \leq\sum_{\left(  I,J\right)
\in\mathcal{R}_{\ast}^{k,r}}2^{-r\kappa}2^{-\left\vert k\right\vert
\min\left\{  \kappa,N-2\kappa\right\}  }\left(  \int_{I_{\eta}}\left\vert
\bigtriangleup_{I;\kappa}^{n-1,\eta}f\right\vert \right)  \left(
\int_{J_{\eta}}\left\vert \bigtriangleup_{J;\kappa}^{n,\eta}g\right\vert
\right) \\
& =2^{-r\kappa}2^{-\left\vert k\right\vert \min\left\{  \kappa,N-2\kappa
\right\}  }\int_{\mathbb{R}^{n}}\sum_{\left(  I,J\right)  \in\mathcal{R}%
_{\ast}^{k,r}}\left(  \int_{J_{\eta}}\left\vert \bigtriangleup_{J;\kappa
}^{n,\eta}g\right\vert \right)  \left\vert \bigtriangleup_{I;\kappa}%
^{n-1,\eta}f\left(  x\right)  \right\vert dx\\
& \leq2^{-r\kappa}2^{-\left\vert k\right\vert \min\left\{  \kappa
,N-2\kappa\right\}  }\int_{\mathbb{R}^{n}}\sqrt{\sum_{\left(  I,J\right)
\in\mathcal{R}_{\ast}^{k,r}}\left(  \int_{J_{\eta}}\left\vert \bigtriangleup
_{J;\kappa}^{n,\eta}g\right\vert \right)  ^{2}}\sqrt{\sum_{\left(  I,J\right)
\in\mathcal{R}_{\ast}^{k,r}}\left\vert \bigtriangleup_{I;\kappa}^{n-1,\eta
}f\left(  x\right)  \right\vert ^{2}}dx\\
& \leq2^{-r\kappa}2^{-\left\vert k\right\vert \min\left\{  \kappa
,N-2\kappa\right\}  }\left(  \int_{\mathbb{R}^{n}}\left(  \sum_{\left(
I,J\right)  \in\mathcal{R}_{\ast}^{k,r}}\left(  \int_{J_{\eta}}\left\vert
\bigtriangleup_{J;\kappa}^{n,\eta}g\right\vert \right)  ^{2}\right)
^{\frac{p^{\prime}}{2}}dx\right)  ^{\frac{1}{p^{\prime}}}\left(
\int_{\mathbb{R}^{n}}\left(  \sum_{\left(  I,J\right)  \in\mathcal{R}_{\ast
}^{k,r}}\left\vert \bigtriangleup_{I;\kappa}^{n-1,\eta}f\left(  x\right)
\right\vert ^{2}\right)  ^{\frac{p}{2}}dx\right)  ^{\frac{1}{p}}%
\end{align*}
where the Alpert square function estimate (\ref{square}) shows that%
\[
\left(  \int_{\mathbb{R}^{n}}\left(  \sum_{\left(  I,J\right)  \in
\mathcal{R}_{\ast}^{k,r}}\left\vert \bigtriangleup_{I;\kappa}^{n-1,\eta
}f\left(  x\right)  \right\vert ^{2}\right)  ^{\frac{p}{2}}dx\right)
^{\frac{1}{p}}\lesssim\left\Vert f\right\Vert _{L^{p}},
\]
since for each $I\in\mathcal{G}$, there is at most one cube $J\in\mathcal{D}$
such that $\left(  I,J\right)  \in\mathcal{R}_{\ast}^{k,r}$. On the other
hand, for each fixed $J\in\mathcal{D}$, the number of cubes $I\in\mathcal{G}$
such that $\left(  I,J\right)  \in\mathcal{R}_{\ast}^{k,r}$ is approximately
$2^{r\left(  n-1\right)  }$, and so%
\begin{align*}
& \sum_{\left(  I,J\right)  \in\mathcal{R}_{\ast}^{k,r}}\left\vert
\left\langle T\bigtriangleup_{I;\kappa}^{n-1,\eta}f,\bigtriangleup_{J;\kappa
}^{n,\eta}g\right\rangle \right\vert \\
& \lesssim2^{-r\kappa}2^{-\left\vert k\right\vert \min\left\{  \kappa
,N-2\kappa\right\}  }\left(  \int_{\mathbb{R}^{n}}\left(  \sum_{J\in
\mathcal{D}}2^{r\left(  n-1\right)  }\left(  \int_{J_{\eta}}\left\vert
\bigtriangleup_{J;\kappa}^{n,\eta}g\right\vert \right)  ^{2}\right)
^{\frac{p^{\prime}}{2}}dx\right)  ^{\frac{1}{p^{\prime}}}\left\Vert
f\right\Vert _{L^{p}}\\
& \approx2^{-r\left(  \kappa-\frac{n-1}{2}\right)  }2^{-\left\vert
k\right\vert \min\left\{  \kappa,N-2\kappa\right\}  }\left\Vert g\right\Vert
_{L^{p^{\prime}}}\left\Vert f\right\Vert _{L^{p}}\ ,
\end{align*}
for $1<p<\infty$ by the Alpert square function estimate (\ref{squ est}) again.

\subsubsection{The enlarged form}

For $k\geq0$ define%
\[
\mathcal{E}_{\ast}^{k,r}\equiv\left\{  \left(  I,J\right)  \in\mathcal{G}%
\left[  U\right]  \times\mathcal{D}:\ell\left(  J\right)  =2^{k}\text{, }%
\ell\left(  \pi_{\tan}J\right)  =2^{r}\ell\left(  I\right)  \text{, and
}I\subset C_{\operatorname{pseudo}}2^{k}\pi_{\tan}J\right\}  ,
\]
and define the \emph{enlarged} form,%
\[
\mathsf{B}_{\operatorname{enlarge}}\left(  f,g\right)  \equiv\sum
_{k=0}^{\infty}\sum_{r=0}^{\infty}\sum_{\left(  I,J\right)  \in\mathcal{E}%
_{\ast}^{k,r}}\left\langle T\bigtriangleup_{I;\kappa}^{n-1,\eta}%
f,\bigtriangleup_{J;\kappa}^{n,\eta}g\right\rangle .
\]
Then for each fixed $J\in\mathcal{D}$, the number of cubes $I\in\mathcal{G} $
such that $\left(  I,J\right)  \in\mathcal{E}_{\ast}^{k,r}$ is approximately
$\frac{\left\vert 2^{k}\pi_{\tan}J\right\vert }{\left\vert I\right\vert
}=\frac{2^{k\left(  n-1\right)  }\left\vert \pi_{\tan}J\right\vert
}{2^{-r\left(  n-1\right)  }\left\vert \pi_{\tan}J\right\vert }=2^{\left(
r+k\right)  \left(  n-1\right)  }$, and so we have%
\begin{align*}
& \sum_{\left(  I,J\right)  \in\mathcal{R}_{\ast}^{k,r}}\left\vert
\left\langle T\bigtriangleup_{I;\kappa}^{n-1,\eta}f,\bigtriangleup_{J;\kappa
}^{n,\eta}g\right\rangle \right\vert \\
& \lesssim2^{-r\kappa}2^{-\left\vert k\right\vert \min\left\{  \kappa
,N-2\kappa\right\}  }\left(  \int_{\mathbb{R}^{n}}\left(  \sum_{J\in
\mathcal{D}}2^{\left(  r+k\right)  \left(  n-1\right)  }\left(  \int_{J_{\eta
}}\left\vert \bigtriangleup_{J;\kappa}^{n,\eta}g\right\vert \right)
^{2}\right)  ^{\frac{p^{\prime}}{2}}dx\right)  ^{\frac{1}{p^{\prime}}%
}\left\Vert f\right\Vert _{L^{p}}\\
& \approx2^{-r\left(  \kappa-\frac{n-1}{2}\right)  }2^{-\left\vert
k\right\vert \min\left\{  \kappa-\frac{n-1}{2},N-2\kappa-\frac{n-1}%
{2}\right\}  }\left\Vert g\right\Vert _{L^{p^{\prime}}}\left\Vert f\right\Vert
_{L^{p}}\ ,
\end{align*}
for $1<p<\infty$ by the Alpert square function estimate (\ref{squ est}) again.

\subsection{Wrapup}

Finally, taking $\kappa>\frac{n-1}{2}$, $N>2\kappa$ and summing the above
estimates over $r\in\mathbb{N}$ and $k\in\mathbb{Z}$, gives,%
\[
\left\vert \sum_{\left(  I,J\right)  \in\mathcal{R}_{\ast}}\left\langle
T\bigtriangleup_{I;\kappa}^{n-1,\eta}f,\bigtriangleup_{J;\kappa}^{n,\eta
}g\right\rangle \right\vert \lesssim\left\Vert f\right\Vert _{L^{p}}\left\Vert
g\right\Vert _{L^{p^{\prime}}}\ .
\]
Combined with the reduction in the first subsection, we obtain the desired
bound,%
\[
\left\vert \mathsf{B}_{\operatorname{above}}\left(  f,g\right)  \right\vert
\lesssim\left\Vert f\right\Vert _{L^{p}}\left\Vert g\right\Vert _{L^{p^{\prime
}}}\ ,\ \ \ \ \ 1<p<\infty,
\]
in fact the stronger bound (\ref{strong above}).

\begin{remark}
The only restriction on $p$ here is $1<p<\infty$, and so the above form
$\mathsf{B}_{\operatorname{above}}\left(  f,g\right)  $ is bounded for all
$1<p<\infty$.
\end{remark}

\section{Control of the $\operatorname*{upper}\operatorname*{disjoint}$ and
$\operatorname*{upper}\operatorname*{distal}$ forms}

The principle of stationary phase is not used for the disjoint or distal
subforms, as the critical point of the phase now lies outside the support of
the amplitude. When $k\geq0$ we must introduce the radial integration by parts
principle of decay to bound the subforms, while in the case $k\leq0$, we must
use the high order vanishing moments of $h_{J;\kappa}^{n,\eta}$. Just as in
the case of the below form $\mathsf{B}_{\operatorname{below}}$, combining the
appropriate formulas, and staying the introduction of absolute values until
the very end, will yield the desired inequalities. There is however a crucial
difference between the cases $d\geq0$ and $d<0$ in the case of both disjoint
subforms $\mathsf{B}_{\operatorname*{disjoint}}^{k,d,m}\left(  f,g\right)  $
and distal subforms $\mathsf{B}_{\operatorname*{distal}}^{k,d}\left(
f,g\right)  $, and we will treat the upper and lower cases in separate
sections, as the resonant lower forms with $d<0$ require probability and
interpolation techniques.

In fact, when $d\geq0$, the standard principles of decay apply to give the
required control. However, as $d$ becomes increasingly negative, resonance
begins to set in more strongly, and by the time $d=-m$, none of the standard
principles of decay are any longer of use. Instead we must invoke classical
methods of estimating $L^{2}$ and $L^{4}$ bounds, but using \emph{probability}
in order to obtain improved bounds for functions restricted to smooth Alpert pseudoprojections.

Recall from (\ref{def up disj}) that%
\begin{align*}
& \mathsf{B}_{\operatorname*{disjoint}}^{\operatorname*{upper}}\left(
f,g\right)  \equiv\sum_{m=1}^{\infty}\sum_{k\in\mathbb{Z}}\sum_{d\geq
0}\mathsf{B}_{\operatorname*{disjoint}}^{k,d,m}\left(  f,g\right)  ,\\
& \text{where }\mathsf{B}_{\operatorname*{disjoint}}^{k,d,m}\left(
f,g\right)  \equiv\sum_{\left(  I,J\right)  \in\mathcal{P}_{m}^{k,d}%
}\left\langle T\bigtriangleup_{I;\kappa}^{n-1,\eta}f,\bigtriangleup_{J;\kappa
}^{n,\eta}g\right\rangle ,\\
& \text{and }\mathcal{P}_{m}^{k,d}\equiv\left\{  \left(  I,J\right)
\in\mathcal{P}_{m}:\ell\left(  J\right)  =2^{k}\text{, and }2^{d}\leq
\ell\left(  I\right)  ^{2}\operatorname*{dist}\left(  0,J\right)  \leq
2^{d+1}\right\}  ,\\
& \text{and }\mathcal{P}_{m}\equiv\left\{  \left(  I,J\right)  \in
\mathcal{G}\left[  U\right]  \times\mathcal{D}:2^{m+1}I\subset S\text{ and
}\pi_{\tan}\left(  J\right)  \subset\Phi\left(  4U\cap2^{m+1}%
C_{\operatorname{pseudo}}I\right)  \setminus\Phi\left(  2^{m}%
C_{\operatorname{pseudo}}I\right)  \right\}  ,
\end{align*}
and similarly from (\ref{def up dist}) that,%
\begin{align*}
& \mathsf{B}_{\operatorname*{distal}}^{\operatorname*{upper}}\left(
f,g\right)  \equiv\sum_{k\in\mathbb{Z}}\sum_{d\geq0}\mathsf{B}%
_{\operatorname*{distal}}^{k,d}\left(  f,g\right)  ,\\
& \text{where }\mathsf{B}_{\operatorname*{distal}}^{k,d}\left(  f,g\right)
\equiv\sum_{\left(  I,J\right)  \in\mathcal{X}^{k,d}}\left\langle
T\bigtriangleup_{I;\kappa}^{n-1,\eta}f,\bigtriangleup_{J;\kappa}^{n,\eta
}g\right\rangle ,\\
& \text{and }\mathcal{X}^{k,d}\equiv\left\{  \left(  I,J\right)
\in\mathcal{X}:\ell\left(  J\right)  =2^{k}\text{, and }2^{d}\leq\ell\left(
I\right)  ^{2}\operatorname*{dist}\left(  0,J\right)  \leq2^{d+1}\right\}  ,\\
& \text{and }\mathcal{X}\equiv\left\{  \left(  I,J\right)  \in\mathcal{G}%
\left[  U\right]  \times\mathcal{D}:2^{m+1}I\subset S\text{ and }\pi_{\tan
}\left(  J\right)  \cap\Phi\left(  2U\right)  =\emptyset\right\}  .
\end{align*}

\subsection{Upper disjoint subforms with $d\geq0$}

When $k=0$, we obtain geometric gain simultaneously in $m\geq1$ and $d\geq0 $
using the tangential integration by parts principle of decay. In order to
handle arbitrary $k\in\mathbb{Z}$, we must include additional principles of
decay combined with tangential integration by parts. For $k\geq0,$ we include
radial integration by parts, and taking absolute values inside the integral at
the very end, we will obtain below that,%
\begin{equation}
\left\vert \left\langle Th_{I;\kappa}^{n-1,\eta},h_{J;\kappa}^{n,\eta
}\right\rangle \right\vert \lesssim2^{-kN_{1}}2^{-N_{2}\left(  m+d\right)
}\sqrt{\left\vert I\right\vert \left\vert J\right\vert }.\label{kmd decay}%
\end{equation}
For $k\leq0$, we include instead the moment vanishing properties of
$h_{J;\kappa}^{n,\eta}$, and taking absolute values inside the integral at the
very end, we will obtain below that,%
\begin{equation}
\left\vert \left\langle Th_{I;\kappa}^{n-1,\eta},h_{J;\kappa}^{n,\eta
}\right\rangle \right\vert \lesssim2^{-\left\vert k\right\vert \kappa
}2^{-N_{2}\left(  m+d\right)  }\sqrt{\left\vert I\right\vert \left\vert
J\right\vert }.\label{kmd decay'}%
\end{equation}

With these estimates in hand, together with the Alpert square function
arguments used repeatedly above, we obtain,
\[
\left\vert \mathsf{B}_{\operatorname*{disjoint}}^{k,d,m}\left(  f,g\right)
\right\vert \lesssim2^{-\delta\left\vert k\right\vert }2^{-\delta\left(
m+d\right)  }\left(  \int\left\vert \bigtriangleup_{I;\kappa}^{n-1,\eta
}f\right\vert \right)  \left(  \int\left\vert \bigtriangleup_{J;\kappa
}^{n,\eta}g\right\vert \right)  ,\ \ \ \ \ \text{for }p\geq\frac{2n}{n-1},
\]
for some $\delta>0$ provided $\kappa$, $N_{1}$ and $N_{2}$ are chosen
sufficiently large, and finally then,%
\[
\sum_{k\in\mathbb{Z}}\sum_{d\geq0}\sum_{m=1}^{\infty}\left\vert \mathsf{B}%
_{\operatorname*{disjoint}}^{k,d,m}\left(  f,g\right)  \right\vert
\lesssim\left\Vert f\right\Vert _{L^{p}}\left\Vert g\right\Vert _{L^{p^{\prime
}}}\ ,\ \ \ \ \ \text{for }p\geq\frac{2n}{n-1}.
\]

Here is a brief sketch of the two inner product estimates mentioned above,
followed by the appropriate Alpert square function estimate.

\subsubsection{The case $k\geq0,d\geq0$\label{Subsub IBP}}

Combining the radial integration by parts formula (\ref{int by parts formula}%
),%
\[
\left\langle Th_{I;\kappa}^{n-1,\eta},h_{J;\kappa}^{n,\eta}\right\rangle
=\int_{\mathbb{R}}\int_{\mathbb{R}^{n-1}}\int_{\mathbb{R}^{n-1}}%
\frac{e^{i\lambda\phi\left(  x,y\right)  }}{\phi\left(  x,y\right)  ^{N_{1}}%
}\varphi_{I}^{\eta}\left(  x\right)  \partial_{\lambda}^{N_{1}}\widehat{\psi
}_{J}^{\eta}\left(  y,\lambda\right)  dxdyd\lambda,
\]
with the tangential integration by parts formula (\ref{tan int by parts}),%
\[
\left\langle Th_{I;\kappa}^{n-1,\eta},h_{J;\kappa}^{n,\eta}\right\rangle
=i^{N}\int_{\mathbb{R}}\int_{\mathbb{R}^{n-1}}\int_{\mathbb{R}^{n-1}%
}e^{i\lambda\phi\left(  x,y\right)  }\left\{  \left(  D_{\mathbf{v}}^{x}%
\frac{1}{\left(  D_{\mathbf{v}}\Phi\right)  \left(  x\right)  \cdot\Phi\left(
y\right)  }\right)  ^{N_{2}}\right\}  \varphi_{I}^{\eta}\left(  x\right)
\widehat{\psi}_{J}^{\eta}\left(  y,\lambda\right)  dxdy\frac{d\lambda}%
{\lambda^{N_{2}}}.
\]
gives%
\begin{align*}
& \left\langle Th_{I;\kappa}^{n-1,\eta},h_{J;\kappa}^{n,\eta}\right\rangle
=i^{N}\int_{\mathbb{R}}\int_{\mathbb{R}^{n-1}}\int_{\mathbb{R}^{n-1}%
}e^{i\lambda\phi\left(  x,y\right)  }\left\{  \left(  D_{\mathbf{v}}^{x}%
\frac{1}{\left(  D_{\mathbf{v}}\Phi\right)  \left(  x\right)  \cdot\Phi\left(
y\right)  }\right)  ^{N_{2}}\right\}  \varphi_{I}^{\eta}\left(  x\right)
\widehat{\psi}_{J}^{\eta}\left(  y,\lambda\right)  dxdy\frac{d\lambda}%
{\lambda^{N_{2}}}\\
& =i^{N}\int_{\mathbb{R}}\int_{\mathbb{R}^{n-1}}\int_{\mathbb{R}^{n-1}}%
\frac{e^{i\lambda\phi\left(  x,y\right)  }}{\phi\left(  x,y\right)  ^{N_{1}}%
}\left\{  \left(  D_{\mathbf{v}}^{x}\frac{1}{\left(  D_{\mathbf{v}}%
\Phi\right)  \left(  x\right)  \cdot\Phi\left(  y\right)  }\right)  ^{N_{2}%
}\right\}  \varphi_{I}^{\eta}\left(  x\right)  \partial_{\lambda}^{N_{1}}%
\frac{\widehat{\psi}_{J}^{\eta}\left(  y,\lambda\right)  }{\lambda^{N_{2}}%
}dxdyd\lambda.
\end{align*}
Taking absolute values inside the integral, and using (\ref{int est}) together
with $\min\left\{  \frac{1}{\eta\ell\left(  J\right)  },\frac{1}{\lambda
}\right\}  \lesssim\frac{1}{\ell\left(  J\right)  }$, and (\ref{tan int est}),
we obtain,%
\begin{equation}
\left\vert \left\langle Th_{I;\kappa}^{n-1,\eta},h_{J;\kappa}^{n,\eta
}\right\rangle \right\vert \lesssim2^{-kN_{1}}2^{-N_{2}\left(  m+d\right)
}\sqrt{\left\vert I\right\vert \left\vert J\right\vert },\label{as req}%
\end{equation}
as required.

\subsubsection{The case $k\leq0,d\geq0$}

This time we use (\ref{tan int by parts}),%
\[
\left\langle Th_{I;\kappa}^{n-1,\eta},h_{J;\kappa}^{n,\eta}\right\rangle
=i^{N}\int_{\mathbb{R}}\int_{\mathbb{R}^{n-1}}\int_{\mathbb{R}^{n-1}%
}e^{i\lambda\phi\left(  x,y\right)  }\left\{  \left(  D_{\mathbf{v}}^{x}%
\frac{1}{\left(  D_{\mathbf{v}}\Phi\right)  \left(  x\right)  \cdot\Phi\left(
y\right)  }\right)  ^{N}\right\}  \varphi_{I}^{\eta}\left(  x\right)
\widehat{\psi}_{J}^{\eta}\left(  y,\lambda\right)  dxdy\frac{d\lambda}%
{\lambda^{N}},
\]
together with (\ref{van mom formula}),
\[
\left\langle Th_{I;\kappa}^{n-1,\eta},h_{J;\kappa}^{n,\eta}\right\rangle
=\int_{\mathbb{R}^{n-1}}e^{-i\Phi\left(  x\right)  \cdot c_{J}}h_{I;\kappa
}^{n-1,\eta}\left(  x\right)  \left\{  \int_{\mathbb{R}^{n}}R_{\kappa}\left(
-i\Phi\left(  x\right)  \cdot\left(  \xi-c_{J}\right)  \right)  h_{J;\kappa
}^{n,\eta}\left(  \xi\right)  d\xi\right\}  dx,
\]
to obtain,%
\begin{align*}
& \left\langle Th_{I;\kappa}^{n-1,\eta},h_{J;\kappa}^{n,\eta}\right\rangle
=i^{N}\int_{\mathbb{R}}\int_{\mathbb{R}^{n-1}}\int_{\mathbb{R}^{n-1}%
}e^{-i\lambda\phi\left(  x,y\right)  }\left\{  \left(  D_{\mathbf{v}}^{x}%
\frac{1}{\left(  D_{\mathbf{v}}\Phi\right)  \left(  x\right)  \cdot\Phi\left(
y\right)  }\right)  ^{N}\right\}  \varphi_{I}^{\eta}\left(  x\right)
\widehat{\psi}_{J}^{\eta}\left(  y,\lambda\right)  dxdy\frac{d\lambda}%
{\lambda^{N}}\\
& =\left(  -i\right)  ^{N}\int_{\mathbb{R}^{n-1}}\left\{  \int_{\mathbb{R}%
^{n}}e^{-i\Phi\left(  x\right)  \cdot\xi}\left\{  \left(  D_{\mathbf{v}}%
^{x}\frac{1}{\left(  D_{\mathbf{v}}\Phi\right)  \left(  x\right)  \cdot
\frac{\xi}{\left\vert \xi\right\vert }}\right)  ^{N}\right\}  \frac
{h_{J;\kappa}^{n,\eta}\left(  \xi\right)  }{\left\vert \xi\right\vert ^{N}%
}d\xi\right\}  h_{I;\kappa}^{n-1,\eta}\left(  x\right)  dx\\
& =\left(  -i\right)  ^{N}\int_{\mathbb{R}^{n}}\left\{  \int_{\mathbb{R}%
^{n-1}}e^{-i\Phi\left(  x\right)  \cdot c_{J}}\left[  \frac{1}{\left\vert
\xi\right\vert ^{N}}\left(  D_{\mathbf{v}}^{x}\frac{1}{\left(  D_{\mathbf{v}%
}\Phi\right)  \left(  x\right)  \cdot\frac{\xi}{\left\vert \xi\right\vert }%
}\right)  ^{N}R_{\kappa}\left(  -i\Phi\left(  x\right)  \cdot\left(  \xi
-c_{J}\right)  \right)  h_{I;\kappa}^{n-1,\eta}\left(  x\right)  \right]
dx\right\}  h_{J;\kappa}^{n,\eta}\left(  \xi\right)  d\xi,
\end{align*}
where in the second line above, we have reversed the change of variable in
(\ref{para}). Now from the estimates used in (\ref{tan int est}) and
(\ref{van mom est}) we obtain,%
\[
\left\vert \left\langle Th_{I;\kappa}^{n-1,\eta},h_{J;\kappa}^{n,\eta
}\right\rangle \right\vert \lesssim2^{-\left\vert k\right\vert \kappa
}2^{-N\left(  m+d\right)  }\sqrt{\left\vert I\right\vert \left\vert
J\right\vert },
\]
as required.

\subsubsection{The Alpert square function argument for $d\geq0$}

We follow the Alpert square function argument used for the below form
$\mathsf{B}_{\operatorname{below}}^{k,d}\left(  f,g\right)  $ when
$k\geq0,d\leq0$. The only difference is that we now accumulate a factor of a
large power of $2^{m}$ depending on $n$ and $p$, but this will be offset by
gains from integration by parts in both parameters $m$ and $d$ - and this uses
in a crucial way that $d\geq0$. We begin by writing the sum over $\left(
I,J\right)  \in\mathcal{P}_{m}^{k,d}$ as,
\[
\sum_{\left(  I,J\right)  \in\mathcal{P}_{m}^{k,d}}=\sum_{\substack{\left(
I,J\right)  \in\mathcal{G}\left[  U\right]  \times\mathcal{D}:\ 2^{m+1}%
I\subset U\text{ and }\pi_{\tan}\left(  J\right)  \subset\Phi\left(
2^{m+1}CI\right)  \setminus\Phi\left(  2^{m}CI\right)  \\\ell\left(  J\right)
=2^{k}\text{ and }2^{d}\leq\ell\left(  I\right)  ^{2}\operatorname*{dist}%
\left(  0,J\right)  \leq2^{d+1}}}\ ,
\]
and%
\begin{align*}
& \left\vert \mathsf{B}_{\operatorname*{disjoint}}^{k,d,m}\left(  f,g\right)
\right\vert =\left\vert \sum_{\left(  I,J\right)  \in\mathcal{P}_{m}^{k,d}%
}\left\langle T\bigtriangleup_{I;\kappa}^{n-1,\eta}f,\bigtriangleup_{J;\kappa
}^{n,\eta}g\right\rangle \right\vert \leq\sum_{\left(  I,J\right)
\in\mathcal{P}_{m}^{k,d}}\left\vert \left\langle T\bigtriangleup_{I;\kappa
}^{n-1,\eta}f,\bigtriangleup_{J;\kappa}^{n,\eta}g\right\rangle \right\vert \\
& \lesssim\sum_{\left(  I,J\right)  \in\mathcal{P}_{m}^{k,d}}2^{-\left\vert
k\right\vert \kappa}2^{-N_{2}\left(  m+d\right)  }\left(  \int\left\vert
\bigtriangleup_{I;\kappa}^{n-1,\eta}f\right\vert \right)  \left(
\int\left\vert \bigtriangleup_{J;\kappa}^{n,\eta}g\right\vert \right) \\
& =2^{-\left\vert k\right\vert \kappa}2^{-N_{2}\left(  m+d\right)  }%
\int_{\mathbb{R}^{n}}\sum_{\left(  I,J\right)  \in\mathcal{P}_{m}^{k,d}%
}\left(  \int_{\mathbb{R}^{n-1}}\left\vert \bigtriangleup_{I;\kappa}%
^{n-1,\eta}f\left(  x\right)  \right\vert dx\right)  \mathbf{1}_{J}\left(
\xi\right)  \left\vert \bigtriangleup_{J;\kappa}^{n,\eta}g\left(  \xi\right)
\right\vert d\xi\\
& \lesssim2^{-\left\vert k\right\vert \kappa}2^{-N_{2}\left(  m+d\right)
}\int_{\mathbb{R}^{n}}\sqrt{\sum_{\left(  I,J\right)  \in\mathcal{P}_{m}%
^{k,d}}\left(  2^{m\left(  n-1\right)  }\int_{I_{\eta}}\left\vert
\bigtriangleup_{I;\kappa}^{n-1,\eta}f\left(  x\right)  \right\vert dx\right)
^{2}\mathbf{1}_{J}\left(  \xi\right)  }\sqrt{\sum_{\left(  I,J\right)
\in\mathcal{P}_{m}^{k,d}}2^{-m\left(  n-1\right)  }\left\vert \bigtriangleup
_{J;\kappa}^{n,\eta}g\left(  \xi\right)  \right\vert ^{2}}d\xi,
\end{align*}
which gives%
\begin{align*}
\left\vert \mathsf{B}_{\operatorname*{disjoint}}^{k,d,m}\left(  f,g\right)
\right\vert  & \lesssim2^{-\left\vert k\right\vert \kappa}2^{-N_{2}\left(
m+d\right)  }\left(  \int_{\mathbb{R}^{n}}\left(  \sum_{\left(  I,J\right)
\in\mathcal{P}_{m}^{k,d}}\left(  2^{m\left(  n-1\right)  }\int_{I_{\eta}%
}\left\vert \bigtriangleup_{I;\kappa}^{n-1,\eta}f\left(  x\right)  \right\vert
dx\right)  ^{2}\mathbf{1}_{J}\left(  \xi\right)  \right)  ^{\frac{p}{2}}%
d\xi\right)  ^{\frac{1}{p}}\\
& \times\left(  \int_{\mathbb{R}^{n}}\left(  \sum_{\left(  I,J\right)
\in\mathcal{P}_{m}^{k,d}}2^{-m\left(  n-1\right)  }\left\vert \bigtriangleup
_{J;\kappa}^{n,\eta}g\left(  \xi\right)  \right\vert ^{2}\right)
^{\frac{p^{\prime}}{2}}d\xi\right)  ^{\frac{1}{p^{\prime}}}\\
& \equiv2^{-\left\vert k\right\vert \kappa}2^{-N_{2}\left(  m+d\right)
}\Gamma_{1}\Gamma_{2}\ .
\end{align*}

We first consider $\Gamma_{2}$ which satisfies,%
\[
\Gamma_{2}^{p^{\prime}}=\int_{\mathbb{R}^{n}}\left(  \sum_{\left(  I,J\right)
\in\mathcal{P}_{m}^{k,d}}2^{-m\left(  n-1\right)  }\left\vert \bigtriangleup
_{J;\kappa}^{n,\eta}g\left(  \xi\right)  \right\vert ^{2}\right)
^{\frac{p^{\prime}}{2}}d\xi\lesssim\int_{\mathbb{R}^{n}}\left(  \sum
_{J\in\mathcal{D}}\left\vert \bigtriangleup_{J;\kappa}^{n,\eta}g\left(
\xi\right)  \right\vert ^{2}\right)  ^{\frac{p^{\prime}}{2}}d\xi
\approx\left\Vert g\right\Vert _{L^{p^{\prime}}}^{p^{\prime}}\ ,
\]
since for a fixed $J$ with $\ell\left(  J\right)  =2^{k}$, the number of cubes
$I$ such that
\[
\left(  I,J\right)  \in\mathcal{P}_{m}^{k,d}=\left\{  \left(  I,J\right)
\in\mathcal{G}\left[  U\right]  \times\mathcal{D}:%
\begin{array}
[c]{c}%
\ 2^{m+1}I\subset U\text{ and }\pi_{\tan}\left(  J\right)  \subset\Phi\left(
2^{m+1}CI\right)  \setminus\Phi\left(  2^{m-1}CI\right) \\
\text{and }\ell\left(  J\right)  =2^{k}\text{ and }2^{d}\leq\ell\left(
I\right)  ^{2}\operatorname*{dist}\left(  0,J\right)  \leq2^{d+1}%
\end{array}
\right\}
\]
is roughly $2^{m\left(  n-1\right)  }$, and where the final approximation is
the Alpert square function estimate (\ref{squ est}).

Now we turn to $\Gamma_{1}$ for which we have the estimate,%
\begin{align*}
\Gamma_{1}^{p}  & =\int_{\mathbb{R}^{n}}\left(  \sum_{\left(  I,J\right)
\in\mathcal{P}_{m}^{k,d}}2^{2m\left(  n-1\right)  }\left(  \int_{I_{\eta}%
}\left\vert \bigtriangleup_{I;\kappa}^{n-1,\eta}f\left(  x\right)  \right\vert
dx\right)  ^{2}\mathbf{1}_{J}\left(  \xi\right)  \right)  ^{\frac{p}{2}}d\xi\\
& =2^{pm\left(  n-1\right)  }\int_{\mathbb{R}^{n}}\left(  \sum_{J\in
\mathcal{D}_{k}}\sum_{I\in\mathcal{G}\left[  U\right]  :\ \left(  I,J\right)
\in\mathcal{P}_{m}^{k,d}}\left(  \int_{I_{\eta}}\left\vert \bigtriangleup
_{I;\kappa}^{n-1,\eta}f\left(  x\right)  \right\vert dx\right)  ^{2}%
\mathbf{1}_{J}\left(  \xi\right)  \right)  ^{\frac{p}{2}}d\xi\\
& =2^{pm\left(  n-1\right)  }\int_{\mathbb{R}^{n}}\sum_{J\in\mathcal{D}_{k}%
}\left(  \sum_{I\in\mathcal{G}\left[  U\right]  :\ \left(  I,J\right)
\in\mathcal{P}_{m}^{k,d}}\left(  \int_{I_{\eta}}\left\vert \bigtriangleup
_{I;\kappa}^{n-1,\eta}f\left(  x\right)  \right\vert dx\right)  ^{2}\right)
^{\frac{p}{2}}\mathbf{1}_{J}\left(  \xi\right)  d\xi\\
& =2^{pm\left(  n-1\right)  }2^{kn}\sum_{J\in\mathcal{D}_{k}}\left(
\sum_{I\in\mathcal{G}\left[  U\right]  :\ \left(  I,J\right)  \in
\mathcal{P}_{m}^{k,d}}\left(  \int_{I_{\eta}}\left\vert \bigtriangleup
_{I;\kappa}^{n-1,\eta}f\left(  x\right)  \right\vert dx\right)  ^{2}\right)
^{\frac{p}{2}}.
\end{align*}

Now for each $J\in\mathcal{D}_{k}$, the number of\ cubes $I\in\mathcal{G}%
\left[  U\right]  $ with $\left(  I,J\right)  \in\mathcal{P}_{m}^{k,d}$ is
approximately $2^{mn}$, and so we compute that,%
\begin{align*}
\left(  \sum_{I\in\mathcal{G}\left[  U\right]  :\ \left(  I,J\right)
\in\mathcal{P}_{m}^{k,d}}\left(  \int_{I_{\eta}}\left\vert \bigtriangleup
_{I;\kappa}^{n-1,\eta}f\left(  x\right)  \right\vert dx\right)  ^{2}\right)
^{\frac{p}{2}}  & \lesssim\left(  \sum_{I\in\mathcal{G}\left[  U\right]
:\ \left(  I,J\right)  \in\mathcal{P}_{m}^{k,d}}1\right)  ^{\frac{p}{2}-1}%
\sum_{I\in\mathcal{G}\left[  U\right]  :\ \left(  I,J\right)  \in
\mathcal{P}_{m}^{k,d}}\left(  \int_{I_{\eta}}\left\vert \bigtriangleup
_{I;\kappa}^{n-1,\eta}f\left(  x\right)  \right\vert dx\right)  ^{p}\\
& \approx2^{mn\left(  \frac{p}{2}-1\right)  }\sum_{I\in\mathcal{G}\left[
U\right]  :\ \left(  I,J\right)  \in\mathcal{P}_{m}^{k,d}}\left(
\int_{I_{\eta}}\left\vert \bigtriangleup_{I;\kappa}^{n-1,\eta}f\left(
x\right)  \right\vert dx\right)  ^{p},
\end{align*}
and hence that%
\begin{align*}
\Gamma_{1}^{p}  & \lesssim2^{pm\left(  n-1\right)  }2^{kn}\sum_{J\in
\mathcal{D}_{k}}\left(  \sum_{I\in\mathcal{G}\left[  U\right]  :\ \left(
I,J\right)  \in\mathcal{P}_{m}^{k,d}}\left(  \int_{I_{\eta}}\left\vert
\bigtriangleup_{I;\kappa}^{n-1,\eta}f\left(  x\right)  \right\vert dx\right)
^{2}\right)  ^{\frac{p}{2}}\\
& \lesssim2^{pm\left(  n-1\right)  }2^{kn}\sum_{J\in\mathcal{D}_{k}%
}2^{mn\left(  \frac{p}{2}-1\right)  }\sum_{I\in\mathcal{G}\left[  U\right]
:\ \left(  I,J\right)  \in\mathcal{P}_{m}^{k,d}}\left(  \int_{I_{\eta}%
}\left\vert \bigtriangleup_{I;\kappa}^{n-1,\eta}f\left(  x\right)  \right\vert
dx\right)  ^{p}\\
& \lesssim2^{m\left[  p\left(  n-1\right)  +n\left(  \frac{p}{2}-1\right)
\right]  }2^{kn}\sum_{J\in\mathcal{D}_{k}}\sum_{I\in\mathcal{G}\left[
U\right]  :\ \left(  I,J\right)  \in\mathcal{P}_{m}^{k,d}}\left\vert
I\right\vert ^{\frac{p}{2}}\left(  \int_{I_{\eta}}\left\vert \bigtriangleup
_{I;\kappa}^{n-1,\eta}f\left(  x\right)  \right\vert ^{2}dx\right)  ^{\frac
{p}{2}}\\
& \approx2^{m\left[  \frac{3}{2}pn-\left(  p+n\right)  \right]  }2^{kn}%
\sum_{I\in\mathcal{G}\left[  U\right]  }\left(  \sum_{J\in\mathcal{D}%
_{k}:\ \left(  I,J\right)  \in\mathcal{P}_{m}^{k,d}}1\right)  \left\vert
I\right\vert ^{p}\left(  \frac{1}{\left\vert I_{\eta}\right\vert }%
\int_{I_{\eta}}\left\vert \bigtriangleup_{I;\kappa}^{n-1,\eta}f\left(
x\right)  \right\vert ^{2}dx\right)  ^{\frac{p}{2}},
\end{align*}
where by the extension of (\ref{card m=0}) to $m\geq1$,%
\[
\sum_{J\in\mathcal{D}_{k}:\ \left(  I,J\right)  \in\mathcal{P}_{m}^{k,d}%
}1\approx2^{m\left(  n-1\right)  }2^{-kn}\left\vert \mathcal{K}_{d}\left(
I\right)  \right\vert \approx2^{m\left(  n-1\right)  }2^{-kn}2^{dn}\left(
\frac{1}{\left\vert I\right\vert }\right)  ^{\frac{n+1}{n-1}}.
\]
Thus we have%
\begin{align*}
\Gamma_{1}^{p}  & \lesssim2^{m\left[  \frac{3}{2}pn-\left(  p+n\right)
\right]  }2^{kn}2^{m\left(  n-1\right)  }2^{-kn}2^{dn}\sum_{I\in
\mathcal{G}\left[  U\right]  }\left(  \frac{1}{\left\vert I\right\vert
}\right)  ^{\frac{n+1}{n-1}}\left\vert I\right\vert ^{p}\left(  \frac
{1}{\left\vert I_{\eta}\right\vert }\int_{I_{\eta}}\left\vert \bigtriangleup
_{I;\kappa}^{n-1,\eta}\right\vert ^{2}\right)  ^{\frac{p}{2}}\\
& =2^{m\left[  \frac{3}{2}pn-\left(  p+1\right)  \right]  }2^{dn}\sum
_{I\in\mathcal{G}\left[  U\right]  }\left\vert I\right\vert ^{p-\frac
{n+1}{n-1}-1}\left(  \frac{1}{\left\vert I_{\eta}\right\vert }\int_{I_{\eta}%
}\left\vert \bigtriangleup_{I;\kappa}^{n-1,\eta}f\right\vert ^{2}\right)
^{\frac{p}{2}}\mathbf{1}_{I}\left(  x\right)  dx\\
& \lesssim2^{m\left[  \frac{3}{2}pn-\left(  p+1\right)  \right]  }2^{dn}%
\int_{\mathbb{R}^{n-1}}\sum_{I\in\mathcal{G}\left[  U\right]  }\left(
\frac{1}{\left\vert I_{\eta}\right\vert }\int_{I_{\eta}}\left\vert
\bigtriangleup_{I;\kappa}^{n-1,\eta}f\right\vert ^{2}\mathbf{1}_{I}\left(
x\right)  \right)  ^{\frac{p}{2}}dx,
\end{align*}
if $p\geq\frac{2n}{n-1}$, and then using $p\geq2$ and the Fefferman Stein
vector valued inequality, we can continue with%
\begin{align*}
\Gamma_{1}^{p}  & \lesssim2^{m\left[  \frac{3}{2}pn-\left(  p+1\right)
\right]  }2^{dn}\int_{\mathbb{R}^{n-1}}\left(  \sum_{I\in\mathcal{G}\left[
U\right]  }\left(  M\left\vert \bigtriangleup_{I;\kappa}^{n-1,\eta
}f\right\vert ^{2}\right)  \left(  x\right)  \right)  ^{\frac{p}{2}}dx\\
& \lesssim2^{m\left[  \frac{3}{2}pn-\left(  p+1\right)  \right]  }2^{dn}%
\int_{\mathbb{R}^{n-1}}\left(  \sum_{I\in\mathcal{G}\left[  U\right]
}\left\vert \bigtriangleup_{I;\kappa}^{n-1,\eta}f\right\vert ^{2}\left(
x\right)  \right)  ^{\frac{p}{2}}dx\lesssim2^{m\left[  \frac{3}{2}pn-\left(
p+1\right)  \right]  }2^{dn}\left\Vert f\right\Vert _{L^{p}}^{p}\ .
\end{align*}
Altogether then we have%
\begin{align*}
\left\vert \mathsf{B}_{\operatorname*{disjoint}}^{k,d,m}\left(  f,g\right)
\right\vert  & \lesssim2^{-\left\vert k\right\vert \kappa}2^{-N_{2}\left(
m+d\right)  }\Gamma_{1}\Gamma_{2}\lesssim2^{-\left\vert k\right\vert \kappa
}2^{-N_{2}\left(  m+d\right)  }2^{m\left[  \frac{3}{2}pn-\left(  p+1\right)
\right]  }2^{dn}\left\Vert f\right\Vert _{L^{p}}\left\Vert g\right\Vert
_{L^{p^{\prime}}}\\
& =2^{-\left\vert k\right\vert \kappa}2^{-\left(  N_{2}-\frac{3}{2}pn+\left(
p+1\right)  \right)  m}2^{-\left(  N_{2}-n\right)  d}2^{dn}\left\Vert
f\right\Vert _{L^{p}}\left\Vert g\right\Vert _{L^{p^{\prime}}}\leq
2^{-\left\vert k\right\vert \delta}2^{-\delta m}2^{-\delta d}\left\Vert
f\right\Vert _{L^{p}}\left\Vert g\right\Vert _{L^{p^{\prime}}}\ ,
\end{align*}
for $d\geq0$ and $p\geq\frac{2n}{n-1}$, so%
\[
\sum_{k\in\mathbb{Z}}\sum_{d=0}^{\infty}\sum_{m=1}^{\infty}\left\vert
\mathsf{B}_{\operatorname*{disjoint}}^{k,d,m}\left(  f,g\right)  \right\vert
\lesssim\sum_{k\in\mathbb{Z}}\sum_{d=0}^{\infty}\sum_{m=1}^{\infty
}2^{-\left\vert k\right\vert \delta}2^{-\delta m}2^{-\delta d}\left\Vert
f\right\Vert _{L^{p}}\left\Vert g\right\Vert _{L^{p^{\prime}}}\lesssim
\left\Vert f\right\Vert _{L^{p}}\left\Vert g\right\Vert _{L^{p^{\prime}}}\ .
\]

\subsection{Upper distal subforms with $d\geq0$}

We can obtain similar estimates for the upper $\operatorname*{distal}$ form,
by treating this form as the sum over pairs $\left(  I,J\right)  $ with $J$ in
the `missing sector', i.e. by setting $m=s$ in the corresponding disjoint form
estimates, as we now do. Indeed, recall that in (\ref{kmd decay}) and
(\ref{kmd decay'}) above we showed that
\[
\left\vert \left\langle Th_{I;\kappa}^{n-1,\eta},h_{J;\kappa}^{n,\eta
}\right\rangle \right\vert \lesssim2^{-\left\vert k\right\vert \min\left\{
N_{1},\kappa\right\}  }2^{-N_{2}\left(  m+d\right)  }\sqrt{\left\vert
I\right\vert \left\vert J\right\vert },
\]
for $\left(  I,J\right)  \in\mathcal{P}_{m}^{k,d}$, $k\in\mathbb{N}$ and
$d\geq0$. The same arguments, when applied to $\left(  I,J\right)
\in\mathcal{X}^{k,d}$, yield%

\[
\left\vert \left\langle Th_{I;\kappa}^{n-1,\eta},h_{J;\kappa}^{n,\eta
}\right\rangle \right\vert \lesssim2^{-\left\vert k\right\vert \min\left\{
N_{1},\kappa\right\}  }2^{-N\left(  s+d\right)  }\sqrt{\left\vert I\right\vert
\left\vert J\right\vert }\lesssim2^{-\left\vert k\right\vert N_{1}}%
2^{-Nd}\sqrt{\left\vert I\right\vert \left\vert J\right\vert },
\]
for $\left(  I,J\right)  \in\mathcal{X}^{k,d}$, $k\in\mathbb{N}$ and $d\geq0
$. Then the Alpert square function argument in the previous subsubsection
applies to give%
\[
\sum_{k\in\mathbb{Z}}\sum_{d=0}^{\infty}\left\vert \mathsf{B}%
_{\operatorname*{distal}}^{k,d}\left(  f,g\right)  \right\vert \lesssim
\sum_{k\in\mathbb{Z}}\sum_{d=0}^{\infty}2^{-\left\vert k\right\vert \delta
}2^{-\delta d}\left\Vert f\right\Vert _{L^{p}}\left\Vert g\right\Vert
_{L^{p^{\prime}}}\lesssim\left\Vert f\right\Vert _{L^{p}}\left\Vert
g\right\Vert _{L^{p^{\prime}}}\ ,
\]
for some $\delta>0$.

\subsection{Wrapup}

If we define%
\begin{align*}
\left\vert \mathsf{B}_{\operatorname*{disjoint}}^{\operatorname*{upper}%
}\right\vert \left(  f,g\right)   & \equiv\sum_{m=1}^{\infty}\sum_{\left(
I,J\right)  \in\mathcal{P}_{m}:\ \ell\left(  I\right)  ^{2}%
\operatorname*{dist}\left(  0,J\right)  \geq1}\left\vert \left\langle
T\bigtriangleup_{I;\kappa}^{n-1,\eta}f,\bigtriangleup_{J;\kappa}^{n,\eta
}g\right\rangle \right\vert ,\\
\left\vert \mathsf{B}_{\operatorname*{distal}}^{\operatorname*{upper}%
}\right\vert \left(  f,g\right)   & \equiv\sum_{\left(  I,J\right)
\in\mathcal{X}:\ \ell\left(  I\right)  ^{2}\operatorname*{dist}\left(
0,J\right)  \geq1}\left\vert \left\langle T\bigtriangleup_{I;\kappa}%
^{n-1,\eta}f,\bigtriangleup_{J;\kappa}^{n,\eta}g\right\rangle \right\vert ,
\end{align*}
in which the absolute values are taken inside the sums, we have proved both%
\begin{equation}
\left\vert \mathsf{B}_{\operatorname*{disjoint}}^{\operatorname*{upper}%
}\right\vert \left(  f,g\right)  \lesssim\left\Vert f\right\Vert _{L^{p}%
}\left\Vert g\right\Vert _{L^{p^{\prime}}},\ \ \ \ \ \,\text{for }p>\frac
{2n}{n-1},\label{strong upper disjoint}%
\end{equation}
and%
\begin{equation}
\left\vert \mathsf{B}_{\operatorname*{distal}}^{\operatorname*{upper}%
}\right\vert \left(  f,g\right)  \lesssim\left\Vert f\right\Vert _{L^{p}%
}\left\Vert g\right\Vert _{L^{p^{\prime}}},\ \ \ \ \ \,\text{for }p>\frac
{2n}{n-1}.\label{strong upper distal}%
\end{equation}

\section{Control of the $\operatorname*{lower}\operatorname*{disjoint}$ and
$\operatorname*{lower}\operatorname*{distal}$ forms}

Momentarily fix $s\in\mathbb{N}$. Let $\left\{  D_{i}\right\}  _{i=1}^{M}$ be
the set of dyadic cubes of side length $2^{2s+1}$ such that $0\in3D_{i}$. Then
$M\leq C_{n}$ and
\[
B\left(  0,2^{2s}\right)  \subset D_{\ast}\equiv\bigcup_{i=1}^{M}D_{i}.
\]
In this section we bundle the \emph{lower} disjoint and distal forms together,
and control their sum by bounding the form
\[
\mathsf{B}^{\operatorname*{lower}}\left(  f,g\right)  \equiv\sum_{s=1}%
^{\infty}\mathsf{B}_{s}^{\operatorname*{lower}}\left(  f,g\right)  ,
\]
where%
\[
\mathsf{B}_{s}^{\operatorname*{lower}}\left(  f,g\right)  \equiv\sum_{i=1}%
^{M}\sum_{\left(  I,J\right)  \in\mathcal{G}_{s}\left[  U\right]
\times\mathcal{D}_{2s}\left[  D_{i}\right]  }\left\langle T\bigtriangleup
_{I;\kappa}^{n-1,\eta}f,\bigtriangleup_{J;\kappa}^{n,\eta}g\right\rangle .
\]
The form $\mathsf{B}^{\operatorname*{lower}}\left(  f,g\right)  $ turns out to
include more pairs $\left(  I,J\right)  $ than occur in the sum $\mathsf{B}%
_{\operatorname*{disjoint}}^{\operatorname*{lower}}\left(  f,g\right)
+\mathsf{B}_{\operatorname*{distal}}^{\operatorname*{lower}}\left(
f,g\right)  $\ defined in (\ref{def up disj}) and (\ref{def up dist}), but the
resulting overcounting is inconsequential because the sum of
the\emph{\ moduli} $\left\vert \left\langle T\bigtriangleup_{I;\kappa
}^{n-1,\eta}f,\bigtriangleup_{J;\kappa}^{n,\eta}g\right\rangle \right\vert $
of the inner products for the overcounted pairs has already been controlled
without using probability in the previous sections. We fix $D\in\left\{
D_{i}\right\}  _{i=1}^{M}$ for the moment and consider just the form%
\[
\mathsf{B}_{s,D}^{\operatorname*{lower}}\left(  f,g\right)  \equiv\sum
_{i=1}^{M}\sum_{\left(  I,J\right)  \in\mathcal{G}_{s}\left[  U\right]
\times\mathcal{D}_{2s}\left[  D\right]  }\left\langle T\bigtriangleup
_{I;\kappa}^{n-1,\eta}f,\bigtriangleup_{J;\kappa}^{n,\eta}g\right\rangle \ ,
\]
where for convenience we assume that $B\left(  0,2^{2s}\right)  \subset D$.

Now we decompose the collection of pairs $\left(  I,J\right)  $ arising in
$\mathsf{B}_{s,D}^{\operatorname*{lower}}\left(  f,g\right)  $ by
\begin{align*}
& \mathcal{G}_{s}\left[  U\right]  \times\mathcal{D}_{2s}\left[  D\right]
=\bigcup_{w=0}^{s}\bigcup_{r=0}^{w}\mathcal{L}_{s,w,r},\\
\mathcal{L}_{s,r}  & \equiv\left\{  \left(  I,J\right)  \in\mathcal{G}%
_{s}\left[  U\right]  \times\mathcal{D}_{2s}\left[  D\right]  :J\subset T%
_{s}^{I}\left[  r\right]  \right\}  ,\ \ \ \ \ 0\leq r\leq s,\\
\mathcal{L}_{s,w,r}  & \equiv\left\{  \left(  I,J\right)  \in\mathcal{G}%
_{s}\left[  U\right]  \times\mathcal{D}_{2s}\left[  D\right]  :J\subset
P_{s,w}^{I}\left[  r\right]  \right\}  ,\ \ \ \ \ 0\leq r\leq w<s,
\end{align*}
where $T_{s}^{I}\left[  r\right]  $ and $P_{s,w}^{I}\left[  r\right]  $ are
tubes and pipes respectively, that are defined in the subsections below. Then
we will control the corresponding subforms,%
\begin{align*}
\mathsf{B}_{s,r,D}^{\operatorname*{lower}}\left(  f,g\right)   & \equiv
\sum_{\left(  I,J\right)  \in\mathcal{L}_{s,r}}\left\langle T\bigtriangleup
_{I;\kappa}^{n-1,\eta}f,\bigtriangleup_{J;\kappa}^{n,\eta}g\right\rangle \ ,\\
\mathsf{B}_{s,w,r,D}^{\operatorname*{lower}}\left(  f,g\right)   & \equiv
\sum_{\left(  I,J\right)  \in\mathcal{L}_{s,w,r}}\left\langle T\bigtriangleup
_{I;\kappa}^{n-1,\eta}f,\bigtriangleup_{J;\kappa}^{n,\eta}g\right\rangle \ ,
\end{align*}
and add in the parameters $r$ and $w$ to control the lower form
\begin{equation}
\mathsf{B}_{s}^{\operatorname*{lower}}\left(  f,g\right)  \equiv\sum_{i=1}%
^{M}\sum_{r=0}^{s}\left\{  \mathsf{B}_{s,r,D}^{\operatorname*{lower}}\left(
f,g\right)  +\sum_{w=0}^{r}\mathsf{B}_{s,w,r,D}^{\operatorname*{lower}}\left(
f,g\right)  \right\}  ,\label{decomp B^lower}%
\end{equation}
by%
\begin{equation}
\mathbb{E}_{\mathcal{G}_{s}\left[  U\right]  }^{\mu}\left\vert \mathsf{B}%
_{s}^{\operatorname*{lower}}\left(  \left(  \mathcal{A}_{\mathbf{a}}%
\mathsf{Q}_{U}^{s}\right)  ^{\spadesuit}f,\mathsf{P}_{2s}\left[  D\right]
g\right)  \right\vert \lesssim2^{-\varepsilon_{p,n}s}\left\Vert f\right\Vert
_{L^{p}\left(  U\right)  }\left\Vert g\right\Vert _{L^{p^{\prime}}\left(
D\right)  },\label{cont by}%
\end{equation}
as well as the stronger average norm estimate,%
\[
\mathbb{E}_{2^{\mathcal{G}\left[  U\right]  }}^{\mu}\left\Vert T\left(
\mathcal{A}_{\mathbf{a}}\mathsf{Q}_{U}^{s}\right)  ^{\spadesuit}f\right\Vert
_{L^{p}\left(  A_{+}\left(  0,2^{2s-w}\right)  \right)  }\lesssim
2^{-\varepsilon_{n,p}s}\left\Vert f\right\Vert _{L^{p}\left(  U\right)  }%
^{p}\ ,\ \ \ \ \ \text{for }p>\frac{2n}{n-1}.
\]
Note that when averaging over the family of `martingale transforms' $T\left(
\mathcal{A}_{\mathbf{a}}\mathsf{Q}_{U}^{s}\right)  ^{\spadesuit}f$, it makes
no difference whether we use $\mathbb{E}_{2^{\mathcal{G}\left[  U\right]  }%
}^{\mu}$ or $\mathbb{E}_{2^{\mathcal{G}_{s}\left[  U\right]  }}^{\mu}$.

Before turning to the details of these estimates, we discuss in the next
subsection the problematic \emph{resonance} that plagues the lower form
$\mathsf{B}^{\operatorname*{lower}}\left(  f,g\right)  $. The details
themselves are found in the second and third subsections using the `pipe' decomposition.

\subsection{Resonance in the lower form}

Note that for fixed $\xi\in\mathbb{R}^{n}$, the wavelength of the oscillation
of the function $x\rightarrow e^{-i\Phi\left(  x\right)  \cdot\xi}$ is roughly
$\frac{1}{\left\vert \xi\right\vert }\approx\frac{\ell\left(  I\right)  ^{2}%
}{2^{d}}$, while the depth of the patch of the sphere $\Phi\left(  I\right)  $
in the direction toward $\xi$ is roughly $\ell\left(  I\right)  \sin
\theta\approx2^{m}\ell\left(  I\right)  ^{2}$. Thus we will have
\emph{oscillation} along the patch $\Phi\left(  I\right)  $ if and only if the
wavelength $\frac{\ell\left(  I\right)  ^{2}}{2^{d}}$ is less than the depth
$2^{m}\ell\left(  I\right)  ^{2}$, i.e. $m\gg\left\vert d\right\vert $, while
we will have \emph{smoothness} along the patch if and only if $m\ll\left\vert
d\right\vert $.

On the other hand, for $\xi\in J$, the wavelength of the oscillation of the
function $\xi\rightarrow e^{-i\Phi\left(  x\right)  \cdot\xi}$ is roughly
$\frac{1}{\cos\measuredangle\left(  \Phi\left(  x\right)  ,c_{J}\right)
}\approx1$ (unless the unit vectors $\frac{c_{J}}{\left\vert c_{J}\right\vert
}$ and $\Phi\left(  c_{I}\right)  $ are nearly orthogonal), while the depth of
the cube in the diretion of $\xi$ is roughly $\ell\left(  J\right)  =2^{k}$.
Thus we will have \emph{oscillation} along the cube $J$ if and only if the
wavelength $1$ is less than the depth $2^{k}$, i.e. $k\gg0$, while we will
have \emph{smoothness} along the cube if and only if $k\ll0$.

\begin{conclusion}
The most problematic case occurs when $d<0$ and both $m\approx\left\vert
d\right\vert $ and $k\approx0$.
\end{conclusion}

We begin by illustrating our approach to controlling resonance in the most
problematic of the subcases in the next subsection, and it is here that we
require the use of \emph{probability} and an interpolation argument. In such
instances where we need to use expectation over `martingale transforms', we
will also need to apply this expectation to \emph{norms} rather than bilinear
forms, which must be addressed.

In order to handle cases with partial resonance in the subsequent subsection,
we introduce a different decomposition of the disjoint form into resonant
pipes that respects resonance when $d<0$, and then apply principles of decay
along with\ probability and the interpolation argument to control these
remaining subcases.

But first we look at the extreme resonant case and show how expectation plays
a role in controlling this simple case before tackling the general case. We
will also show why the annular cone decomposition used in $\mathcal{P}_{m}$
must be replaced by a pipe decomposition, namely because pipes respect
resonance while sectors do not.

\subsubsection{The extreme resonant case}

The most resonant of\ the disjoint subforms is $\mathsf{B}%
_{\operatorname*{disjoint}}^{k,d,m}\left(  f,g\right)  =\mathsf{B}%
_{\operatorname*{disjoint}}^{0,-m,m}\left(  f,g\right)  $ when $\ell\left(
J\right)  =1$ and $d=-m$. Fix $\left(  I,J\right)  \in\mathcal{P}_{m}^{0,-m}$
and let $J_{\max}^{m}\left[  I\right]  $ be any dyadic cube in $\mathcal{D}$
satisfying the following conditions,
\begin{align}
\ell\left(  J_{\max}^{m}\left[  I\right]  \right)   & =\frac{1}{\ell\left(
I\right)  },\label{def Jmax}\\
\operatorname*{dist}\left(  0,J_{\max}^{m}\left[  I\right]  \right)   &
\approx\frac{2^{-m}}{\ell\left(  I\right)  ^{2}},\nonumber\\
\pi_{\tan}J_{\max}^{m}\left[  I\right]   & \subset2^{m+1}I\setminus
2^{m-1}I,\nonumber\\
\ell\left(  \pi_{\tan}J_{\max}^{m}\left[  I\right]  \right)   & =2^{m}%
\ell\left(  I\right)  ,\nonumber
\end{align}
where $\ell\left(  \pi_{\tan}J_{\max}^{m}\left[  I\right]  \right)  $ denotes
the diameter of the quasicube $\pi_{\tan}J_{\max}^{m}\left[  I\right]  $. If
$\ell\left(  I\right)  =2^{-s}$ with $s\geq m$ (which follows from
(\ref{def Jmax}) and $\ell\left(  \pi_{\tan}J_{\max}^{m}\left[  I\right]
\right)  \lesssim1$), then we have%
\[
\ell\left(  J_{\max}^{m}\left[  I\right]  \right)  =2^{s}%
,\ \ \ \operatorname*{dist}\left(  0,J_{\max}^{m}\left[  I\right]  \right)
\approx2^{2s-m},\ \ \ \ell\left(  \pi_{\tan}J_{\max}^{m}\left[  I\right]
\right)  =\frac{\ell\left(  J_{\max}^{m}\left[  I\right]  \right)
}{\operatorname*{dist}\left(  0,J_{\max}^{m}\left[  I\right]  \right)
}=2^{m-s}.
\]

At this point we note that the cubes $J_{\max}^{m}\left[  I\right]  $ are
essentially the maximal dyadic cubes that fit inside the annular conic region
given by (\ref{def Jmax}), and hence there are roughly $\frac
{\operatorname*{dist}\left(  0,J_{\max}^{m}\left[  I\right]  \right)  }%
{\ell\left(  J_{\max}^{m}\left[  I\right]  \right)  }\approx\frac{2^{2s-m}%
}{2^{s}}\approx2^{s-m}$ such cubes stacked away from the origin. We enumerate
these cubes by $\left\{  J_{\max}^{m,t}\left[  I\right]  \right\}
_{t=1}^{c2^{s-m}}$ and let
\begin{equation}
J_{\max}^{m,\ast}\left[  I\right]  \equiv%
%TCIMACRO{\dbigcup \limits_{t=1}^{c2^{s-m}}}%
%BeginExpansion
{\displaystyle\bigcup\limits_{t=1}^{c2^{s-m}}}
%EndExpansion
J_{\max}^{m,t}\left[  I\right] \label{def quasi}%
\end{equation}
denote their union. Thus $J_{\max}^{m,\ast}\left[  I\right]  $ is a
\emph{quasirectangle} of `length' roughly $\operatorname*{dist}\left(
0,J_{\max}^{m}\left[  I\right]  \right)  \approx2^{2s-m}$, and `width' roughly
$2^{s}$ - we say `quasi' because $J_{\max}^{m,\ast}\left[  I\right]  $ is a
union of dyadic cubes $J_{\max}^{m,t}\left[  I\right]  $ staggered in the
direction of the annular conic region. Note that there are at most $C_{n}$
such quasirectangles $J_{\max}^{m,\ast}\left[  I\right]  $ associated to any
given cube $I\in\mathcal{G}\left[  S\right]  $.

\begin{remark}
Since quasirectangles do not respect resonance (which varies along the
quasirectangle), they will not play a part in the proof going forward, but
will instead be replaced by pipes in the next subsection.
\end{remark}

If $\phi\equiv\measuredangle\left(  c_{J_{\max}^{m}\left[  I\right]  }%
-\Phi\left(  c_{I}\right)  ,\Phi\left(  c_{I}\right)  ^{\perp}\right)  $ is
the angle between the vector $c_{J_{\max}^{m}\left[  I\right]  }-\Phi\left(
c_{I}\right)  $ and the unit vector $\Phi\left(  c_{I}\right)  $, and if
$\theta\equiv\measuredangle\left(  \frac{c_{J_{\max}^{m}\left[  I\right]  }%
}{\left\vert c_{J_{\max}^{m}\left[  I\right]  }\right\vert },\Phi\left(
c_{I}\right)  \right)  $ is the angle between the unit vectors $\frac
{c_{J_{\max}^{m}\left[  I\right]  }}{\left\vert c_{J_{\max}^{m}\left[
I\right]  }\right\vert }$ and $\Phi\left(  c_{I}\right)  $, then
$\theta\approx2^{m}\ell\left(  I\right)  $ and we have%
\begin{align}
\frac{\pi}{2}-\phi & =\measuredangle\left(  c_{J_{\max}^{m}\left[  I\right]
}-\Phi\left(  c_{I}\right)  ,\Phi\left(  c_{I}\right)  \right) \label{angle}\\
& =\measuredangle\left(  c_{J_{\max}^{m}\left[  I\right]  }-\frac{c_{J_{\max
}^{m}\left[  I\right]  }}{\left\vert c_{J_{\max}^{m}\left[  I\right]
}\right\vert },\Phi\left(  c_{I}\right)  \right)  +\measuredangle\left(
c_{J_{\max}^{m}\left[  I\right]  }-\Phi\left(  c_{I}\right)  ,c_{J_{\max}%
^{m}\left[  I\right]  }-\frac{c_{J_{\max}^{m}\left[  I\right]  }}{\left\vert
c_{J_{\max}^{m}\left[  I\right]  }\right\vert }\right) \nonumber\\
& =\measuredangle\left(  \frac{c_{J_{\max}^{m}\left[  I\right]  }}{\left\vert
c_{J_{\max}^{m}\left[  I\right]  }\right\vert },\Phi\left(  c_{I}\right)
\right)  +O\left(  \frac{\left\vert \Phi\left(  c_{I}\right)  -\frac
{c_{J_{\max}^{m}\left[  I\right]  }}{\left\vert c_{J_{\max}^{m}\left[
I\right]  }\right\vert }\right\vert }{\left\vert c_{J_{\max}^{m}\left[
I\right]  }-\Phi\left(  c_{I}\right)  \right\vert }\right)  \approx2^{m}%
\ell\left(  I\right)  +\frac{2^{m}\ell\left(  I\right)  }{\operatorname*{dist}%
\left(  0,J_{\max}^{m}\left[  I\right]  \right)  }\nonumber\\
& =2^{m}\ell\left(  I\right)  \left\{  1+\frac{1}{\operatorname*{dist}\left(
0,J_{\max}^{m}\left[  I\right]  \right)  }\right\}  \approx2^{m}\ell\left(
I\right)  \left\{  1+2^{m-2s}\right\}  \approx2^{m}\ell\left(  I\right)
,\nonumber
\end{align}
since $s\geq m$. Thus it follows that there is neither oscillation nor
smoothness of the inner product
\[
\left\langle T\bigtriangleup_{I;\kappa}^{n-1,\eta}f,\bigtriangleup_{J;\kappa
}^{n,\eta}g\right\rangle =\int_{\mathbb{R}^{n}}\left\{  \int_{\mathbb{R}%
^{n-1}}\left\langle f,h_{I;\kappa}^{n-1,\eta}\right\rangle h_{I;\kappa
}^{n-1,\eta}\left(  x\right)  e^{i\Phi\left(  x\right)  \cdot\xi}dx\right\}
\bigtriangleup_{J;\kappa}^{n,\eta}g\left(  \xi\right)  d\xi
\]
in the integral over $I$ in braces, since the `tilted depth' of $\Phi\left(
I\right)  $ in the direction $\frac{\pi}{2}-\phi$ is given by%
\[
\operatorname*{tilted}\operatorname*{depth}\approx\ell\left(  I\right)
\cos\phi=\ell\left(  I\right)  \sin\left(  \frac{\pi}{2}-\phi\right)
\approx2^{m}\ell\left(  I\right)  ^{2},
\]
and so
\begin{equation}
\operatorname{wavelength}\approx\frac{1}{\operatorname*{dist}\left(
0,J_{\max}^{m}\left[  I\right]  \right)  }=2^{m}\ell\left(  I\right)
^{2}\approx\operatorname*{tilted}\operatorname*{depth}.\label{wavelength}%
\end{equation}
Of course there is neither oscillation nor smoothness in the integral over $J
$ either since $\ell\left(  J\right)  =1$ and the wavelength coming from the
sphere is approximately $\ell\left(  J\right)  =1$ as well.

Then $\left(  I,J\right)  \in\mathcal{P}_{m}^{0,-m}$ \emph{essentially} if and
only if $J\subset J_{\max}^{m,\ast}\left[  I\right]  $ and $\ell\left(
J\right)  =1$. There are roughly $\frac{1}{\ell\left(  I\right)  ^{n}}$ cubes
$J\subset J_{\max}^{m,t}\left[  I\right]  $ of side length $1$ for each $1\leq
t\leq c2^{s-m}$, and we may restrict our attention to the cubes $I$ having
side length $2^{-s}$ with $s\geq m$, that are contained in a cube $Q$ where
\begin{equation}
Q\subset S\text{ with }\ell\left(  Q\right)  \approx2^{m-s}\text{, such that
}J_{\max}^{m,\ast}\left[  I\right]  \approx J_{\max}^{m,\ast}\left[
I^{\prime}\right]  \text{ for all such cubes }I\subset Q.\label{bundle}%
\end{equation}
We also then set
\begin{equation}
Q^{\ast}\equiv\bigcup_{I\subset Q}J_{\max}^{m,\ast}\left[  I\right]
,\label{bundle'}%
\end{equation}
which is approximately equal to any of the $J_{\max}^{m,\ast}\left[  I\right]
$ taken individually, and thus $Q^{\ast}$ is a quasirectangle of length
roughly $2^{2s-m}$, and width roughly $2^{s}$. Thus we have defined cube /
quasirectangle pairs $\left(  Q,Q^{\ast}\right)  $ which we now analyze a bit
further. Recall from (\ref{def Jmax}) that $\ell\left(  \pi_{\tan}Q^{\ast
}\right)  \approx2^{m}\ell\left(  I\right)  =2^{m-s}$.

We write%
\begin{equation}
\mathsf{Q}_{Q}^{s}g\equiv\sum_{I\in\mathsf{Q}_{Q}^{s}}\bigtriangleup
_{I;\kappa}^{n-1}g\text{ and }\mathsf{P}_{m,s}^{\eta,0,Q^{\ast}}g\equiv
\sum_{J\subset Q^{\ast}:\ \ell\left(  J\right)  =1}\bigtriangleup_{J;\kappa
}^{n,\eta}g,\label{def Qm and Pm}%
\end{equation}
and recalling that $\left(  \mathcal{A}_{\mathbf{a}}\mathsf{Q}_{Q}^{s}\right)
^{\spadesuit}=\left(  \mathcal{A}_{\mathbf{a}}\mathsf{Q}_{Q}^{s}\right)
^{S_{\kappa,\eta}}=S_{\kappa,\eta}\mathcal{A}_{\mathbf{a}}\mathsf{Q}_{Q}%
^{s}\left(  S_{\kappa,\eta}\right)  ^{-1}$ is the conjugation of
$\mathcal{A}_{\mathbf{a}}\mathsf{Q}_{Q}^{s}$ by $S_{\kappa,\eta}$, we claim
that%
\begin{align}
\mathbb{E}_{2^{\mathcal{D}}}^{\mu}\left\vert \sum_{m=1}^{\infty}\sum
_{s=m}^{\infty}\sum_{Q}\left\langle T\left(  \mathcal{A}_{\mathbf{a}%
}\mathsf{Q}_{Q}^{s}\right)  ^{\spadesuit}f,\mathsf{P}_{m,s}^{\eta,0,Q^{\ast}%
}g\right\rangle \right\vert  & \leq\sum_{m=1}^{\infty}\sum_{s=m}^{\infty}%
\sum_{Q}\mathbb{E}_{2^{\mathcal{D}}}^{\mu}\left\vert \left\langle T\left(
\mathcal{A}_{\mathbf{a}}\mathsf{Q}_{Q}^{s}\right)  ^{\spadesuit}%
,\mathsf{P}_{m,s}^{\eta,0,Q^{\ast}}g\right\rangle \right\vert \label{no osc}\\
& \lesssim\left\Vert f\right\Vert _{L^{p}}\left\Vert g\right\Vert
_{L^{p^{\prime}}}\ ,\ \ \ \ \ p\geq\frac{2n}{n-1},\nonumber
\end{align}
where we recall that the parameters $k$ and $d$ are fixed at $k=0$ and $d=-m$.
It is here in (\ref{no osc}) that our argument requires averaging over all
involutive smooth Alpert multipliers on the left hand side of the inequality.
Note that we have replaced the large projection $\mathsf{Q}_{S}$ with the
smaller projections $\mathsf{Q}_{Q}^{s}$ for $Q\subset S$.

\subsubsection{The interpolation argument\label{Subinterp}}

In order to illustrate the probabilistic methods in a relatively simple
situation, we first prove (\ref{no osc}) when the sum is taken only over
$s=m\in\mathbb{N}$, so that both $Q$ and $Q^{\ast}$ reduce to cubes of side
length roughly $1$. Thus there are only a bounded number of such cube / cube
pairs $\left(  Q,Q^{\ast}\right)  $, which for convenience we treat as a
single pair $\left(  Q_{0},Q_{0}^{\ast}\right)  $. We claim,%
\begin{equation}
\mathbb{E}_{2^{\mathcal{G}}}^{\mu}\left\vert \sum_{m=1}^{\infty}\left\langle
T\left(  \mathcal{A}_{\mathbf{a}}\mathsf{Q}_{Q_{0}}^{m}\right)  ^{\spadesuit
}f,\mathsf{P}_{m,m}^{\eta,0,Q_{0}^{\ast}}g\right\rangle \right\vert
\lesssim\left\Vert f\right\Vert _{L^{p}}\left\Vert g\right\Vert _{L^{p^{\prime
}}}\ ,\ \ \ \ \ p>\frac{2n}{n-1}.\label{no osc s=m}%
\end{equation}
We note that the expectation $\mathbb{E}_{2^{\mathcal{G}}}^{\mu}$ will
circumvent some of the geometric $L^{4}$ arguments that go back to Fefferman
\cite{Fef} (see also \cite{Bou}, \cite{Gut} and \cite{Tao4}). Recall that we
are in the case $d=-m$, and that%
\[
\mathsf{Q}_{Q_{0}}^{m}g=\sum_{I\subset Q_{0}:\ \ell\left(  I\right)  =2^{-m}%
}\bigtriangleup_{I;\kappa}^{n-1}g\text{ and }\mathsf{P}_{m,m}^{\eta
,0,Q_{0}^{\ast}}g\equiv\sum_{J\subset Q_{0}^{\ast}:\ \ell\left(  J\right)
=1}\bigtriangleup_{J;\kappa}^{n,\eta}g,
\]
where $Q_{0}$ is a cube in $\mathbb{R}^{n-1}$ centered at the origin with side
length approximately $1$, and $Q_{0}^{\ast}$ is a cube in $\mathbb{R}^{n}$ at
distance $2^{m}$ from the origin with side length approximately $2^{m}$, and
such that $\operatorname*{dist}\left(  Q_{0},\pi_{\tan}Q_{0}^{\ast}\right)
\approx1$. We will again use $\widehat{\varphi}$ to denote the Fourier
transform of $\varphi$. Thus we must estimate the average of the moduli of the
inner products,%
\begin{align}
& \left\langle T\left(  \mathcal{A}_{\mathbf{a}}\mathsf{Q}_{Q_{0}}^{m}\right)
^{\spadesuit}f,\mathsf{P}_{m,m}^{\eta,0,Q_{0}^{\ast}}g\right\rangle
=\left\langle T\sum_{I\in\mathcal{G}_{m}\left[  Q_{0}\right]  }a_{I}%
\bigtriangleup_{I;\kappa}^{n-1,\eta}f,\sum_{J\subset Q_{0}^{\ast}%
:\ \ell\left(  J\right)  =1}\bigtriangleup_{J;\kappa}^{n,\eta}g\right\rangle
\label{must est}\\
& =\sum_{I\in\mathcal{G}_{m}\left[  Q_{0}\right]  }\sum_{J\subset Q_{0}^{\ast
}:\ \ell\left(  J\right)  =1}\int_{S}\int_{\mathbb{R}^{n}}e^{-i\Phi\left(
x\right)  \cdot\xi}a_{I}\bigtriangleup_{I;\kappa}^{n-1,\eta}f\left(  x\right)
\bigtriangleup_{J;\kappa}^{n,\eta}g\left(  \xi\right)  dxd\xi\nonumber\\
& =\int_{\mathbb{R}^{n}}\left\{  \int e^{-iz\cdot\xi}\sum_{I\in\mathcal{G}%
_{m}\left[  Q_{0}\right]  }a_{I}\bigtriangleup_{I;\kappa}^{n-1,\eta}f\left(
\Phi^{-1}\left(  z\right)  \right)  \partial\Phi^{-1}\left(  z\right)
dz\right\}  \sum_{J\subset Q_{0}^{\ast}:\ \ell\left(  J\right)  =1}%
\bigtriangleup_{J;\kappa}^{n,\eta}g\left(  \xi\right)  d\xi\nonumber\\
& \equiv\int_{\mathbb{R}^{n}}\widehat{f_{\mathbf{a},\Phi}}\left(  \xi\right)
g_{m}\left(  \xi\right)  d\xi,\nonumber
\end{align}
where $\widehat{f_{\mathbf{a},\Phi}}$ denotes the Fourier transform of
$f_{\mathbf{a},\Phi}$ as in Section \ref{Sub interp}, and
\begin{align*}
g_{m}\left(  \xi\right)   & \equiv\sum_{J\subset Q_{0}^{\ast}:\ \ell\left(
J\right)  =1}\bigtriangleup_{J;\kappa}^{n,\eta}g\left(  \xi\right)
=\mathsf{P}_{m,m}^{\eta,0,Q_{0}^{\ast}}g\left(  \xi\right)  \ ,\\
f_{\mathbf{a},\Phi}\left(  z\right)   & \equiv\left(  \mathcal{A}_{\mathbf{a}%
}\mathsf{Q}_{Q_{0}}^{m}\right)  ^{\spadesuit}f\left(  \Phi^{-1}\left(
z\right)  \right)  \partial\Phi^{-1}\left(  z\right)  =\sum_{I\in
\mathcal{G}_{m}\left[  Q_{0}\right]  }a_{I}\bigtriangleup_{I;\kappa}%
^{n-1,\eta}f\left(  \Phi^{-1}\left(  z\right)  \right)  \partial\Phi
^{-1}\left(  z\right) \\
& =\sum_{I\in\mathcal{G}_{m}\left[  Q_{0}\right]  }a_{I}\left\langle
f,h_{I;\kappa}^{n-1,\eta}\right\rangle h_{I;\kappa}^{n-1,\eta}\left(
\Phi^{-1}\left(  z\right)  \right)  \partial\Phi^{-1}\left(  z\right)
\equiv\sum_{I\in\mathcal{G}_{m}\left[  Q_{0}\right]  }f_{\mathbf{a},\Phi}%
^{I}\left(  z\right)  \ ,
\end{align*}
and where the spherical measure $f_{\mathbf{a},\Phi}^{I}$ has mass roughly
$\left\vert \widehat{f}\left(  I\right)  \right\vert 2^{-m\left(  n-1\right)
}$ and is supported in $\mathbb{S}^{n-1}$.

The bound (\ref{no osc s=m}) now follows immediately from H\"{o}lder's
inequality and Proposition \ref{prop interp}, upon noting that $\mathsf{Q}%
_{S}^{s}$ in Proposition \ref{prop interp} is the projection $\mathsf{Q}%
_{Q_{0}}^{m}$ here. Indeed, from Proposition \ref{prop interp} we have%
\[
\sum_{m=1}^{\infty}\mathbb{E}_{2^{\mathcal{G}}}^{\mu}\left\Vert T\left(
\mathcal{A}_{\mathbf{a}}\mathsf{Q}_{Q_{0}}^{m}\right)  ^{\spadesuit
}f\right\Vert _{L^{p}\left(  \left\vert \varphi_{m}\right\vert ^{4}\right)
}\lesssim\sum_{m=1}^{\infty}2^{-m\varepsilon_{n,p}}\left\Vert f\right\Vert
_{L^{p}\left(  \left\vert \varphi_{m}\right\vert ^{4}\right)  }%
\]
and then in particular,
\begin{align*}
& \mathbb{E}_{2^{\mathcal{G}}}^{\mu}\left\vert \sum_{m=1}^{\infty}\left\langle
T\left(  \mathcal{A}_{\mathbf{a}}\mathsf{Q}_{Q_{0}}^{m}\right)  ^{\spadesuit
}f,\mathsf{P}_{m,m}^{\eta,0,Q_{0}^{\ast}}g\right\rangle \right\vert \leq
\sum_{m=1}^{\infty}\mathbb{E}_{2^{\mathcal{G}}}^{\mu}\left\Vert T\left(
\mathcal{A}_{\mathbf{a}}\mathsf{Q}_{Q_{0}}^{m}\right)  ^{\spadesuit
}f\right\Vert _{L^{p}\left(  \left\vert \varphi_{m}\right\vert ^{4}\right)
}\left\Vert \mathsf{P}_{m,m}^{\eta,0,Q_{0}^{\ast}}g\right\Vert _{L^{p^{\prime
}}\left(  \left\vert \varphi_{m}\right\vert ^{4}\right)  }\\
& \leq\sum_{m=1}^{\infty}2^{-m\varepsilon_{n,p}}\left\Vert f\right\Vert
_{L^{p}\left(  \left\vert \varphi_{m}\right\vert ^{4}\right)  }\left\Vert
g\right\Vert _{L^{p^{\prime}}\left(  \left\vert \varphi_{m}\right\vert
^{4}\right)  }\lesssim\left\Vert f\right\Vert _{L^{p}}\left\Vert g\right\Vert
_{L^{p^{\prime}}}\ ,\ \ \ \ \ \text{where }\varepsilon_{n,p}>0\text{ for
}p>\frac{2n}{n-1},m\in\mathbb{N}.
\end{align*}

But we can in fact obtain more. Define the smooth Alpert pseudoprojection
\begin{equation}
\mathsf{P}_{m,m}^{\eta,Q_{0}^{\ast}}g\equiv\sum_{k\in\mathbb{Z}}\sum_{J\subset
Q_{0}^{\ast}:\ \ell\left(  J\right)  =2^{k}}\bigtriangleup_{J;\kappa}^{n,\eta
}g,\label{k included}%
\end{equation}
where of course the restriction $J\subset Q_{0}^{\ast}$ means that $k\leq m$
in the sum above (contrast this with the restriction to $k=0$ in
$\mathsf{P}_{m,m}^{\eta,0,Q_{0}^{\ast}}g$). Then we have the stronger
inequality in which the sum over $k$ is included,%
\begin{align}
& \mathbb{E}_{2^{\mathcal{G}}}^{\mu}\left\vert \sum_{m=1}^{\infty}\left\langle
T_{S}\left(  \mathcal{A}_{\mathbf{a}}\mathsf{Q}_{Q_{0}}^{m}\right)
^{\spadesuit}f,\mathsf{P}_{m,m}^{\eta,Q_{0}^{\ast}}g\right\rangle \right\vert
\leq\sum_{m=1}^{\infty}\mathbb{E}_{2^{\mathcal{G}}}^{\mu}\left\Vert
T_{S}\left(  \mathcal{A}_{\mathbf{a}}\mathsf{Q}_{Q_{0}}^{m}\right)
^{\spadesuit}f\right\Vert _{L^{p}}\left\Vert \mathsf{P}_{m,m}^{\eta
,Q_{0}^{\ast}}g\right\Vert _{L^{p^{\prime}}}\label{no osc s=m,k}\\
& \leq\sum_{m=1}^{\infty}2^{-m\varepsilon_{p,n}}\left\Vert S_{\kappa,\eta
}\mathcal{A}_{\mathbf{a}}\mathsf{Q}_{Q_{0}}^{m}\left(  S_{\kappa,\eta}\right)
^{-1}f\right\Vert _{L^{p}}\left\Vert \mathsf{P}_{m,m}^{\eta,Q_{0}^{\ast}%
}g\right\Vert _{L^{p^{\prime}}}\lesssim\left\Vert f\right\Vert _{L^{p}%
}\left\Vert g\right\Vert _{L^{p^{\prime}}}\ ,\ \ \ \ \ p>\frac{2n}{n-1}%
,m\in\mathbb{N}.\nonumber
\end{align}

\subsection{The resonant pipe decomposition}

We now abandon the decomposition into annular cones parameterized by $m$, and
distances parameterized by $d$, since this decomposition does not respect
resonance in the inner products. Instead, we will use (\ref{decomp B^lower})
to decompose the lower form as%
\begin{align*}
& \mathsf{B}^{\operatorname*{lower}}\left(  f,g\right)  =\sum_{s=1}^{\infty
}\sum_{i=1}^{M}\sum_{r=0}^{s}\left\{  \mathsf{B}_{s,r,Q}%
^{\operatorname*{lower}}\left(  f,g\right)  +\sum_{w=0}^{r}\mathsf{B}%
_{s,w,r,Q}^{\operatorname*{lower}}\left(  f,g\right)  \right\} \\
& =\sum_{s=1}^{\infty}\sum_{i=1}^{M}\sum_{r=0}^{s}\left\{  \sum_{\left(
I,J\right)  \in\mathcal{L}_{s,r}}\left\langle T\bigtriangleup_{I;\kappa
}^{n-1,\eta}f,\bigtriangleup_{J;\kappa}^{n,\eta}g\right\rangle +\sum_{w=0}%
^{r}\sum_{\left(  I,J\right)  \in\mathcal{L}_{s,w,r}}\left\langle
T\bigtriangleup_{I;\kappa}^{n-1,\eta}f,\bigtriangleup_{J;\kappa}^{n,\eta
}g\right\rangle \right\}  \ ,
\end{align*}
where%
\begin{align*}
\mathcal{L}_{s,r}  & \equiv\left\{  \left(  I,J\right)  \in\mathcal{G}%
_{s}\left[  U\right]  \times\mathcal{D}_{2s}\left[  Q\right]  :J\subset T%
_{s}^{I}\left[  r\right]  \right\}  ,\ \ \ \ \ 0\leq r\leq s,\\
\mathcal{L}_{s,w,r}  & \equiv\left\{  \left(  I,J\right)  \in\mathcal{G}%
_{s}\left[  U\right]  \times\mathcal{D}_{2s}\left[  Q\right]  :J\subset
P_{s,w}^{I}\left[  r\right]  \right\}  ,\ \ \ \ \ 0\leq r\leq w<s.
\end{align*}
Thus for each $I\in\mathcal{G}_{s}\left[  U\right]  $, we are now decomposing
the set of cubes $J\in\mathcal{D}_{2s}\left[  Q\right]  $ into `truncated
tubes' $T_{s}^{I}\left[  r\right]  $ and `truncated pipes' $P_{s,w}^{I}\left[
r\right]  $, instead of the quasirectangles $J_{\max}^{m,\ast}\left[
I\right]  $ introduced in (\ref{def quasi}) above, using \emph{new} parameters
$w,r$ in place of $m,d$ above. The advantage of this new decomposition into
pipes is that it does indeed respect resonance.

In the remainder of this section, we will define the tubes $T_{s}^{I}\left[
r\right]  $ and pipes $P_{s,w}^{I}\left[  r\right]  $, and prove the
associated subform and norm estimates.

Fix $s\in\mathbb{N}$ and consider a cube $I\in G_{s}\left[  U\right]  $. Let
$\mathbf{u}_{n}^{I}$ be the unit outward normal to the sphere at the point
$\Phi\left(  c_{I}\right)  $, and let $\left(  \mathbf{u}^{I}\right)
^{\prime}=\left\{  \mathbf{u}_{1}^{I},...,\mathbf{u}_{n-1}^{I}\right\}  $ be
an orthonormal basis for the space $\left(  \mathbf{u}_{n}^{I}\right)
^{\perp}$ perpendicular to $\mathbf{u}_{n}^{I}$. We will use the coordinate
system $\left\{  \left(  \mathbf{u}^{I}\right)  ^{\prime},\mathbf{u}_{n}%
^{I}\right\}  $ in $\mathbb{R}^{n}$ in connection with the cube $I\in
G_{s}\left[  U\right]  $, so that as we vary $I\in G_{s}\left[  U\right]  $
the coordinate systems $\left\{  \left(  \mathbf{u}^{I}\right)  ^{\prime
},\mathbf{u}_{n}^{I}\right\}  $ rotate ($\operatorname*{Span}\left\{
\mathbf{u}_{n}^{I}\right\}  $ and $\operatorname*{Span}\left(  \mathbf{u}%
^{I}\right)  ^{\prime}$ are determined canonically under rotation, but not the
individual basis vectors $\mathbf{u}_{1}^{I},...,\mathbf{u}_{n-1}^{I}$).

For convenience in notation, we momentarily suppose without loss of generality
that $I=I_{0}\in G_{s}\left[  U\right]  $ is centered at the origin in $S$,
and consequently we can take $\left\{  \mathbf{u}_{1}^{I},...,\mathbf{u}%
_{n-1}^{I},\mathbf{u}_{n}^{I}\right\}  $ to be the standard orthonormal basis
$\left\{  \mathbf{e}_{1},...,\mathbf{e}_{n-1},\mathbf{e}_{n}\right\}  $ in
$\mathbb{R}^{n}$, and $\xi=\left(  \xi_{1},...,\xi_{n}\right)  =\left(
\xi^{\prime},\xi_{n}\right)  \in\mathbb{R}^{n}$ is the usual representation of
a point $\xi$ in $\mathbb{R}^{n}$. Then the pairs $\left(  I_{0},J\right)
\in\mathcal{G}\left[  U\right]  \times\mathcal{D}$ for which we have resonance
on both sides of the inner product, are precisely those satisfying
$\ell\left(  J\right)  \approx1$ and,%
\begin{align}
\frac{1}{\operatorname*{dist}\left(  0,J\right)  }  & \approx
\operatorname*{tilted}\operatorname*{depth}\approx2^{-s}\sin\theta
,\label{reso}\\
\text{i.e. }\left\vert \xi\right\vert  & \approx\frac{2^{s}}{\sin\theta}%
=2^{s}\frac{\left\vert \xi\right\vert }{\left\vert \xi^{\prime}\right\vert
},\ \ \ \ \text{ for }\xi\in J,\nonumber\\
\text{i.e. }2^{s-1}  & \leq\left\vert \xi^{\prime}\right\vert \leq
2^{s+1},\ \ \ \ \text{ for }\xi\in J,\nonumber
\end{align}
where $\theta$ is the angle $\xi$ makes with the positive $\xi_{n}$-axis. Thus
the union $P_{s}^{I_{0}}$ of the $J^{\prime}s$ satisfying $\ell\left(
J\right)  \approx1$ and (\ref{reso}) is essentially the difference of two
infinite tubes, namely the $\left(  2^{s+1}\times2^{s+1}\times\infty\right)
$-tube and the $\left(  2^{s-1}\times2^{s-1}\times\infty\right)  $-tube that
are oriented vertically with infinite length. We refer to $P_{s}^{I_{0}}$ as
the resonant $2^{s} $-pipe for $I_{0}$. In terms of the projection $\pi
_{\Phi\left(  c_{I_{0}}\right)  ^{\perp}}$ of $\mathbb{R}^{n}$ onto the
horizontal plane perpendicular to $\Phi\left(  c_{I_{0}}\right)  $, we have%
\[
P_{s}^{I_{0}}\approx\left\{  \xi\in\mathbb{R}^{n}:\operatorname*{dist}\left(
c_{I_{0}},\pi_{\Phi\left(  c_{I_{0}}\right)  ^{\perp}}\xi\right)  \approx
2^{s}\right\}  ,
\]
since $\left\vert \xi^{\prime}\right\vert \approx\operatorname*{dist}\left(
c_{I_{0}},\pi_{\left(  c_{I_{0}}\right)  ^{\perp}}\xi\right)  $.

\begin{definition}
We define the \emph{truncated} pipe%
\[
P_{s,w}^{I_{0}}\equiv P_{s}^{I_{0}}\cap L_{w}^{I_{0}},\ \ \ \ \ 1\leq w\leq s,
\]
to be the intersection of the infinite pipe $P_{s}^{I_{0}}$ and the horizontal
slab
\[
L_{w}^{I_{0}}\equiv\left\{  \xi\in\mathbb{R}^{n}:2^{2s-w-1}<\xi_{n}%
\leq2^{2s-w}\right\}  ,
\]
that is distance $2^{2s-w-1}$ above the plane $\xi_{n}=0$, and has height
roughly $2^{2s-w}$. We also define the truncated pipes $P_{s,w}^{I_{0}}$ for
$-s\leq w\leq-1$ by reflecting the pipes $P_{s,-w}^{I_{0}}$ across the plane
$\xi_{n}=0$, so that these pipes lie below the $\xi_{n}=0$.\newline Finally,
we define the truncated tubes $T_{s,+}^{I_{0}}\equiv P_{s}^{I_{0}}\cap L_{+}$
where $L_{+}\equiv\left\{  \xi\in\mathbb{R}^{n}:0\leq\xi_{n}\leq2^{s}\right\}
$, and their reflections $T_{s,-}^{I_{0}}\equiv-T_{s,\ast}^{I_{0}}$ across the
plane $\xi_{n}=0$.
\end{definition}

We now extend these notions of tubes and pipes to all $I\in\mathcal{G}%
_{s}\left[  U\right]  $.

\begin{definition}
For $I\in\mathcal{G}_{s}\left[  S\right]  $ and $0\leq w\leq s$, define the
truncated pipe $P_{s,w}^{I}$ to be the rotation of the pipe $P_{s,w}^{I_{0}}$
by any rotation $R$ that takes $\Phi\left(  c_{I_{0}}\right)  $ to
$\Phi\left(  c_{I}\right)  $, i.e.%
\[
P_{s,w}^{I}\equiv RP_{s,w}^{I_{0}}\approx\left\{  \xi\in\mathbb{R}%
^{n}:\operatorname*{dist}\left(  c_{I_{0}},\pi_{\Phi\left(  c_{I}\right)
^{\perp}}\xi\right)  \approx2^{s}\right\}  ,
\]
where $\pi_{\Phi\left(  c_{I}\right)  ^{\perp}}=\pi_{R\Phi\left(  c_{I_{0}%
}\right)  ^{\perp}}$. Similarly we define tubes $T_{s,+}^{I}$ and $T_{s,-}%
^{I}$.
\end{definition}

We will define expanded versions of these tubes and pipes below as needed.

Note that if $\left\vert \xi^{\prime}\right\vert \gg2^{s}$ then $e^{-i\Phi
\left(  x\right)  \cdot\xi}$ oscillates at least $\frac{\left\vert \xi
^{\prime}\right\vert }{2^{s}}$ times along the span of $\Phi\left(  I\right)
$, so that integration by parts is effective, while if $\left\vert \xi
^{\prime}\right\vert \ll2^{s}$ then $e^{-i\Phi\left(  x\right)  \cdot\xi}$
varies by at most $\frac{\left\vert \xi^{\prime}\right\vert }{2^{s}} $ along
the span of $\Phi\left(  I\right)  $, so that the vanishing moment properties
of $h_{I;\kappa}^{\eta}$ are effective.

\begin{definition}
For $r>0$ and $n\geq2$, define the $n$-dimensional annulus $A\left(
0,r\right)  =A_{n}\left(  0,r\right)  $ by%
\[
A\left(  0,r\right)  \equiv B\left(  0,r\right)  \setminus B\left(  0,\frac
{r}{2}\right)  ,
\]
where $B\left(  0,r\right)  =B_{n}\left(  0,r\right)  $ is the ball of radius
$r>0$ in $\mathbb{R}^{n}$ centered at the origin. Define the \emph{upper half}
ball $B_{\frac{1}{2}}\left(  0,r\right)  $ by%
\[
B_{+}\left(  0,r\right)  \equiv\left\{  \xi\in B\left(  0,r\right)  :\xi
_{n}\geq0\right\}  .
\]
and the \emph{upper half} annulus $A_{+}\left(  0,r\right)  $ by%
\[
A_{+}\left(  0,r\right)  \equiv\left\{  \xi\in A\left(  0,r\right)  :\xi
_{n}\geq0\right\}  .
\]

\end{definition}

To complete control of the lower disjoint form, in which $d<0$, we will use
the decomposition,%
\[
B_{+}\left(  0,2^{2s}\right)  =B_{+}\left(  0,2^{s}\right)  \cup\bigcup
_{w=0}^{s}A_{+}\left(  0,2^{2s-w}\right)  .
\]
We will later establish average control of $L^{p}$ norms, but first we turn to
controlling inner products.

\begin{lemma}
Suppose $s\in\mathbb{N}$ and $0\leq w\leq s$. Then%
\[
\mathbb{E}_{2^{\mathcal{G}_{s}\left[  U\right]  }}^{\mu}\left\vert
\left\langle T_{S}\left(  \mathcal{A}_{\mathbf{a}}\mathsf{Q}_{U}^{s}\right)
^{\spadesuit}f,\mathsf{P}_{A_{+}\left(  0,2^{2s-w}\right)  }^{\eta
}g\right\rangle \right\vert \lesssim2^{-\varepsilon_{n,p}s}\left\Vert
f\right\Vert _{L^{p}}\left\Vert g\right\Vert _{L^{p^{\prime}}}%
\ ,\ \ \ \ \ \text{for }p>\frac{2n}{n-1},
\]
where the implied constant is \emph{independent} of $s$ and $w$.
\end{lemma}

To prove the lemma, fix $0\leq w\leq s$ and $\mathbf{a}\in2^{\mathcal{G}%
_{s}\left[  S\right]  }$, and consider the positive expression,%
\begin{equation}
Z_{s,w}^{\mathbf{a}}\equiv\left\vert \sum_{I\in\mathcal{G}_{s}\left[
U\right]  }\sum_{J\subset P_{s,w}^{I}}\int_{\mathbb{R}^{n}}\left\{
\int_{\mathbb{R}^{n-1}}e^{-i\Phi\left(  x\right)  \cdot\xi}\left(
\mathcal{A}_{\mathbf{a}}\bigtriangleup_{I;\kappa}^{n-1}\mathsf{Q}_{S}%
^{s}\right)  ^{\spadesuit}f\left(  x\right)  dx\right\}  \bigtriangleup
_{J;\kappa}^{n,\eta}g\left(  \xi\right)  d\xi\right\vert .\label{mod inner}%
\end{equation}
We begin by establishing control of $Z_{s,w}^{\mathbf{a}}$, and then control
the sums over cubes $J$ in expanding geometric annuli away from the truncated
pipes $P_{s,w}^{I}$, by applying decay principles to obtain geometric decay
factors. Finally we apply the arguments used to bound $Z_{s,w}^{\mathbf{a}}$
to each of these collections of annuli, and then sum up the annuli to cover
all of the upper half annulus $A_{+}\left(  0,2^{2s-w}\right)  $, which
completes the proof of the lemma.

\begin{definition}
\label{expanded pipes}Define the \emph{expanded} truncated pipes
\[
P_{s,w}^{I_{0}}\left[  r\right]  =\left\{  \xi\in\mathbb{R}^{n}:\delta_{r}%
\xi\in P_{s,w}^{I_{0}}\right\}  ,
\]
where $\delta_{r}\xi=\left(  \frac{\xi^{\prime}}{2^{r}},\frac{\xi_{n}}%
{C_{n}2^{r}}\right)  $ is a (slightly nonisotropic) dilation for
$r\in\mathbb{Z}$, and $C_{n}$ is chosen sufficiently large. Thus $P%
_{s,w}^{I_{0}}\left[  r\right]  $ is a truncated pipe of height roughly
$C_{n}2^{2s-w+r}$ and width roughly $2^{s+r}$ centered at a point horizontally
located away from that of $P_{s,w}^{I_{0}}$. Then define the rotated expanded
truncated pipes $P_{s,w}^{I}\left[  r\right]  $ for $I\in\mathcal{G}%
_{s}\left[  S\right]  $, by $P_{s,w}^{I}\left[  r\right]  \equiv RP%
_{s,w}^{I_{0}}\left[  r\right]  $ for any rotation $R$ in $\mathbb{R}^{n}$
that takes $c_{I_{0}}$ to $c_{I}$.
\end{definition}

Note that if $C_{n}$ is chosen sufficiently large in the definition of
$P_{s,w}^{I_{0}}\left(  r\right)  $, then for every $I\in\mathcal{G}%
_{s}\left[  U\right]  $, the upper half annulus $A_{+}\left(  0,2^{2s-w}%
\right)  $ is contained in the union of the\ tube $T_{s,w}^{I}$, which we
recall is the convex hull of the truncated pipe $P_{s,w}^{I}$, and the
expanded truncated pipes $P_{s,w}^{I}\left[  r\right]  $ for $r\leq w$, i.e.%
\begin{equation}
A_{+}\left(  0,2^{2s-w}\right)  \subset T_{s,w}^{I}\cup\left(
%TCIMACRO{\dbigcup \limits_{r=1}^{w}}%
%BeginExpansion
{\displaystyle\bigcup\limits_{r=1}^{w}}
%EndExpansion
P_{s,w}^{I}\left[  r\right]  \right)  ,\ \ \ \ \ \text{for all }%
I\in\mathcal{G}_{s}\left[  S\right]  .\label{covers}%
\end{equation}
Moreover, the overlap of the truncated pipes $P_{s,w}^{I}$ is approximately%
\[
\frac{\left(  \#\ \text{pipes }P_{s,w}^{I}\right)  \times\left(  \text{volume
of a pipe }P_{s,w}^{I}\right)  }{\text{volume of annulus }A_{+}\left(
0,2^{2s-w}\right)  }\approx\frac{\left(  2^{s}\right)  ^{n-1}\times\left(
2^{s}\right)  ^{n-1}2^{2s-w}}{\left(  2^{2s-w}\right)  ^{n}}=2^{w\left(
n-1\right)  }.
\]
We will need to choose $C_{n}$ even larger in Subsubsection \ref{subsub gen} below.

\begin{definition}
For $\mathbf{a}\in2^{\mathcal{G}_{s}\left[  S\right]  }$ and $r\geq0$, define%
\begin{equation}
Z_{s,w}^{\mathbf{a}}\left[  r\right]  \equiv\left\vert \sum_{I\in
\mathcal{G}_{s}\left[  U\right]  }\sum_{J\subset P_{s,w}^{I}\left[  r\right]
}\int_{\mathbb{R}^{n}}\left\{  \int e^{-i\Phi\left(  x\right)  \cdot\xi
}\left(  \mathcal{A}_{\mathbf{a}}\bigtriangleup_{I;\kappa}^{n-1}\right)
^{\spadesuit}f\left(  x\right)  dx\right\}  \bigtriangleup_{J;\kappa}^{n,\eta
}g\left(  \xi\right)  d\xi\right\vert .\label{mod inner'}%
\end{equation}

\end{definition}

We will now control the average of this sum of inner products, as well as the
stronger average norm estimates, see (\ref{Lp est}) below. First, we consider
the two extreme cases $w=0$ and $w=s$, which are easily handled by two
different techniques. Then we combine these two proofs to give a single
argument for the general case.

\begin{definition}
Define
\[
\mathcal{R}_{s}^{k,w}\left(  r\right)  \equiv\left\{  \left(  I,J\right)
\in\mathcal{G}_{s}\left[  U\right]  \times\mathcal{D}_{k}:J\subset P_{s,w}%
^{I}\left[  r\right]  \right\}
\]
to be the set of pairs $\left(  I,J\right)  \in\mathcal{G}\left[  U\right]
\times\mathcal{D}$ with $\ell\left(  I\right)  =2^{-s}$, $\ell\left(
J\right)  =2^{k}$ and $J\subset P_{s,w}^{I}\left(  r\right)  $. When $r=0$ we
write simply
\[
\mathcal{R}_{s}^{k,w}=\mathcal{R}_{s}^{k,w}\left(  0\right)  .
\]

\end{definition}

For symmetry of notation, we also introduce tubes $\widehat{I_{0}}\left[
w\right]  $ that are essentially the same as the tubes $T_{s,w}^{I}$. For
$I\in\mathcal{G}_{s}\left[  U\right]  $ and $0\leq w\leq s$, define
\[
\widehat{I_{0}}\left[  w\right]  \equiv\left[  -2^{s},2^{s}\right]
^{n-1}\times\left[  2^{2s-w-1},2^{2s-w}\right]  \approx T_{s,w}^{I_{0}},
\]
and extend this definition to $\widehat{I}\left[  w\right]  $ by rotation , so
that $\widehat{I}\left[  w\right]  \approx T_{s,w}^{I}$ and $\widehat
{I}\left[  0\right]  \approx\widehat{I}$.

\subsubsection{The case $w=0$ (Direct Argument):\label{subsub direct}}

In the case $w=0$, we first consider $Z_{s,0}^{\mathbf{a}}$ with the sequence
$\mathbf{a}=\mathbf{1}\ $of all $1^{\prime}s$, since the arguments in this
subsubsection take absolute values inside anyways, and do not use probability.
The bound for the subform
\[
Z_{s,0}^{\mathbf{1}}=\left\vert \sum_{s=1}^{\infty}\sum_{I\in\mathcal{G}%
_{s}\left[  U\right]  }\sum_{J\in\mathcal{D}:\ J\subset\widehat{I}%
}\left\langle T\bigtriangleup_{I;\kappa}^{n-1,\eta}f,\bigtriangleup_{J;\kappa
}^{n,\eta}g\right\rangle \right\vert
\]
applies more generally to indicators $\mathbf{1}_{I}$ times $f$, in place of
smooth Alpert pseudoprojections $\bigtriangleup_{I;\kappa}^{n-1,\eta}$ applied
to $f$, and to $\mathbf{1}_{\widehat{I}}$ in place of $\sum_{J\in
\mathcal{D}:\ J\subset\widehat{I}}\bigtriangleup_{J;\kappa}^{n,\eta}$. To see
this, we first note that%
\begin{align*}
\left\Vert T\mathbf{1}_{I}f\right\Vert _{L^{p}\left(  \widehat{I}\right)  }  &
=\left(  \int_{\widehat{I}}\left\vert \int_{I}e^{-i\Phi\left(  x\right)
\cdot\xi}f\left(  x\right)  dx\right\vert ^{p}d\xi\right)  ^{\frac{1}{p}}%
\leq\left\vert \widehat{I}\right\vert ^{\frac{1}{p}}\left\vert I\right\vert
^{\frac{1}{p^{\prime}}}\left(  \int_{I}\left\vert f\left(  x\right)
\right\vert ^{p}dx\right)  ^{\frac{1}{p}}\\
& =2^{s\frac{n+1}{p}}2^{-s\frac{n-1}{p^{\prime}}}\left\Vert \mathbf{1}%
_{I}f\right\Vert _{L^{p}\left(  \mathbb{R}^{n-1}\right)  }=2^{-s\varepsilon
_{p,n}}\left\Vert \mathbf{1}_{I}f\right\Vert _{L^{p}\left(  \mathbb{R}%
^{n-1}\right)  }\ ,
\end{align*}
where%
\[
\varepsilon_{p,n}\equiv\frac{n-1}{p^{\prime}}-\frac{n+1}{p}=\frac{n-1}%
{p}\left(  p-1-\frac{n+1}{n-1}\right)  =\frac{n-1}{p}\left(  p-\frac{2n}%
{n-1}\right)  .
\]

Then with $s$ fixed, we continue with%
\begin{align*}
& \sum_{I\in\mathcal{G}_{s}\left[  U\right]  }\left\vert \left\langle
T\mathbf{1}_{I}f,\mathbf{1}_{\widehat{I}}g\right\rangle \right\vert \leq
\sum_{I\in\mathcal{G}_{s}\left[  U\right]  }\left\Vert T\mathbf{1}%
_{I}f\right\Vert _{L^{p}\left(  \widehat{I}\right)  }\left\Vert g\right\Vert
_{L^{p^{\prime}}\left(  \widehat{I}\right)  }\leq\left(  \sum_{I\in
\mathcal{G}_{s}\left[  U\right]  }\left\Vert T\mathbf{1}_{I}f\right\Vert
_{L^{p}\left(  \widehat{I}\right)  }^{p}\right)  ^{\frac{1}{p}}\left(
\sum_{I\in\mathcal{G}_{s}\left[  U\right]  }\left\Vert g\right\Vert
_{L^{p^{\prime}}\left(  \widehat{I}\right)  }^{p^{\prime}}\right)  ^{\frac
{1}{p^{\prime}}}\\
& \lesssim\left(  \sum_{I\in\mathcal{G}_{s}\left[  U\right]  }%
2^{-sp\varepsilon_{p,n}}\left\Vert \mathbf{1}_{I}f\right\Vert _{L^{p}\left(
\mathbb{R}^{n-1}\right)  }^{p}\right)  ^{\frac{1}{p}}\left\Vert g\right\Vert
_{L^{p^{\prime}}\left(  \cup_{I\in\mathcal{G}_{s}\left[  U\right]  }%
\widehat{I}\right)  }\leq2^{-s\varepsilon_{p,n}}\left\Vert f\right\Vert
_{L^{p}\left(  \mathbb{R}^{n-1}\right)  }\left\Vert g\right\Vert
_{L^{p^{\prime}}\left(  \mathbb{R}^{n}\right)  }\ ,
\end{align*}
and finally we sum over $s\in\mathbb{N}$ to obtain%
\[
\left\vert \sum_{s=1}^{\infty}\sum_{I\in\mathcal{G}_{s}\left[  U\right]
}\left\langle T\mathbf{1}_{I}f,\mathbf{1}_{\widehat{I}}g\right\rangle
\right\vert \leq\sum_{s=1}^{\infty}\sum_{I\in\mathcal{G}_{s}\left[  U\right]
}\left\vert \left\langle T\mathbf{1}_{I}f,\mathbf{1}_{\widehat{I}%
}g\right\rangle \right\vert \leq C_{n}\left\Vert f\right\Vert _{L^{p}\left(
\mathbb{R}^{n-1}\right)  }\left\Vert g\right\Vert _{L^{p^{\prime}}\left(
\mathbb{R}^{n}\right)  }\ ,
\]
where%
\[
C_{n}\equiv\sum_{s=1}^{\infty}2^{-\varepsilon_{p,n}s}<\infty\text{ for
}p>\frac{2n}{n-1}.
\]

\begin{corollary}
\label{cor enlarge}If we enlarge the cubes $I$ by a factor $2^{t}$ to
$I\left[  t\right]  \equiv2^{t}I$, and if we enlarge the tubes $\widehat{I}$
transversally (meaning perpendicular to $\Phi\left(  c_{I}\right)  $) by a
factor of $2^{r}$ to $\widehat{I}\left[  r\right]  $, then we obtain the
estimate,%
\[
\left\vert \sum_{I\in\mathcal{G}_{s}\left[  U\right]  }\left\langle
T\mathbf{1}_{I\left[  t\right]  }f,\mathbf{1}_{\widehat{I}\left[  r\right]
}g\right\rangle \right\vert \leq C2^{t\frac{n}{p^{\prime}}}2^{r\frac{n-1}{p}%
}2^{-s\varepsilon_{p,n}}\left\Vert f\right\Vert _{L^{p}\left(  \mathbb{R}%
^{n-1}\right)  }\left\Vert g\right\Vert _{L^{p^{\prime}}\left(  \mathbb{R}%
^{n}\right)  }\ .
\]

\end{corollary}

\begin{proof}
Apply the above argument and use $\left(  \left\vert \widehat{I}\left[
r\right]  \right\vert \left\vert I\left[  t\right]  \right\vert ^{p-1}\right)
^{\frac{1}{p}}=2^{r\frac{n}{p}}2^{t\frac{n-1}{p^{\prime}}}\left(  \left\vert
\widehat{I}\right\vert \left\vert I\right\vert ^{p-1}\right)  ^{\frac{1}{p}}$.
\end{proof}

We now turn to obtaining the stronger norm estimate for smooth Alpert
pseudoprojections,%
\begin{equation}
\left\Vert T\left(  \mathsf{Q}_{U}^{s}\right)  ^{\spadesuit}f\right\Vert
_{L^{p}\left(  A_{+}\left(  0,2^{2s}\right)  \right)  }\lesssim2^{-\varepsilon
_{p,n}s}\left\Vert f\right\Vert _{L^{p}}\ ,\ \ \ \ \ \text{for }s\in
\mathbb{N},\label{Lp est}%
\end{equation}
where integration by parts in the $x$-variable in the expanded pipes
$\widehat{I}\left[  r\right]  $ will compensate for the growth $2^{r\frac
{n}{p}} $ in Corollary \ref{cor enlarge}.

\medskip

\textbf{Expanded pipes}

\medskip

Consider an expanded truncated pipe $P_{s,0}^{I_{0}}\left[  r\right]  $. For
$r\gg0$, we claim that the wavelength on $I_{0}$ in the inner product is much
smaller than the diameter $2^{-s}$ of $I_{0}$, and so we can use integration
by parts to gain a geometric decay factor of $C_{N}2^{-rN}$ for all $N\geq1$.
Indeed, for $\xi\in J$ with $J\subset P_{s,0}^{I_{0}}\left[  r\right]  $ and
$0\leq r\lesssim s$, the wavelength of the exponential factor $e^{-i\Phi
\left(  x\right)  \cdot\xi}$ is roughly $\frac{1}{\left\vert \xi\right\vert
}\approx\frac{1}{2^{2s}}$, and referring to (\ref{reso}), we see that the
tilted depth of $I_{0}$ in the direction $\xi$, is roughly $\ell\left(
I\right)  \sin\theta$, where $\sin\theta=\frac{\left\vert \xi^{\prime
}\right\vert }{\left\vert \xi\right\vert }\approx\frac{2^{r+s}}{2^{2s}}$.
Altogether then, since $\xi\in B\left(  0,2^{2s}\right)  \cap P_{s,0}^{I_{0}%
}\left[  r\right]  $, we have
\[
\operatorname*{tilted}\operatorname*{depth}\approx\ell\left(  I\right)
\sin\theta\gtrsim2^{-s}\frac{2^{r+s}}{2^{2s}}=2^{r}\frac{1}{2^{2s}}%
=2^{r}\operatorname*{wavelength},
\]
and so the exponential factor $e^{-i\Phi\left(  x\right)  \cdot\xi}$
oscillates at least $2^{r}$ times as $x$ traverses $I_{0}$.

Thus%
\[
\left\langle T\bigtriangleup_{I;\kappa}^{n-1,\eta}f,\bigtriangleup_{J;\kappa
}^{n,\eta}g\right\rangle =\int_{\mathbb{R}^{n}}\left\{  \int_{\mathbb{R}%
^{n-1}}e^{-i\Phi\left(  x\right)  \cdot\xi}\bigtriangleup_{I;\kappa}%
^{n-1,\eta}f\left(  x\right)  dx\right\}  \bigtriangleup_{J;\kappa}^{n,\eta
}g\left(  \xi\right)  d\xi,
\]
where for $\xi\in J$ and $J\subset P_{s,0}^{I_{0}}\left(  r\right)  $, the
integral in braces satisfies,%
\begin{align*}
\int_{\mathbb{R}^{n-1}}e^{-i\Phi\left(  x\right)  \cdot\xi}\bigtriangleup
_{I;\kappa}^{n-1,\eta}f\left(  x\right)  dx  & =\int_{\mathbb{R}^{n-1}}\left(
\frac{1}{-i\partial_{x}\left(  \Phi\left(  x\right)  \cdot\xi\right)
}\partial_{x}\right)  ^{N}e^{-i\Phi\left(  x\right)  \cdot\xi}\bigtriangleup
_{I;\kappa}^{n-1,\eta}f\left(  x\right)  dx\\
& =\left(  -1\right)  ^{N}\int_{\mathbb{R}^{n-1}}e^{-i\Phi\left(  x\right)
\cdot\xi}\left(  \partial_{x}\frac{1}{-i\Phi^{\prime}\left(  x\right)
\cdot\xi}\right)  ^{N}\bigtriangleup_{I;\kappa}^{n-1,\eta}f\left(  x\right)
dx,
\end{align*}
and hence is dominated in modulus by $C_{N}2^{-rN}\int\left\vert \partial
^{N}\bigtriangleup_{I;\kappa}^{n-1,\eta}f\left(  x\right)  \right\vert dx$
since
\[
\left\vert \Phi^{\prime}\left(  x\right)  \cdot\xi\right\vert \approx
\left\vert \xi^{\prime}\right\vert \approx2^{r+s}\ \ \ \ \ \left(  \text{also
}\approx\frac{1}{\ell\left(  I\right)  }\frac{\operatorname*{tilted}%
\operatorname*{depth}}{\operatorname*{wavelength}}\gtrsim2^{r+s}\right)
,\ \ \ \ \ \text{for }\xi\in P_{s,0}^{I_{0}}\left(  r\right)  .
\]
In conclusion, for any cube $I\in\mathcal{G}_{s}\left[  S\right]  $ we have%
\begin{equation}
\left\vert \int_{\mathbb{R}^{n-1}}e^{-i\Phi\left(  x\right)  \cdot\xi
}\bigtriangleup_{I;\kappa}^{n-1,\eta}f\left(  x\right)  dx\right\vert \lesssim
C_{N}2^{-\left(  r+s\right)  N}\int_{\mathbb{R}^{n-1}}\left\vert \partial
^{N}\bigtriangleup_{I;\kappa}^{n-1,\eta}f\left(  x\right)  \right\vert
dx,\ \ \ \ \ \xi\in P_{s,0}^{I}\left[  r\right]  .\label{dom in mod}%
\end{equation}

Plugging this estimate back into the inner product gives%
\begin{align}
\left\vert \left\langle T\bigtriangleup_{I;\kappa}^{n-1,\eta}f,\bigtriangleup
_{J;\kappa}^{n,\eta}g\right\rangle \right\vert  & \leq\int_{\mathbb{R}^{n}%
}\left\vert \int_{\mathbb{R}^{n-1}}e^{-i\Phi\left(  x\right)  \cdot\xi
}\bigtriangleup_{I;\kappa}^{n-1,\eta}f\left(  x\right)  dx\right\vert
\left\vert \bigtriangleup_{J;\kappa}^{n,\eta}g\left(  \xi\right)  \right\vert
d\xi\label{back dom}\\
& \lesssim C_{N}2^{-\left(  r+s\right)  N}\left(  \int_{\mathbb{R}^{n-1}%
}\left\vert \partial^{N}\bigtriangleup_{I;\kappa}^{n-1,\eta}f\right\vert
\right)  \left(  \int_{\mathbb{R}^{n}}\left\vert \bigtriangleup_{J;\kappa
}^{n,\eta}g\right\vert \right)  .\nonumber
\end{align}
For use later on, we note that for any $K\in\mathcal{G}\left[  S\right]  $
with $\ell\left(  K\right)  \geq2^{-s}$, we can sum over $I\in\mathcal{G}%
_{s}\left[  K\right]  $ in (\ref{dom in mod}) to obtain%
\begin{equation}
\left\vert \int_{\mathbb{R}^{n-1}}e^{-i\Phi\left(  x\right)  \cdot\xi}\left(
\mathsf{Q}_{K}^{s}\right)  ^{\spadesuit}f\left(  x\right)  dx\right\vert
\lesssim C_{N}2^{-\left(  r+s\right)  N}\int_{\mathbb{R}^{n-1}}\left\vert
\partial^{N}\left(  \mathsf{Q}_{K}^{s}\right)  ^{\spadesuit}f\left(  x\right)
\right\vert dx,\ \ \ \ \ \xi\in P_{s,0}^{K}\left[  r\right]
,\label{dom in mod'}%
\end{equation}
and with a similar estimate of the corresponding inner product.

We now apply the argument used above for bounding%
\[
Z_{s,0}^{\mathbf{1}}\equiv\left\vert \sum_{I\in\mathcal{G}_{s}\left[
U\right]  }\sum_{J\subset T_{s}^{I}\left[  0\right]  }\int_{\mathbb{R}^{n}%
}\left\{  \int_{\mathbb{R}^{n-1}}e^{-i\Phi\left(  x\right)  \cdot\xi
}\bigtriangleup_{I;\kappa}^{n-1,\eta}f\left(  x\right)  dx\right\}
\bigtriangleup_{J;\kappa}^{n,\eta}g\left(  \xi\right)  d\xi\right\vert ,
\]
to the expanded truncated pipes $P_{s,0}^{I}\left[  r\right]  $ in place of
the tubes $T_{s}^{I}\left[  0\right]  $, to obtain from Corollary
\ref{cor enlarge} and the estimate (\ref{dom in mod}), that%
\begin{align}
& \left\Vert T\bigtriangleup_{I;\kappa}^{n-1,\eta}f\right\Vert _{L^{p}\left(
P_{s,0}^{I}\left[  r\right]  \right)  }=\left(  \int_{P_{s,0}^{I}\left[
r\right]  }\left\vert \int_{\mathbb{R}^{n-1}}e^{-i\Phi\left(  x\right)
\cdot\xi}\bigtriangleup_{I;\kappa}^{n-1,\eta}f\left(  x\right)  dx\right\vert
^{p}d\xi\right)  ^{\frac{1}{p}}\label{norm est}\\
& \leq\left\vert P_{s,0}^{I}\left[  r\right]  \right\vert ^{\frac{1}{p}%
}\left\vert I\right\vert ^{\frac{1}{p^{\prime}}}\left(  C_{N}2^{-\left(
r+s\right)  Np}\int_{\mathbb{R}^{n-1}}\left\vert \partial^{N}\bigtriangleup
_{I;\kappa}^{n-1,\eta}f\left(  x\right)  \right\vert ^{p}dx\right)  ^{\frac
{1}{p}}\nonumber\\
& \leq C_{N}2^{-\left(  r+s\right)  N}2^{r\frac{n}{p}}\left\vert P_{s,0}%
^{I}\right\vert ^{\frac{1}{p}}\left\vert I\right\vert ^{\frac{1}{p^{\prime}}%
}\left(  \int_{\mathbb{R}^{n-1}}\left\vert \partial^{N}\bigtriangleup
_{I;\kappa}^{n-1,\eta}f\left(  x\right)  \right\vert ^{p}dx\right)  ^{\frac
{1}{p}}\nonumber\\
& \leq C_{N}2^{-r\left(  N-\frac{n}{p}\right)  }2^{-s\varepsilon_{p,n}}%
2^{-sN}\left\Vert \partial^{N}\bigtriangleup_{I;\kappa}^{n-1,\eta}f\right\Vert
_{L^{p}\left(  \mathbb{R}^{n-1}\right)  }\ ,\nonumber
\end{align}
since $\left\vert P_{s,0}^{I}\right\vert \approx\left(  2^{s+r}\right)
^{n-1}2^{2s+r}$ implies%
\[
\left\vert P_{s,0}^{I}\right\vert ^{\frac{1}{p}}\left\vert I\right\vert
^{\frac{1}{p^{\prime}}}\approx2^{s\frac{n+1}{p}}2^{r\frac{n}{p}}%
2^{-s\frac{n-1}{p^{\prime}}}=2^{s\left(  \frac{n+1}{p}-\frac{n-1}{p^{\prime}%
}\right)  }2^{r\frac{n}{p}}=2^{-\varepsilon_{p,n}s}2^{r\frac{n}{p}}.
\]
Thus%
\begin{align*}
\left(  \sum_{I\in\mathcal{G}_{s}\left[  U\right]  }\left\Vert T\bigtriangleup
_{I;\kappa}^{n-1,\eta}f\right\Vert _{L^{p}\left(  P_{s,0}^{I}\left[  r\right]
\right)  }^{p}\right)  ^{\frac{1}{p}}  & \lesssim C_{N}2^{-r\left(  N-\frac
{n}{p}\right)  }2^{-s\varepsilon_{p,n}}\left(  \sum_{I\in\mathcal{G}%
_{s}\left[  U\right]  }2^{-sNp}\left\Vert \partial^{N}\bigtriangleup
_{I;\kappa}^{n-1,\eta}f\right\Vert _{L^{p}\left(  \mathbb{R}^{n-1}\right)
}^{p}\right)  ^{\frac{1}{p}}\\
& \lesssim C_{N}2^{-r\left(  N-\frac{n}{p}\right)  }2^{-s\varepsilon_{p,n}%
}\left\Vert f\right\Vert _{L^{p}\left(  \mathbb{R}^{n-1}\right)  }^{p},
\end{align*}

and so also,
\begin{align}
& Z_{s,0}^{\mathbf{1}}\left[  r\right]  \equiv\left\vert \sum_{I\in
\mathcal{G}_{s}\left[  U\right]  }\sum_{J\subset P_{s}^{I}\left[  r\right]
}\int_{\mathbb{R}^{n}}\left\{  \int_{\mathbb{R}^{n-1}}e^{-i\Phi\left(
x\right)  \cdot\xi}\bigtriangleup_{I;\kappa}^{n-1,\eta}f\left(  x\right)
dx\right\}  \bigtriangleup_{J;\kappa}^{n,\eta}g\left(  \xi\right)
d\xi\right\vert \label{above arg}\\
& \leq\sum_{I\in\mathcal{G}_{s}\left[  S\right]  }\left\Vert T\bigtriangleup
_{I;\kappa}^{n-1,\eta}f\right\Vert _{L^{p}\left(  P_{s,0}^{I}\left[  r\right]
\right)  }\left\Vert g\right\Vert _{L^{p^{\prime}}\left(  P_{s,0}^{I}\left[
r\right]  \right)  }\nonumber\\
& \leq\left(  \sum_{I\in\mathcal{G}_{s}\left[  U\right]  }\left\Vert
T\bigtriangleup_{I;\kappa}^{n-1,\eta}f\right\Vert _{L^{p}\left(  P_{s,0}%
^{I}\left[  r\right]  \right)  }^{p}\right)  ^{\frac{1}{p}}\left(  \sum
_{I\in\mathcal{G}_{s}\left[  U\right]  }\left\Vert g\right\Vert _{L^{p^{\prime
}}\left(  P_{s,0}^{I}\left[  r\right]  I\right)  }^{p^{\prime}}\right)
^{\frac{1}{p^{\prime}}}\nonumber\\
& \leq C_{N}2^{-r\left(  N-\frac{n}{p}\right)  }2^{-s\varepsilon_{p,n}%
}\left\Vert f\right\Vert _{L^{p}\left(  \mathbb{R}^{n-1}\right)  }\left\Vert
g\right\Vert _{L^{p^{\prime}}\left(  \mathbb{R}^{n}\right)  }.\nonumber
\end{align}

Summing in $r$ gives%
\begin{equation}
\left\vert \int_{\mathbb{R}^{n}}T\left(  \mathsf{Q}_{U}^{s}\right)
^{\spadesuit}f,\left(  \mathsf{P}_{A_{+}\left(  0,2^{2s}\right)  }^{s}\right)
^{\spadesuit}g\right\vert \lesssim\sum_{r=0}^{\infty}Z_{s,0}^{\mathbf{1}%
}\left[  r\right]  \lesssim C_{N}2^{-s\varepsilon_{p,n}}\left\Vert
f\right\Vert _{L^{p}\left(  \mathbb{R}^{n-1}\right)  }\left\Vert \left(
\mathsf{P}_{A_{+}\left(  0,2^{2s}\right)  }^{s}\right)  ^{\spadesuit
}g\right\Vert _{L^{p^{\prime}}\left(  \mathbb{R}^{n}\right)  }%
,\label{above arg'}%
\end{equation}
and a standard argument then yields,%
\begin{equation}
\left\vert \int_{\mathbb{R}^{n}}T\left(  \mathsf{Q}_{U}^{s}\right)
^{\spadesuit}f,\mathbf{1}_{A_{+}\left(  0,2^{2s}\right)  }g\right\vert
\lesssim C_{N}2^{-s\varepsilon_{p,n}}\left\Vert f\right\Vert _{L^{p}\left(
\mathbb{R}^{n-1}\right)  }\left\Vert g\right\Vert _{L^{p^{\prime}}\left(
A_{+}\left(  0,2^{2s}\right)  \right)  }\ .\label{above arg''}%
\end{equation}

\medskip

\textbf{Norm estimate}

\medskip

Simce there is no expectation involved, we can extend the inner product
estimate (\ref{above arg''}) to a norm estimate by duality. Indeed, for each
$s\in\mathbb{N}$, choose an appropriate function $g_{s}$ with $\left\Vert
g_{s}\right\Vert _{L^{p^{\prime}}\left(  \mathbb{R}^{n}\right)  }=1$ and
\begin{equation}
\left\langle T\left(  \mathsf{Q}_{U}^{s}\right)  ^{\spadesuit}f,g_{s}%
\right\rangle =\left\Vert T\left(  \mathsf{Q}_{U}^{s}\right)  ^{\spadesuit
}f\right\Vert _{L^{p}\left(
%TCIMACRO{\dbigcup \limits_{I\in\mathcal{G}_{s}\left[  U\right]  }}%
%BeginExpansion
{\displaystyle\bigcup\limits_{I\in\mathcal{G}_{s}\left[  U\right]  }}
%EndExpansion
\left\{  T_{s}^{I}\cup%
%TCIMACRO{\dbigcup \limits_{r\geq0}}%
%BeginExpansion
{\displaystyle\bigcup\limits_{r\geq0}}
%EndExpansion
P_{s}^{I}\left[  r\right]  \right\}  \right)  }\ ,\label{choose g_s}%
\end{equation}
and then with $N>\frac{n}{p}$ and $p>\frac{2n}{n-1}$, sum in $r$ and $s$ to
obtain%
\begin{align*}
& \sum_{s=1}^{\infty}\left\Vert T\left(  \mathsf{Q}_{U}^{s}\right)
^{\spadesuit}f\right\Vert _{L^{p}\left(  A_{+}\left(  0,2^{2s}\right)
\right)  }\leq\sum_{s=1}^{\infty}\left\Vert T\left(  \mathsf{Q}_{U}%
^{s}\right)  ^{\spadesuit}f\right\Vert _{L^{p}\left(
%TCIMACRO{\dbigcup \limits_{I\in\mathcal{G}_{s}\left[  U\right]  }}%
%BeginExpansion
{\displaystyle\bigcup\limits_{I\in\mathcal{G}_{s}\left[  U\right]  }}
%EndExpansion
\left\{  T_{s}^{I}\cup%
%TCIMACRO{\dbigcup \limits_{r\geq0}}%
%BeginExpansion
{\displaystyle\bigcup\limits_{r\geq0}}
%EndExpansion
P_{s}^{I}\left[  r\right]  \right\}  \right)  }\\
& =\sum_{s=1}^{\infty}\left\vert \left\langle T\left(  \mathsf{Q}_{U}%
^{s}\right)  ^{\spadesuit}f,g_{s}\right\rangle \right\vert \leq\sum
_{s=1}^{\infty}\sum_{r=0}^{\infty}C_{N}2^{-r\left(  N-\frac{n}{p}\right)
}2^{-s\varepsilon_{p,n}}\left\Vert f\right\Vert _{L^{p}\left(  \mathbb{R}%
^{n-1}\right)  }\left\Vert g_{s}\right\Vert _{L^{p^{\prime}}\left(
\mathbb{R}^{n}\right)  }\lesssim\left\Vert f\right\Vert _{L^{p}\left(
\mathbb{R}^{n-1}\right)  }\ ,
\end{align*}
which is (\ref{Lp est}). Here we have used (\ref{choose g_s}) in the first
equality in the second line above, (\ref{above arg''}) in the second
inequality, and $\left\Vert g_{s}\right\Vert _{L^{p^{\prime}}\left(
\mathbb{R}^{n}\right)  }=1$ in the final inequality.

\subsubsection{The case $w=s$}

In this case we need to take expectation. Since each fixed cube $J$ in
the\ upper half annulus $A_{+}\left(  0,2^{s}\right)  $ belongs to the
truncated tube $T_{s,s}^{I}\equiv T_{s}^{I}\cap L_{s}^{I} $ for essentially
all $I\in\mathcal{G}_{s}\left[  S\right]  $, we get%
\begin{align*}
Z_{s,s}^{\mathbf{a}}  & =\left\vert \sum_{I\in\mathcal{G}_{s}\left[  U\right]
}\sum_{J\subset T_{s,s}^{I}}\int_{\mathbb{R}^{n}}\left\{  \int e^{-i\Phi
\left(  x\right)  \cdot\xi}\left(  \mathcal{A}_{\mathbf{a}}\mathsf{Q}_{U}%
^{s}\right)  ^{\spadesuit}f\left(  x\right)  dx\right\}  \bigtriangleup
_{J;\kappa}^{n,\eta}g\left(  \xi\right)  d\xi\right\vert \\
& \approx\left\vert \sum_{Q_{0}}\left\langle T\left(  \mathcal{A}_{\mathbf{a}%
}\mathsf{Q}_{Q_{0}}^{s}\right)  ^{\spadesuit}f,\mathsf{P}_{Q_{0}^{\ast
},s;\kappa}^{n,\eta}g\right\rangle \right\vert \lesssim\sum_{Q_{0}}\left\Vert
T\left(  \mathcal{A}_{\mathbf{a}}\mathsf{Q}_{Q_{0}}^{s}\right)  ^{\spadesuit
}f\right\Vert _{L^{p}}\left\Vert \mathsf{P}_{Q_{0}^{\ast},s;\kappa}^{n,\eta
}g\right\Vert _{L^{p^{\prime}}},
\end{align*}
where $\mathsf{Q}_{Q_{0}}^{s}=\sum_{I\in\mathcal{G}_{s}\left[  Q_{0}\right]
}\bigtriangleup_{I;\kappa}^{n-1}$ and $\mathsf{P}_{Q_{0}^{\ast},s;\kappa
}^{n,\eta}=\sum_{J\in\mathcal{D}_{k}\left[  Q_{0}^{\ast}\right]
}\bigtriangleup_{I;\kappa}^{n-1,\eta}$, and where $Q_{0}$ ranges over a
bounded number of cubes in $S$ with side length approximately $1$. Also note
that
\[
\left(  \mathcal{A}_{\mathbf{a}}\mathsf{Q}_{Q_{0}}^{s}\right)  ^{\spadesuit
}f=S_{\kappa,\eta}\mathcal{A}_{\mathbf{a}}\sum_{I\in\mathcal{G}_{s}\left[
Q_{0}\right]  }\left\langle \left(  S_{\kappa,\eta}\right)  ^{-1}%
f,h_{I;\kappa}^{n-1}\right\rangle h_{I;\kappa}^{n-1}=\sum_{I\in\mathcal{G}%
_{s}\left[  Q_{0}\right]  }a_{I}\bigtriangleup_{I;\kappa}^{n-1,\eta}\ .
\]
Now we apply just part of the estimate (\ref{no osc s=m,k}), which followed
from Proposition \ref{prop interp},\ to obtain%
\[
\mathbb{E}_{2^{\mathcal{G}_{s}\left[  U\right]  }}^{\mu}Z_{s,s}^{\mathbf{a}%
}\lesssim\mathbb{E}_{2^{\mathcal{G}_{s}\left[  U\right]  }}^{\mu}\left\Vert
T_{S}\left(  \mathcal{A}_{\mathbf{a}}\mathsf{Q}_{Q_{0}}^{s}\right)
^{\spadesuit}f\right\Vert _{L^{p}\left(  B\left(  0,2^{s}\right)  \right)
}\left\Vert \mathsf{P}_{Q_{0}^{\ast},s;\kappa}^{n,\eta}g\right\Vert
_{L^{p^{\prime}}}\lesssim2^{-\varepsilon_{p,n}s}\left\Vert f\right\Vert
_{L^{p}}\left\Vert g\right\Vert _{L^{p^{\prime}}}\ ,
\]
for $p>\frac{2n}{n-1}$ and $m=s\in\mathbb{N}$. We do not need to make use of
expanded pipes in this case, due to the small size of the ball $B\left(
0,2^{s}\right)  $.

However, we actually obtain from Proposition \ref{prop interp} the stronger
average norm inequality,%
\begin{equation}
\mathbb{E}_{2^{\mathcal{G}_{s}\left[  U\right]  }}^{\mu}\left\Vert
T_{S}\left(  \mathcal{A}_{\mathbf{a}}\mathsf{Q}_{Q_{0}}^{s}\right)
^{\spadesuit}f\right\Vert _{L^{p}\left(  B\left(  0,2^{s}\right)  \right)
}\lesssim2^{-\varepsilon_{p,n}s}\left\Vert f\right\Vert _{L^{p}}%
\ ,\ \ \ \ \ \text{for }s\in\mathbb{N},\label{strong avg ineq}%
\end{equation}
and this is what we will use going forward.

\subsection{The general case $0\leq w\leq s$ via Fourier square functions
\label{subsub gen}}

In this subsection we prove the average norm estimate for each $s\in
\mathbb{N}$ and $0\leq w\leq s$,%
\begin{equation}
\mathbb{E}_{2^{\mathcal{G}\left[  U\right]  }}^{\mu}\left\Vert T\left(
\mathcal{A}_{\mathbf{a}}\mathsf{Q}_{U}^{s}\right)  ^{\spadesuit}f\right\Vert
_{L^{p}\left(  A_{+}\left(  0,2^{2s-w}\right)  \right)  }\lesssim
2^{-\varepsilon_{n,p}s}\left\Vert f\right\Vert _{L^{p}\left(  U\right)
}\ ,\ \ \ \ \ \text{for }p>\frac{2n}{n-1}.\label{prob K rem square}%
\end{equation}
Note that we have already proved the endpoint case $w=0$ in (\ref{Lp est}),
and the other endpoint case $w=s$ in (\ref{strong avg ineq}). It will be
convenient to pass back and forth between average norm estimates and Fourier
square function estimates using Khintchine's inequalities. For example
(\ref{prob K rem square}) is equivalent to,%
\begin{equation}
\left\Vert \mathcal{S}_{T,s}^{\eta}f\right\Vert _{L^{p}\left(  A_{+}\left(
0,2^{2s-w}\right)  \right)  }\lesssim2^{-\varepsilon_{n,p}s}\left\Vert
f\right\Vert _{L^{p}}\ ,\ \ \ \ \ \text{for }p>\frac{2n}{n-1}%
,\label{Fourier square est}%
\end{equation}
where
\begin{equation}
\mathcal{S}_{T,s}^{\eta}f\equiv\left(  \sum_{I\in\mathcal{G}_{s}\left[
U\right]  }\left\vert T\bigtriangleup_{I;\kappa}^{n-1,\eta}f\right\vert
^{2}\right)  ^{\frac{1}{2}}\label{def squ}%
\end{equation}
is the Fourier square function associated with the random decomposition
\[
T\left(  \mathcal{A}_{\mathbf{a}}\mathsf{Q}_{U}^{s}\right)  ^{\spadesuit
}f=\sum_{I\in\mathcal{G}_{s}\left[  U\right]  }a_{I}T\bigtriangleup_{I;\kappa
}^{n-1,\eta}f,\ \ \ \ \ \text{for\ each }\mathbf{a}\in2^{\mathcal{G}\left[
U\right]  }.
\]

We will prove (\ref{Fourier square est}) in three steps, the first two being
local estimates requiring \emph{probabilistic} arguments, and the third being
a global estimate that uses \emph{Fourier square function} arguments. The
probabilistic local estimates are used to control the sums over cubes
$I\in\mathcal{G}_{s}\left[  K\right]  $ which are typically close together,
while the Fourier square function estimate is used to control the sums of
cubes $K\in\mathcal{G}_{s-w}\left[  S\right]  $ in which the subcubes $I$ of
different $K^{\prime}s$ are typically farther apart. Once we have established
(\ref{Fourier square est}), we use the decomposition
\[
B_{+}\left(  0,2^{2s}\right)  =P_{s,0}^{I_{0}}\left[  r\right]  \ \ \ \cup
\ \ \
%TCIMACRO{\dbigcup \limits_{w=0}^{s-1}}%
%BeginExpansion
{\displaystyle\bigcup\limits_{w=0}^{s-1}}
%EndExpansion
A_{+}\left(  0,2^{2s-w}\right)  ,
\]
and then appeal to reflection across the horizontal plane to conclude that,%
\begin{equation}
\left\Vert \mathcal{S}_{T,s}^{\eta}f\right\Vert _{L^{p}\left(  B_{+}\left(
0,2^{2s}\right)  \right)  }\lesssim2^{-\varepsilon_{n,p}s}\left\Vert
f\right\Vert _{L^{p}}\ ,\ \ \ \ \ \text{for }p>\frac{2n}{n-1}%
.\label{conclude that}%
\end{equation}

\subsubsection{Step 1: The local probabilistic argument}

Here we prove the local Fourier square function inequality,%
\[
\left\Vert \mathcal{S}_{T,s}^{\eta}\left(  \mathsf{Q}_{K}^{s}\right)
^{\spadesuit}f\right\Vert _{L^{p}\left(  A_{+}\left(  0,2^{2s-w}\right)
\right)  }^{{}}\lesssim2^{-s\varepsilon_{p,n}}\left\Vert \left(
\mathsf{Q}_{K}^{s}\right)  ^{\spadesuit}f\right\Vert _{L^{p}\left(
\mathbb{R}^{n-1}\right)  },\ \ \ \ \ \text{for all }K\in\mathcal{G}%
_{s-w}\left[  U\right]  \text{ and }s\in\mathbb{N},
\]
which by Khintchine's inequalities is equivalent to the local average
expectation inequality,%
\[
\mathbb{E}_{2^{\mathcal{G}\left[  U\right]  }}^{\mu}\left\Vert T\left(
\mathcal{A}_{\mathbf{a}}\mathsf{Q}_{K}^{s}\right)  ^{\spadesuit}f\right\Vert
_{L^{p}\left(  A_{+}\left(  0,2^{2s-w}\right)  \right)  }\lesssim
2^{-s\varepsilon_{p,n}}\left\Vert \left(  \mathsf{Q}_{K}^{s}\right)
^{\spadesuit}f\right\Vert _{L^{p}\left(  \mathbb{R}^{n-1}\right)
},\ \ \ \ \ \text{for all }K\in\mathcal{G}_{s-w}\left[  U\right]  \text{ and
}s\in\mathbb{N}.
\]

Consider $\left(  I,J\right)  \in\mathcal{R}_{s}^{k,w}$, i.e. $I\in
\mathcal{G}_{s}\left[  S\right]  $, $\ell\left(  J\right)  =2^{k}$ and
$J\subset P_{s,w}^{I}$. Recall that $T_{s,w}^{I}$ is the tube given by the
convex hull of the pipe $P_{s,w}^{I}$. For $0<w<s$, these tubes have bounded
overlap approximately $2^{w\left(  n-1\right)  }$.

\begin{definition}
For each $K\in\mathcal{G}_{s-w}$ define a `tube' $T_{s,w}^{K,\natural}\equiv%
%TCIMACRO{\dbigcup \limits_{I\in\mathcal{G}_{s}\left[  K\right]  }}%
%BeginExpansion
{\displaystyle\bigcup\limits_{I\in\mathcal{G}_{s}\left[  K\right]  }}
%EndExpansion
T_{s,w}^{I}$ consisting of all the tubes $T_{s,w}^{I}$ with $I\subset K$,
where each tube $T_{s,w}^{I}$ has dimensions $C_{1}2^{s}\times2^{2s-w}$, and
due to the $2^{w\left(  n-1\right)  }$-overlap, each of the `tubes' $T%
_{s,w}^{K,\natural}$ also has dimensions $C_{2}2^{s}\times2^{2s-w}$, but with
a larger constant $C_{2}$.
\end{definition}

We begin with the following more elementary local average inequality for
$0\leq w\leq s$, in which we restrict the integration over $\mathbb{R}^{n}$ to
the tubes $T_{s,w}^{K,\natural}$,\
\begin{equation}
\mathbb{E}_{2^{\mathcal{G}_{s}\left[  S\right]  }}^{\mu}\left\Vert
T_{S}\left(  \mathcal{A}_{\mathbf{a}}\mathsf{Q}_{K}^{s}\right)  ^{\spadesuit
}f\right\Vert _{L^{p}\left(  T_{s,w}^{K,\natural}\right)  }\lesssim2^{-\left(
2s-w\right)  \varepsilon_{p,n}}\left\Vert f\right\Vert _{L^{p}\left(
U\right)  }\ ,\ \ \ \ \ \text{for }K\in\mathcal{G}_{s-w}\left[  U\right]
\text{ and }p>\frac{2n}{n-1}.\label{difficult}%
\end{equation}
To prove this, we consider the $L^{2}$ and average $L^{4}$ bounds separately
and then interpolate.

\medskip

\textbf{Step 1(a): local }$L^{2}$\textbf{\ estimate}

\medskip

We first compute the norm of $\Lambda_{\mathsf{Q}_{K}^{s}}^{2s}$ from
$L^{2}\left(  \lambda_{n-1}\right)  $ to $L^{2}\left(  T_{s,w}^{K,\natural
}\right)  $, where we recall that
\[
\Lambda_{\mathsf{Q}_{K}^{s}}^{2s}f\equiv\widehat{\left(  \left(
\mathsf{Q}_{K}^{s}\right)  ^{\spadesuit}f\right)  _{\Phi,2s}}.
\]
Consistent with (\ref{def Phi and I}), we write and%
\begin{align}
f_{K}^{s}  & \equiv\left(  \mathsf{Q}_{K}^{s}\right)  ^{\spadesuit
}f\ ,\label{cont not}\\
\left(  f_{K}^{s}\right)  _{\Phi}  & \equiv\Phi_{\ast}\left[  \left(
\mathsf{Q}_{K}^{s}\right)  ^{\spadesuit}f\right]  =\sum_{I\in\mathcal{G}%
_{s}\left(  K\right)  }\Phi_{\ast}\left[  \left(  \bigtriangleup_{I;\kappa
}^{n-1}\right)  ^{\spadesuit}f\right]  =\sum_{I\in\mathcal{G}_{s}\left(
K\right)  }f_{\Phi}^{I}\ ,\nonumber\\
\left(  f_{K}^{s}\right)  _{\Phi,r}  & =\sum_{I\in\mathcal{G}_{s}\left(
K\right)  }f_{\Phi,r}^{I}\ .\nonumber
\end{align}
For $I_{0}\in\mathcal{G}_{s}\left[  K\right]  $, whose normal is
$\mathbf{e}_{n}$, we will use the rectangular convolver $\varphi
_{s,2s-w}\left(  z\right)  $ that has dimensions $2^{-s}\times...\times
2^{-s}\times2^{w-2s}$, and we will multiply by a modulation $m\left(
z\right)  $ that translates the associated Fourier tube $\left[  -2^{s}%
,2^{s}\right]  ^{n-1}\times\left[  -2^{2s-w},2^{2s-w}\right]  $ to be
positioned near $T_{s,w}^{K,\natural}$. For convenience we momentarily set
\begin{equation}
\psi\left(  z\right)  \equiv m\left(  z\right)  \varphi_{s,2s-w}\left(
z\right)  .\label{def psi}%
\end{equation}
We then have with $f_{K}^{s}=\left(  \mathsf{Q}_{K}^{s}\right)  ^{\spadesuit}f
$,%
\begin{align*}
& \left\Vert \Lambda_{\mathsf{Q}_{K}^{s}}^{2s}f\right\Vert _{L^{2}\left(
\left\vert \widehat{\psi}\right\vert ^{2}\lambda_{n}\right)  }^{2}%
=\int_{\mathbb{R}^{n}}\left\vert \widehat{\left(  f_{K}^{s}\right)  _{\Phi
,2s}}\left(  \xi\right)  \right\vert ^{2}\left\vert \widehat{\psi}\left(
\xi\right)  \right\vert ^{2}d\xi=\int_{\mathbb{R}^{n}}\overline{\widehat
{\left(  f_{K}^{s}\right)  _{\Phi,2s}\ast\psi}\left(  \xi\right)  }%
\ \widehat{\left(  f_{K}^{s}\right)  _{\Phi,2s}\ast\psi}\left(  \xi\right)
d\xi\\
& =\sum_{I,J\in\mathcal{G}_{s}\left[  K\right]  }\int_{\mathbb{R}^{n}%
}\overline{\widehat{f_{\Phi,2s}^{I}\ast\psi}\left(  \xi\right)  }%
\ \widehat{f_{\Phi,2s}^{J}\ast\psi}\left(  \xi\right)  d\xi=\sum
_{I,J\in\mathcal{G}_{s}\left[  K\right]  }\int_{S}\overline{f_{\Phi,2s}%
^{I}\ast\psi\left(  x\right)  }\ \left(  f_{\Phi,2s}^{J}\ast\psi\right)
\left(  x\right)  dx.
\end{align*}

Note first that the supports of $f_{\Phi,2s}^{I}\ast\psi$ and $f_{\Phi,2s}%
^{J}\ast\psi$ are essentially disjoint unless $I\sim J$. Next, if we define
the fattened cube
\[
I_{0}^{\ast}\equiv\left(  \left[  -2^{-s},2^{-s}\right]  ^{n-1}\times\left[
-2^{w-2s},2^{w-2s}\right]  \right)  +\mathbf{e}_{n}\ ,
\]
and $I^{\ast}$ by rotation, then we have%
\[
\left\vert f_{\Phi,2s}^{I}\ast\psi\left(  z\right)  \right\vert \lesssim
\left\vert \left\langle S_{\kappa,\eta}^{-1}f,h_{I;\kappa}^{n-1}\right\rangle
\right\vert 2^{2s-w}2^{s\frac{n-1}{2}}\mathbf{1}_{I^{\ast}}\left(  z\right)  ,
\]
since%
\[
\left\vert f_{\Phi,2s}^{I}\ast\psi\right\vert \approx\left\vert f_{\Phi}%
^{I}\ast\psi\right\vert \lesssim\left\Vert \frac{df_{\Phi}^{I}}{d\sigma_{n-1}%
}\right\Vert _{\infty}\times\left(  \mathbf{1}_{\Phi\left(  I\right)  }%
\sigma_{n-1}\right)  \ast\varphi_{s,2s-w}\left(  z\right)  \approx\left\vert
\left\langle S_{\kappa,\eta}^{-1}f,h_{I;\kappa}^{n-1}\right\rangle \right\vert
2^{s\frac{n-1}{2}}\times\left(  \operatorname*{density}\right)  \mathbf{1}%
_{I^{\ast}}\left(  z\right)  ,
\]
where the quantity $\operatorname*{density}$ (of the convolution with
$\varphi_{s,2s-w}$ ) satisfies,
\begin{align*}
\left(  \operatorname*{density}\right)  2^{-s\left(  n-1\right)  }2^{w-2s}  &
=\left(  \operatorname*{density}\right)  \left\vert I^{\ast}\right\vert
=\left\Vert \mathbf{1}_{\Phi\left(  I\right)  }\sigma_{n-1}\right\Vert
=2^{-s\left(  n-1\right)  }\\
& \Longrightarrow\operatorname*{density}=\frac{2^{-s\left(  n-1\right)  }%
}{2^{-s\left(  n-1\right)  }2^{w-2s}}=2^{2s-w}.
\end{align*}

Altogether then, using $\left\vert I^{\ast}\right\vert =2^{-s\left(
n-1\right)  }2^{w-2s}$, we have from (\ref{cont not}) that%
\begin{align*}
& \ \ \ \ \ \ \ \ \ \ \ \ \ \ \ \left\Vert \Lambda_{\mathsf{Q}_{K}^{s}}%
^{2s}f\right\Vert _{L^{2}\left(  \left\vert \widehat{\psi}\right\vert
^{2}\lambda_{n}\right)  }^{2}\lesssim\int_{\mathbb{R}^{n}}\left\vert \left(
f_{K}^{s}\right)  _{\Phi,2s}\ast\psi\left(  \xi\right)  \right\vert ^{2}%
d\xi=\sum_{I\in\mathcal{G}_{s}\left[  K\right]  }\int_{\mathbb{R}^{n}%
}\left\vert f_{\Phi,2s}^{I}\ast\psi\left(  \xi\right)  \right\vert ^{2}d\xi\\
& \lesssim\sum_{I\in\mathcal{G}_{s}\left[  K\right]  }\int_{\mathbb{R}^{n}%
}\left\vert \left\vert \left\langle S_{\kappa,\eta}^{-1}f,h_{I;\kappa}%
^{n-1}\right\rangle \right\vert 2^{2s-w}2^{s\frac{n-1}{2}}\mathbf{1}_{I^{\ast
}}\left(  \xi\right)  \right\vert ^{2}d\xi\lesssim\sum_{I\in\mathcal{G}%
_{s}\left[  K\right]  }\left\vert \left\langle S_{\kappa,\eta}^{-1}%
f,h_{I;\kappa}^{n-1}\right\rangle \right\vert ^{2}\left(  2^{2s-w}%
2^{s\frac{n-1}{2}}\right)  ^{2}\left\vert I^{\ast}\right\vert \\
& =2^{4s-2w}2^{s\left(  n-1\right)  }2^{-s\left(  n-1\right)  }2^{w-2s}%
\sum_{I\in\mathcal{G}_{s}\left[  K\right]  }\left\vert \left\langle
S_{\kappa,\eta}^{-1}f,h_{I;\kappa}^{n-1}\right\rangle \right\vert
^{2}=2^{2s-w}\sum_{I\in\mathcal{G}_{s}\left[  K\right]  }\left\vert
\left\langle S_{\kappa,\eta}^{-1}f,h_{I;\kappa}^{n-1}\right\rangle \right\vert
^{2}\lesssim2^{2s-w}\left\Vert f_{K}^{s}\right\Vert _{L^{2}\left(  U\right)
}^{2}.
\end{align*}
In terms of the notation $T\left(  \mathsf{Q}_{K}^{s}\right)  ^{\spadesuit}f$,
this implies%
\begin{equation}
\left\Vert T\left(  \mathsf{Q}_{K}^{s}\right)  ^{\spadesuit}f\right\Vert
_{L^{2}\left(  T_{s,w}^{K,\natural}\right)  }^{2}\lesssim2^{2s-w}\left\Vert
\left(  \mathsf{Q}_{K}^{s}\right)  ^{\spadesuit}f\right\Vert _{L^{2}\left(
U\right)  }^{2}\label{est above exp'}%
\end{equation}

\medskip

\textbf{Step 1(b): local average }$L^{4}$\textbf{\ estimate}

\medskip

We run the argument in Subsection \ref{subsubsection L4} up until the estimate
for $\Omega_{t}=\Omega_{t}\left[  K\right]  $, where $2^{-t}\approx
\operatorname*{dist}\left(  I,J\right)  $ for $I,J\in\mathcal{G}_{s}\left[
K\right]  $, i.e. $2^{-t}\lesssim\ell\left(  K\right)  =2^{w-s}$ or $s-w\leq
t\leq s$. It is this restriction to large $t$ that yields the geometric gain
needed for the average $L^{4}$ estimate when $I,J\in\mathcal{G}_{s}\left[
K\right]  $. Then for $s-w<t<s$, and \emph{with notation as in Subsection
\ref{subsubsection L4}}, we have%
\begin{align*}
& \Omega_{t}\left[  K\right]  \lesssim\sum_{I,J\in\mathcal{G}_{s}\left[
K\right]  :\ \operatorname*{dist}\left(  I,J\right)  \approx2^{-t}%
}2^{-s\left(  n-2\right)  }2^{t}\left\vert \left\langle \left(  S_{\kappa
,\eta}\right)  ^{-1}f,h_{I;\kappa}\right\rangle \left\langle \left(
S_{\kappa,\eta}\right)  ^{-1}f,h_{J;\kappa}\right\rangle \right\vert ^{2}\\
& \lesssim2^{-s\left(  n-2\right)  }2^{t}\sum_{I,J\in\mathcal{G}_{s}\left[
K\right]  :\ \operatorname*{dist}\left(  I,J\right)  \approx2^{-t}}\left\vert
\left\langle \left(  S_{\kappa,\eta}\right)  ^{-1}f,h_{I;\kappa}\right\rangle
\right\vert ^{4}\\
& \lesssim2^{-s\left(  n-2\right)  }2^{t}2^{\left(  s-t\right)  \left(
n-1\right)  }\sum_{I\in\mathcal{G}_{s}\left[  K\right]  }\left\vert
\left\langle \left(  S_{\kappa,\eta}\right)  ^{-1}f,h_{I;\kappa}\right\rangle
\right\vert ^{4}=2^{-t\left(  n-2\right)  }2^{-s\left(  n-2\right)
}\left\Vert \mathsf{Q}_{K}^{s}\left(  S_{\kappa,\eta}\right)  ^{-1}%
f\right\Vert _{L^{4}\left(  U\right)  }^{4}\ ,
\end{align*}
which gives%
\begin{align*}
\sum_{t=s-w}^{s}\Psi_{t}\left[  K\right]   & \lesssim\sum_{t=s-w}^{s}%
\Omega_{t}\left[  K\right]  \lesssim\sum_{t=s-w}^{s}2^{-t\left(  n-2\right)
}2^{-s\left(  n-2\right)  }\left\Vert \mathsf{Q}_{K}^{s}\left(  S_{\kappa
,\eta}\right)  ^{-1}f\right\Vert _{L^{4}\left(  U\right)  }^{4}\\
& \approx2^{-\left(  s-w\right)  \left(  n-2\right)  }2^{-s\left(  n-2\right)
}\left\Vert \mathsf{Q}_{K}^{s}\left(  S_{\kappa,\eta}\right)  ^{-1}%
f\right\Vert _{L^{4}\left(  U\right)  }^{4}=2^{-\left(  2s-w\right)  \left(
n-2\right)  }\left\Vert \mathsf{Q}_{K}^{s}\left(  S_{\kappa,\eta}\right)
^{-1}f\right\Vert _{L^{4}\left(  U\right)  }^{4}.
\end{align*}
Similarly we obtain%
\[
\Psi\lesssim2^{-\left(  2s-w\right)  \left(  n-2\right)  }\left\Vert
\mathsf{Q}_{K}^{s}\left(  S_{\kappa,\eta}\right)  ^{-1}f\right\Vert
_{L^{4}\left(  U\right)  }^{4},
\]
and adding these last two inequalities gives,
\[
\mathbb{E}_{2^{\mathcal{G}}}^{\mu}\left\Vert \Lambda_{\mathcal{A}_{\mathbf{a}%
}\mathsf{Q}_{K}^{s}}^{2s}f\right\Vert _{L^{4}\left(  \left\vert \widehat{\psi
}\right\vert ^{2}\lambda_{n}\right)  }^{4}\lesssim2^{-\left(  2s-w\right)
\left(  n-2\right)  }\left\Vert f\right\Vert _{L^{4}\left(  U\right)  }^{4}.
\]
In terms of the notation $T\left(  \mathsf{Q}_{K}^{s}\right)  ^{\spadesuit}f$,
this implies%
\begin{equation}
\mathbb{E}_{2^{\mathcal{G}}}^{\mu}\left\Vert T\left(  \mathsf{Q}_{K}%
^{s}\right)  ^{\spadesuit}f\right\Vert _{L^{4}\left(  T_{s,w}^{K,\natural
}\right)  }^{4}\lesssim2^{-\left(  2s-w\right)  \left(  n-2\right)
}\left\Vert \left(  \mathsf{Q}_{K}^{s}\right)  ^{\spadesuit}f\right\Vert
_{L^{4}\left(  U\right)  }^{4}\label{est above exp''}%
\end{equation}

\medskip

\textbf{Step 1(c): local interpolation}

\medskip

Collecting the bounds (\ref{est above exp'}) and (\ref{est above exp''})
gives,%
\begin{align*}
\left\Vert T\left(  \mathsf{Q}_{K}^{s}\right)  ^{\spadesuit}f\right\Vert
_{L^{2}\left(  T_{s,w}^{K,\natural}\right)  }  & \lesssim2^{\frac{2s-w}{2}%
}\left\Vert f\right\Vert _{L^{2}\left(  K\right)  },\\
\mathbb{E}_{2^{\mathcal{G}}}^{\mu}\left\Vert T\left(  \mathsf{Q}_{K}%
^{s}\right)  ^{\spadesuit}f\right\Vert _{L^{4}\left(  T_{s,w}^{K,\natural
}\right)  }  & \lesssim2^{-\frac{2s-w}{2}\frac{n-2}{2}}\left\Vert f\right\Vert
_{L^{4}\left(  S\right)  }^{4}.
\end{align*}
Now we claim that an application of the interpolation Lemma \ref{interp red}
yields,
\begin{equation}
\mathbb{E}_{2^{\mathcal{G}}}^{\mu}\left\Vert T\left(  \mathcal{A}_{\mathbf{a}%
}\mathsf{Q}_{K}^{s}\right)  ^{\spadesuit}f\right\Vert _{L^{p}\left(  T%
_{s,w}^{K,\natural}\right)  }\lesssim2^{-\left(  2s-w\right)  \varepsilon
_{p,n}^{\prime}}\left\Vert f\right\Vert _{L^{p}\left(  U\right)
},\ \ \ \ \ \text{for }p>\frac{2n}{n-1}.\label{Lp avg}%
\end{equation}
Indeed, the calculation at the end of the proof of Lemma \ref{interp red}
shows that if $p>\frac{2n}{n-1}$, then (with notation as in that proof)
$\theta=\frac{4}{p}-1$ and so
\[
\left[  2^{-\frac{2s-w}{2}\frac{n-2}{2}}\right]  ^{1-\theta}\left[
2^{\frac{2s-w}{2}}\right]  ^{\theta}=2^{-\frac{2s-w}{2}\frac{n-2}{2}%
}2^{\left(  \frac{2s-w}{2}+\frac{2s-w}{2}\frac{n-2}{2}\right)  \theta
}=2^{-\frac{2s-w}{2}\frac{n-2}{2}}2^{\left(  \frac{2s-w}{2}\frac{n}{2}\right)
\theta}=2^{-\left(  2s-w\right)  \varepsilon_{p,n}^{\prime}},
\]
where
\begin{align*}
\varepsilon_{p,n}^{\prime}  & \equiv\frac{1}{2s-w}\left\{  \frac{2s-w}{2}%
\frac{n-2}{2}-\left(  \frac{2s-w}{2}\frac{n}{2}\right)  \left(  \frac{4}%
{p}-1\right)  \right\} \\
& =\frac{n-2}{4}-\frac{n}{4}\left(  \frac{4}{p}-1\right)  =\frac{n-1}{2}%
-\frac{n}{p}=\frac{n-1}{2p}\left(  p-\frac{2n}{n-1}\right)  .
\end{align*}
This completes our proof of (\ref{difficult}) in \emph{Step 1}.

\subsubsection{Step 2: The local expanded probabilistic argument}

Now we turn to proving the expanded analogue of (\ref{difficult}) given by,%
\begin{align}
& \mathbb{E}_{2^{\mathcal{G}_{s}\left[  U\right]  }}^{\mu}\left\Vert T\left(
\mathcal{A}_{\mathbf{a}}\mathsf{Q}_{K}^{s}\right)  ^{\spadesuit}f\right\Vert
_{L^{p}\left(  P_{s,w}^{K}\left[  r\right]  \right)  }^{p}\lesssim
2^{-rp\left(  N-\frac{n}{p^{\prime}}\right)  }2^{-\left(  2s-w\right)
p\varepsilon_{p,n}}\left\Vert f\right\Vert _{L^{p}\left(  \mathbb{R}%
^{n-1}\right)  }^{p}\label{difficult exp}\\
& \ \ \ \ \ \ \ \ \ \ \ \ \ \ \ \ \ \ \ \ \text{for all }K\in\mathcal{G}%
_{s-w}\left[  S\right]  \text{ and }p>\frac{2n}{n-1},\nonumber
\end{align}
where $\delta>0$ and $P_{s,w}^{K}\left[  r\right]  $ is the expanded pipe
corresponding to the tube $T_{s,w}^{K}$. This is proved in the same way as the
case of the tube $T_{s,w}^{K,\natural}$ in the previous subsubsection, except
that we use the geometric decay in $r$ derived from integration by parts and
the fact that the expanded pipe $P_{s,w}^{K}\left[  r\right]  $ is far from
the tube $T_{s,w}^{K}$, to compensate the geometric growth in $r$ that arises
from the expanded pipes.

We will repeat\ the above proof of (\ref{difficult}), but with expanded pipes
$P_{s,w}^{K}\left[  r\right]  $ in place of the tube $T_{s,w}^{K}$, to get
(\ref{difficult exp}). Indeed, the $L^{2}$ and average $L^{4}$ estimates
(\ref{est above exp'}) and (\ref{est above exp''}) are now multiplied by an
additional factor $C_{\delta}2^{-r\delta}$ for some $\delta>0$, which
percolates through the interpolation to give (\ref{difficult exp}).

More precisely, we adapt the arguments surrounding (\ref{Lp avg}),%
\[
\mathbb{E}_{2^{\mathcal{G}}}^{\mu}\left\Vert \Lambda_{\mathcal{A}_{\mathbf{a}%
}\mathsf{Q}_{K}^{s}}^{2s}f\right\Vert _{L^{p}\left(  \left\vert \widehat{\psi
}\right\vert ^{2}\lambda_{n}\right)  }\lesssim2^{-\left(  2s-w\right)
\varepsilon_{p,n}^{\prime}}\left\Vert f\right\Vert _{L^{p}\left(
\mathbb{R}^{n-1}\right)  }\ ,
\]
and (\ref{norm est}),%
\[
\left\Vert T\bigtriangleup_{I;\kappa}^{n-1,\eta}f\right\Vert _{L^{p}\left(
P_{s,0}^{I}\left[  r\right]  \right)  }\leq C_{N}2^{-r\left(  N-\frac{n}%
{p}\right)  }2^{-s\varepsilon_{p,n}}2^{-sN}\left\Vert \partial^{N}%
\bigtriangleup_{I;\kappa}^{n-1,\eta}f\right\Vert _{L^{p}\left(  \mathbb{R}%
^{n-1}\right)  }\ ,
\]
to conclude that%
\begin{align*}
\mathbb{E}_{2^{\mathcal{G}\left[  U\right]  }}^{\mu}\left\Vert T\left(
\mathcal{A}_{\mathbf{a}}\mathsf{Q}_{K}^{s}\right)  ^{\spadesuit}f\right\Vert
_{L^{p}\left(  P_{s,w}^{K}\left[  r\right]  \right)  }  & \lesssim
C_{N}2^{-r\left(  N-\frac{n}{p}\right)  }2^{-\varepsilon_{p,n}s}\left\Vert
2^{-sN}\partial^{N}\left(  \mathsf{Q}_{K}^{s}\right)  ^{\spadesuit
}f\right\Vert _{L^{p}\left[  \mathbb{R}^{n-1}\right]  }\ ,\\
\ \ \ \ \ \text{for }K  & \in\mathcal{G}_{s-w}\left[  U\right]  \text{ and
}p>\frac{2n}{n-1}.
\end{align*}
The following three steps are almost verbatim analogues of Steps 1(a), (b) and
(c) above, but we include the details for the sake of completeness. For use in
Step 2(a) below, we note that the analogue of (\ref{dom in mod'}) in the case
$0\leq w\leq s$ is,%
\begin{equation}
\left\vert \int_{\mathbb{R}^{n-1}}e^{-i\Phi\left(  x\right)  \cdot\xi}\left(
\mathsf{Q}_{K}^{s}\right)  ^{\spadesuit}f\left(  x\right)  dx\right\vert
\lesssim C_{N}2^{-\left(  r+s\right)  N}\int_{\mathbb{R}^{n-1}}\left\vert
\partial^{N}\left(  \mathsf{Q}_{K}^{s}\right)  ^{\spadesuit}f\left(  x\right)
\right\vert dx,\ \ \ \ \ \text{for }\xi\in P_{s,w}^{K}\left[  r\right]
.\label{dom in mod w}%
\end{equation}

\medskip

\textbf{Step 2(a): local expanded }$L^{2}$\textbf{\ estimate}

\medskip

We compute the norm of $\Lambda_{\mathsf{Q}_{K}^{s}}^{2s}$ from $L^{2}\left(
\mathbb{R}^{n-1}\right)  $ to $L^{2}\left(  P_{s,w}^{K}\left[  r\right]
\right)  $. For $I_{0}\in\mathcal{G}_{s}\left[  K\right]  $, whose normal is
$\mathbf{e}_{n}$, we now use the \emph{cylindrical}\ convolver $\varphi
_{s,2s-w}^{r}\left(  z\right)  $ that has outer dimensions $2^{-s-r}%
\times2^{w-2s}$, and we will multiply by a modulation $m\left(  z\right)  $
that translates the pipe whose convex hull is the tube $\left[  -2^{s+r}%
,2^{s+r}\right]  ^{n-1}\times\left[  -2^{2s-w},2^{2s-w}\right]  $ to be
positioned near $P_{s,w}^{K}\left[  r\right]  $. For convenience we
momentarily set%
\begin{equation}
\psi\left(  z\right)  \equiv m\left(  z\right)  \varphi_{s,2s-w}^{r}\left(
z\right)  .\label{def psi'}%
\end{equation}
We then have with using (\ref{cont not}) that,%
\begin{align*}
& \left\Vert \Lambda_{\mathsf{Q}_{K}^{s}}^{2s}f\right\Vert _{L^{2}\left(
\left\vert \widehat{\psi}\right\vert ^{2}\lambda_{n}\right)  }^{2}%
=\int_{\mathbb{R}^{n}}\left\vert \widehat{\left(  f_{K}^{s}\right)  _{\Phi
,2s}}\left(  \xi\right)  \right\vert ^{2}\left\vert \widehat{\psi}\left(
\xi\right)  \right\vert ^{2}d\xi=\int_{\mathbb{R}^{n}}\overline{\widehat
{\left(  f_{K}^{s}\right)  _{\Phi,2s}\ast\psi}\left(  \xi\right)  }%
\ \widehat{\left(  f_{K}^{s}\right)  _{\Phi,2s}\ast\psi}\left(  \xi\right)
d\xi\\
& =\sum_{I,J\in\mathcal{G}_{s}\left[  K\right]  }\int_{\mathbb{R}^{n}%
}\overline{\widehat{f_{\Phi,2s}^{I}\ast\psi}\left(  \xi\right)  }%
\ \widehat{f_{\Phi,2s}^{J}\ast\psi}\left(  \xi\right)  d\xi=\sum
_{I,J\in\mathcal{G}_{s}\left[  K\right]  }\int_{S}\overline{f_{\Phi,2s}%
^{I}\ast\psi\left(  x\right)  }\ \left(  f_{\Phi,2s}^{J}\ast\psi\right)
\left(  x\right)  dx.
\end{align*}
The supports of $f_{\Phi,2s}^{I}\ast\psi$ and $f_{\Phi,2s}^{J}\ast\psi$ are
essentially disjoint unless $I\sim J$. Next, if we define
\[
I_{w}^{\ast}\left[  r\right]  \equiv\left(  \left[  -2^{-s},2^{-s}\right]
^{n-1}\times\left[  -2^{w-2s+r},2^{w-2s+r}\right]  \right)  +\mathbf{e}_{n}\ ,
\]
and $I_{w}^{\ast}\left[  r\right]  $ by rotation, then we have%
\begin{equation}
\left\vert f_{\Phi,2s}^{I}\ast\psi\left(  z\right)  \right\vert \lesssim
2^{-rN}\left\vert \left\langle S_{\kappa,\eta}^{-1}f,h_{I;\kappa}%
^{n-1}\right\rangle \right\vert 2^{2s-w-r}2^{s\frac{n-1}{2}}\mathbf{1}%
_{I_{w}^{\ast}\left[  r\right]  }\left(  z\right)  ,\label{conv equ}%
\end{equation}
since $N$ integrations by part gains $2^{-\left(  r+s\right)  N}$ as in
(\ref{dom in mod w}), while $N$ differentiations%
\[
\partial^{N}\bigtriangleup_{I;\kappa}^{n-1,\eta}f=\left\langle S_{\kappa,\eta
}^{-1}f,h_{I;\kappa}^{n-1}\right\rangle \partial^{N}h_{I;\kappa}^{n-1,\eta}.
\]
loses $2^{sN}$, all of which leads to
\begin{align*}
\left\vert f_{\Phi,2s}^{I}\ast\psi\right\vert  & \approx\left\vert f_{\Phi
}^{I}\ast\psi\right\vert \lesssim\left\Vert 2^{-rN}\frac{df_{\Phi}^{I}%
}{d\sigma_{n-1}}\right\Vert _{\infty}\times\left(  \mathbf{1}_{\Phi\left(
I\right)  }\sigma_{n-1}\right)  \ast\varphi_{s,2s-w}^{r}\left(  z\right) \\
& \approx2^{-rN}\left\vert \left\langle S_{\kappa,\eta}^{-1}f,h_{I;\kappa
}^{n-1}\right\rangle \right\vert 2^{s\frac{n-1}{2}}\times\left(
\operatorname*{density}\right)  \mathbf{1}_{I_{w}^{\ast}\left[  r\right]
}\left(  z\right)  ,
\end{align*}
where the quantity $\operatorname*{density}$ satisfies,
\begin{align*}
\left(  \operatorname*{density}\right)  2^{-s\left(  n-1\right)  }2^{w-2s+r}
& =\left(  \operatorname*{density}\right)  \left\vert I^{\ast}\left[
r\right]  \right\vert =\left\Vert \mathbf{1}_{\Phi\left(  I^{\ast}\left[
r\right]  \right)  }\sigma_{n-1}\right\Vert =2^{-s\left(  n-1\right)  }\\
& \Longrightarrow\operatorname*{density}=\frac{2^{-s\left(  n-1\right)  }%
}{2^{-s\left(  n-1\right)  }2^{w-2s+r}}=2^{2s-w-r}.
\end{align*}

Altogether then, using (\ref{conv equ}) and $\left\vert I^{\ast}\left[
r\right]  \right\vert =2^{-s\left(  n-1\right)  }2^{w-2s+r}$, we have%
\begin{align*}
& \ \ \ \ \ \ \ \ \ \ \ \ \ \ \ \left\Vert \Lambda_{\mathsf{Q}_{K}^{s}}%
^{2s}f\right\Vert _{L^{2}\left(  \left\vert \widehat{\psi}\right\vert
^{2}\lambda_{n}\right)  }^{2}\lesssim\sum_{I\in\mathcal{G}_{s}\left[
K\right]  }\int_{\mathbb{R}^{n}}\left\vert f_{\Phi,2s}^{I}\ast\psi\left(
\xi\right)  \right\vert ^{2}d\xi\\
& \lesssim2^{-2rN}\sum_{I\in\mathcal{G}_{s}\left[  K\right]  }\int
_{\mathbb{R}^{n}}\left\vert \left\vert \left\langle S_{\kappa,\eta}%
^{-1}f,h_{I;\kappa}^{n-1}\right\rangle \right\vert 2^{2s-w-r}2^{s\frac{n-1}%
{2}}\mathbf{1}_{I_{w}^{\ast}\left[  r\right]  }\left(  \xi\right)  \right\vert
^{2}d\xi\\
& \lesssim2^{-2rN}\sum_{I\in\mathcal{G}_{s}\left[  K\right]  }\left\vert
\left\langle S_{\kappa,\eta}^{-1}f,h_{I;\kappa}^{n-1}\right\rangle \right\vert
^{2}\left(  2^{2s-w-r}2^{s\frac{n-1}{2}}\right)  ^{2}\left\vert I_{w}^{\ast
}\right\vert \\
& =2^{-2rN}2^{2s-w-r}\sum_{I\in\mathcal{G}_{s}\left[  K\right]  }\left\vert
\left\langle S_{\kappa,\eta}^{-1}f,h_{I;\kappa}^{n-1}\right\rangle \right\vert
^{2}\lesssim2^{-\left(  2N+1\right)  r}2^{2s-w}\left\Vert f\right\Vert
_{L^{2}\left(  \mathbb{R}^{n-1}\right)  }^{2},
\end{align*}
which in terms of $T\left(  \mathsf{Q}_{K}^{s}\right)  ^{\spadesuit}f$ implies%
\begin{equation}
\left\Vert T\left(  \mathsf{Q}_{K}^{s}\right)  ^{\spadesuit}f\right\Vert
_{L^{2}\left(  P_{s,w}^{K}\left[  r\right]  \right)  }^{2}\lesssim2^{-\left(
2N+1\right)  r}2^{2s-w}\left\Vert \left(  \mathsf{Q}_{K}^{s}\right)
^{\spadesuit}f\right\Vert _{L^{2}\left(  U\right)  }^{2}%
\label{est above exp'''}%
\end{equation}

\medskip

\textbf{Step 2(b): local average expanded }$L^{4}$\textbf{\ estimate}

\medskip

We begin by using (\ref{dom in mod w}) to estimate the $L^{4}\left(  P%
_{s,w}^{K}\left[  r\right]  \right)  $ norm of $\Lambda_{\mathsf{Q}_{K}^{s}%
}^{2s}f$:%
\begin{align*}
& \left\Vert \Lambda_{\mathsf{Q}_{K}^{s}}^{2s}f\right\Vert _{L^{4}\left(
P_{s,w}^{K}\left[  r\right]  \right)  }^{4}=\int_{P_{s,w}^{K}\left[  r\right]
}\left\vert \widehat{\left(  f_{K}^{s}\right)  _{\Phi,2s}}\left(  \xi\right)
\right\vert ^{4}d\xi=\int_{P_{s,w}^{K}\left[  r\right]  }\left\vert \sum
_{I\in\mathcal{G}_{s}\left[  K\right]  }\widehat{\left(  f_{K}^{s}\right)
_{\Phi,2s}^{I}}\left(  \xi\right)  \right\vert ^{4}d\xi\\
& \lesssim2^{-4\left(  r+s\right)  N}\int_{P_{s,w}^{K}\left[  r\right]
}\left\vert \sum_{I\in\mathcal{G}_{s}\left[  K\right]  }\widehat{\partial
^{N}f_{\Phi,2s}^{I}}\left(  \xi\right)  \right\vert ^{4}d\xi=2^{-4\left(
r+s\right)  N}\int_{P_{s,w}^{K}\left[  r\right]  }\left\vert \sum
_{I,J\in\mathcal{G}_{s}\left[  K\right]  }\widehat{\partial^{N}f_{\Phi,2s}%
^{I}}\left(  \xi\right)  \widehat{\partial^{N}f_{\Phi,2s}^{J}}\left(
\xi\right)  \right\vert ^{2}d\xi\\
& =2^{-4\left(  r+s\right)  N}\int_{P_{s,w}^{K}\left[  r\right]  }\left\vert
\sum_{I,J\in\mathcal{G}_{s}\left[  K\right]  }\widehat{\partial^{N}f_{\Phi
,2s}^{I}\ast\partial^{N}f_{\Phi,2s}^{J}}\widehat{\partial}\left(  \xi\right)
\right\vert ^{2}d\xi.
\end{align*}
Then we run the argument in Subsection \ref{subsubsection L4}, with notation
as used there, with the above estimate up until the estimate for $\Omega
_{t}=\Omega_{t}\left[  K\right]  $, where $2^{-t}\approx\operatorname*{dist}%
\left(  I,J\right)  $ for $I,J\in\mathcal{G}_{s}\left[  K\right]  $, i.e.
$2^{-t}\lesssim\ell\left(  K\right)  =2^{w-s}$ or $s-w\leq t\leq s$. Then for
$s-w<t<s$ we have%
\begin{align*}
& \Omega_{t}\left[  K\right]  \lesssim2^{-\left(  4N+2\right)  r}\sum
_{I,J\in\mathcal{G}_{s}\left[  K\right]  :\ \operatorname*{dist}\left(
I,J\right)  \approx2^{-t}}2^{-s\left(  n-2\right)  }2^{t}\left\vert
\left\langle \left(  S_{\kappa,\eta}\right)  ^{-1}f,h_{I;\kappa}\right\rangle
\left\langle \left(  S_{\kappa,\eta}\right)  ^{-1}f,h_{J;\kappa}\right\rangle
\right\vert ^{2}\\
& \lesssim2^{-\left(  4N+2\right)  r}2^{-s\left(  n-2\right)  }2^{t}%
\sum_{I,J\in\mathcal{G}_{s}\left[  K\right]  :\ \operatorname*{dist}\left(
I,J\right)  \approx2^{-t}}\left\vert \left\langle \left(  S_{\kappa,\eta
}\right)  ^{-1}f,h_{I;\kappa}\right\rangle \right\vert ^{4}\\
& \lesssim2^{-\left(  4N+2\right)  r}2^{-s\left(  n-2\right)  }2^{t}2^{\left(
s-t\right)  \left(  n-1\right)  }\sum_{I\in\mathcal{G}_{s}\left[  K\right]
}\left\vert \left\langle \left(  S_{\kappa,\eta}\right)  ^{-1}f,h_{I;\kappa
}\right\rangle \right\vert ^{4}\approx2^{-\left(  4N+2\right)  r}2^{-t\left(
n-2\right)  }2^{-s\left(  n-2\right)  }\left\Vert \left(  \mathsf{Q}_{K}%
^{s}\right)  ^{\spadesuit}f\right\Vert _{L^{4}\left(  S\right)  }^{4}\ ,
\end{align*}
which gives%
\begin{align*}
\sum_{t=s-w}^{s}\Psi_{t}\left[  K\right]   & \lesssim\sum_{t=s-w}^{s}%
\Omega_{t}\left[  K\right]  \lesssim2^{-\left(  4N+2\right)  r}\sum
_{t=s-w}^{s}2^{-t\left(  n-2\right)  }2^{-s\left(  n-2\right)  }\left\Vert
\left(  \mathsf{Q}_{K}^{s}\right)  ^{\spadesuit}f\right\Vert _{L^{4}\left(
S\right)  }^{4}\\
& \approx2^{-\left(  4N+2\right)  r}2^{-\left(  s-w\right)  \left(
n-2\right)  }2^{-s\left(  n-2\right)  }\left\Vert \left(  \mathsf{Q}_{K}%
^{s}\right)  ^{\spadesuit}f\right\Vert _{L^{4}\left(  S\right)  }^{4}%
\lesssim2^{-\left(  4N+2\right)  r}2^{-\left(  2s-w\right)  \left(
n-2\right)  }\left\Vert f\right\Vert _{L^{4}\left(  S\right)  }^{4}.
\end{align*}
Similarly we obtain%
\[
\Psi\lesssim2^{-\left(  4N+2\right)  r}2^{-\left(  2s-w\right)  \left(
n-2\right)  }\left\Vert f\right\Vert _{L^{4}\left(  S\right)  }^{4},
\]
and adding these results gives,
\[
\mathbb{E}_{2^{\mathcal{G}}}^{\mu}\left\Vert \Lambda_{\mathsf{Q}_{K}^{s}}%
^{2s}f\right\Vert _{L^{4}\left(  \lambda_{n}\right)  }^{4}\lesssim2^{-\left(
4N+2\right)  r}2^{-\left(  2s-w\right)  \left(  n-2\right)  }\left\Vert
f\right\Vert _{L^{4}\left(  S\right)  }^{4}.
\]
In terms of $T\left(  \mathsf{Q}_{K}^{s}\right)  ^{\spadesuit}f$ this implies%
\begin{equation}
\mathbb{E}_{2^{\mathcal{G}}}^{\mu}\left\Vert T\left(  \mathcal{A}_{\mathbf{a}%
}\mathsf{Q}_{K}^{s}\right)  ^{\spadesuit}f\right\Vert _{L^{4}\left(  P%
_{s,w}^{K}\left[  r\right]  \right)  }^{4}\lesssim2^{-\left(  4N+2\right)
r}2^{-\left(  2s-w\right)  \left(  n-2\right)  }\left\Vert \left(
\mathsf{Q}_{K}^{s}\right)  ^{\spadesuit}f\right\Vert _{L^{4}\left(  U\right)
}^{4}\label{est above exp''''}%
\end{equation}

\medskip

\textbf{Step 2(c): local expanded interpolation}

\medskip

Collecting the bounds (\ref{est above exp'''}) and (\ref{est above exp''''})
gives,%
\begin{align*}
\left\Vert T\left(  \mathcal{A}_{\mathbf{a}}\mathsf{Q}_{K}^{s}\right)
^{\spadesuit}f\right\Vert _{L^{2}\left(  P_{s,w}^{K}\left[  r\right]  \right)
}  & \lesssim2^{-\left(  N+\frac{1}{2}\right)  r}2^{\frac{2s-w}{2}}\left\Vert
f\right\Vert _{L^{2}\left(  U\right)  },\\
\mathbb{E}_{2^{\mathcal{G}}}^{\mu}\left\Vert T\left(  \mathcal{A}_{\mathbf{a}%
}\mathsf{Q}_{K}^{s}\right)  ^{\spadesuit}f\right\Vert _{L^{4}\left(  P%
_{s,w}^{K}\left[  r\right]  \right)  }  & \lesssim2^{-\left(  N+\frac{1}%
{2}\right)  r}2^{-\frac{2s-w}{2}\frac{n-2}{2}}\left\Vert f\right\Vert
_{L^{4}\left(  U\right)  }.
\end{align*}
Now we claim that an application of the interpolation Lemma \ref{interp red}
yields,
\[
\mathbb{E}_{2^{\mathcal{G}}}^{\mu}\left\Vert T\left(  \mathcal{A}_{\mathbf{a}%
}\mathsf{Q}_{K}^{s}\right)  ^{\spadesuit}f\right\Vert _{L^{p}\left(  P%
_{s,w}^{K}\left[  r\right]  \right)  }\lesssim2^{-\left(  N+\frac{1}%
{2}\right)  r}2^{-\left(  2s-w\right)  \varepsilon_{p,n}^{\prime}}\left\Vert
f\right\Vert _{L^{p}\left(  \mathbb{R}^{n-1}\right)  }.
\]
Indeed, the calculation at the end of the proof of Lemma \ref{interp red}
shows that if $p>\frac{2n}{n-1}$, then (with notation as in that proof)
$\theta=\frac{4}{p}-1$ and so
\begin{align*}
& \left[  2^{-\left(  N+\frac{1}{2}\right)  r}2^{-\frac{2s-w}{2}\frac{n-2}{2}%
}\right]  ^{1-\theta}\left[  2^{-\left(  N+\frac{1}{2}\right)  r}%
2^{\frac{2s-w}{2}}\right]  ^{\theta}=2^{-\left(  N+\frac{1}{2}\right)
r}2^{-\frac{2s-w}{2}\frac{n-2}{2}}2^{\left(  \frac{2s-w}{2}+\frac{2s-w}%
{2}\frac{n-2}{2}\right)  \theta}\\
& =2^{-\left(  N+\frac{1}{2}\right)  r}2^{-\frac{2s-w}{2}\frac{n-2}{2}%
}2^{\left(  \frac{2s-w}{2}\frac{n}{2}\right)  \theta}=2^{-\left(  N+\frac
{1}{2}\right)  r}2^{-\left(  2s-w\right)  \varepsilon_{p,n}^{\prime}},
\end{align*}
where
\begin{align*}
\varepsilon_{p,n}^{\prime}  & \equiv\frac{1}{2s-w}\left\{  \frac{2s-w}{2}%
\frac{n-2}{2}-\left(  \frac{2s-w}{2}\frac{n}{2}\right)  \left(  \frac{4}%
{p}-1\right)  \right\} \\
& =\frac{n-2}{4}-\frac{n}{4}\left(  \frac{4}{p}-1\right)  =\frac{n-1}{2}%
-\frac{n}{p}=\frac{n-1}{2p}\left(  p-\frac{2n}{n-1}\right)  .
\end{align*}
This completes our proof of (\ref{difficult exp}) in \emph{Step 2}.

\medskip

\subsubsection{Step 3: The Fourier square function argument}

\medskip

Momentarily fix $0\leq w\leq s$, and recall (\ref{difficult exp}),%
\[
\mathbb{E}_{2^{\mathcal{G}_{s}\left[  U\right]  }}^{\mu}\left\Vert T\left(
\mathcal{A}_{\mathbf{a}}\mathsf{Q}_{K}^{s}\right)  ^{\spadesuit}f\right\Vert
_{L^{p}\left(  P_{s,w}^{K}\left[  r\right]  \right)  }\lesssim2^{-r\left(
N-\frac{n}{p^{\prime}}\right)  }2^{-\left(  2s-w\right)  \varepsilon_{p,n}%
}\left\Vert f\right\Vert _{L^{p}\left(  U\right)  },
\]
which in terms of the Fourier square function $\mathcal{S}_{T,s}^{\eta
,K}\equiv\left(  \sum_{I\in\mathcal{G}_{s}\left[  K\right]  }\left\vert
\bigtriangleup_{I;\kappa}^{\eta}f\right\vert ^{2}\right)  ^{\frac{1}{2}}$ is%
\begin{equation}
\left\Vert \mathcal{S}_{T,s}^{\eta,K}f\right\Vert _{L^{p}\left(  P_{s,w}%
^{K}\left[  r\right]  \right)  }\lesssim2^{-r\left(  N-\frac{n}{p^{\prime}%
}\right)  }2^{-\left(  2s-w\right)  \varepsilon_{p,n}}\left\Vert f\right\Vert
_{L^{p}\left(  U\right)  },\label{difficult exp'}%
\end{equation}
For every $K\in\mathcal{G}_{s-w}\left[  U\right]  $, we have
\[
\mathbf{1}_{A_{+}\left(  0,2^{2s-w}\right)  }\lesssim\sum_{r=0}^{s-w}%
\mathbf{1}_{P_{s,w}^{K}\left[  r\right]  }\ ,
\]
and so from $\left\vert \mathcal{S}_{T,s}^{\eta}f\right\vert ^{2}=\sum
_{K\in\mathcal{G}_{s-w}\left[  U\right]  }\left\vert \mathcal{S}_{T,s}%
^{\eta,K}f\right\vert ^{2}$ (where $\mathcal{S}_{T,s}^{\eta}\equiv
\mathcal{S}_{T,s}^{\eta,U}$), we obtain using $\left(  \frac{p}{2}\right)
^{\prime}=\frac{p}{p-2}$ that,%
\begin{align*}
& \left\Vert \mathcal{S}_{T,s}^{\eta}f\right\Vert _{L^{p}\left(  A_{+}\left(
0,2^{2s-w}\right)  \right)  }^{p}=\int\left\vert \mathcal{S}_{T,s}^{\eta
}f\right\vert ^{p}\mathbf{1}_{A_{+}\left(  0,2^{2s-w}\right)  }=\int\left\vert
\mathcal{S}_{T,s}^{\eta}f\right\vert ^{p-2}\sum_{K\in\mathcal{G}_{s-w}\left[
U\right]  }\left\vert \mathcal{S}_{T,s}^{\eta,K}f\right\vert ^{2}%
\mathbf{1}_{A_{+}\left(  0,2^{2s-w}\right)  }\\
& \lesssim\int\left\vert \mathcal{S}_{T,s}^{\eta}f\right\vert ^{p-2}\sum
_{K\in\mathcal{G}_{s-w}\left[  U\right]  }\left\vert \mathcal{S}_{T,s}%
^{\eta,K}f\right\vert ^{2}\sum_{r=0}^{s-w}\mathbf{1}_{P_{s,w}^{K}\left[
r\right]  }=\sum_{r=0}^{s-w}\sum_{K\in\mathcal{G}_{s-w}\left[  U\right]  }%
\int_{P_{s,w}^{K}\left[  r\right]  }\left\vert \mathcal{S}_{T,s}^{\eta
,K}f\right\vert ^{2}\left\vert \mathcal{S}_{T,s}^{\eta}f\right\vert ^{p-2}\\
& \leq\sum_{r=0}^{s-w}\sum_{K\in\mathcal{G}_{s-w}\left[  U\right]  }\left(
\int_{P_{s,w}^{K}\left[  r\right]  }\left\vert \mathcal{S}_{T,s}^{\eta
,K}f\right\vert ^{p}\right)  ^{\frac{2}{p}}\left(  \int_{P_{s,w}^{K}\left[
r\right]  }\left\vert \mathcal{S}_{T,s}^{\eta}f\right\vert ^{p}\right)
^{\frac{p-2}{p}}\\
& \leq\sum_{r=0}^{s-w}\left[  \sum_{K\in\mathcal{G}_{s-w}\left[  U\right]
}\left(  \int_{P_{s,w}^{K}\left[  r\right]  }\left\vert \mathcal{S}%
_{T,s}^{\eta,K}f\right\vert ^{p}\right)  \right]  ^{\frac{2}{p}}\left[
\sum_{K\in\mathcal{G}_{s-w}\left[  U\right]  }\int_{P_{s,w}^{K}\left[
r\right]  }\left\vert \mathcal{S}_{T,s}^{\eta}f\right\vert ^{p}\right]
^{\frac{p-2}{p}}.
\end{align*}

Then from (\ref{difficult exp'}) we obtain%
\begin{align*}
& \left\Vert \mathcal{S}_{T,s}^{\eta}f\right\Vert _{L^{p}\left(  A_{+}\left(
0,2^{2s-w}\right)  \right)  }^{p}\\
& \lesssim\sum_{r=0}^{s-w}\left[  \sum_{K\in\mathcal{G}_{s-w}\left[  S\right]
}2^{-rp\left(  N-\frac{n}{p^{\prime}}\right)  }2^{-\left(  2s-w\right)
p\varepsilon_{p,n}}\left\Vert \left(  \mathsf{Q}_{K}^{s}\right)  ^{\spadesuit
}f\right\Vert _{L^{p}\left(  U\right)  }^{p}\right]  ^{\frac{2}{p}}\left[
2^{rn}\int_{A_{+}\left(  0,2^{2s-w}\right)  }\left\vert \mathcal{S}%
_{T,s}^{\eta}f\right\vert ^{p}\right]  ^{\frac{p-2}{p}}\\
& \lesssim\sum_{r=0}^{s-w}\left[  \sum_{K\in\mathcal{G}_{s-w}\left[  S\right]
}2^{-rp\left(  N-\frac{n}{p^{\prime}}\right)  }2^{rn\frac{p-2}{2}}2^{-\left(
2s-w\right)  p\varepsilon_{p,n}}\left\Vert \left(  \mathsf{Q}_{K}^{s}\right)
^{\spadesuit}f\right\Vert _{L^{p}\left(  U\right)  }^{p}\right]  ^{\frac{2}%
{p}}\left\Vert \mathcal{S}_{T,s}^{\eta}f\right\Vert _{L^{p}\left(
A_{+}\left(  0,2^{2s-w}\right)  \right)  }^{p-2}%
\end{align*}
since the overlap constant of the pipes $\left\{  P_{s,w}^{K}\left[  r\right]
\right\}  _{K\in\mathcal{G}_{s-w}\left[  S\right]  }$ is $C2^{rn}$. Using%
\[
\sum_{K\in\mathcal{G}_{s-w}\left[  S\right]  }\left\Vert \left(
\mathsf{Q}_{K}^{s}\right)  ^{\spadesuit}f\right\Vert _{L^{p}\left(  U\right)
}^{p}=\left\Vert \left(  \mathsf{Q}_{U}^{s}\right)  ^{\spadesuit}f\right\Vert
_{L^{p}\left(  U\right)  }^{p}\ ,
\]
we conclude that%
\begin{align*}
& \left\Vert \mathcal{S}_{T,s}^{\eta}f\right\Vert _{L^{p}\left(  A_{+}\left(
0,2^{2s-w}\right)  \right)  }^{2}\lesssim\sum_{r=0}^{s-w}\left[  2^{-rp\left(
N-\frac{n}{p^{\prime}}+\frac{n}{p}-\frac{n}{2}\right)  }2^{-\left(
2s-w\right)  p\varepsilon_{p,n}}\left\Vert \left(  \mathsf{Q}_{U}^{s}\right)
^{\spadesuit}f\right\Vert _{L^{p}\left(  U\right)  }^{p}\right]  ^{\frac{2}%
{p}}\\
& \lesssim\sum_{r=0}^{s-w}2^{-r2\left(  N-\frac{n}{p^{\prime}}+\frac{n}%
{p}-\frac{n}{2}\right)  }2^{-\left(  2s-w\right)  2\varepsilon_{p,n}%
}\left\Vert \left(  \mathsf{Q}_{U}^{s}\right)  ^{\spadesuit}f\right\Vert
_{L^{p}\left(  U\right)  }^{2}\lesssim2^{-\left(  2s-w\right)  2\varepsilon
_{p,n}}\left\Vert \left(  \mathsf{Q}_{U}^{s}\right)  ^{\spadesuit}f\right\Vert
_{L^{p}\left(  U\right)  }^{2}%
\end{align*}
provided $N>\frac{n}{p^{\prime}}-\frac{n}{p}+\frac{n}{2}$. Thus we have proved
the Fourier square function estimate%
\[
\left\Vert \mathcal{S}_{T,s}^{\eta}f\right\Vert _{L^{p}\left(  A_{+}\left(
0,2^{2s-w}\right)  \right)  }\lesssim2^{-\left(  2s-w\right)  \varepsilon
_{p,n}}\left\Vert \left(  \mathsf{Q}_{U}^{s}\right)  ^{\spadesuit}f\right\Vert
_{L^{p}\left(  U\right)  }\ ,
\]
which is (\ref{prob K rem square}) by Khintchine's inequalities,%
\[
\left\Vert \mathcal{S}_{T,s}^{\eta}f\right\Vert _{L^{p}\left(  A_{+}\left(
0,2^{2s-w}\right)  \right)  }\approx\mathbb{E}_{2^{\mathcal{G}\left[
U\right]  }}^{\mu}\left\Vert T\left(  \mathcal{A}_{\mathbf{a}}\mathsf{Q}%
_{U}^{s}\right)  ^{\spadesuit}f\right\Vert _{L^{p}\left(  A_{+}\left(
0,2^{2s-w}\right)  \right)  }.
\]

\subsection{Wrapup}

We have established the norm expectation,%
\begin{equation}
\mathbb{E}_{2^{\mathcal{G}\left[  U\right]  }}^{\mu}\left\Vert T\left(
\mathcal{A}_{\mathbf{a}}\mathsf{Q}_{U}^{s}\right)  ^{\spadesuit}f\right\Vert
_{L^{p}\left(  B\left(  0,2^{2s}\right)  \right)  }\lesssim2^{-\varepsilon
_{n,p}s}\left\Vert f\right\Vert _{L^{p}\left(  U\right)  }^{p}%
\ ,\;\ \ \ \ \text{for }p>\frac{2n}{n-1},\label{norm exp}%
\end{equation}
which will play a critical role in completing the proof of our main theorem in
the next section.

\section{Completion of the proof of the probabilisitic extension Theorem
\ref{FEC}}

Consider the norm $\left\Vert \widehat{\left(  \left(  \mathcal{A}%
_{\mathbf{a}}\mathsf{Q}_{U}^{s}\right)  ^{\spadesuit}f\right)  _{\Phi,2s}%
}\right\Vert _{L^{p}\left(  \mathbf{1}_{\mathbb{R}^{n}\setminus B\left(
0,2^{2s}\right)  }\lambda_{n}\right)  }$ for each fixed $f\in L^{p}$,
$s\in\mathbb{N}$ and $a\in\mathbf{a}$, and choose $g_{f,s,a}\in L^{p^{\prime}%
}\left(  \lambda_{n}\right)  $ such that%
\begin{align}
\bigtriangleup_{J;\kappa}g_{f,s,a}  & =0\text{ for }J\in\mathcal{D}\left[
B\left(  0,2^{2s}\right)  \right]  ,\label{Haar supp}\\
\left\Vert \widehat{\left(  \left(  \mathcal{A}_{\mathbf{a}}\mathsf{Q}_{U}%
^{s}\right)  ^{\spadesuit}f\right)  _{\Phi,2s}}\right\Vert _{L^{p}\left(
\mathbf{1}_{\mathbb{R}^{n}\setminus B\left(  0,2^{2s}\right)  }\lambda
_{n}\right)  }  & =\left\vert \left\langle T\left(  \left(  \mathcal{A}%
_{\mathbf{a}}\mathsf{Q}_{U}^{s}\right)  ^{\spadesuit}f\right)  _{2s}%
,g_{f,s,a}\right\rangle \right\vert \text{ and }\left\Vert g_{f,s,a}%
\right\Vert _{L^{p^{\prime}}\left(  \lambda_{n}\right)  }=1.\nonumber
\end{align}
Since $\mathsf{B}_{\operatorname*{disjoint}}^{\operatorname*{lower}}\left(
\left(  \mathcal{A}_{\mathbf{a}}\mathsf{Q}_{U}^{s}\right)  ^{\spadesuit
}f,g_{f,s,\mathbf{a}}\right)  $ and $\mathsf{B}_{\operatorname*{distal}%
}^{\operatorname*{lower}}\left(  \left(  \mathcal{A}_{\mathbf{a}}%
\mathsf{Q}_{U}^{s}\right)  ^{\spadesuit}f,g_{f,s,\mathbf{a}}\right)  $ each
vanish by the assumption on the Alpert support of $g_{f,s,a}$ in
(\ref{Haar supp}), and the definitions of the lower disjoint and distal forms,
we have%
\begin{align*}
& \mathbb{E}_{2^{\mathcal{G}\left[  S\right]  }}^{\mu}\left\vert \left\langle
T\left(  \left(  \mathcal{A}_{\mathbf{a}}\mathsf{Q}_{U}^{s}\right)
^{\spadesuit}f\right)  _{2s},g_{f,s,\mathbf{a}}\right\rangle \right\vert
=\mathbb{E}_{2^{\mathcal{G}_{s}\left[  S\right]  }}^{\mu}\left\vert
\left\langle T\mathcal{A}\left(  \left(  \mathcal{A}_{\mathbf{a}}%
\mathsf{Q}_{U}^{s}\right)  ^{\spadesuit}f\right)  _{2s},g_{f,s,\mathbf{a}%
}\right\rangle \right\vert \\
& =\mathbb{E}_{2^{\mathcal{G}_{s}\left[  S\right]  }}^{\mu}\left\vert
\mathsf{B}_{\operatorname*{below}}\left(  T\left(  \left(  \mathcal{A}%
_{\mathbf{a}}\mathsf{Q}_{U}^{s}\right)  ^{\spadesuit}f\right)  _{2s}%
,g_{f,s,\mathbf{a}}\right)  +\mathsf{B}_{\operatorname*{above}}\left(
T\left(  \left(  \mathcal{A}_{\mathbf{a}}\mathsf{Q}_{U}^{s}\right)
^{\spadesuit}f\right)  _{2s},g_{f,s,\mathbf{a}}\right)  \right. \\
& \ \ \ \ \ \ \ \ \ \ \ \ \ \ \ \ \ \ \ \ \ \ \ \ \ \ \ \ \left.
+\mathsf{B}_{\operatorname*{disjoint}}^{\operatorname*{upper}}\left(  T\left(
\left(  \mathcal{A}_{\mathbf{a}}\mathsf{Q}_{U}^{s}\right)  ^{\spadesuit
}f\right)  _{2s},g_{f,s,\mathbf{a}}\right)  +\mathsf{B}%
_{\operatorname*{distal}}^{\operatorname*{upper}}\left(  T\left(  \left(
\mathcal{A}_{\mathbf{a}}\mathsf{Q}_{U}^{s}\right)  ^{\spadesuit}f\right)
_{2s},g_{f,s,\mathbf{a}}\right)  \right\} \\
& \lesssim\sup_{\mathbf{a}}2^{-\varepsilon_{n,p}s}\left\Vert \left(
\mathcal{A}_{\mathbf{a}}\mathsf{Q}_{U}^{s}\right)  ^{\spadesuit}f\right\Vert
_{L^{p}\left(  \mathbb{R}^{n}\right)  }\left\Vert g_{f,s,\mathbf{a}%
}\right\Vert _{L^{p^{\prime}}\left(  \mathbb{R}^{n}\right)  }\ ,
\end{align*}
from estimates proved in previous sections, namely (\ref{strong below}),
(\ref{strong above}), (\ref{strong upper disjoint}) and
(\ref{strong upper distal}). From this and (\ref{norm exp}) we conclude that%
\begin{align*}
& \ \ \ \ \ \ \ \ \ \ \ \ \ \ \mathbb{E}_{2^{\mathcal{G}\left[  S\right]  }%
}^{\mu}\left\Vert T\left(  \mathcal{A}_{\mathbf{a}}\mathsf{Q}_{U}^{s}\right)
^{\spadesuit}f\right\Vert _{L^{p}\left(  \mathbb{R}^{n}\right)  }\\
& \lesssim\mathbb{E}_{2^{\mathcal{G}_{s}\left[  S\right]  }}^{\mu}\left\Vert
\widehat{\left(  \left(  \mathcal{A}_{\mathbf{a}}\mathsf{Q}_{U}^{s}\right)
^{\spadesuit}f\right)  _{\Phi,2s}}\right\Vert _{L^{p}\left(  \mathbf{1}%
_{\mathbb{R}^{n}\setminus B\left(  0,2^{2s}\right)  }\lambda_{n}\right)
}+\mathbb{E}_{2^{\mathcal{G}_{s}\left[  S\right]  }}^{\mu}\left\Vert
\widehat{\left(  \left(  \mathcal{A}_{\mathbf{a}}\mathsf{Q}_{U}^{s}\right)
^{\spadesuit}f\right)  _{\Phi,2s}}\right\Vert _{L^{p}\left(  B\left(
0,2^{2s}\right)  \right)  }\\
& =\mathbb{E}_{2^{\mathcal{G}_{s}\left[  S\right]  }}^{\mu}\left\vert
\left\langle T\left(  \left(  \mathcal{A}_{\mathbf{a}}\mathsf{Q}_{U}%
^{s}\right)  ^{\spadesuit}f\right)  _{2s},g_{f,s,\mathbf{a}}\right\rangle
\right\vert +\mathbb{E}_{2^{\mathcal{G}_{s}\left[  S\right]  }}^{\mu
}\left\Vert \widehat{\left(  \left(  \mathcal{A}_{\mathbf{a}}\mathsf{Q}%
_{U}^{s}\right)  ^{\spadesuit}f\right)  _{\Phi,2s}}\right\Vert _{L^{p}\left(
B\left(  0,2^{2s}\right)  \right)  }\\
& \lesssim\sup_{\mathbf{a}}2^{-\varepsilon_{n,p}s}\left\Vert \left(
\mathcal{A}_{\mathbf{a}}\mathsf{Q}_{U}^{s}\right)  ^{\spadesuit}f\right\Vert
_{L^{p}\left(  \mathbb{R}^{n}\right)  }\left\Vert g_{f,s,\mathbf{a}%
}\right\Vert _{L^{p^{\prime}}\left(  \mathbb{R}^{n}\right)  }+2^{-\varepsilon
_{n,p}s}\left\Vert \left(  \mathsf{Q}_{U}^{s}\right)  ^{\spadesuit
}f\right\Vert _{L^{p}\left(  \mathbb{R}^{n}\right)  }\lesssim2^{-\varepsilon
_{n,p}s}\left\Vert f\right\Vert _{L^{p}\left(  U\right)  }\ ,
\end{align*}
since the multipliers $\left(  \mathcal{A}_{\mathbf{a}}\mathsf{Q}_{U}%
^{s}\right)  ^{\spadesuit}$ and the conjugated projection $\left(
\mathsf{Q}_{U}^{s}\right)  ^{\spadesuit}$ are both bounded on $L^{p}$ by the
Alpert square function estimates (\ref{square}). Finally we have%
\begin{align*}
& \mathbb{E}_{2^{\mathcal{G}\left[  S\right]  }}^{\mu}\left\Vert T\left(
\mathcal{A}_{\mathbf{a}}\mathsf{P}_{U}\right)  ^{\spadesuit}f\right\Vert
_{L^{p}\left(  \mathbb{R}^{n}\right)  }=\mathbb{E}_{2^{\mathcal{G}\left[
S\right]  }}^{\mu}\left\Vert \sum_{s=1}^{\infty}T\left(  \mathcal{A}%
_{\mathbf{a}}\mathsf{Q}_{U}^{s}\right)  ^{\spadesuit}f\right\Vert
_{L^{p}\left(  \mathbb{R}^{n}\right)  }\\
& \leq\sum_{s=1}^{\infty}\mathbb{E}_{2^{\mathcal{G}\left[  S\right]  }}^{\mu
}\left\Vert T_{S}\left(  \mathcal{A}_{\mathbf{a}}\mathsf{Q}_{U}^{s}\right)
^{\spadesuit}f\right\Vert _{L^{p}\left(  \mathbb{R}^{n}\right)  }\leq
\sum_{s=1}^{\infty}2^{-\varepsilon_{n,p}s}\left\Vert f\right\Vert
_{L^{p}\left(  U\right)  }\lesssim\left\Vert f\right\Vert _{L^{p}\left(
U\right)  }\ .
\end{align*}
This completes the proof of (\ref{prob ext''}), and hence that of Theorem
\ref{FEC}.

\section{Concluding remarks}

The two weight testing methods used in this paper might also be applicable to
the following open \emph{probabilistic} problems:

\begin{enumerate}
\item proving a probabilistic analogue of the Bochner-Riesz conjecture or even
the stronger local smoothing conjecture. In the context of the
(nonprobabilistic) extension conjecture, see Sogge \cite{Sog} for a proof that
local smoothing implies Bochner-Riesz, and Tao \cite{Tao1}\ for a proof that
Bochner-Riesz implies Fourier restriction,

\item replacing the sphere in Theorem \ref{FEC} with any smooth surface of
nonvanishing Gaussian curvature, and possibly with appropriate smooth surfaces
of finite type (and with altered indices $p$),

\item replacing the Fourier kernel $e^{-ix\cdot\xi}$ in Theorem \ref{FEC} with
a more general kernel $\Omega\left(  x,\xi\right)  $,

\item to multilinear probabilistic variants of the extension conjecture,

\item deciding the endpoint case $q=p^{\prime}\frac{n+1}{n-1}$ when
$2<p<\frac{2n}{n-1}$ in (\ref{prob ext}),

\item and finally to the much more challenging problem of boundedness of the
maximal spherical partial sum operator in a probabilistic sense.
\end{enumerate}

The main open problem is of course the full \emph{deterministic} Fourier
extension conjecture (\ref{extension}).

\end{document}